\documentclass[11pt, letterpaper, reqno]{amsart}
\usepackage[utf8]{inputenc}
\usepackage[margin=1.4in]{geometry}
\usepackage{enumerate}
\usepackage{amsmath,amssymb}
\usepackage[noadjust]{cite}
\usepackage{amsfonts}
\usepackage{mathrsfs} 
\usepackage{galois}
\usepackage[toc,page]{appendix}
\usepackage{xcolor}
\usepackage[colorlinks, linkcolor = blue, anchorcolor = blue, citecolor = blue]{hyperref}
\usepackage{marginnote}
\usepackage{comment}

\numberwithin{equation}{section}

\newtheorem{Def}{Definition}[section]
\newtheorem{Thm}[Def]{Theorem}

\newtheorem{Lem}[Def]{Lemma}
\newtheorem{Prop}[Def]{Proposition}
\newtheorem{Cor}[Def]{Corollary}
\newtheorem{Cla}[Def]{Claim}

\theoremstyle{remark}
\newtheorem{Rem}[Def]{Remark}

\newcommand{\MM}{\mathbb{M}}

\newcommand{\RR}{\mathbb{R}}
\newcommand{\SSp}{\mathbb{S}}

\newcommand{\ZZ}{\mathbb{Z}}

\newcommand{\cC}{\mathcal{C}}

\newcommand{\cE}{\mathcal{E}}

\newcommand{\cH}{\mathcal{H}}

\newcommand{\cJ}{\mathcal{J}}

\newcommand{\cL}{\mathcal{L}}

\newcommand{\cN}{\mathcal{N}}

\newcommand{\cS}{\mathcal{S}}
\newcommand{\cT}{\mathcal{T}}

\newcommand{\mbfA}{\mathbf{A}}
\newcommand{\mbfB}{\mathbf{B}}

\newcommand{\mbfE}{\mathbf{E}}

\newcommand{\mbfG}{\mathbf{G}}

\newcommand{\mbfO}{\mathbf{O}}

\newcommand{\mbfS}{\mathbf{S}}
\newcommand{\mbfT}{\mathbf{T}}

\newcommand{\mbfV}{\mathbf{V}}

\newcommand{\mbfd}{\mathbf{d}}
\newcommand{\mbfe}{\mathbf{e}}
\newcommand{\mbfp}{\mathbf{p}}
\newcommand{\mbfq}{\mathbf{q}}

\newcommand{\mbfv}{\mathbf{v}}
\newcommand{\mbfw}{\mathbf{w}}
\newcommand{\mbfx}{\mathbf{x}}
\newcommand{\mbfy}{\mathbf{y}}
\newcommand{\mbfz}{\mathbf{z}}
\newcommand{\mbfo}{\mathbf{0}}

\newcommand{\frT}{\mathfrak{T}}

\newcommand{\supp}{\operatorname{spt}}
\newcommand{\spt}{\operatorname{spt}}

\newcommand{\Area}{\operatorname{Area}}

\newcommand{\diverg}{\operatorname{div}}
\newcommand{\Div}{\operatorname{div}}
\newcommand{\graph}{\operatorname{graph}}

\newcommand{\injrad}{\operatorname{injrad}}
\newcommand{\dist}{\operatorname{dist}}
\newcommand{\proj}{\operatorname{proj}}
\newcommand{\reg}{\operatorname{reg}}
\newcommand{\sing}{\operatorname{sing}}

\newcommand{\Rm}{\operatorname{Rm}}
\newcommand{\Id}{\operatorname{id}}
\newcommand{\Jac}{\operatorname{Jac}}
\newcommand{\spine}{\operatorname{spine}}
\newcommand{\Span}{\operatorname{span}}
\newcommand{\Ord}{\operatorname{ord}}

\newcommand{\Sdist}{\overline{\dist}}
\newcommand{\Rot}{\operatorname{Rot}}

\DeclareMathOperator{\Dom}{dom}

\newcommand{\orig}{\mathbf{0}}

\newcommand{\Lip}{\mathrm{Lip}}

\newcommand{\eps}{\varepsilon}

\newcommand{\llbracket}{[\![}
\newcommand{\rrbracket}{]\!]}
\newcommand{\mres}{\lfloor}

\newcommand{\loc}{\mathrm{loc}}

\newcommand{\param}{\kappa}
\newcommand{\Rdist}{\widehat{\dist}}

\title{Generic regularity for minimizing hypersurfaces in dimension 11}
\author{Otis Chodosh} 
\address{OC: Department of Mathematics, Bldg.\ 380, Stanford University, Stanford, CA 94305, USA}
\email{ochodosh@stanford.edu}
\author{Christos Mantoulidis} 
\address{CM: Department of Mathematics, Rice University, Houston, TX 77005, USA}
\email{christos.mantoulidis@rice.edu}
\author{Felix Schulze}
\address{FS: Department of Mathematics, Zeeman Building, University of Warwick, Gibbet Hill Road, Coventry CV4 7AL, UK}
\email{felix.schulze@warwick.ac.uk} 
\author{Zhihan Wang}
\address{ZW: Department of Mathematics, Cornell University, Ithaca, New York 14850, USA}
\email{zw782@cornell.edu}
\date{\today}

\allowdisplaybreaks 

\begin{document}

\begin{abstract}
We prove that area-minimizing hypersurfaces are generically smooth in ambient dimension $11$ in the context of the Plateau problem and of area minimization in integral homology. For higher ambient dimensions, $n+1 \geq 12$, we prove in the same two contexts that area-minimizing hypersurfaces have at most an $n-10-\epsilon_n$ dimensional singular set after an arbitrarily $C^\infty$-small perturbation of the Plateau boundary or the ambient Riemannian metric, respectively.
\end{abstract}

\maketitle

\textit{This article is dedicated to Leon Simon on the upcoming occasion of his 80th birthday for his groundbreaking contributions to the study of singularities in geometric analysis.}

\setcounter{tocdepth}{1}
\tableofcontents

\section{Introduction}

Consider a smooth, closed, oriented $(n-1)$-dimensional submanifold $\Gamma\subset \RR^{n+1}$. We are interested here in Plateau's problem. Among all smooth, compact hypersurfaces $M\subset \RR^{n+1}$ with $\partial M = \Gamma$, we want to find one of least area. It's now well-known that such $M$ always exists for $n+1\leq 7$, while when $n+1\geq 8$ and for certain choices of $\Gamma$, no minimizer $M$ can be found among smooth hypersurfaces. Using geometric measure theory, one can prove the existence of a minimizer among a wider class of objects which are smooth hypersurfaces except perhaps along an $(n-7)$-dimensional singular set. See \cite{Federer:GMT, MassariMiranda,Giusti,Maggi:finite-per,Fleming:plateau, DeGiorgi:bernstein, Almgren:regularity, Simons:minvar, BDG:Simons,Hardt-Simon:boundary-regularity}. 

In $\RR^8$, the first dimension that singularities can appear, a fundamental result of Hardt--Simon \cite{HardtSimon:isolated} shows that for a {generic} choice of Plateau boundary $\Gamma$, there does exist a smooth $M$ minimizing area. An analogous result in $8$-dimensional manifolds was proven by Smale \cite{Smale:generic}. These generic regularity results were recently extended to cover $\RR^9$ and $\RR^{10}$ in \cite{CMS:generic.9.10} using new ideas from the works of the first three authors with K.\ Choi on generic mean curvature flows, and specifically \cite{CCMS:gen.1, CCMS:gen.low.ent.1}. 

In this paper we prove that solutions to Plateau's problem in $\RR^{11}$ are generically smooth. We also prove that in any $\RR^{n+1}$, an area-minimizing $M$ will have a $\leq n-10-\epsilon_n$ dimensional singular set after perhaps a $C^\infty$-perturbation of the Plateau boundary. See Theorem \ref{theo:main.rn.11}. We had previously obtained the upper bound $\leq n - 9 - \epsilon_n'$ in \cite{CMS:generic.codim.plus.2}.

We also prove the analogous results in the context of area-minimization in integral homology classes of a closed oriented manifold $(N^{n+1}, g)$. See Theorem \ref{theo:main.homology.11}. As is well-known, this extends Schoen--Yau's stable minimal hypersurface obstruction to positive scalar curvature up to dimension $11$ and also implies the positive mass theorem in these dimensions after a well-known reduction of Lohkamp. See also the work of Schoen--Yau and Lohkamp \cite{SY:sing,Lohkamp:PSC}.

\subsection{Non-technical overview of new ideas} Before giving the precise statement of the results, we describe the new contributions in this paper as compared to \cite{CMS:generic.9.10,CMS:generic.codim.plus.2}. Recall that the idea there was to study and exploit the interaction of singular points in a ``foliation'' (possibly with gaps) of area-minimizers coming from variations of the Plateau boundary or Riemannian metric. There were two main ingredients. (i) A cone-splitting argument implied that the singular set of the entire foliation obeys the same $(n-7)$-dimensional Hausdorff estimate as a single leaf does. (ii) Analyzing positive Jacobi fields on minimizing hypercones (following L.\ Simon \cite{Simon:decay}) implied that distinct leaves of the foliation obey a quantitative separation estimate at small scales. Using a Sard-type theorem and a stratification argument, the rate of quantitative separation in (ii) is sufficient to overcome the $(n-7)$-dimensional singular set in (i) when $n+1\in\{8,9,10\}$. 

Our starting point here is the observation that for each possible spine dimension $k$ of a tangent hypercone, i.e., $\cC = \cC_\circ \times \RR^k$ with $\cC_\circ \subset \RR^{n+1-k}$, the singular points \emph{least} amenable to being perturbed away by our previous strategy are {product cylinders} $\cC$ over {quadratic hypercones} (generalizations of Simons cones) $\cC_\circ \subset \RR^{n+1-k}$. 

We proceed to closely analyze the behavior of the minimizers near such singularities.\footnote{In fact, we consider a broader class of cylindrical hypercones. See Remark \ref{rema:intro.quadratic.generalized}.} We go one blow-up level deeper and model the minimizer locally as a graph of a Jacobi field over such a cylindrical hypercone. This requires the use of {non-concentration estimates} pioneered by L.\ Simon in \cite{Simon:cylindrical-sing,Simon:uniqueness.some}.\footnote{Simon's non-concentration estimates have been extremely influential in regularity theory for minimal hypersurfaces; we refer to the works \cite{Wickramasekera:stable,CES:tetrahedron,MW:structure,Szkelyhidi:unique,Szkelyhidi:cylindrical,EdelenSzekelyhidi:liouville}.} Crucial in our analysis is the rate of decay toward the origin of these graph Jacobi fields. It is measured by a novel discrete {Almgren frequency} with a suitable monotonicity.\footnote{Almgren's frequency monotonicity has also played a key role in regularity theory of higher codimension minimizing surfaces; we refer to the works \cite{Almgren:book, DeLellisSpadaro:iii, KrummelWickramasekera:i, KrummelWickramasekera:ii, DeLellisSkorobogatova:i, DeLellisSkorobogatova:ii, DeLellisMinterSkorobogatova:iii}.} (See Definition \ref{defi:doubling.constant.nonlinear} for the definition.) The decay is super-linear (``fast'') or linear (``slow''). These two cases bootstrap our previous genericity strategy in different ways. The case of fast decay (which, a posteriori, implies tangent hypercone uniqueness) improves the separation estimate in (ii). The case of slow decay does not, but new cone-splitting arguments at the Jacobi field level (still exploiting the pairwise disjointness of two such graphical models in $\RR^{n+1}$) provide a second-order refinement of (i) in the previous strategy.

As we explain in Section \ref{subsec:strat-proof}, our previous strategy in $\RR^9$, $\RR^{10}$ was borderline in $\RR^{11}$ for the ``least non-generic'' hypercones in $\RR^{11}$ of largest spine, i.e., $\cC_\circ \times \RR^3$, with $\cC_\circ \subset \RR^8$ any quadratic hypercone (e.g., a Simons cone). Our refined second order analysis near generalizations of this setting provides enough of an improvement to imply generic regularity in $\RR^{11}$. The improved dimension of the singular set in all dimensions $\geq 12$ follows similarly.

Let us end with a brief remark concerning this work in the context of the generic regularity theorem of Figalli--Ros-Oton--Serra for the obstacle problem \cite{FROS:ihes}. Already in \cite{CMS:generic.9.10} it was observed that the strategy of combining an estimate for the singular set of a ``foliation'' with a ``separation'' estimate was a posteriori similar to the strategy used in \cite{FROS:ihes}. The nature of the area minimization problem is such that \cite{CMS:generic.9.10} had to approach these estimates very differently than \cite{FROS:ihes}. In this paper, we obtain improved estimates over \cite{CMS:generic.9.10} by further analyzing higher-order ``blowup'' information (here, the Jacobi field). This remains in the spirit of \cite{FROS:ihes} (see also \cite{FigalliSerra}) but, again, there are important differences in this problem. A core issue is that our Almgren-type frequency monotonicity quantity involves geometric data, but its monotonicity is at the linear (``blow up'') level. Going between the two requires wrestling with the possibility of tangent cone non-uniqueness, the lack of an equation for the singular set, and with non-concentration estimates. We must also treat curved background metrics, and curvature interferes with second-order analysis, particularly non-concentration estimates. These facts create considerable technical complications. On the other hand, we do not carry out nearly as refined of an analysis of second-order blowups; cf. Remark \ref{rem:R12}.

\subsection{The main theorems}
We now describe the results of this paper more precisely.  We obtain here the following generalizations of the main result of \cite{CMS:generic.9.10, CMS:generic.codim.plus.2}. All submanifolds in this paper are embedded, and if $\Sigma$ is such, we denote $\sing \Sigma = \bar \Sigma \setminus \Sigma$. All singular set dimensions are Hausdorff dimensions. 

For the Plateau problem in $\RR^{n+1}$ we have:
\begin{Thm} \label{theo:main.rn.11}
	Consider a smooth, closed, oriented, $(n-1$)-dimensional submanifold $\Gamma \subset \RR^{n+1}$. There exist $C^\infty$-small perturbations $\Gamma'$ of $\Gamma$ (in the space of $C^\infty$ submanifolds) such that every minimizing integral $n$-current with boundary $\llbracket \Gamma' \rrbracket$ is of the form $\llbracket \Sigma' \rrbracket$ for a smooth, precompact, oriented hypersurface $\Sigma' \subset \RR^{n+1}$ with $\partial \Sigma' = \Gamma'$, and
	\[ \sing \Sigma' = \emptyset \text{ if } n+1 \leq 11, \text{ else } \dim \sing \Sigma' \leq n-10-\epsilon_n, \]
	where $\epsilon_n > 0$ is a dimensional constant.
\end{Thm}
For the homological Plateau problem in a manifold we have:
\begin{Thm} \label{theo:main.homology.11}
	Consider a closed, oriented, $(n+1$)-dimensional Riemannian manifold $(N, g)$. Let $[\alpha] \in H_n(N, \ZZ) \setminus \{ [0] \}$. There exist $C^\infty$-small perturbations $g'$ of $g$ such that every $g'$-minimizing integral $n$-current in $[\alpha]$ is of the form $\sum_{i=1}^Q k'_i \llbracket \Sigma'_i \rrbracket$ for disjoint, smooth, precompact, oriented hypersurfaces $\Sigma'_1, \ldots, \Sigma'_Q \subset N$ without boundary and
	\[ \sing \Sigma'_i = \emptyset \text{ if } n+1 \leq 11, \text{ else } \dim \sing \Sigma'_i \leq n-10-\epsilon_n, \]
	and multiplicities $k'_1, \ldots, k'_Q \in \ZZ$; again, $\epsilon_n > 0$ is a dimensional constant.
\end{Thm}

The respective proofs are given in Sections \ref{sec:plateau.rn} and \ref{sec:homology.min}.

\begin{Rem}
    It is a well-known consequence of Allard's interior regularity theorem \cite{Allard:first-variation} and Hardt--Simon's boundary regularity theorem \cite{Hardt-Simon:boundary-regularity} that $\sing \Sigma' = \emptyset$ is an open condition in such a multiplicity-one setting. Therefore, when $n+1 \leq 11$, the set of $\Gamma'$, $g'$ for which the corresponding minimizers are smooth objects is simultaneously open (by this observation) and dense (by Theorems \ref{theo:main.rn.11}, \ref{theo:main.homology.11}), and thus Baire generic.
\end{Rem}

\subsection{The technical ingredient}

Theorems \ref{theo:main.rn.11} and \ref{theo:main.homology.11} will be consequences (after some appropriate geometric considerations) of the following main result.\footnote{The geometric construction in the homological setting is improved here relative to \cite{CMS:generic.9.10}, so we can work here with a genuine ``foliation'' rather than ``sufficiently dense set of leaves.'' This relieves some complications in this part of \cite{CMS:generic.9.10}.} 

\begin{Thm} \label{theo:small.sing.dim}
	Fix an $(n+1)$-dimensional Riemannian manifold without boundary $(M, g)$, $n+1 \geq 8$. Let $\mathscr{F}$ be a set of minimizing boundaries in $M$. Denote:
	\begin{align*}
		\supp \mathscr{F} & := \cup_{T \in \mathscr{F}} \supp T, \\
		\sing \mathscr{F} & := \cup_{T \in \mathscr{F}} \sing T.
	\end{align*}
    Assume the properties:
	\begin{enumerate}
		\item[(1)] $\sing \mathscr{F} \Subset M$,
		\item[(2)] elements of $\mathscr{F}$ have pairwise disjoint supports.
	\end{enumerate}
    Moreover, let $\frT : \supp \mathscr{F} \to \RR$ be such that:
	\begin{enumerate}
		\item[(3)] $\frT$ is constant on each $\supp T$, $T \in \mathscr{F}$, and
		\item[(4)] $\frT$ is Lipschitz continuous with the induced distance metric from $(M, g)$.
	\end{enumerate}
	Then we have the following Hausdorff dimension bounds:
	\begin{enumerate}
		\item[(i)] $\dim \frT(\sing \mathscr{F}) \leq \min \{ \mathfrak{d}_{n}^{\textnormal{img}}, 1 \}$, and
		\item[(ii)] For a.e. $t \in \RR$, every $\Sigma \in \mathscr{F}$ with $\frT = t$ on $\supp \Sigma$ has $\dim \sing \Sigma \leq \max \{ \mathfrak{d}_{n}^{\textnormal{dom}}, 0 \}$.
	\end{enumerate}
	Above, $\mathfrak{d}_{n}^{\textnormal{img}}$, $\mathfrak{d}_{n}^{\textnormal{dom}}$ are:
	\begin{equation} \label{eq:delta.n}
		\mathfrak{d}_{n}^{\textnormal{img}} := \max_{k = 0, 1, \ldots, n-7} \left\{ \tfrac{k}{1+ \alpha_{n-k} + \min \{ \Delta^{\textnormal{non-qd}}_{n-k}, \Delta^{\textnormal{qd}}_{n-k} \} }, \tfrac{k-1}{1 + \alpha_{n-k}} \right\},
	\end{equation}
	\begin{equation} \label{eq:cee.n}
		\mathfrak{d}_{n}^{\textnormal{dom}} := \max_{k = 0, 1, \ldots, n-7} \left\{ k - (1 + \alpha_{n-k} + \min \{ \Delta^{\textnormal{non-qd}}_{n-k},  \Delta^{\textnormal{qd}}_{n-k} \} ) \right \},
	\end{equation}
	where
	\begin{itemize}
		\item $\alpha_{n-k} \in (1, 2]$ is as in Lemma \ref{lemm:kappa.zhu},
		\item $\Delta^{\textnormal{non-qd}}_{n-k} \in (0, \infty]$ is as in Lemma \ref{lemm:kappa.zhu.nonqd},
		\item $\Delta^{\textnormal{qd}}_{n-k} \in (0, 1]$ is as in \eqref{eq:delta.hf.gap}.
	\end{itemize}
\end{Thm}

\begin{Rem}
    It will appear at first sight that \eqref{eq:cee.n} takes a maximum over fewer terms than \eqref{eq:delta.n}. Specifically, it seems as if we are neglecting to include a term of the form
    \[ k-1-(1+\alpha_{n-k}) \]
    in \eqref{eq:cee.n}. However, accounting for such a term is unnecessary since the aforementioned fact that $\Delta^{\textnormal{qd}}_{n-k} \in (0, 1]$ guarantees
    \begin{align*}
        k-1-(1+\alpha_{n-k}) 
            & \leq k-(1 + \alpha_{n-k} + \Delta^{\textnormal{qd}}_{n-k}) \\
            & \leq k-(1 + \alpha_{n-k} + \min \{ \Delta^{\textnormal{non-qd}}_{n-k}, \Delta^{\textnormal{qd}}_{n-k} \}).
    \end{align*} 
\end{Rem}

It is not hard to see by tracing out the algebraic inequalities involved in \eqref{eq:delta.n}, \eqref{eq:cee.n} that 
\[ \mathfrak{d}_{n}^{\textnormal{img}} < 1 \iff \mathfrak{d}_{n}^{\textnormal{dom}} < 0 \iff n+1 \leq 11. \]
Tracing these inequalities carefully in fact yields:

\begin{Cor} \label{coro:small.sing.dim}
	Assume the setting of Theorem \ref{theo:small.sing.dim}. 
	\begin{enumerate}
		\item[(i)] Assume $n+1 \leq 11$. For a.e. $t \in \RR$,
			\[ \Sigma \in \mathscr{F}, \; \frT = t \text{ on } \supp \Sigma \implies \sing \Sigma = \emptyset. \]
		\item[(ii)] Assume $n+1 \geq 12$. There exists $\epsilon_n > 0$ such that for a.e. $t \in \RR$, 
			\[ \Sigma \in \mathscr{F}, \; \frT = t \text{ on } \supp \Sigma \implies \dim \sing \Sigma \leq n-10-\epsilon_n. \]
	\end{enumerate}
\end{Cor}

\subsection{Strategy of proof} \label{subsec:strat-proof} Let us briefly compare this result and its strategy to that of the previous results in  \cite{CMS:generic.9.10, CMS:generic.codim.plus.2, CMS:generic.mcf.high.dim}.\footnote{The proof in \cite{CMS:generic.mcf.high.dim} is most similar, as it required (the parabolic analog of) $\dim \spine \cC$-aware estimates for $\alpha(\cC)$, unlike those in \cite{CMS:generic.9.10, CMS:generic.codim.plus.2}.} The strategy in our previous works involved modifications of the following main theme:
\begin{enumerate}
	\item[(i)] Use cone-splitting arguments to show that, near points of $\sing \mathscr{F}$ suitably modeled by a hypercone $\cC$, we have
		\[ \sing \mathscr{F} \text{ is locally (coarsely)} \leq (\dim \spine \cC)\text{-dimensional}. \]
		 See also Proposition \ref{prop:conical.structure} (ii).
	\item[(ii)] Use positive Jacobi field decay estimates (originally due to Simon \cite{Simon:decay}) to show, near points of $\cS$ modeled by a hypercone $\cC$, that
		\[ \frT \text{ is locally (coarsely) } \mathfrak{h}\text{-H\"older} \text{ for all } \mathfrak{h} < 1 + \alpha(\cC), \]
		with $\alpha(\cC)$ as in Definition \ref{defi:kappa.lambda}. See also Proposition \ref{prop:conical.structure} (iii) and Lemma \ref{lemm:conical.structure.iter}.
	\item[(iii)] Deduce using a covering argument such as \cite[Corollary 7.8]{FROS:ihes} that
		\[ \dim \frT(\sing \mathscr{F}) \leq \sup_{\cC} \tfrac{\dim \spine \cC}{1 + \alpha(\cC)}, \]
		and that, for a.e. $t \in \RR$, every $\Sigma \in \mathscr{F}$ with $\frT(\supp \Sigma) = t$ also satisfies
		\[ \dim \sing \Sigma \leq \sup_{\cC} (\dim \spine \cC - (1 + \alpha(\cC))). \]
		The supremum is over all nonflat minimizing hypercones $\cC \subset \RR^{n+1}$. See also Proposition \ref{prop:covering}.
\end{enumerate}

This strategy establishes full generic regularity for minimizers up to ambient dimension $n+1 \leq 10$. That is because, in this case, one has $\inf_\cC (1 + \alpha(\cC)) > 2$ (see Lemma \ref{lemm:kappa.zhu}) and $\dim \spine \cC \leq n-7 \leq 2$, so
\[ \sup_{\cC} \tfrac{\dim \spine \cC}{1 + \alpha(\cC)} < 1. \]
This strategy will no longer suffice in ambient dimension $n+1=11$. Indeed, when
\[ \cC \in \Rot(\cC_\circ \times \RR^3), \; \cC_\circ \subset \RR^8 \text{ any minimizing quadratic hypercone}, \]
(e.g., $\cC_\circ$ could be a Simons cone) then Lemma \ref{lemm:kappa.zhu} and Proposition \ref{prop:kappa.spine} predict  that
\[ \tfrac{\dim \spine \cC}{1 + \alpha(\cC)} = \tfrac{3}{3} = 1, \]
which is insufficient for showing $\frT(\sing \mathscr{F}) \subset \RR$ has measure zero. See Appendix \ref{app:quadratic} for some background on quadratic hypercones.

In this paper we delve into a refined analysis of this ``borderline'' case, i.e., the case where $\sing \mathscr{F}$ is locally suitably modeled by a hypercone
\[ \cC \subset \RR^{n+1}, \; \cC \in \Rot(\cC_\circ \times \RR^k), \]
where
\[ \cC_\circ \subset \RR^{n+1-k} \text{ is any minimizing quadratic hypercone}. \]
We show a dichotomy in this situation. One can either improve the bound $\leq \dim \spine \cC$ in (i) or prove H\"older-regularity beyond $1 + \alpha(\cC)$ in (ii). 

More specifically, we make use of the fact that Jacobi fields on such $\cC$ which can be used to model nearby minimizing hypersurfaces obey an important \emph{gap} in their degrees of homogeneity: either the degree equals $1$, or it is $\geq 1 + \Delta^{>1}_\cC$ for some $\Delta^{>1}_\cC > 0$ (see Remark \ref{rema:gamma.star.c}). These cases behave differently from each other in a qualitative manner:

\begin{enumerate}
	\item Near points modeled by $\cC$ with degree $\geq 1 + \Delta^{>1}_\cC$, we can improve the (coarse) H\"older regularity in (ii) from $1 + \alpha(\cC)$ to $1 + \alpha(\cC) + \Delta^{>1}_\cC$. See Sections \ref{sec:nonlinear.hf}, \ref{sec:hf.qd.points}.
	\item Near points modeled by $\cC$ with degree equal to $1$, we can improve the (coarse) dimension bound in (i) from $\dim \spine \cC$ to $\dim \spine \cC - 1$. See Sections \ref{sec:nonlinear.lf}, \ref{sec:lf.nonqd.points}.
\end{enumerate}

\begin{Rem} \label{rema:intro.quadratic.generalized}
	Our technical ingredients in fact apply more generally to all 
	\[ \cC \subset \RR^{n+1}, \; \cC \in \Rot(\cC_\circ \times \RR^k), \]
	where $\cC_\circ$ is any hypercone which is:
	\begin{itemize}
		\item regular (i.e., $\sing \cC_\circ = \{ 0 \}$),
		\item strictly stable (Definition \ref{defi:strictly.stable}), 
		\item strictly minimizing (Definition \ref{def:strictly-minimizing}), and 
		\item strongly integrable (Definition \ref{defi:strongly.integrable}). 
	\end{itemize}
	All minimizing quadratic hypercones $\cC_\circ$ satisfy these properties. See Appendix \ref{app:quadratic}.
\end{Rem}

\begin{Rem}\label{rem:R12}
The classification of area-minimizing hypercones in $\RR^8$ is a well-known open problem. It is not known if there exist non-strictly stable, non-strictly minimizing, or non strongly integrable hypercones in $\RR^8$ (or higher dimensions). Such a question is likely to present a major issue in any study of generic regularity in $\RR^{12}$ and beyond.
\end{Rem}

\subsection{An explanation in the slow-decay case: computing the homogeneous degree-one Jacobi fields on cylinders over Simons cones}
Let $\cC_\circ = C(\mathbb{S}^3\times \mathbb{S}^3) \subset \RR^8$ be the Simons cone. We briefly describe the homogeneous degree-one Jacobi fields on $\cC_\circ \times \RR^k$ and indicate how the improved ``effective spine'' estimate (see Section \ref{subsec:strat-proof}) can be seen at the linear level. 

The homogeneous Jacobi fields $w$ on $\cC_\circ$ (cf.\ Definition \ref{def:gamma.c}) satisfy $w(r\omega) = r^\gamma \psi(\omega)$ for $\psi$ an eigenfunction of $-(\Delta_{\cL_\circ} + |A_{\cL_\circ}|^2)$ with corresponding eigenvalue $\gamma^2 + 5 \gamma$ where $\cL_\circ =\cC_\circ \cap \mathbb{S}^7$. Using spectral theory on $\mathbb{S}^3$, it's easy to compute (cf.\ \cite{SimonSolomon:quadratic}) the small eigenvalues of $-(\Delta_{\cL_\circ} + |A_{\cL_\circ}|^2)$ to be $-6$ (with constant eigenfunction), $0$ (corresponding to translations of the cone) and $6$ (corresponding to rotations). The associated degrees of homogeneous eigenfunctions on $\cC_\circ$ are (respectively) $\gamma = -2,0$ and $1$. 

Suppose that $u$ is an appropriate (cf.\ Definition \ref{defi:cylindrical.energies}) Jacobi field on $\cC_\circ \times \RR^k$ that's homogeneous of degree one. (Due to Simon's non-concentration result, this is exactly the  model for minimal surfaces with a $\cC_\circ \times \RR^k$  tangent cone that are exhibiting ``slow decay.'') Using coordinates $(r,\omega,y)$ on $\cC_\circ \times \RR^k$ where $r\omega \in\cC_\circ$ is a ``polar coordinate'' and $y \in \RR^k$ is the direction along the spine $\{0\}\times \RR^k$, work of Simon shows that $u$ can be decomposed as $u=u_{-2}+u_0 + u_1$ where 
\[
u_\gamma(r,\omega,y) = r^\gamma \psi(\omega) p_{1-\gamma}(r,y)
\]
for $\psi$ an eigenfunction of $-(\Delta_{\cL_\circ} + |A_{\cL_\circ}|^2)$ with eigenvalue corresponding to $\gamma$ as above and $p_{1-\gamma}(r,y)$ a polynomial in $(r,y)$ with only even powers of $r$ that's homogeneous of degree  $1-\gamma$ satisfying the ``$\beta$-harmonic equation''
\[
\left( \partial^2_r + \frac{1+\beta}{r} \partial_r + \Delta_y\right) p_{1-\gamma} = 0
\]
for $\beta = 2\gamma+5$ (cf.\ Lemma \ref{Lem_spect proj Jac are beta harmonic}). By adjusting by an appropriate rotation, it's possible (when $\cC_\circ$ is strongly integrable, as is the case for the Simons cone; cf.\ Definition \ref{defi:strongly.integrable}) to ensure that\footnote{To clarify this, in the strongly integrable case, for $\gamma=1$, we see that $p_0$ is constant so $u_1$ corresponds to a rotation fixing the spine $\{0\}\times \RR^k$. Similarly, for $\gamma=0$, $p_1$ is a linear function in $y$ (it cannot contain $r$ since it only contains powers of $r^2$), so this corresponds to a rotation fixing $\cC_\circ$ but rotating the spine.   } $u_0=u_1=0$. It thus remains to consider $\gamma=-2$. In this case, $\psi(\omega) = 1$ (up to a constant) and $\beta = 1$. We need to determine all $1$-harmonic polynomials $p_3(r, y)$ in $(r,y)$ that are homogeneous of degree $3$ and contain only powers of $r^2$. 

When $k=1$, i.e., when $\cC=\cC_\circ\times \RR$, there's a one-dimensional space of $p_3$, spanned by $y^3-r^2 y$ by \cite[Section 3]{Simon:liouville}. As such, if a minimal hypersurface $\Sigma$ in $\RR^9$ has a tangent cone $\cC_\circ\times \RR$, and exhibits slow decay at some scale, then we'll see that (after a slight rotation), the hypersurface is more closely modeled on a Jacobi field of the form
\[
u(r,\omega,y) = c (r^{-2}y^3-y). 
\]
for some $c\neq 0$. In particular, at any non-zero point $(0,0) \neq (0,y) \in \{0\}\times \RR $ in the spine of $\cC$ we see $r \mapsto u(r,\omega,y)$ is growing at rate $r^{-2}$ as $r\to 0$. On the other hand, if $\Sigma$ is singular (and modeled on $\cC$ with slow decay) near $(0,y)$, this decay rate must be linear (or faster). In sum, the traditional analysis of the singular set of $\Sigma$ near the origin based on the singular set of the cone $\cC$ would suggest that there could be a one-dimensional singular set near the origin. However, in this slow-decay case this estimate can be improved to show that there are no nearby singularities modeled on $\cC$ (i.e.\ this part of the singular set is $0$-dimensional). 

We have ignored several crucial issues here. First of all, in the actual analysis, we need to consider $\cC_\circ\times \RR^k$ for $k\geq 3$. In this case, there are multiple possibilities for $p_{3}$. For example, when $k=3$ one possibility is the $1$-dimensional solution extended to trivially to $\RR^3$:
\[
p_3(r,y_1,y_2,y_3) =  y_1^3 -  r^2 y_1
\]
(see \cite[Section 3]{Simon:liouville} for the general solution) in which the Jacobi field would be
\[
u(r,\omega,y_1,y_2,y_3) = c(r^{-2}(y_1)^3-y_1)
\]
which has linear decay rate as $r\to0$ precisely on the $2$ dimensional subspace $\{y_1  =0\}\subset \{0\}\times \RR^3$. Thus, if $\Sigma$  has a tangent cone $\cC=\cC_\circ\times \RR^3$ and is modeled to higher order by this Jacobi field, we expect that the singularities of $\Sigma$ that are also modeled on $\cC$ with slow decay are concentrated near $\{y_1=0\}$. The reader may check that less trivial possibilties for $p_3$ such as $p_3(r,y_1,y_2,y_3) = 3 (y_1)^2 y_2 - r^2 y_2$ have similar behavior (but the resulting ``effective spine'' has  even smaller dimension). 

Secondly, we have omitted several essential additional geometric considerations, the most important of which are as follows. We need to allow for the case that  $u$  also contains a translation term, corresponding to a singularity near---but not exactly on---the spine of the cone. Moreover, we need to extend the above analysis to the case of two disjoint minimizers $\Sigma,\Sigma^+$. We describe here how to handle the second issue. In particular, we prove that if $\Sigma$ has slow-decay, and thus is modeled on $\cC = \cC_\circ\times \RR^k$ and homogeneous degree-one Jacobi field $u$ as above, then under appropriate conditions, a disjoint minimizer $\Sigma^+$ is modeled on $\cC$ and a Jacobi field in the span of $u$, a translation, and a positive (corresponding to $\Sigma,\Sigma^+$ being disjoint) Jacobi field. Using Edelen--Sz\'ekelyhidi's \cite{EdelenSzekelyhidi:liouville} Liouville theorem for positive Jacobi fields (with appropriate decay) on $\cC$ we find that this positive Jacobi field is equal to a multiple of $r^{-2}$ (cf.\ Lemma \ref{Lem_Jac field_Jac_(trl, rot, pos)}). The key observation in this case is that adding $r^{-2}$ (or translations) cannot sufficiently cancel the decay of $u$ to create linear decay at some new points on the spine of $\cC$. For example, when $k=1$ we have (in the case where there's only a positive Jacobi field and no translation)
\[
u(r,\omega,y)+dr^{-2} = (cy^3+d)r^{-2} - c y,
\]
so even if $cy^3+d=0$ at some $y\in\RR$, the decay of the remaining term $-cy$ will be $O(1)$ rather than $O(r)$ as $r\to 0$ at such a point. Translation terms can be handled in an analogous manner. 

\subsection{Acknowledgements}
We are grateful to Nick Edelen, Xavier Ros-Oton, Rick Schoen, Joaquim Serra, Leon Simon, Ao Sun, G\'abor Sz\'ekelyhidi and Jinxin Xue for their encouragement and interest. O.C. was supported by a Terman Fellowship and an NSF grant
(DMS-2304432). C.M. was supported by an NSF grant (DMS-2403728).

\section{Preliminaries and notation}

In addition to standard notation in geometric measure theory (cf.\ \cite{Simon:GMT}) we adopt the following notation:
\begin{align}
	\mathscr{C}_n & := \{ \text{nonflat mininimizing hypercones } \cC^n \subset \RR^{n+1} \}, \label{eq:mathscr.n} \\
	\mathscr{C}_n^{\textnormal{qd}} & := \{ \text{quadratic } \cC^n \subset \RR^{n+1} \}, \label{eq:mathscr.n.qd} \\
        \mathscr{C}_n^{\textnormal{qd-cyl}} & := \cup_{k=0,1,\ldots,n-7} \{ \cC \in \mathscr{C}_n : \cC \in \Rot(\cC_\circ \times \RR^k), \; \cC_\circ \in \mathscr{C}_{n-k}^{\textnormal{qd}} \}. \label{eq:mathscr.n.qd-cyl}
\end{align}
For a subset $S\subset \RR^n$ we write $\Rot(S) : = \{O(S) : O \in SO(n)\}$. 

When it's unlikely to cause confusion we will conflate peremeter-minimizing Caccioppoli sets $E \subset (M,g)$ with their reduced boundary $\Sigma = \partial^*E$. We then set $\sing \Sigma = \bar\Sigma\setminus \Sigma$. For $p \in \bar \Sigma$ we define the \emph{regularity scale} $r_\Sigma(p)$ to be the supremum of $r \in (0,\textrm{inj}_{(M,g)}(x))$ so that $\bar\Sigma \cap B_r(p)$ is a smooth hypersurface with second fundamental form $|A|\leq r^{-1}$. As in \cite[Lemma 2.4]{CMS:generic.9.10} the regularity scale is continuous with respect to weak convergence $\Sigma$ and convergence of $x$. We also define the relatively open sets
\[
\reg_{>\rho}\Sigma : = \{p \in \Sigma : r_\Sigma(p) > \rho\} \subset \Sigma. 
\]

\section{Jacobi fields on cylindrical hypercones}

\subsection{General notions}

Let $\cC \subset \RR^{n+1}$ a minimal hypercone, not necessarily regular.\footnote{In our subsequent geometric applications, we will take $\cC$ to be $\cC = \cC_\circ \times \RR^k$, with $\cC_\circ$ some minimizing quadratic hypercone (see Appendix \ref{app:quadratic}).} 

Recall that we adopt the convention that $\cC$ coincides with its regular part, and can be viewed as an $n$-dimensional submanifold of $\RR^{n+1}$. In particular, the covariant derivative $\nabla_\cC$ and Laplace--Beltrami operator $\Delta_\cC$ and the second fundamental form $A_\cC$ are well-defined on $\cC$.

\begin{Def} \label{defi:jf}
We denote the space of Jacobi fields on $\cC$ by 
  \[
  \Jac(\cC): = \{w \in C^2_\loc(\cC) : \Delta_\cC u + |A_{\cC}|^2 u = 0\}. 
  \]
We will also be interested in locally defined Jacobi fields within some open $U\subset \RR^{n+1}$:
  \[
  \Jac(\cC \cap U): = \{w \in C^2_\loc(\cC\cap U) : \Delta_\cC u + |A_{\cC}|^2 u = 0\}. 
  \]
\end{Def}

Certain canonical Jacobi fields, which are induced by ambient isometries, will play an important role in our analysis. Fix a global unit normal $\nu_\cC$ defined on $\cC$.

\begin{Def} \label{defi:translation.jf}
  Any fixed $\mbfv \in \RR^{n+1}$ induces a \textbf{translation Jacobi field},
  \[ W_\mbfv(x) : = \langle \mbfv, \nu_\cC(x) \rangle, \; x \in \cC. \]
  The space of such Jacobi fields will be denoted
  \[
  \Jac_\textnormal{trl}(\cC) : =\{W_\mbfv : \mbfv \in \RR^{n+1}\}. 
  \]
\end{Def}

\begin{Def} \label{defi:rotation.jf}
Any fixed $\mbfA \in \mathfrak{so}(n+1)$ induces a \textbf{rotation Jacobi field}
 \[ W_\mbfA(x) : = \langle \mbfA x, \nu_\cC(x) \rangle, \; x \in \cC. \]
 The space of such Jacobi fields will be denoted
  \[
  \Jac_\textnormal{rot}(\cC) : = \{W_\mbfA : \mbfA \in \mathfrak{so}(n+1)\}. 
  \]
\end{Def}

In performing our spectral analysis of Jacobi fields on suitable minimal hypercones, we will be especially interested in homogeneous Jacobi fields:

  \begin{Def}
  A Jacobi field $w\in \Jac(\cC)$ is \textbf{homogeneous of degree $\gamma \in \RR$} if
  \[ w(tx) = t^\gamma w(x) \]
  for every $t>0$ and $x \in \cC$. The space of such Jacobi fields will be denoted
  \[ \Jac_\gamma(\cC) = \{ w \in \Jac(\cC) \text{ is homogeneous of degree $\gamma$} \}. \]
\end{Def}

Clearly:
\begin{align} 
	\Jac_\textnormal{trl}(\cC) \subset \Jac_0(\cC), \label{eq:translation.inclusion} \\
	\Jac_\textnormal{rot}(\cC) \subset \Jac_1(\cC). \label{eq:rotation.inclusion}
\end{align}

\subsection{Jacobi fields on regular minimal hypercones} \label{Subsec_reg Min Cone}
Fix a nonflat minimal hypercone
\[ \cC_\circ\subset\RR^{n_\circ+1}, \]
that is regular, i.e., one with $\sing \cC_\circ = \{0\}$, and denote its link by 
\[ \cL_\circ := \cC_\circ\cap \SSp^{n_\circ}. \]
We may parameterize $\cC_\circ$ by 
\[
  (0, +\infty)\times \cL_\circ \to \cC_\circ, \ \ \ (r, \omega) \mapsto x= r\omega\,.
\] 

\begin{Def} \label{def:gamma.c}
	Denote
	\begin{align*}
		\Gamma(\cC_\circ) 
			& := \{\deg(v): v \text{ is a nontrivial homogeneous Jacobi field on }\cC_\circ\} \\
			& = \{ \gamma \in \RR : \Jac_\gamma(\cC_\circ) \neq 0 \, \}  \subset \RR.
	\end{align*}
\end{Def}

Note that a homogeneous Jacobi field $w \in \Jac_\gamma(\cC_\circ)$ can be written as 
\begin{equation} \label{eq:jf.gamma}
w(r\omega) = r^\gamma \psi(\omega).
\end{equation}
Following the notation of Definition \ref{defi:kappa.lambda}, denote the Laplace--Beltrami operator on $\cL_\circ$ by $\Delta_{\cL_\circ}$, and the second fundamental form of $\cL_\circ$ in $\SSp^{n_\circ}$ by $A_{\cL_\circ}$. Then, by the simple computation
\[
	\Delta_{\cC_\circ} w + |A_{\cC_\circ}|^2 w = r^{-2} (\Delta_{\cL_\circ} \psi + |A_{\cL_\circ}|^2 \psi) + r^{-2} (\gamma^2 + (n_\circ-2) \gamma)  \psi.
\]
It easily follows that:

\begin{Lem} \label{lemm:gamma.c0.discrete}
	The subset $\Gamma(\cC_\circ) \subset \RR$ of Definition \ref{def:gamma.c} is discrete. Elements $\gamma \in \Gamma(\cC_\circ)$ and $w \in \Jac_\gamma(\cC_\circ) \setminus \{ 0 \}$ as in \eqref{eq:jf.gamma} are characterized by the fact that
	\begin{equation} \label{eq:lambda.gamma}
		\mu(\gamma):=\gamma^2 + (n_\circ-2)\gamma
	\end{equation}
	is an eigenvalue of $-(\Delta_{\cL_\circ} + |A_{\cL_\circ}|^2)$ on $\cL_\circ$ and $\psi$ is a corresponding eigenfunction.
\end{Lem}

An easy computation shows:

\begin{Lem} \label{lemm:gamma.star.c}
	Homogeneous Jacobi fields on $\cC_\circ$ with finite Dirichlet energy on $\cC_\circ\cap B_1$ are those whose degrees lie in
	\begin{equation*}
		\Gamma^*(\cC_\circ) := \Gamma(\cC_\circ) \cap \RR_{> -(n_\circ-2)/2} \,. 
	\end{equation*}
\end{Lem}

\begin{Rem} \label{rema:gamma.star.c0.contains.01}
	It is an easy consequence of Definition \ref{def:gamma.c}, Lemma \ref{lemm:gamma.star.c} \eqref{eq:translation.inclusion}, \eqref{eq:rotation.inclusion} that $0, 1 \in \Gamma^*(\cC_\circ)$.
\end{Rem}

When solving for homogeneity degrees $\gamma$ in terms of eigenvalues $\mu$ of $-(\Delta_{\cL_\circ}+|A_{\cL_\circ}|^2)$ via \eqref{eq:lambda.gamma}, it will be important for us that each $\mu$ gives rise to a $\gamma \in \Gamma^*(\cC_\circ)$. Equivalently, we want $-\tfrac{n_\circ-2}{2} \not \in \Gamma(\cC_\circ)$. It was shown in \cite[Section 1.6]{HardtSimon:isolated} that the following notion of the ``strict stability'' of $\cC_\circ$ guarantees this. Henceforth $\cC_\circ$ will be thus assumed to be strictly stable.

\begin{Def} \label{defi:strictly.stable}
	A regular minimal hypercone $\cC_\circ$ is said to be \textbf{strictly stable} if there exists a $\mu_{\cC_\circ}> 0$ such that for every $\varphi\in C_c^1(\cC_\circ)$, the second variation 
      \[ Q_{\cC_\circ}(\varphi, \varphi) := \int_{\cC_\circ} |\nabla_{\cC_\circ} \varphi|^2 - |A_{\cC_\circ}|^2\varphi^2 \]
      satisfies
      \[ Q_{\cC_\circ}(\varphi, \varphi) \geq \mu_{\cC_\circ}\int_{\cC_\circ} r^{-2}\varphi^2. \]
\end{Def}

We proceed and fix an $L^2(\cL_\circ)$-orthonormal basis consisting of eigenfunctions of $-(\Delta_{\cL_\circ} + |A_{\cL_\circ}|^2)$,
\[ \psi_1, \psi_2, \psi_3, \ldots \]
where $\psi_1>0$, with corresponding eigenvalues 
\[ \mu(\gamma_1) < \mu(\gamma_2) \leq \mu(\gamma_3) \leq \ldots \nearrow +\infty, \]
and where the homogeneity degrees
\[ -\tfrac{n_\circ-2}{2} < \gamma_1< \gamma_2\leq \gamma_3\leq \dots \nearrow +\infty, \]
are related to the eigenvalues via this relationship stemming from \eqref{eq:lambda.gamma}:
\[ \gamma_j = - \tfrac{n_\circ-2}{2} + \sqrt{\tfrac{(n_\circ-2)^2}{4} + \mu(\gamma_j)}. \]
Finally, choose corresponding homogeneous Jacobi fields on $\cC_\circ$ with finite Dirichlet energy on $\cC_\circ\cap B_1$
\[ w_j(x) := r^{\gamma_j}\psi_j(\omega). \]
After first choosing $\psi_j$ appropriately, we can assume that for each $j\geq 1$:
\begin{itemize}
	\item $w_j \in \Jac_\textrm{trl}(\cC_\circ)$, or
	\item $w_j \in \Jac_\textrm{rot}(\cC_\circ)$, or
	\item $w_j\perp \Jac_\textrm{trl}(\cC_\circ) \oplus \Jac_\textrm{rot}(\cC_\circ)$ in $L^2(\cC_\circ \cap B_1)$.
\end{itemize}

\begin{Rem} \label{rema:gamma1.negative}
	Since $\cC_\circ$ is nonflat, taking $\varphi =1$ in the variational characterization of eigenvalues implies that $\mu(\gamma_1) < 0$ and thus, in particular, that $\gamma_1 < 0$.
\end{Rem}

\begin{Rem}
    We chose not to use $\circ$'s in our notation for the spectral data $\psi_j, \gamma_j, w_j$ so as to minimize notational clutter. For cylindrical hypercones $\cC = \cC_\circ \times \RR^k$, we plan to express and understand all spectral quantities in terms of these ones for $\cC_\circ$.
\end{Rem}

We will eventually have to place another assumption on $\cC_\circ$, i.e., that it be strongly integrable in the sense of \cite{Simon:uniqueness.some}. We recall the definition below for this later use.
    
\begin{Def} \label{defi:strongly.integrable}
	We say that $\cC_\circ$ is \textbf{strongly integrable}, if
	\[ \Jac_0(\cC_\circ) = \Jac_\textnormal{trl}(\cC_\circ) \text{ and } \Jac_1(\cC_\circ) = \Jac_\textnormal{rot}(\cC_\circ), \]
	i.e., all homogeneous degree-$0$, resp.\ $1$, Jacobi fields are translations, resp.\ rotations.
\end{Def}

\subsection{Jacobi fields on cylindrical minimal hypercones} \label{Subsec_Cylind Min Cone} 
  
In this subsection we discuss the case of minimal hypercones of interest. We take
\[ \cC = \cC_\circ \times \RR^k, \]
where $\cC_\circ \subset \RR^{n_\circ+1}$, $n_\circ = n - k$, is a regular and strictly stable minimal hypercone. All our results also clearly apply to all rotations of such $\cC \subset \RR^{n+1}$.

With $\cL_\circ$ still denoting the link of $\cC_\circ$, we can parametrize $\cC$ by 
    \begin{align}
       (0, +\infty)\times \cL_\circ \times \RR^k \to \cC,\ \ \ (r, \omega, y)\mapsto (x = r\omega, y)\,. \label{Equ_Parametrize minimal hypercone}
    \end{align}
Following Section \ref{Subsec_reg Min Cone} above, we continue to denote:
\begin{itemize}
	\item $\gamma_j \in \Gamma^*(\cC_\circ)$, the homogeneity degrees of homogeneous Jacobi fields on $\cC_\circ$ with locally finite Dirichlet energy, and
	\item $w_j = r^{\gamma_j} \psi_j \in \Jac(\cC_\circ)$, the corresponding homogeneous Jacobi fields on $\cC_\circ$, normalized as discussed immediately prior to Remark \ref{rema:gamma1.negative}.
\end{itemize}
Note that, in this cylindrical setting, $r$ coincides with the distance to $\spine \cC = \{ 0 \} \times \RR^k$. We will be concerned with Jacobi fields on $\cC$ for which the following $r$-weighted Dirichlet energy is locally finite:

\begin{Def} \label{defi:cylindrical.energies}
For $u\in \Jac(\cC)$ and open $U\subset \RR^{n+1}$, define 
    \begin{align*}
      \|u\|_{H^1_*(\cC \cap U)}^2 := \int_{\cC \cap U} |\nabla_\cC u|^2 + r^{-2}u^2.
    \end{align*}
    We also define the following spaces of Jacobi fields for open $U \subset \RR^{n+1}$:
      \begin{align*}
      \Jac^*(\cC \cap U) & := \{u\in \Jac(\cC \cap U): \|u\|_{H^1_*(\cC \cap U')}< \infty \text{ for all } U' \Subset U \} \,,
      \end{align*}
      where of course $\Jac^*(\cC) = \Jac^*(\cC \cap \RR^{n+1})$. We also define
      \begin{align*}
      \Jac^*_\textnormal{pos}(\cC) & := \{ t u : t \in \RR, u \in \Jac^*(\cC), u > 0 \} \,, \\
      \Jac^*_d(\cC) & := \Jac^*(\cC) \cap \Jac_d(\cC), \; d \in \RR.
    \end{align*}
\end{Def}

For later use, we record here a Caccioppoli inequality that plays an important role in our subsequent construction of Jacobi fields. It only relies on our standing assumptions, i.e., on $\cC^\circ$ being regular and strictly stable.

    \begin{Lem} \label{Lem_Jac field_Caccioppoli-type Ineq}
       For every $u \in \Jac^*(\cC \cap B_\varrho)$ and $0 < {\tau_1} < {\tau_2} < \varrho$, we have
        \[ \int_{\cC \cap B_{\tau_1}} |\nabla_{\cC} u|^2 + r^{-2} u^2 \leq \frac{C(\cC)}{({\tau_2}-{\tau_1})^2} \int_{\cC \cap B_{\tau_2}} u^2. \]
    \end{Lem}
    \begin{proof}
      The strict stability of $\cC_\circ$ gives $\mu=\mu(\cC_\circ)$ with
      \[
        \int_{\cC_\circ} |\nabla_{\cC_\circ} \varphi|^2 - |A_{\cC_\circ}|^2\varphi^2 \geq \mu \int_{\cC_\circ} r^{-2}\varphi^2, \text{ for all } \varphi\in C^1_c(\cC_\circ).
      \] 
      Fubini's theorem then implies that
      \[
       \int_{\cC} |\nabla_{\cC} \varphi|^2 - |A_{\cC}|^2\varphi^2 \geq \mu \int_{\cC} r^{-2}\varphi^2, \text{ for all } \varphi\in C^1_c(\cC_\circ).
      \]
       Using that $|A_{\cC}|^2 \leq C(\cC_\circ)r^{-2}$ we can take $\mu$ smaller so that 
        \[
       \int_{\cC} |\nabla_{\cC} \varphi|^2 - |A_{\cC}|^2\varphi^2 \geq \mu \int_{\cC} |\nabla_{\cC}\varphi|^2 + r^{-2}\varphi^2, \text{ for all } \varphi\in C^1_c(\cC_\circ).
      \]
      On the other hand, for $u\in \Jac^*(\cC; B_\varrho)$ and $\eta \in C^1_c(\cC \cap B_{\tau_2})$ we can integrate by parts to find that 
      \[
       \int_{\cC} |\nabla_{\cC} (\eta u)|^2 - |A_{\cC}|^2\eta^2u^2 = \int_{\cC} |\nabla \eta|^2 u^2. 
      \]
      Combining these expressions and choosing $\eta$ equal to $1$ on $B_{\tau_1}$, cutting off smoothly inside of $B_{\tau_2}$ completes the proof.
          \end{proof}

Jacobi fields in $\Jac^*(\cC)$ have been studied extensively by Simon \cite{Simon:uniqueness.some, Simon:liouville} and Edelen--Sz\'ekelyhidi \cite[Section 2.5]{EdelenSzekelyhidi:liouville}. It will be important for our geometric applications involving ``one-sided'' perturbations (which yield positive Jacobi fields) that $\Jac^*_\textnormal{pos}(\cC) = \Jac^*_{\gamma_1}(\cC)$ and this space is generated by the corresponding underlying positive eigenfunctions from $\cC_\circ$. This will be shown in Lemma \ref{Lem_Jac field_Jac_(trl, rot, pos)}, but first we need to introduce more notation.

\begin{Def} \label{defi:beta.harmonic}
	Let $\varrho > 0$ and
	\[ B^+_\varrho:=\{(r,y) \in \RR_{>0}\times \RR^k : r^2 + |y|^2 < \varrho^2\} \subset \RR^{k+1}. \]
	We say that $h \in C^2_\loc(B^+_\varrho)$ is $\beta$-harmonic on $B^+_\varrho$ if
	\begin{align}
       \begin{cases}
        \;\; r^{-1-\beta} \operatorname{div}_{\RR^{k+1}}(r^{1+\beta} \nabla_{\RR^{k+1}} h) = \displaystyle\left(\partial_r^2 + \tfrac{1+\beta}{r}\partial_r + \Delta_y\right)h = 0\,, & \\
        \;\; \displaystyle\int_{B^+_\tau} |\nabla_{\RR^{k+1}} h|^2 \, r^{1+\beta} \, dr \, dy < +\infty \text{ for all } 0 < \tau < \varrho \,.
       \end{cases} \label{Equ_Jac field_beta harmonic equ}
      \end{align}
\end{Def}

The following result regarding $\beta$-harmonic functions was proven by Simon \cite[Section 3 \& Appendix A.19]{Simon:liouville}.

    \begin{Prop} \label{Prop_Jac field_beta harmonic}
      Suppose that $h$ is $\beta$-harmonic on $B^+_1$. Then:
      \begin{enumerate}
      	\item $h$ can be smoothly extended to $\{r=0\}$, and
	\item the extension is real analytic in $r^2$ and $y$, i.e., there exists $\tau=\tau(k, \beta)\in (0, 1/2)$ and degree-$i$ homogeneous polynomials $p_{ij}(y)$ so that,
      \[ h(r, y) = \sum_{i, j\geq 0} r^{2j}p_{ij}(y) \]
      with
      \[ \sup_{(r, y)\in B_{\tau}^+} |D_r^jD_y^i h(r, y)| \leq C(\beta, k)^{i+j}i!j!\left(\int_{B_{3/4}^+} h^2 \, r^{1+\beta} \, dr \, dy\right) \]
      for all $i, j \geq 0$.
      \end{enumerate}
      Moreover, if $d\in \RR$ and $h$ is a non-vanishing homogeneous function of degree $d$, then $d \in \ZZ_{\geq 0}$ and $h$ is a polynomial in $r^2$ and $y$ and is uniquely determined by the non-vanishing degree-$d$ polynomial $h(0, y)$ in $y$.  
\end{Prop}

Below, we will always implicitly extend $\beta$-harmonic functions to $\{r=0\}$. 

\begin{Def} \label{defi:jac.c.star.gamma.q}
	For $\gamma \in \Gamma^*(\cC_\circ)$ and an integer $q\geq 0$ we let
\[ \Jac^*_{\gamma,q}(\cC) = \Jac^*_{\gamma+q}(\cC) \cap \operatorname{span} \{ r^{\gamma} \psi(\omega) p_q(r,y) \} \]
where the span is taken over all:
\begin{itemize}
	\item $\psi(\omega)$ that are eigenfunctions of $-\Delta_{\cL_\circ}-|A_{\cL_\circ}|^2$ on $\cL_\circ$ with eigenvalue $\mu(\gamma)$ as in \eqref{eq:lambda.gamma}, and
	\item $p_q(r,y)$ that are homogeneous degree-$q$ polynomials in $(r, y)$ that are $(2\gamma+n-2)$-harmonic (see Definition \ref{defi:beta.harmonic}); by Proposition \ref{Prop_Jac field_beta harmonic}, $p_q(r,y)$ is a posteriori a polynomial in $r^2$ and $y$, i.e., it contains no odd powers of $r$. 
\end{itemize}
Clearly,
\[
\Jac^*_{\gamma,q}(\cC) \subset \Jac^*_{\gamma+q}(\cC).
\]
We denote by
    \begin{align}
      \Gamma^*(\cC) := \Gamma^*(\cC_\circ)+\ZZ_{\geq 0} = \{\gamma+q: \gamma\in \Gamma^*(\cC_\circ), q\in \ZZ_{\geq 0}\}\,. \label{Equ_Jac field_Asymp Spectrum of C=C_0*R^k}
    \end{align}
    the set of homogeneities obtained in this way. 
\end{Def}

\begin{Rem} \label{rema:gamma.star.c}
	It follows from Lemma \ref{lemm:gamma.c0.discrete} and Remark \ref{rema:gamma.star.c0.contains.01} that $\Gamma^*(\cC) \subset \RR$ is discrete and contains $\ZZ_{\geq 0}$. Thus, the following quantities are well-defined,
	\begin{align*}
		\Delta_\cC^{<1} & := 1 - \max (\Gamma^*(\cC)\cap \RR_{<1}), \\
		\Delta_\cC^{>1} & := \min (\Gamma^*(\cC)\cap \RR_{>1}) - 1,
	\end{align*}
	and take values in $(0, 1]$. Moreover,
	\[ (1-\Delta_\cC^{<1}, 1) \cap \Gamma^*(\cC) = (1, 1+\Delta_\cC^{>1}) \cap \Gamma^*(\cC) = \emptyset. \]
\end{Rem}

The following crucial result was essentially proven by Simon \cite[Section 3]{Simon:liouville}. 

\begin{Prop} \label{Prop_Jac field_Decomp}
    Fix $d \in \RR$.
    \begin{enumerate}[(i)]
    \item If $d \not \in \Gamma^*(\cC)$ then $\Jac^*_d(\cC) = \{0\}$. 
    \item If $d \in \Gamma^*(\cC)$ then we have an $L^2(\cC \cap B_1)$-orthogonal decomposition
    \[
          \Jac_d^*(\cC) = \bigoplus_{\substack{\gamma\in \Gamma^*(\cC_\circ)\\ q\in \ZZ_{\geq 0}\\ \gamma+q = d}} \Jac^*_{\gamma, q}(\cC)\,.     \]   
       \item For $d'\neq d \in \Gamma^*(\cC)$, we have with respect to $L^2(\cC \cap B_1)$ that 
       \[ \Jac^*_d(\cC) \perp \Jac_{d'}^*(\cC). \]
     \item Let $u\in \Jac^*(\cC \cap B_R)$. We may assume that $R > 1$ by rescaling the domain. Then, $u$ can be uniquely decomposed in $L^2(\cC \cap B_1)$ as
       \begin{equation}\label{Equ_Jac field_L^2 decomp of u}
         u(x, y) = \sum_{d \in \Gamma^*(\cC)} u_d(x, y),
       \end{equation}
       where $u_d \in \Jac^*_d(\cC)$ for all $d \in \Gamma^*(\cC)$. Moreover, for all $\varrho \in (0, R)$, 
       \[
       \| u \|_{L^2(\cC \cap B_\varrho)}^2 = \sum_{d \in\Gamma^*(\cC)} c_d^2 \varrho^{n+2d}, \text{ where } c_d^2 : = \frac{1}{n+2d} \int_{\cC \cap \SSp^n} u_d^2 \, .
       \]
    \end{enumerate}
 \end{Prop}
 
We also have:

    \begin{Lem}\label{Lem_spect proj Jac are beta harmonic}
    Let $u \in \Jac^*(\cC \cap B_R)$. Then, for $r^2+|y|^2 < R^2$ we define
    \[
    h_j(r,y) : = r^{-\gamma_j} \int_{\cL_\circ} u(r\omega,y) \psi_j(\omega) \, d\omega, \; j \geq 1.
    \]
    Then $h_j$ is $\beta$-harmonic on $B^+_R$ with $\beta = 2\gamma_j + n_\circ-2>0$. 
    \end{Lem}
    \begin{proof}
This is a straightforward computation. Indeed, we have    \begin{align*}
     \left(\partial_r^2 + \tfrac{1+\beta}{r}\partial_r + \Delta_y\right)h   & = \gamma_j(\beta-\gamma_j) r^{-2} h + r^{-\gamma_j} \int_{\cL_\circ}  \left( \partial^2_r u + \tfrac{n_\circ-1}{r} \partial_r u + \Delta_y u\right)  \psi_j \\
    & = \mu(\gamma_j) r^{-2} h - r^{-\gamma_j-2} \int_{\cL_\circ}  \left( \Delta_{\cL_\circ} u +  |A_{\cL_\circ}|^2 u \right)  \psi_j \\
        & = \mu(\gamma_j) r^{-2} h - r^{-\gamma_j-2} \int_{\cL_\circ}  \left( \Delta_{\cL_\circ} \psi_j +  |A_{\cL_\circ}|^2 \psi_j \right) u \\
        & = 0. 
    \end{align*}
    The weighted Dirichlet energy bound follows from the definition of the $H^1_*$ norm. 
\end{proof}
The following lemma relates the translation, rotation, and positive Jacobi fields to the spectral decomposition obtained in Proposition \ref{Prop_Jac field_Decomp}. 

    \begin{Lem}\label{Lem_Jac field_Jac_(trl, rot, pos)}
    The translation and rotation Jacobi fields of $\cC$ (see Definitions \ref{defi:translation.jf}, \ref{defi:rotation.jf}) satisfy
         \begin{align*}
       \Jac_\textnormal{trl}(\cC)& \subset \Jac^*_{0, 0}(\cC)\,, \\
       \Jac_\textnormal{rot}(\cC) & \subset \Jac^*_{1, 0}(\cC)\oplus \Jac^*_{0, 1}(\cC) \,;
     \end{align*}
     cf.\  \eqref{eq:translation.inclusion}, \eqref{eq:rotation.inclusion}. Moreover
     \[
      \Jac_\textnormal{pos}^*(\cC) = \Jac^*_{\gamma_1}(\cC) = \Jac^*_{\gamma_1, 0}(\cC)   \]
      is $1$-dimensional and spanned by $\ r^{\gamma_1}\psi_1(\omega)$. 
    \end{Lem}

    \begin{proof} 
    The first two statements are clear. 
    
    The claim about positive Jacobi fields actually follows from the work of Edelen--Sz\'ekelyhidi \cite[Lemma 2.7]{EdelenSzekelyhidi:liouville}. Since their hypotheses are stated slightly differently we recall their proof here. 
    
    First, note that by Proposition \ref{Prop_Jac field_Decomp} {(ii)} we have $\Jac^*_{\gamma_1}(\cC) = \Jac^*_{\gamma_1,0}(\cC)$. Since 
    \[
    \Jac^*_{\gamma_1,0}(\cC) \subset \Jac^*_\textnormal{pos}(\cC)
    \] is clear from the fact that $\psi_1 > 0$, it suffices to prove the reverse inclusion. To that end, let $u \in \Jac^*_\textrm{pos}(\cC)$ be arbitrary. For $j\geq 1$, recall that
    \begin{equation*} 
        h_j(r, y):= r^{-\gamma_j}\int_{\cL_\circ} u(r\omega, y)\psi_j(\omega)\ d\omega \,.
      \end{equation*}
      is $\beta_j$-harmonic with $\beta_j = 2\gamma_j + n-2>0$ by Lemma \ref{Lem_spect proj Jac are beta harmonic}. Moreover, $h_j$ extends to $\{r=0\}$ and the extension is real analytic in $r^2$ and $y$ by Proposition \ref{Prop_Jac field_beta harmonic}. 
      
      As in \cite[(19)]{EdelenSzekelyhidi:liouville}, the Liouville property for (nonnegative) $\beta$-harmonic functions implies that $h_1(r,y)$ is constant. Thus, we get
    \[
    \int_{\cC \cap B_\varrho} r^{\gamma_1} u \psi_1 = h_1(0,0) \int_0^\varrho r^{n-1+2\gamma_1} \, dr = \frac{h_1(0,0)}{n+2\gamma_1} \varrho^{n+2\gamma_1}. 
    \]
    Using elliptic estimates and positivity of $u$, there is $C=C(\cC) > 0$ with
    \[
u(\varrho \omega,y) \leq C \varrho^{-n} \int_{\cC \cap B_{\varrho/2}(\varrho\omega,y)} u \leq C  \varrho^{-n-\gamma_1} \int_{\cC \cap B_{2\varrho}} r^{\gamma_1} u \psi_1 \leq C \varrho^{\gamma_1} ,
\] 
so
   \[
   \int_{\cC \cap B_\varrho} u^2 \leq C \varrho^{n+2\gamma_1}. 
   \]
   On the other hand, if $u \in \Jac^*_\textrm{pos}(\cC) \setminus \Jac^*_{\gamma_1}(\cC)$ then {(iv)} in Proposition \ref{Prop_Jac field_Decomp} implies that
\[
\varrho^{-n - 2\gamma_1} \int_{\cC \cap B_\varrho} u^2 \to \infty
\]
as $\varrho\to\infty$, a contradiction. Thus, $u \in \Jac^*_{\gamma_1}(\cC)$. However, $\Jac^*_{\gamma_1}(\cC) = \Jac^*_{\gamma_1,0}(\cC)$ by {(ii)} in Proposition \ref{Prop_Jac field_Decomp}. This completes the proof. 
    \end{proof}

\begin{Rem} \label{rema:strongly.integrable}
It follows from Lemma \ref{Lem_Jac field_Jac_(trl, rot, pos)} that $\cC_\circ$ is strongly integrable (Definition \ref{defi:strongly.integrable}) if and only if
\[ \Jac_\textnormal{trl}(\cC) = \Jac^*_{0,0}(\cC) \text{ and } \Jac_\textnormal{rot}(\cC) = \Jac^*_{1,0}(\cC) \oplus \Jac^*_{0,1}(\cC). \] 
\end{Rem}

\subsection{The decay order}

We introduce one final definition, whose study we defer to later in the section. We retain all   conventions and assumptions on
\[ \cC = \cC_\circ \times \RR^k \]
from the previous section, namely, that $\cC_\circ \subset \RR^{n_\circ+1}$ is a regular, strictly stable minimal hypercone with link $\cL_\circ$ and associated spectral quantities $\gamma_j$, $w_j$, $\psi_j$.

\begin{Def} \label{Def_doubling constant} \label{defi:linear.doubling.constant}
    Consider a nonzero Jacobi field $u \in \Jac(\cC \cap B_{2\varrho}) \cap L^2(\cC \cap B_{2\varrho})$. For $\kappa \geq 0$, we define the $\kappa$-adjusted \textbf{decay order} of $u$ at scale $\varrho > 0$ to be
    \begin{align*}
      G^\kappa_\cC(u; \varrho) 
      	& := 1 + \frac{1}{2}\log_2 \left((2\varrho)^{-n-2}\int_{\cC \cap B_{2\varrho}} u^2 + \kappa \cdot (2\varrho)^2\right)  \\
      	& \qquad - \frac{1}{2}\log_2\left(\varrho^{-n-2}\int_{\cC \cap B_{\varrho}} u^2 + \kappa \cdot \varrho^2 \right).
    \end{align*}
    We'll write $G_\cC(u;\varrho)$ for $G^0_\cC(u;\varrho)$. 
  \end{Def}
  
\begin{Rem} \label{rema:kappa.linear}
	The constant $\kappa \geq 0$ is an auxiliary damping factor that will be necessary in the subsequent Riemannian setting. We encourage the first-time reader to read this section keeping in mind that $\kappa = 0$ suffices for all applications to $\RR^{n+1}$.
\end{Rem}

It follows from the definition above that $G_\cC(u; \varrho) = \gamma$ if $u\in \Jac_\gamma(\cC) \cap L^2(\cC \cap B_{2\varrho})$. See also Corollary \ref{Cor_Jac field_Freq Monoton} below.

For non-homogeneous $L^2$ Jacobi fields $u$, we will prove a crucial monotonicity property of the decay order. First note that:

\begin{Lem} \label{Lem_Jac field_Freq Monoton}
      Let $u\in \Jac^*(\cC \cap B_R) \setminus \{0\}$, $\kappa \geq 0$. The auxiliary function
      \[
        H^\kappa(t):= \log\left( \int_{\cC \cap B_{e^t}} u^2 + \kappa \cdot e^{(n+4)t} \right)
      \]
      is well-defined and smooth and satisfies:
      \begin{enumerate}
      	\item[(i)] $H^\kappa$ is a $C^\infty$ convex function in $(-\infty, \log R)$,
	\item[(ii)] $(H^\kappa)'$ equals the constant $2n+d \in \RR$ on some sub-interval of $(-\infty, \log R)$ if and only if $u$ is the restriction to $B_R$ of a $\Jac_d^*(\cC)$ element, with $d=2$ if $\kappa > 0$.
      \end{enumerate}
\end{Lem}
\begin{proof}
      By Proposition \ref{Prop_Jac field_Decomp} {(iv)} we may write
      \[
      \int_{\cC \cap B_\varrho} u^2 = \sum_{d\in \Gamma^*(\cC)} c_d^2 \varrho^{n+2d}\,
      \]
      where this series converges for all $\varrho \in (0,R)$. Let\footnote{Recall that $2 \in \Gamma^*(\cC)$ by Remark \ref{rema:gamma.star.c}.}
      \[
      	\tilde c_d = \begin{cases} c_d & \text{ if } d \neq 2, \\ \sqrt{c_2^2+\kappa} & \text{ if } d = 2. \end{cases}
      \]
      We find that
       \[
        H^\kappa(t) = \log\left( \sum_{d\in \Gamma^*(\cC)} \tilde c_d^2 e^{(n+2d)t} \right)
      \]
      is $C^\infty$ in $t\in (-\infty, \log R)$ and 
      \begin{align*}
        (H^\kappa)'(t) & = \left( \sum_{d\in \Gamma^*(\cC)} \tilde c_d^2 e^{(n+2d)t} \right)^{-1}\left( \sum_{d\in \Gamma^*(\cC)} \tilde c_d^2 e^{(n+2d)t}(n+2d) \right)\,; \\
        (H^\kappa)''(t) & = \left( \sum_{d\in \Gamma^*(\cC)} \tilde c_d^2 e^{(n+2d)t} \right)^{-2}\left[\left( \sum_{d\in \Gamma^*(\cC)} \tilde c_d^2 e^{(n+2d)t} \right)\left( \sum_{d\in \Gamma^*(\cC)}\tilde  c_d^2 e^{(n+2d)t}(n+2d)^2 \right) \right. \\
        & \;\;\;\;\;\;\;\;\;\;\;\;\;\;\;\;\;\;\;\;\;\;\;\;\;\;\;\;\;\;\;\;\;\;\;\;\;\;\;\;\;\;\;\;\;\;\;\;\;\;\;\;\;\;\;\;\;\;\;\;\;\;\;\;\;\;\;\;\;\;\;\;\;\;\;\;\; - \left. \left( \sum_{d\in \Gamma^*(\cC)} \tilde c_d^2 e^{(n+2d)t}(n+2d) \right)^2 \right]\,.        
      \end{align*}
      By Cauchy-Schwarz, $(H^\kappa)''(t)\geq 0$ and equality holds for some $t$ if and only if all the $\tilde c_d$ but one vanish. This easily implies both {(i)} and {(ii)}.
\end{proof}

\begin{Cor}[Decay order monotonicity] \label{Cor_Jac field_Freq Monoton}
	Let $u \in \Jac^*(\cC \cap B_R) \setminus \{0\}$, $\kappa \geq 0$. Then:
	\begin{enumerate}
		\item[(i)] $\varrho \mapsto G^\kappa_\cC(u; \varrho)$ is non-decreasing in $(0, R/2)$. 
		\item[(ii)] $\varrho \mapsto G^\kappa_\cC(u; \varrho)$ equals the constant $d \in \RR$ on some sub-interval of $(0, R/2)$ if and only if $u$ is the restriction to $B_R$ of a $\Jac^*_d(\cC)$ element, with $d = 2$ if $\kappa > 0$.
		\item[(iii)] The following limit exists, 
		\[
			G^\kappa_\cC(u; 0):= \lim_{\varrho \to 0} G^\kappa_\cC(u; \varrho) = \begin{cases}
				d_0(u)  & \text{ if } \kappa = 0, \\
				\min\{2,d_0(u)\} & \text{ if } \kappa > 0 ,
        \end{cases}
      \] 
		with $d_0(u) \in \Gamma^*(\cC)$ as in Definition \ref{defi:d.0.u} below.
	\end{enumerate}
\end{Cor}
     
 \begin{Def} \label{defi:d.0.u}
 	For all $u \in \Jac^*(\cC \cap B_R) \setminus \{0\}$, we define
	\[ d_0(u) = \min \{ d : u_d \neq 0 \text{ in \eqref{Equ_Jac field_L^2 decomp of u}} \}. \]
	This is well-defined by {(iii)} of Proposition \ref{Prop_Jac field_Decomp}.
 \end{Def}

    \begin{proof}[Proof of Corollary \ref{Cor_Jac field_Freq Monoton}]
      Recall that 
      \begin{align*}
        & G^\kappa_\cC(u; \varrho) \\
        & = \frac12 \log_2\left((2\varrho)^{-n-2}\int_{\cC \cap B_{2\varrho}} u^2 + \kappa \cdot (2\varrho)^2 \right) - \frac12 \log_2\left(\varrho^{-n-2}\int_{\cC \cap B_{\varrho}} u^2 + \kappa \cdot \varrho^2 \right) + 1 \\
        & = \frac12 \left( - n - 2 + \frac{H^\kappa(\log(2\varrho))-H^\kappa(\log(\varrho))}{\log 2}\right) + 1 \,,
      \end{align*}
      where $H^\kappa$ is as in Lemma \ref{Lem_Jac field_Freq Monoton}. By {(i)} in Lemma \ref{Lem_Jac field_Freq Monoton}, \[
        \frac{d}{d\varrho} G^\kappa_\cC(u; \varrho) = \frac{ (H^\kappa)'(\log(2\varrho))- (H^\kappa)'(\log(\varrho))}{2\varrho \log 2} \geq 0 \,,
      \]
      since $(H^\kappa)'$ is non-decreasing. This also gives {(ii)} thanks to {(ii)} in Lemma \ref{Lem_Jac field_Freq Monoton}.

      We turn to {(iii)}. It follows from our definition of $d_0 = d_0(u)$ and from \eqref{Equ_Jac field_L^2 decomp of u} that
      \[ H^0(\log(\varrho)) = \log(c_{d_0}^2 \varrho^{n+2d_0}(1+o(1))) \text{ as } \varrho \to 0. \]
      Therefore, 
      \[
        \lim_{\varrho \to 0} G^0_\cC(u; \varrho) = \frac12 \left( - n+ \frac{(n+2d_0)\log 2}{\log 2}\right) = d_0 \,.
      \]
      The argument when $\kappa>0$ is similar, so this completes the proof.
    \end{proof}
    
We recall that $G_\cC(u;\varrho) = G^0_\cC(u; \varrho)$. This is all we will need to discuss in the remainder of this section.\footnote{We will always know that $\kappa = 0$ whenever the subsequent results are invoked.} We introduce suitably translated and rescaled decay orders.

\begin{Def}
	Let $u \in \Jac^*(\cC \cap B_2) \setminus \{0\}$, $\mbfy \in \spine \cC \cap B_2$, and $\varrho \in (0, 2-|\mbfy|)$. We define
	\begin{align*}
		G_\cC(u; \mbfy, \varrho) & := G_\cC(u(\mbfy + \cdot); \varrho)\,; \\
		G_\cC(u; \mbfy) & := G_\cC(u; \mbfy, 0) := G_\cC(u(\mbfy + \cdot); 0)\,.
	\end{align*}
\end{Def}

\begin{Lem} \label{lemm:doubling.upper.semicontinuity}
$G_\cC(u; \mbfy)$ is upper-semi-continuous in convergence of $\mbfy\in \spine \cC\cap B_2$ and $u\in L^2(\cC \cap B_\tau(\mbfy))$, $\tau > 0$.
\end{Lem}
\begin{proof}
This follows from the fact that  $\varrho \mapsto G_\cC(u; \varrho)$ is non-decreasing (Corollary \ref{Cor_Jac field_Freq Monoton}). 
\end{proof}

    Next we investigate the relationship between the spectral decomposition of Jacobi fields and the decay order. We agree on the following notation. If $f$ is an $\RR$-valued real-analytic function on $U \subset \RR^k$, and $y_0 \in U$, we'll write
    \[ \Ord_f(y_0) \]
    for the degree of the lowest non-zero term in the Taylor expansion of $f$ about $y_0$. We will apply this to $\beta$-harmonic $h : B_R^+ \to \RR$ as follows. Recall that  $h$ extends to $\{r=0\}$ by Proposition \ref{Prop_Jac field_beta harmonic}. We may then write
    \[ \Ord_{h(0, \cdot)}(y_0) \]
to mean what was defined above.  

    \begin{Lem} \label{Lem_Jac field_N_C(u)=min(gamma_j+ord_h)}
      Let $u\in \Jac^*(\cC \cap B_R) \setminus \{0\}$, and define $h_j$, $j \geq 1$, as in Lemma \ref{Lem_spect proj Jac are beta harmonic}. Then, for every $\mbfy_0=(0,y_0)\in \spine \cC \cap B_R$ we have 
      \begin{align}
        G_\cC(u; \mbfy_0) = \inf\{\gamma_j+\Ord_{h_j(0, \cdot)}(y_0): j\geq 1\} \,. \label{Equ_Jac field_N_C(u)=min(gamma+ord)}
      \end{align}
    \end{Lem}
    \begin{proof}
    It suffices to consider $\mbfy_0 =\orig = (0,0)$. By Proposition \ref{Prop_Jac field_Decomp} {(iii)}, \[
        u = \sum_{d\in \Gamma^*(\cC)} u_d
      \]
      in $L^2$. Observe that 
      \[
       p_{d, j}(r, y):= r^{-\gamma_j}\int_{\cL_\circ} u_d(r\omega, y)\psi_j(\omega)\ d\omega 
      \]
      either vanishes or is a degree-$(d-\gamma_j)$ homogeneous polynomial. Using elliptic estimates to exchange the sum and integral, we see that
      \[
      h_j(r,y) = \sum_{d \in \Gamma^*(\cC)} p_{d,j}(r,y)
      \]
      so
      \[
      \gamma_j + \Ord_{h_j(0,\cdot)}(0) = \inf\{d : p_{d,j}\neq 0, \; d\in\Gamma^*(\cC)\}. 
      \]
      
      On the other hand, there must be some $j$ so that $p_{d_0,j} \neq 0$ by completeness of the $\psi_j$. Thus we find
      \[
      \inf\{\gamma_j + \Ord_{h_j(0,\cdot)}(0) : j\geq 1\}=\inf\{d : p_{d,j}\neq 0, \; d\in\Gamma^*(\cC), \; j \geq 1\} = d_0.
      \]
      This finishes the proof.
    \end{proof}

We now arrive at our fundamental ``splitting'' result for homogeneous Jacobi fields, Proposition \ref{Prop_Jac field_Max deg mod transl}, according to high decay order points. This is akin to cone-splitting along high-density points.

    \begin{Prop} \label{Prop_Jac field_Max deg mod transl}
        Let $u\in \Jac^*_1(\cC) \cap (\Jac^*_{0,1}(\cC))^\perp \setminus \{0\}$, where $\perp$ is with respect to $L^2(\cC \cap B_1)$. Denote
        \[ \cJ_u:= \operatorname{span} \{ u \} + \Jac^*_{0,0}(\cC) +\Jac_\textnormal{pos}^*(\cC). \]
        Then:
        \begin{enumerate}
        	\item[(i)] The following is a linear subspace of  $\spine \cC$:
        		\[ V_u := \{\mbfy_0 \in \spine \cC : G_\cC(w; \mbfy_0) \geq 1 \text{ for some } w \in \cJ_u \setminus \{0\} \}. \]
        \item[(ii)] $u$ is invariant under translations in $V_u$, i.e., 
        		\[ u(\mbfy_0 + \cdot) = u \text{ for all } \mbfy_0 \in V_0. \]
        \end{enumerate}
    \end{Prop}

\begin{Rem} \label{rema:splitting}
	In our geometric applications of Proposition \ref{Prop_Jac field_Max deg mod transl}, via Corollary \ref{Cor_Jac field_High degree pts concentrate}, we will take $\cC_\circ$ strongly integrable to have access to Remark \ref{rema:strongly.integrable}. We will study minimal hypersurfaces modeled on $\graph_\cC u$ modulo translations ($\Jac^*_{\textnormal{trl}}(\cC) = \Jac^*_{0,0}(\cC)$) and one-sided deformations ($\Jac^*_{\textnormal{pos}}(\cC)$). Our function $u$ will have linear decay and $\cC$ will have been tilted to begin with so that $u$ is $\perp$ to $\Jac^*_{\textnormal{rot}}(\cC)$ ($= \Jac^*_{1,0}(\cC) \oplus \Jac^*_{0,1}(\cC) \supset \Jac^*_{0,1}(\cC)$). 
\end{Rem}

    \begin{proof}[Proof of Proposition \ref{Prop_Jac field_Max deg mod transl}]
	 By Proposition \ref{Prop_Jac field_Decomp} {(ii)} we have the decomposition
	 	 
\begin{equation}
 \label{eq:decomp}
 u(r\omega, y) = \sum_{j\geq 1} r^{\gamma_j} \psi_j(\omega) p_j(r, y)
\end{equation}
where 
      \[
      p_j(r,y) = r^{-\gamma_j} \int_{\cL_\circ} u(r\omega,y) \psi_j(\omega)\, d\omega\,.
      \]
      Note that $u \in \Jac^*_1(\cC) \cap (\Jac^*_{0,1}(\cC))^\perp$ implies
      \begin{equation} \label{eq:decomp-pj-zero}
      p_j \neq 0 \implies \gamma_j \neq 0.
      \end{equation}
      Moreover, Proposition \ref{Prop_Jac field_beta harmonic} implies that the non-vanishing $p_j$ are homogeneous polynomials in $(r,y)$ with only even powers of $r$ and with
      \begin{equation} \label{eq:decomp-degrees}
      \deg p_j = \deg p_j(0,\cdot) = 1 - \gamma_j. 
      \end{equation}
      Together, \eqref{eq:decomp-pj-zero} and \eqref{eq:decomp-degrees} imply that
      \begin{equation} \label{eq:decomp-gamma}
      p_j \neq 0 \implies \gamma_j \in \{ \ldots, -3, -2, -1 \} \cup \{ 1 \}
      \end{equation}
      In particular, there are only finitely many $p_j \neq 0$ since $\gamma_j \to \infty$ as $j \to \infty$.

    We will show that $V_u = \hat V_u$ where
      \[
      \hat V_u := \bigcap_{\substack{j \geq 1 \\ p_j \neq 0}} \{\mbfy = (0, y) \in \spine \cC : \Ord_{p_j(0,\cdot)}(y) \geq 1 - \gamma_j = \Ord_{p_j(0,\cdot)}(0)\}. 
      \]
      Indeed, by Lemma \ref{Lem_homog-poly-splitting}, $\hat V_u$ is an intersection of linear subspaces of $\spine \cC$, so a linear subspace itself, and each  $p_j \neq 0$ is invariant under translations in the $\hat V_u$ directions (of course, so are $p_j = 0$). Thus, $V_u = \hat V_u$ will imply both {(i)} and {(ii)}.
      
      \textit{Proof that $V_u \subset \hat V_u$}. Let $\mbfy_0 \in V_u$. Note that 
      \[ G_\cC(\varphi; \mbfy_0) \leq 0 
      \]
      for every $\varphi \in \Jac^*_{0,0}(\cC) + \Jac^*_\textrm{pos}(\cC)$, so having $\mbfy_0 \in V$ means that 
      \begin{equation}\label{Eq_Jac field_Max deg mod transl_y0}
      G_\cC(u + \varphi; \mbfy_0) \geq 1
      \end{equation}
      for some $\varphi \in \Jac^*_{0,0}(\cC) + \Jac^*_\textrm{pos}(\cC)$.
      Define, for $j \geq 1$,
      \[
      \hat p_j(r,y) := r^{-\gamma_j} \int_{\cL_\circ}(u+\varphi)(r\omega,y)\psi_j(\omega)\, d\omega.
      \]
      Our choice of $\mbfy_0 = (0, y_0) \in V_u$, Lemma \ref{Lem_Jac field_N_C(u)=min(gamma_j+ord_h)}, \eqref{eq:decomp-degrees}, and \eqref{Eq_Jac field_Max deg mod transl_y0} give
      \begin{equation} \label{eq:splitting-hat}
      \Ord_{\hat p_j(0,\cdot)}(y_0) \geq 1 - \gamma_j  = \deg p_j(0, \cdot) = \Ord_{p_j(0,\cdot)}(0) 
      \end{equation}
      whenever $p_j \neq 0$. To conclude that $\mbfy_0 \in \hat V_u$, we need to replace $\hat p_j$ with $p_j$ on the left hand side of \eqref{eq:splitting-hat}. There are two cases to consider. 
      \begin{itemize}
      	\item If $j \geq 2$, then $\gamma_j \neq \gamma_1$. Moreover, $\gamma_j \neq 0$ by \eqref{eq:decomp-gamma}. Then, Lemma \ref{Lem_Jac field_Jac_(trl, rot, pos)} together with $\varphi \in \Jac^*_{0,0}(\cC) + \Jac^*_{\gamma_1,0}(\cC)$ imply $\hat p_j = p_j$, as desired.
      	\item If $j = 1$, then $\gamma_j=\gamma_1 <0$ by Remark \ref{rema:gamma1.negative}, so
       \[ \deg p_1(0, \cdot) = 1-\gamma_1 \geq 2 \]
       by \eqref{eq:decomp-degrees} whenever $p_j \neq 0$. Lemma \ref{Lem_Jac field_Jac_(trl, rot, pos)} and $\varphi \in \Jac^*_{0,0}(\cC) + \Jac^*_\textrm{pos}(\cC)$ guarantee that
       \[ \hat p_1 = p_1 + c_1 \]
       for some $c_1 \in \RR$. Therefore, \eqref{eq:splitting-hat} and Lemma \ref{lemm:polynomial.ord.characterization} imply that
       \[ D_y^\alpha \hat p_1(0, y_0) = 0 \text{ for all } |\alpha| \leq \deg p_1 - 1. \]
       Applying this with $|\alpha| = \deg p_1 - 1 \geq 1$ cancels the $c_1 \in \RR$ in $\hat p_1 = p_1 + c_1$, so
       \[ D_y^\alpha p_1(0, y_0) = 0 \text{ for all } |\alpha| = \deg p_1 - 1. \]
       Thus, Lemma \ref{Lem_homog-poly-splitting} and \eqref{eq:decomp-degrees} imply that \eqref{eq:splitting-hat} does hold with $p_1$ in place of $\hat p_1$ on the left hand side.
      \end{itemize}
      This completes the proof that $V_u \subset \hat V_u$.
      
      \textit{Proof that $\hat V_u \subset V_u$}. Let $\mbfy_0 = (0, y_0) \in \hat V_u$. Note that
      \[
      u(r \omega,y_0 + y) = u(r\omega,y),
      \]
      where we used the decomposition \eqref{eq:decomp} and that, by Lemma \ref{Lem_homog-poly-splitting}, each $p_j$ is invariant under translations in the $\hat V_u$ directions. Thus 
      $$u(\mbfy_0 + \cdot) = u \in \cJ_u. $$ 
      By assumption, $G_\cC(u; \mbfy_0) \geq 1$, which implies that $\mbfy_0 \in V_u$. Thus, $\hat V_u \subset V_u$ as claimed, which completes the proof.
      \end{proof}
      
      \begin{Rem} \label{rema:splitting-weaker}
      One may of course prove more general versions of Proposition \ref{Prop_Jac field_Max deg mod transl}. 
      
      We could have taken $u \in \Jac^*_1(\cC) \setminus \{0\}$ and (i) would remain true but (ii) would not, as one can see by taking $u \in \Jac^*_{0,1}(\cC)$ that rotates $\spine \cC$. One can instead show that (ii) holds in general for the projection of $u$ onto $(\Jac^*_{0,1}(\cC))^\perp$ with minor modifications. 
      
      Generalizing further, one could have taken $\Jac^*_d(\cC)$ in place of $\Jac^*_1(\cC)$ and suitable replacements of $\Jac^*_{0,0}(\cC)$, $\Jac^*_{\textnormal{pos}}(\cC)$ with $\Jac^*_{\hat \gamma, \hat q}(\cC)$ with $\hat \gamma + \hat q < d$. Each such $\Jac^*_{\hat \gamma, \hat q}(\cC)$ would introduce a constraint of working $\perp$ to $\Jac^*_{\hat \gamma,d-\hat \gamma}(\cC)$ provided $\hat \gamma + \hat q > d-2$.
      \end{Rem}    

    We arrive at a quantitative version of Proposition \ref{Prop_Jac field_Max deg mod transl} whose proof strongly requires an additional assumption on our underlying regular hypercone $\cC_\circ$, namely, that it be strongly integrable (see Definition \ref{defi:strongly.integrable}).

    \begin{Cor} \label{Cor_Jac field_High degree pts concentrate}
        Suppose that the factor $\cC_\circ$ of $\cC$ is also strongly integrable.\footnote{So far, $\cC_\circ$ had only been assumed to be a regular and strictly stable minimal hypercone.} Then, for each $\theta >0$, there exists $\varrho=\varrho(\cC, \theta)>0$ with the following significance.
        
        Consider any $u\in \Jac_1^*(\cC) \cap (\Jac_\textnormal{rot}(\cC))^\perp \setminus \{0\}$, where $\perp$ is with respect to $L^2(\cC \cap B_1)$. Let
 	\[ \cJ_u := \operatorname{span} \{ u \} + \Jac_\textnormal{trl}(\cC)+\Jac_\textnormal{pos}^*(\cC). \]
	Then, there exists a proper linear subspace $V_u \subsetneq \spine \cC$ so that
        \[ G_\cC(w; \mbfy, \varrho) < 1 - \tfrac12 \Delta_\cC^{<1} \]
        for every $w \in \cJ_u \setminus \{ 0 \}$, $\mbfy\in \spine \cC \cap \bar B_1 \setminus B_\theta(V)$; here $\Delta_\cC^{<1}$ is as in Remark \ref{rema:gamma.star.c}.
    \end{Cor}
    
    Of course, the right hand side could have been any $1-\delta$ with $\delta \in (0, \Delta_\cC^{<1})$, in which case, $\varrho$ would also have depend on our choice of $\delta$.
    
    \begin{proof}
      Suppose for contradiction that the assertion fails for some $\theta>0$ and a sequence of $u_j\in \Jac^*_1(\cC)\cap (\Jac_\textnormal{rot}(\cC))^\perp \setminus \{0\}$, $\varrho_j = j^{-1}$. 
      
      Normalize so that $\|u_j\|_{L^2(\cC \cap B_1)}=1$. Using the Caccioppoli inequality in Lemma \ref{Lem_Jac field_Caccioppoli-type Ineq} and homogeneity we can pass to a subsequence (not relabled) so that $u_j$ converges to a Jacobi field $u_\infty\in \Jac_1^*(\cC)\cap (\Jac_\textrm{rot}(\cC))^\perp \setminus \{ 0 \}$ in $L^2(\cC \cap B_1)$. Note that $(\Jac_\textrm{rot}(\cC))^\perp\subset \Jac^*_{0,1}(\cC)^{\perp}$ and $\Jac_\textnormal{trl}(\cC) = \Jac^*_{0,0}(\cC)$ by the strong integrability assumption (Remark \ref{rema:strongly.integrable}), so Proposition \ref{Prop_Jac field_Max deg mod transl} applies and furnishes the linear subspace 
      \[ V_{u_\infty} = \{ \mbfy \in \spine \cC : G_\cC(w; \mbfy) \geq 1 \text{ for some } w \in \cJ_{u_\infty} \setminus \{0\} \} \]
      of $\spine \cC$ along which $u$ is translation invariant.
      
      We claim that $V_{u_\infty} \subsetneq \spine \cC$. Otherwise, $u_\infty$ would be the trivial extension of a $u_\infty^\circ \in \Jac^*_1(\cC_\circ)$ along $\spine \cC$. The strong integrability of $\cC_\circ$ implies that $u_\infty^\circ \in \Jac^*_\textnormal{rot}(\cC_\circ)$ and thus $u_\infty \in \Jac^*_\textnormal{rot}(\cC)$, a contradiction to $u_\infty \in (\Jac_\textrm{rot}(\cC))^\perp \setminus \{ 0 \}$. Thus,
      \[ V_{u_\infty} \subsetneq \spine \cC, \]
      as claimed.
      
      We may thus invoke $V_{u_\infty}$ in our construction of $u_j$ as counterexamples to the theorem. For each $j$ there exist $w_j \in \cJ_{u_j} \setminus \{ 0 \}$ and $\mbfy_j \in \spine \cC \cap \bar B_1 \setminus B_\theta(V_{u_\infty})$ such that
      \[ G_\cC(w_j; \mbfy_j) \geq 1 - \tfrac12 \Delta_\cC^{<1}. \]
	By renormalizing the $w_j$ as we did the $u_j$, the scaling invariance of decay order, upper-semicontinuity of the decay order under $L^2$-convergence (see Lemma \ref{lemm:doubling.upper.semicontinuity}), and after passing to a subsequence, we obtain $w_\infty \in \cJ_{u_\infty} \setminus \{ 0 \}$ and $\mbfy_\infty \in \spine \cC \cap \bar B_1 \setminus B_\theta(V_{u_\infty})$ so that
	\[ \cN_\cC(w_\infty; \mbfy_\infty) \geq 1 - \tfrac12 \Delta_\cC^{<1}. \]
      Recalling Corollary \ref{Cor_Jac field_Freq Monoton} (iii) and the definition of $\Delta_\cC^{<1}$, it follows that
      \[ \cN_\cC(w_\infty; \mbfy_\infty) \geq 1. \]
      This implies $\mbfy_\infty \in V_{u_\infty}$, a contradiction. This completes the proof.
\end{proof}

\section{Minimal hypersurfaces modeled by Jacobi fields on cylindrical hypercones} \label{sec:nonlinear}

We build on the previous section's linear theory to study a (geometric) $L^2$ decay order over certain cylindrical hypercones $\cC \subset \RR^{n+1}$ for minimal hypersurfaces in $(n+1)$-dimensional Riemannian manifolds.\footnote{In our subsequent geometric applications, we will take $\cC = \cC_\circ \times \RR^k$, with $\cC_\circ$ a minimizing quadratic hypercone (see Appendix \ref{app:quadratic}).}

We fix once and for all a Riemannian manifold $(M^{n+1},g)$ satisfying
\[
|\Rm| + |\nabla \Rm| + |\nabla^2\Rm| \leq C^{-1} \text{ on } M
\]
where $C \geq C(n+1)$, with $C(n+1)$ as in Appendix \ref{Lem:norm-coord-derivatives}. We also take $C$ large enough (depending only on $C(n+1)$) for the monotonicity inequality to hold for the weighted area-ratios 
\[ \Theta^{1}_\Sigma(p, R) := e^{R} \Theta_\Sigma(p, R) = e^R \frac{\Vert \Sigma \Vert(B_R(p))}{\omega_n R^n}, \; R \in (0, 2]; \]
see \cite[\S 17]{Simon:GMT}. We will often work with points in the subset 
\[
	\breve M = \{ p \in M : \injrad(M, g, p) \geq 2 \}.
\]
We also fix the notation
\begin{align*}
	\eta_{p,\tau}(\mbfv) & := \exp_{p}(\tau \mbfv),\\
	g_{p,\tau} & := \tau^{-2} \eta_{p,\tau}^* \; g.
\end{align*}

\subsection{The decay order}

Fix $p \in \breve M$ and endow $T_p M$ with the flat metric $\bar g = g_p$. Let $C \subset T_p M$ be a nonempty closed set that divides $T_p M$ into two components, $\Omega_+$, $\Omega_-$. Denote the signed distance function to $C$ in $\subset T_p M$ as:
\begin{equation}\label{eq:Sdist-signed}
\Sdist_C(\mbfx) : = \begin{cases} d_{\bar g}(\mbfx, C) &  \text{ if } \mbfx \in \Omega_+ \\ - d_{\bar g}(\mbfx, C) & \text{ if } \mbfx \in \Omega_-, \\ 0 & \text { if } \mbfx \in C.  \end{cases}
\end{equation}
In practice, $C$ will be $\spt \cC_p$ for some copy $\cC_p := \iota_p(\cC)$ of $\cC$ into $T_p M$ under a \textbf{linear} isometry $\iota_p : \RR^{n+1} \to (T_p M, g_p)$. We will always abuse notation and denote such $\Sdist_C$ by $\Sdist_{\cC_p}$.

Let $\Sigma \subset M$ be a hypersurface with finite area in compact subsets of $M$. For $\tau \in (0, 2]$, we define
\begin{equation} \label{eq:dist.to.cone}
\mbfd_C(\Sigma; p, \tau) : = \|\Sdist_{\tau^{-1} C}\|_{L^2(\Sigma_{p,\tau}\cap B_1)},
\end{equation}
where $\tau^{-1} C$ denotes a standard rescaling in the vector space $T_p M$ (note that it equals $C$ whenever $C$ is a hypercone in $T_p M$) and 
\[
	\Sigma_{p,\tau} := \eta_{p,\tau}^{-1}(\Sigma \cap B_2(p)) \subset T_p M, \\
\]
It is easy to check that for $\varrho \in (0,1)$, $\tau \in (0, 2]$,
\[
\mbfd_C(\Sigma; p, \varrho\tau) = \varrho^{-\frac n 2- 1}\|\Sdist_{\tau^{-1} C}\|_{L^2(\Sigma_{p,\tau}\cap B_\varrho)}. 
\]

\begin{Def}[cf.\ Definition {\ref{defi:linear.doubling.constant}}]  \label{defi:doubling.constant.nonlinear}
	Let $p$, $S$, $\Sigma$ be as above, and $\hat \kappa > 0$. We define the $\hat \kappa$-adjusted \textbf{decay order} of $\Sigma$ with respect to $S \subset T_p M$ at scale $\tau \in (0, 1]$  as:
\begin{align*}
	\cN^{\hat \kappa}_C(\Sigma; p, \tau) 
    & : = 1 + \frac 12 \log_2(\mbfd_C(\Sigma; p, 2\tau)^2+ \hat \kappa \cdot (2\tau)^2) \\
    & \qquad - \frac 12 \log_2(\mbfd_C(\Sigma; p, \tau)^2 + \hat \kappa \cdot \tau^2).
\end{align*}
\end{Def}

We fix $\hat \kappa := 1$ once and for all. 

\begin{Rem}[cf.\ Remark {\ref{rema:kappa.linear}}] \label{rema:kappa.nonlinear}
	We encourage the first-time reader to read this section keeping in mind that for all applications to $\RR^{n+1}$ we can make these simplifications:
	\begin{itemize}
		\item $\cN^{{1}}_S$ can be replaced with $\cN^0_S$ after suitably extending to $\hat \kappa = 0$ in Definition \ref{defi:doubling.constant.nonlinear}.\footnote{Note, however, that $\cN^0_S(\Sigma; p, \tau)$ suffers from the drawback of being ill-defined if $\Sigma = S$ in $B_\tau(p)$.}
		\item $\Theta^{1}_\Sigma$ can be replaced by the usual monotone density $\Theta_\Sigma$.
		\item Decay bounds $\cN^{\hat \kappa}_C(\Sigma; p, \tau) \leq \Lambda \in \RR_{<2}$ can be relaxed to $\cN^{0}_C(\Sigma; p, \tau) \leq \Lambda \in \RR$.\footnote{The requirement $\Lambda<2$ allows us to disregard the influence of a non-flat background metric $g$; cf.\ Proposition \ref{Prop:nonconc-Jac}. If we were working in $\RR^{n+1}$ this would not be necessary. }
	\end{itemize}
\end{Rem}

\begin{Lem}\label{Lem:freq-bd-below}
It always holds that $\cN^{1}_{C}(\Sigma;p,\tau) \geq - \frac n 2$. 
\end{Lem}
\begin{proof}
Note that
\begin{align*}
	\mbfd_C(\Sigma;p,\tau)^2 & \leq 2^{n+2} \mbfd_C(\Sigma;p,2\tau)^2, \\
	\tau^2 & \leq 2^{n+2}(2\tau)^2.
\end{align*}
This proves the assertion. 
\end{proof} 

\begin{Lem}\label{Lem:freq-lss-2-metric-dies}
If $\cN^{1}_C(\Sigma;p,\tau) \leq \Lambda < 2$ then:\footnote{If we were working in $\RR^{n+1}$ with $\hat \kappa = 0$ (cf. Remark \ref{rema:kappa.nonlinear}), then (i) would need to be omitted.}
\begin{enumerate}
	\item[(i)] $\tau^2 \leq C(\Lambda) \mbfd_C(\Sigma;p,\tau)^2$, and
	\item[(ii)] $\mbfd_C(\Sigma;p,2\tau) \leq 2^{\Lambda-1} \mbfd_C(\Sigma;p,\tau)$.
\end{enumerate}
\end{Lem}
\begin{proof}
We may rearrange $\cN^{1}_C(\Sigma;p,\tau) \leq \Lambda$ into
\[
\mbfd_C(\Sigma;p,2\tau)^2 + 2^2(1-2^{2(\Lambda-2)}) \tau^2 \leq 2^{2\Lambda-2} \mbfd_C(\Sigma;p,\tau)^2. 
\]
This proves the assertion. 
\end{proof}

\subsection{Jacobi fields on cylindrical hypercones} 

We construct Jacobi fields whenever mass-minimizing boundaries in $(M, g)$ suitably converge to suitable cylindrical hypercones, and track properties about their decay rates. Throughout the section we fix
\[ \cC = \cC_\circ \times \RR^k \subset \RR^{n+1}, \]
where $\cC_\circ$ is a regular, strictly stable (see Definition \ref{defi:strictly.stable}), and strictly minimizing (see Definition \ref{def:strictly-minimizing}) hypercone. Note that $\cC \subset \RR^{n+1}$ does \textbf{not} live in $T_p M$, so we will need to study its copies in $T_p M$, i.e. $\iota_p(\cC)$ for \textbf{linear} isometries $\iota : \RR^{n+1} \to (T_p M, g_p)$. All our results also clearly apply to all rotations of such $\cC \subset \RR^{n+1}$.

For notational brevity, we will say $\Sigma$ is minimizing if it is an minimizing boundary in $M$, and will continue to identify $\Sigma$ with its regular part in interpreting $\mbfd_{\iota_p(\cC)}(\Sigma; p, \tau)$ and $\cN^{1}_{\iota_p(\cC)}(\Sigma; p, \tau)$.
 
\begin{Prop}[Jacobi fields via non-concentration]\label{Prop:nonconc-Jac}
Fix $\Lambda \in \RR_{< 2}$.

Consider minimizing $\Sigma_i$, $p_i \in \breve M \cap \spt \Sigma_i$, copies $\cC_i := \iota_{p_i}(\cC) \subset T_{p_i} M$, and $\tau_i > 0$ such that
\begin{enumerate}
\item $\tau_i \to 0$,
\item $\mbfd_{\cC_i}(\Sigma_i;p_i,\tau_i) \to 0$, and
\item $\cN^{1}_{\cC_i}(\Sigma_i;p_i,\tau_i) \leq \Lambda$.
\end{enumerate}
Then:
\begin{enumerate}
	\item[(i)] For each fixed $\varrho \in (0,2)$ we have the following non-concentration inequality
		\[
		\limsup_{i\to\infty} \mbfd_{\cC_i}(\Sigma_i;p_i,\tau_i)^{-2} \int_{(\Sigma_i)_{p_i,\tau_i} \cap B_\varrho}  r^{-2}\Sdist_{\cC_i}^2 <\infty \, ,
		\]
		where $r$ denotes the distance to $\spine \cC_i \subset T_{p_i} M$ with respect to $g_{p_i}$.
	\item[(ii)] For each $i$, there is an open $U_i \subset \cC \cap B_2$ such that the closest-point projection $\Pi$ to $\spt \cC$ is a diffeomorphism from some $V_i \subset \iota_{p_i}^{-1} ((\Sigma_i)_{p_i, \tau_i}) \cap B_2$ onto $U_i$. We may take $U_i$ to be an increasing exhaustion of $\cC \cap B_2$. If $h_i : U_i \to \RR$ is 
		\[ h_i \circ \Pi := \Sdist_{\cC_i} \circ \iota_{p_i} \; (= \Sdist_{\cC}) \text{ on } V_i \]
		(the ``height'' function) then, after passing to a subsequence, we have
		\[
		\mbfd_{\cC_i}(\Sigma_i;p_i,\tau_i)^{-1} h_i \to u \text{ in } C^2_\textnormal{loc}(\cC \cap B_2),
		\]
		for some $u \in \Jac^*(\cC \cap B_2)$ with $\|u\|_{L^2(\cC \cap B_1)} = 1$. 
	\item[(iii)] After passing to a further subsequence so that (recall Lemma \ref{Lem:freq-lss-2-metric-dies})
		\[ \mbfd_{\cC_i}(\Sigma_i;p_i,\tau_i)^{-2} \tau_i^2 \to \kappa \in [0,C(\Lambda)), \]
		it holds that $\cN^{1}_{\cC_i}(\Sigma_i; p_i, \varrho \tau_i) \to G^\kappa _\cC(u;\varrho)$ for any $\varrho \in (0, 1)$ and $G_{\cC}^\kappa (u;1) \leq \Lambda$.
\end{enumerate}
\end{Prop}
\begin{proof}
By combining Simon's non-concentration estimates (see Theorem \ref{Thm_Simon's L^2 Noncon}) with Lemma \ref{Lem:norm-coord-derivatives} we have that for $\varrho \in (0,2)$ fixed, 
\[
 \int_{(\Sigma_i)_{p_i,\tau_i} \cap B_\varrho}  r^{-2}\Sdist_{\cC_i}^2 \leq C(\varrho) \left( \mbfd_{\cC_i}(\Sigma_i;p_i,2\tau_i)^2 + \tau_i^2 \right). 
\]
Using Lemma \ref{Lem:freq-lss-2-metric-dies}, we find 
\[
\int_{(\Sigma_i)_{p_i,\tau_i} \cap B_\varrho}  r^{-2}\Sdist_{\cC_i}^2 \leq C(\varrho,\Lambda) \, \mbfd_{\cC_i}(\Sigma_i;p_i,\tau_i)^2.
\]
This proves {(i)}, the non-concentration inequality.

For {(ii)}, the well-definedness of $h_i$ is standard; namely, it follows from how (2) implies that $\iota_{p_i}^{-1} ((\Sigma_i)_{p_i, \tau_i}) \cap B_2 \rightharpoonup \cC \cap B_2$ together with Allard's theorem \cite{Allard:first-variation}. Then, again using Lemma \ref{Lem:freq-lss-2-metric-dies} and standard elliptic estimates we find that the rescaling
\[ \mbfd_{\cC_i}(\Sigma_i;p_i,\tau_i)^{-1} h_i \]
converges in $C^2_\textrm{loc}(\cC\cap B_2)$ to a Jacobi field $u \in \Jac(\cC \cap B_2)$. We used Lemmas \ref{Lem:freq-lss-2-metric-dies}, \ref{Lem:norm-coord-derivatives} to see that the difference between the Riemannian metrics $g_{p_i,\tau_i}$ (with respect to which the graph of $h_i$ is minimal) and $g_{p_i}$ implies that its mean curvature satisfies the following estimates 
$$H_{\Sigma_i,g_{p_i}} = H_{\Sigma_i,g_{p_i,\tau_i}}+ O(\tau^2) =  O(\tau^2)\, ,$$
uniformly on compact subsets of $B_2 \setminus \sing \cC$. Thus, $H_{\Sigma_i,g_{p_i}}$ is negligible compared to the normalizing factor $\mbfd_{\cC_i}(\Sigma_i;p_i,\tau_i)$ in the $C^2_\textrm{loc}(\cC\cap B_2)$ convergence described above. 

By covering $\spine \cC \cap B_2$ by cubes whose doubling is contained in $B_2$ the estimates of the graphical function Simon's non-concentration estimates (see Theorem \ref{Thm_Simon's L^2 Noncon}) imply that $u \in \Jac^*(\cC \cap B_2)$. Note that the non-concentration estimate gives
\begin{equation} \label{eq:nonconc-Jac-L2-limit}
\lim_{i\to\infty} \mbfd_{\cC_i}(\Sigma_i;p_i,\tau_i)^{-2}\mbfd_{\cC_i}(\Sigma_i;p_i,\varrho\tau_i)^2 = \varrho^{-n-2} \int_{\cC \cap B_\varrho} u^2
\end{equation}
for all $\varrho \in (0,2)$. This proves that $\|u\|_{L^2(B_1)} = 1$, completing the proof of {(ii)}.

Finally, we relate the non-linear and linear decay orders. Lemma \ref{Lem:freq-lss-2-metric-dies} implies that after passing to a subsequence, we have
\[
\kappa : = \lim_{i\to\infty} \frac{\tau_i^2}{\mbfd_{\cC_i}(\Sigma_i;p_i,\tau_i)^2}  \in [0,C(\Lambda)). 
\]
Combined with \eqref{eq:nonconc-Jac-L2-limit} this gives 
\begin{align*}
& \lim_{i\to\infty} (\mbfd_{\cC_i}(\Sigma_i;p_i,\tau_i)^{2} + \tau_i^2)^{-1}(\mbfd_{\cC_i}(\Sigma_i;p_i,\varrho \tau_i)^2 + (\varrho \tau_i)^2) \\
& = (1+\kappa)^{-1} \varrho^{-n-2} \int_{\cC \cap B_\varrho} u^2 + (1+\kappa)^{-1} \kappa \varrho^2
\end{align*}
for all $\varrho \in (0,2)$. In particular,
\[
\cN^{1}_{\cC_i}(\Sigma_i; p_i, \varrho\tau_i) \to G_\cC^\kappa(u;\varrho) \text{ for all } \varrho \in (0, 1). 
\]
Finally, we have that 
\[
\mbfd_{\cC_i}(\Sigma_i;p_i,2\tau_i)^2 \geq \varrho^{n+2} \mbfd_{\cC_i}(\Sigma_i;p_i,2\varrho \tau_i)^2 
\]
for any $\varrho \in (0,1)$ so 
\begin{align*}
& \liminf_{i\to\infty} (\mbfd_{\cC_i}(\Sigma_i;p_i,\tau_i)^{2} + \tau_i^2)^{-1}(\mbfd_{\cC_i}(\Sigma_i;p_i,2\tau_i)^2 + (2\tau_i)^2) \\
& \geq (1+\kappa)^{-1}  2^{-n-2}  \int_{\cC \cap B_{2\varrho}} u^2 + (1+\kappa)^{-1} \kappa 2^2.
\end{align*}
This gives $G_\cC^\kappa(u;1) \leq \Lambda$ after sending $\varrho \nearrow 1$. 
\end{proof}

\subsection{Decay order monotonicity} 

We continue to denote
\[ \cC = \cC_\circ \times \RR^k \subset \RR^{n+1}, \]
where $\cC_\circ$ is a regular, strictly stable, strictly minimizing hypercone, and to study its isometric copies in various $T_p M$'s.

Recall that, for $u \in \Jac^*(\cC \cap B_R) \setminus \{ 0 \}$, the decay order $\varrho \mapsto G^\kappa_\cC(u;\varrho)$ is non-decreasing in $(0, R/2)$ by Corollary \ref{Cor_Jac field_Freq Monoton}, and if it is constant on an interval then $u$ is homogeneous of degree $G^\kappa_\cC(u;\cdot) \in \Gamma^*(\cC)$.

\begin{Lem}[Decay order monotonicity]\label{Lem_Refined_Freq Almst Monon}
Fix $\delta > 0$. There exist $\eps$, $\lambda > 0$, $L \in \ZZ_{\geq 0}$ depending on $\cC$, $\delta$ with the following property.

Consider minimizing $\Sigma$, $p \in \breve M \cap \spt \Sigma$, a copy $\cC_p := \iota_p(\cC) \subset T_p M$, $\ell \in \ZZ_{\geq 0}$, and $\Lambda \in \RR_{< 2} \setminus B_\delta(\Gamma^*(\cC))$ such that:
\begin{enumerate}
\item $\ell \geq L$,
\item $\mbfd_{\cC_p}(\Sigma;p,2^{-\ell}) \leq \eps$, and
\item $\cN^{1}_{\cC_p}(\Sigma;p,2^{-\ell}) \geq \Lambda$.
\end{enumerate} 
Then,  $\cN^{1}_{\cC_p}(\Sigma;p,2^{-\ell+1}) \geq \Lambda + \lambda$.
\end{Lem}
\begin{proof}
Suppose not. Then, there are $\Sigma_i$, $p_i \in \breve M \cap \spt \Sigma_i$, $\cC_i := \iota_{p_i}(\cC) \subset T_{p_i}M$, $\ell_i \in \ZZ_{\geq 0}$, and $\Lambda_i \in \RR_{< 2} \setminus B_\delta(\Gamma^*(\cC))$ as above so that
\begin{enumerate}
\item[(1')] $\ell_i \to \infty$,
\item[(2')] $\mbfd_{\cC_i}(\Sigma_i;p_i,2^{-\ell_i}) \to 0 $, 
\item[(3')] $\cN^{1}_{\cC_i}(\Sigma_i;p_i,2^{-\ell_i}) \geq \Lambda_i$,
\end{enumerate}
but
\begin{equation} \label{eq:refined-freq-almost-mon-contr}
	\cN^{1}_{\cC_i}(\Sigma_i;p_i,2^{-\ell_i+1}) \leq \Lambda_i + \lambda_i.
\end{equation}
with $\lambda_i \to 0$. Passing to a subsequence, we may assume that
\begin{equation} \label{eq:refined-freq-almost-mon-limit}
	\Lambda := \lim_i \Lambda_i \in \RR_{\leq 2} \setminus B_\delta(\Gamma^*(\cC)).
\end{equation}
Since $\lambda_i \to 0$, \eqref{eq:refined-freq-almost-mon-limit} and Remark \ref{rema:gamma.star.c} imply that for sufficiently large $i$,
\begin{equation} \label{eq:refined-freq-almost-mon-limit-bound}
	\Lambda_i + \lambda_i \leq \Lambda + \tfrac12 \delta \in \RR_{< 2}
\end{equation}

Note that (1'), (2'), \eqref{eq:refined-freq-almost-mon-contr}, and \eqref{eq:refined-freq-almost-mon-limit-bound} imply that Proposition \ref{Prop:nonconc-Jac} is applicable for all sufficiently large $i$ with $\tau_i := 2^{-\ell_i+1}$ and $\Lambda + \tfrac12 \delta$ in place of $\Lambda$. After perhaps passing to a subsequence, we obtain a Jacobi field $u \in \Jac^*(\cC \cap B_2)$ satisfying, for some $\kappa \geq 0$,
\[
\Lambda \leq \lim_{i\to\infty}\cN^{1}_{\cC_i}(\Sigma_i;p_i,2^{-\ell_i}) = G^\kappa_{\cC}(u;2^{-1}) \leq G^\kappa_\cC(u;1) \leq \Lambda,
\]
where we used (3') in the first step and the monotonicity of frequency (Corollary \ref{Cor_Jac field_Freq Monoton}) in the penultimate step. Thus, equality holds in the monotonicity of frequency so $u$ is homogeneous with degree $\Lambda$. This contradicts the inclusion statement in \eqref{eq:refined-freq-almost-mon-limit}.
\end{proof}

\begin{Lem}[Implication of far-from-linear decay] \label{Lem:close-to-other-cone-freq-close-1}
Fix $\delta$, $\eps > 0$. There exist $\eta > 0$, $L \in \ZZ_{\geq 0}$  depending on $\cC$, $\delta$, $\eps$ with the following property. 

Consider minimizing $\Sigma$, $p \in \breve M \cap \spt \Sigma$, a copy $\cC_p := \iota_p(\cC) \subset T_p M$, and $\ell \in \ZZ_{\geq 0}$ such that:
\begin{enumerate}
\item $\ell \geq L$,
\item $\Theta^{1}_{\Sigma}(p, 2^{-\ell+1}) - \Theta^{1}_\Sigma(p, 2^{-\ell}) \leq \eta$, and
\item $|\cN^{1}_{\cC_p}(\Sigma;p,2^{-\ell}) - 1| \geq \delta$.
\end{enumerate}
Then, $\mbfd_{\cC_p}(\Sigma;p,2^{-\ell+k}) < \eps$ for each $k \in \{-1,0,1\}$. 
\end{Lem}
\begin{proof}
If not, there are $\Sigma_i$, $p_i \in \breve M \cap \spt \Sigma_i$, $\cC_i := \iota_{p_i}(\cC) \subset T_{p_i}M$, and $\ell_i \in \ZZ_{\geq 0}$ so that 
\begin{enumerate}
\item[(1')] $\ell_i \to \infty$,
\item[(2')] $\Theta^{1}_{\Sigma_i}(p_i,2^{-\ell_i+1}) - \Theta^{1}_{\Sigma_i}(p_i,2^{-\ell_i}) \to 0$, 
\item[(3')] $|\cN^{1}_{\cC_i}(\Sigma_i;p_i,2^{-\ell_i}) - 1| \geq \delta$, 
\end{enumerate}
but, for some fixed $k \in \{ -1, 0, 1 \}$,
\begin{equation} \label{eq:close-to-other-close-freq-close-1-contr}
	\mbfd_{\cC_i}(\Sigma_i;p_i,2^{-\ell_i+k})\geq \eps.
\end{equation} 

Let
\[ \tilde\Sigma_i := (\Sigma_i)_{p_i,2^{-\ell_i+k}} \subset T_{p_i}M. \]
By (1') and (2') we can pass to a subsequence so that $\iota_{p_i}^{-1} (\tilde\Sigma_i) \rightharpoonup \tilde \cC \subset \RR^{n+1}$, a nonempty minimizing hypercone. Since $\Sdist_\cC$ is $1$-Lipschitz, we have
\begin{equation} \label{eq:close-to-other-close-freq-close-1-limit}
	\mbfd_{\cC_i}(\Sigma_i;p_i, \varrho 2^{-\ell_i}) \to \mbfd_\cC(\tilde\cC; \orig, \varrho) =  \mbfd_\cC(\tilde\cC;\orig,1) \neq 0
\end{equation}
for any fixed $\varrho > 0$; the penultimate step is a consequence of the dilation invariance of $\cC$ and $\tilde\cC$, and the last inequality follows since $\tilde\cC \neq \cC$ in view of \eqref{eq:close-to-other-close-freq-close-1-contr}.

Finally, \eqref{eq:close-to-other-close-freq-close-1-limit} and the definition of decay order imply
\[
\cN^{1}_{\cC_i}(\Sigma_i;p_i,2^{-\ell_i}) \to 1 + \tfrac12 \log_2 \frac{\mbfd_{\cC}(\tilde \cC; \orig, 2)}{\mbfd_{\cC}(\tilde \cC; \orig, 1)} = 1,
\]
contradicting (3').
\end{proof}

\begin{Lem}[Decay is at least almost linear]\label{Lem:freq-not-too-low}  
Fix $\delta > 0$. There exist $\eta > 0$, $L \in \ZZ_{\geq 0}$ depending on $\cC$, $\delta$ with the following property.

Consider minimizing $\Sigma$, $p \in \breve M \cap \spt \Sigma$, and $\ell \in \ZZ_{\geq 0}$ such that:
\begin{enumerate}
\item $\ell \geq L$, and
\item  $\Theta^{1}_{\Sigma}(p,2^{-\ell+1}) - \Theta_\Sigma(p) \leq \eta$.
\end{enumerate}
Then, $\cN^{1}_{\iota_p(\cC)}(\Sigma;p,2^{-\ell}) > 1-\delta$ for all linear isometries $\iota_p : \RR^{n+1} \to (T_p M, g_p)$.
\end{Lem}
\begin{proof}
Take $\eps$, $\lambda$, $L$ so that decay order monotonicity (Lemma \ref{Lem_Refined_Freq Almst Monon}) holds with $\delta$ as given. Then take $\eta > 0$ and possibly increase $L$ so that Lemma \ref{Lem:close-to-other-cone-freq-close-1} holds with $\delta$ and $\eps$. Let $\iota_p$ be any linear isometry.

Let us first show that 
\begin{equation} \label{eq:freq-not-too-low-some-ell}
	\cN^{1}_{\cC_p}(\Sigma; p, 2^{-\ell'}) > 1-\delta \text{ for some } \ell' \geq \ell.
\end{equation}
If \eqref{eq:freq-not-too-low-some-ell} failed, then iterating it would yield
\[
\mbfd_{\cC_p}(\Sigma;p,2^{-\ell})^2 + 2^{-2\ell} \leq 2^{-\delta(\ell'-\ell)}(\mbfd_{\cC_p}(\Sigma;p,2^{-\ell'})^2 + 2^{-2\ell'}) \leq C(n) 2^{-\delta(\ell'-\ell)},
\]
and sending $\ell'\to\infty$ would yield a contradiction. This proves \eqref{eq:freq-not-too-low-some-ell}.

Let $\ell'$ be such that \eqref{eq:freq-not-too-low-some-ell} holds. If $\ell' = \ell$, we are done. Otherwise, we apply Lemma \ref{Lem:close-to-other-cone-freq-close-1} at scale $2^{-\ell'}$ to get $\mbfd_{\cC_p}(\Sigma;p,2^{-\ell'}) \leq \eps$. Lemma \ref{Lem_Refined_Freq Almst Monon} applied at scale $2^{-\ell'}$ gives
\[ \cN^{1}_{\cC_p}(\Sigma;p,2^{-\ell'+1}) > 1- \delta + \lambda > 1-\delta, \]
so \eqref{eq:freq-not-too-low-some-ell} holds with $\ell'+1$, too. Iterating upwards, we can get to $\ell' = \ell$. This completes the proof.
\end{proof}

\subsection{Fast decay analysis} \label{sec:nonlinear.hf}

We continue to denote
\[ \cC = \cC_\circ \times \RR^k \subset \RR^{n+1}, \]
where $\cC_\circ$ is a regular, strictly stable, strictly minimizing hypercone. 
We now define the isometry-independent quantities
\begin{align*}
\overline{\cN}^{1}_{\cC}(\Sigma;p,\tau) & := \sup \{ \cN^{1}_{\iota_p(\cC)}(\Sigma;p,\tau) : \text{ all isometries } \iota : \RR^{n+1} \to (T_p M, g_p) \}, \\
\overline{\cN}^{1}_{\cC}(\Sigma;p) & := \limsup_{\ell\to\infty} \overline{\cN}^{1}_{\cC}(\Sigma; p, 2^{-\ell}).
\end{align*}
It follows from Lemma \ref{Lem:freq-not-too-low} that $\overline{\cN}^{1}_{\cC}(\Sigma; p) \geq 1$ when $\Sigma$ is minimizing and $p \in \spt \Sigma$.\footnote{In fact, Lemma \ref{Lem:freq-not-too-low} implies $\liminf_{\ell \to \infty} \cN^{1}_{\iota_p(\cC)}(\Sigma; p, 2^{-\ell}) \geq 1$ for each fixed isometry $\iota$.}  We call $\overline{\cN}^{1}_{\cC}(\Sigma;p)>1$ the \textbf{fast} decay case, and $\overline{\cN}^{1}_{\cC}(\Sigma;p)=1$ the \textbf{slow} decay case.

The following lemma is an important observation in our study of fast decay. Recall the definition of $\Delta_{\cC}^{>1}$ from Remark \ref{rema:gamma.star.c}.

\begin{Lem}[Propagating fast decay to bigger scales]\label{Lem_Refined_Uniq tang cone if limsup N(Sigma, tau)>1}
Fix $\delta \in (0, \Delta_\cC^{>1})$. There exist $\eta>0$, $L\in\ZZ_{\geq 0}$ depending on $\cC$, $\delta$ with the following property. 

Consider minimizing $\Sigma$, $p \in \breve M \cap \spt \Sigma$, a copy $\cC_p := \iota_p(\cC) \subset T_p M$, and $L_0, L_1 \in \ZZ_{\geq 0}$ such that: 
\begin{enumerate}
\item $L_1 > L_0 \geq L$,
\item $\Theta^{1}_{\Sigma}(p,2^{-\ell}) - \Theta_\Sigma(p) \leq \eta$ at $\ell = L_0$, and
\item $\cN^{1}_{\cC_p}(\Sigma;p,2^{-\ell})\geq 1+\delta$ at $\ell = L_1$. 
\end{enumerate}
Then, $\cN^{1}_{\cC_p}(\Sigma;p,2^{-\ell}) \geq 1+ \delta$ for all $\ell \in \ZZ_{\geq 0}$, $L_0 \leq \ell \leq L_1$.
\end{Lem}
\begin{proof}
Take $\eps$, $\lambda$, $L$ so that decay order monotonicity (Lemma \ref{Lem_Refined_Freq Almst Monon}) holds with $\delta$ as given. Then take $\eta > 0$ and possibly increase $L$ so that Lemma \ref{Lem:close-to-other-cone-freq-close-1} holds with $\delta$ and $\eps$.

Assuming the assertion fails, we can choose $\ell \in \{L_0, \dots,L_1-1\}$ to be the maximal (so the scale $2^{-\ell}$ is minimal) counterexample. This gives
\begin{equation} \label{eq:refined-uniq-tang-cone-fast-decay-contr}
	\cN^{1}_{\cC_p}(\Sigma;p,2^{-\ell}) < 1 + \delta \leq \cN^{1}_{\cC_p}(\Sigma;p,2^{-\ell-1}).
\end{equation}
By Lemma \ref{Lem:close-to-other-cone-freq-close-1}, applied at scale $2^{-\ell-1}$ by virtue of the right hand side of \eqref{eq:refined-uniq-tang-cone-fast-decay-contr}, we find that $\mbfd_{\cC_p}(\Sigma;p,2^{-\ell-1})\leq \eps$. Lemma \ref{Lem_Refined_Freq Almst Monon} applied at scale $2^{-\ell-1}$ yields
\[
\cN^{1}_{\cC_p}(\Sigma;p,2^{-\ell}) \geq 1 + \delta + \lambda > 1 + \delta,
\]
which contradicts the left hand side of \eqref{eq:refined-uniq-tang-cone-fast-decay-contr}.
\end{proof}

\begin{Cor}[Fast decay, qualitatively]\label{Cor:unif-high-freq-qualitative}
Consider minimizing $\Sigma$, $p \in \spt \Sigma$, and assume that $\overline{\cN}^{1}_{\cC}(\Sigma; p) > 1$. 

Then, there exists a fixed linear isometry $\iota_p : \RR^{n+1} \to (T_p M, g_p)$ such that $\cC_p := \iota_p(\cC)$ satisfies:
\begin{enumerate}
	\item[(i)] $\liminf_{\ell \to \infty} \cN^{1}_{\cC_p}(\Sigma; p, 2^{-\ell}) \geq 1 + \Delta_\cC^{>1}$, and
	\item[(ii)] $\cC_p$ is the unique tangent cone to $\Sigma$ at $p$.
\end{enumerate}
\end{Cor}
\begin{proof}
	Claim (i) follows from iterating Lemma \ref{Lem_Refined_Uniq tang cone if limsup N(Sigma, tau)>1} and Lemma \ref{Lem_Refined_Freq Almst Monon}. Claim (ii) follows from (i) and Lemma \ref{Lem:close-to-other-cone-freq-close-1}.
\end{proof}

A straightforward quantitative consequence is:

\begin{Cor}[Fast decay, quantitative] \label{Cor:unif-high-freq-quantitative}
Fix $\delta \in (0, \Delta_{\cC}^{>1})$. There exist $\eta>0$, $L \in \ZZ_{\geq 0}$ depending on $\cC$, $\delta$ with the following property. 

Consider minimizing $\Sigma$, $p \in \breve M \cap \spt \Sigma$, and $L_0 \in \ZZ_{\geq 0}$ such that:
\begin{enumerate}
	\item $L_0 \geq L$,
	\item $\Theta^{1}_\Sigma(p,2^{-\ell}) - \Theta_\Sigma(p)\leq \eta$ at $\ell = L_0$, and
	\item $\overline{\cN}^{1}_{\cC}(\Sigma; p) > 1$.
\end{enumerate}
Let $\iota_p : \RR^{n+1} \to (T_p M, g_p)$ be a linear isometry so that $\cC_p := \iota_p(\cC)$ is the unique tangent cone to $\Sigma$ at $p$ (per Corollary \ref{Cor:unif-high-freq-qualitative}). Then, for all $\ell \in \ZZ_{\geq 0}$, $\ell \geq L_0$,
\begin{enumerate}
	\item[(i)] $\cN^{1}_{\cC_p}(\Sigma;p,2^{-\ell}) \geq 1+ \Delta_\cC^{>1} - \delta$, and
	\item[(ii)] $\mbfd_{\cC_p}(\Sigma;p,2^{-\ell}) \leq C(n) \cdot 2^{-(\Delta_\cC^{>1}- \delta)(\ell-L_0)}$.
\end{enumerate}
\end{Cor}

\subsection{Slow decay analysis} \label{sec:nonlinear.lf}

\begin{Lem}[Jacobi fields in $\Jac^*_1 \cap \Jac_{\textnormal{rot}}^\perp$] \label{Lem:ortho-rot}
Consider minimizing $\Sigma_i$, $p_i \in \breve M \cap \spt \Sigma_i$, isometric copies $\cC_i := \iota_{p_i}(\cC) \subset T_{p_i}M$, $L_{0,i}, L_{1,i}, L_{2,i} \in \ZZ_{\geq 0}$,\footnote{By monotonicity, one can always take $L_{0,i} := L_{1,i}$. However, these two variables will be decoupled in our applications (see Remark \ref{rema:L0.L1}) so we have chosen to decouple them here too. \label{footnote:L0.L1}} and $\delta_i \to 0$ such that:
\begin{enumerate}
\item $L_{1,i} \geq L_{0,i} \to \infty$,
\item $L_{2,i} - L_{1,i} \to\infty$,
\item $\mbfd_{\cC_i}(\Sigma_i;p_i,2^{-L_{2,i}}) \to 0$, 
\item $\Theta^{1}_{\Sigma_i}(p_i,2^{-L_{0,i}}) - \Theta_{\Sigma_i}(p_i) \to 0$, and
\item $\overline{\cN}^{1}_{\cC}(\Sigma_i;p_i,2^{-\ell})\leq 1+\delta_i$ for all $\ell \in \ZZ_{\geq 0}$, $\ell \geq L_{1,i}$.
\end{enumerate}
Then, there exist rotations $\cC_i' = \iota_{p_i}'(\cC)$ of $\cC_i$ with $\mbfd_{\cC_i'}(\Sigma_i;p,2^{-L_{2,i}}) \to 0$ so that Proposition \ref{Prop:nonconc-Jac} applies with $\cC_i'$ in place of $\cC_i$ at all scales and yields a Jacobi field $u \in \Jac^*_1(\cC)$ which is $L^2(\cC \cap B_1)$-orthogonal to $\Jac_{\textnormal{rot}}(\cC)^\perp$.
\end{Lem}
\begin{proof}
We first slightly regularize $\Sdist$. 
Let $\param>0$ be such that $\Sdist_{\cC_\circ}$ is a smooth function in $\{x\in \RR^{n_\circ+1}: |\Sdist_{\cC_\circ}(x)|< 2\param |x|\}$; And let $\chi: \RR\to \RR$ be an odd nondecreasing function such that $\chi(s) = s$ for $|s|\leq \param$ and $\chi(s) = 2\param$ for $s\geq 2\param$ and $\chi(s) = -2\param$ for $s\leq -2\param$. 

For a fixed copy $\cC'=\iota(\cC)\subset T_pM$, we denote by $r_{\cC'}(\mbfx)$ the distance from $\mbfx\in T_pM$ to $\spine(\cC')$. 
We define the regularized signed distance function $\Rdist_\cC: T_pM \to \RR$ by \[
   \Rdist_{\cC'}(\mbfx) := r_{\cC'}(\mbfx)\cdot\chi\left(\Sdist_{\cC'}(\mbfx)/r_{\cC'}(\mbfx)\right)\,.
\]
Note that $(\Rdist_{\cC'})^2$ is a 2-homogeneous $C^1$ function on $T_pM$ and there exists some constant $\hat C(n, \param)>1$ such that 
\begin{align}
   & \Rdist_{\cC'} = \Sdist_{\cC'} & & \text{on } \{\mbfx: |\Sdist_{\cC'}(\mbfx)|\leq \param\, r_{\cC'}(\mbfx)\}\,, \label{Equ_Rdist = Sdist near cC'} \\ 
   \hat C^{-1} \leq\; &\Sdist_{\cC'}(\mbfx)^{-1}\cdot\Rdist_{\cC'}(\mbfx) \leq \hat C ,  & & \forall\, \mbfx\in T_pM\setminus \cC'\,. \label{Equ_Rdist approx Sdist}
\end{align}

Choose $\cC_i' \in \Rot(\cC_i)$ minimizing 
\begin{equation} \label{eq:ortho-rot-min}
	\Rot(\cC_i) \ni \cC' \mapsto 
    \hat\mbfd_{\cC'}(\Sigma_i; p_i, 2^{-L_{2,i}})^2
    := \int_{ (\Sigma_i)_{p_i,L_{2,i}} \cap B_1} (\Rdist_{\cC'})^2.
\end{equation}
Clearly, (3) and \eqref{Equ_Rdist approx Sdist} implies
\begin{equation} \label{eq:ortho-rot-0}
	\hat\mbfd_{\cC_i'}(\Sigma_i;p_i,2^{-L_{2,i}}),\ \mbfd_{\cC_i'}(\Sigma_i;p_i,2^{-L_{2,i}}) \leq \hat C^2 \mbfd_{\cC_i}(\Sigma_i;p_i,2^{-L_{2,i}}) \to 0.
\end{equation}
Enlarging $\delta_i$ if necessary while still preserving $\delta_i\to0$, we can use (1), (4), and (5) to invoke Lemma \ref{Lem:freq-not-too-low}  to conclude:
\begin{equation} \label{eq:ortho-rot-1}
	|\cN^{1}_{\cC'_i}(\Sigma_i;p_i,2^{-\ell})- 1| \leq \delta_i \text{ for all } \ell > L_{1,i}.
\end{equation}
In view of (2), (3), \eqref{eq:ortho-rot-0}, and \eqref{eq:ortho-rot-1}, we can apply Lemma \ref{Lem:freq-lss-2-metric-dies} and Proposition \ref{Prop:nonconc-Jac} at all scales and obtain $u \in \Jac^*(\cC)$ and $\kappa \geq 0$ satisfying 
\[
G_{\cC}^\kappa(u;2^k) = \lim_{i\to\infty} \cN^{1}_{\cC_i'}(\Sigma_i;p_i,2^{-L_{2,i}+k}) = 1
\]
for all $k$. Thus, $u \in \Jac^*_1(\cC)$ and $\kappa = 0$ by Corollary \ref{Cor_Jac field_Freq Monoton}.

It remains to show that $u$ is also $L^2(\cC \cap B_1)$-orthogonal to $\Jac_{\textnormal{rot}}(\cC)$. We will use the minimizing nature of $\cC_i'$ for \eqref{eq:ortho-rot-min}. To that end note that, for all $\mbfA \in \mathfrak{so}({n+1})$,
\[
\frac{d}{ds}\Big|_{s=0} \Rdist_{e^{s\mbfA}\cC_i'}(\mbfx) =  \frac{d}{ds}\Big|_{s=0} \Rdist_{\cC_i'}(e^{-s\mbfA}\mbfx) = - \langle \nabla \Rdist_{\cC_i'}(\mbfx),\mbfA \mbfx\rangle
\]
at all points $\mbfx\notin\spine(\cC_i')$, where $\Rdist_{\cC_i'}$ is differentiable; above, $\langle \cdot, \cdot \rangle$ and $\nabla$ are taken in $T_{p_i} M$ with respect to $g_{p_i}$. Since $(\Rdist_{\cC_i})^2$ is $C^1$, the minimizing nature of $\cC_i'$ in \eqref{eq:ortho-rot-min} implies
\begin{equation}\label{eq:ortho-rotation-fixing-cone}
\int_{ (\Sigma_i)_{p_i,L_{2,i}} \cap B_1} \langle \nabla \Rdist_{\cC_i'}(\mbfx), \mbfA \mbfx  \rangle \, \Rdist_{\cC_i'}(\mbfx)\ = 0
\end{equation}
for any $A \in \mathfrak{so}(n+1)$. We also recall that Proposition \ref{Prop:nonconc-Jac} and \eqref{Equ_Rdist approx Sdist} give the non-concentration estimate
\[
\limsup_{i\to\infty} \hat\mbfd_{\cC_i'}(\Sigma_i;p_i,L_{2,i})^{-2} \int_{(\Sigma_j)_{p_i,L_{2,i}} \cap B_1}  r^{-2}\Rdist_{\cC_i'}^2 <\infty \, .
\]
Since $ \langle \nabla \Rdist_{\cC_i'}(\mbfx) , \mbfA \mbfx  \rangle$ limits to $\langle \mbfA \mbfx,\nu_\cC\rangle$ in $C^2_\textrm{loc}(\cC)$ we can combine this with \eqref{eq:ortho-rotation-fixing-cone} and \eqref{Equ_Rdist = Sdist near cC'} to conclude that 
\[
\int_{\cC \cap B_1} \langle \mbfA \mbfx,\nu_\cC\rangle \, u = 0. 
\]
for all $\mbfA \in \mathfrak{so}(n+1)$, completing the proof. 
\end{proof}

We now impose the additional assumption that $\cC_\circ$ is strongly integrable (cf.\ Definition \ref{defi:strongly.integrable}).

\begin{Thm} \label{theo:freq-1-subspace}
        Suppose that the factor $\cC_\circ$ of $\cC$ is also strongly integrable.\footnote{So far, $\cC_\circ$ had only been assumed to be a regular, strictly stable, and strictly minimizing hypercone.} Fix $\theta>0$. There exist $\eps$, $\delta$, $\eta>0$, $L \in \ZZ_{\geq 0}$, depending on $\cC$, $\theta$, with the following properties. 

Consider minimizing $\Sigma$, $p \in \breve M \cap \spt \Sigma$, a copy $\cC_p := \iota_p(\cC) \subset T_p M$, and $L_0, L_1, L_2 \in \ZZ_{\geq 0}$\footnote{Again, one could have taken $L_0 := L_1$, but see footnote \ref{footnote:L0.L1}.} such that:
\begin{enumerate}
\item $L_1 \geq L_0 \geq L$,
\item $L_2 - L_1 \geq L$,
\item $\mbfd_{\cC_p}(\Sigma; p, 2^{-\ell}) \leq \eps$ at $\ell = L_2$, 
\item $\Theta^{1}_\Sigma(p,2^{-\ell}) - \Theta_\Sigma(p) \leq \eta$ at $\ell = L_0$, and
\item $\overline{\cN}^{1}_{\cC}(\Sigma;p,2^{-\ell}) \leq 1+\delta$ for all $\ell \in \ZZ_{\geq 0}$, $\ell \geq L_1$.
\end{enumerate}
There exists a proper linear subspace $V\subsetneq \spine \cC$, so that if (4), (5) hold with $\Sigma^+$, $p^+$ in place of $\Sigma$, $p$, and
\begin{enumerate}
\item[(6)] $p^+ \in B_{2^{-L_2}}(p)$, 
\item[(7)] $\Sigma^+$ does not cross $\Sigma$ smoothly,
\end{enumerate}
then $p^+ \in (\eta_{p,2^{-L_2}} \circ \iota_p)(B_\theta(V))$. 
\end{Thm}

\begin{Rem} \label{rema:freq-1-subspace}
    Since $\Sigma$ doesn't cross itself, we are allowed to take $\Sigma^+ = \Sigma$ in Theorem \ref{theo:freq-1-subspace} above. The theorem's conclusion gives novel information in this case as well. 
\end{Rem}

\begin{proof}[Proof of Theorem \ref{theo:freq-1-subspace}]
Consider a sequence of counterexamples. That is, $\Sigma_i$, $p_i \in \breve M \cap \spt \Sigma_i$, $p_i \in M$, $\cC_i := \iota_{p_i}(\cC) \subset T_{p_i}M$, $\delta_i \to 0$, $L_{0,i}, L_{1,i}, L_{2,i}, L_i \to \infty$ contradicting the statement of the theorem. In particular, we are assuming that:
\begin{enumerate}
	\item[(1')] $L_{1,i} \geq L_{0,i} \geq L_i$,
	\item[(2')] $L_{2,i} - L_{1,i} \geq L_i$, 
	\item[(3')] $\mbfd_{\cC_i}(\Sigma_i; p_i, 2^{-L_{2,i}}) \to 0$, 
	\item[(4')] $\Theta^{1}_{\Sigma_i}(p_i, 2^{-L_{0,i}}) - \Theta_{\Sigma_i}(p_i) \to 0$, and
	\item[(5')] $\overline{\cN}^{1}_\cC(\Sigma_i; p_i, 2^{-\ell}) \leq 1 + \delta_i$ for all $\ell \geq L_{1,i}$.
\end{enumerate}
By $L_i \to \infty$ and (1'), (2'), (3'), (4'), and (5') we can apply Lemma \ref{Lem:ortho-rot} to replace $\cC_i$ with a nearby element of $\Rot(\cC_i)$, not relabeled, so that Proposition \ref{Prop:nonconc-Jac} and Corollary \ref{Cor_Jac field_Freq Monoton} applied at all scales produce a Jacobi field $u \in  \Jac_1^*(\cC) \cap (\Jac_\textnormal{rot}(\cC))^\perp \setminus \{ 0 \}$. For later use, we introduce the notation
\[ \tau_i := 2^{-L_{2,i}}, \]
and, per Proposition \ref{Prop:nonconc-Jac} (ii) applied at scale $\tau_i$, we denote
\[ h_{p_i,\tau_i}^{\Sigma_i} \circ \Pi := \Sdist_{\cC_i} \circ \iota_{p_i} = \Sdist_{\cC} \]
on a subset of $\iota_{p_i}^{-1}((\Sigma_i)_{p_i, \tau_i})$ on which the closest-point projection $\Pi$ to $\spt \cC$ in $\RR^{n+1}$ is a diffeomorphism onto some open $U_{p_i, \tau_i}^{\Sigma_i} \subset \cC$. Note that $U_{p_i, \tau_i}^{\Sigma_i} \to \cC$ as $i \to \infty$ by (2'). By construction,
\begin{equation} \label{eq:freq-1-subspace-u0}
	\mbfd_{\cC_i}(\Sigma_i; p_i, \tau_i)^{-1} h^{\Sigma_i}_{p_i, \tau_i} \to u \text{ in } C^2_{\textnormal{loc}}(\cC),
\end{equation}
by passing to subsequences (not relabeled) as appropriate.

We fix
\[ \cJ_u := \Span\{u \} + \Jac_\textrm{trl}(\cC) + \Jac^*_\textrm{pos}(\cC), \]
and then
\[ V_u \subsetneq \spine \cC \subset \RR^{n+1} \]
as in Corollary \ref{Cor_Jac field_High degree pts concentrate}. We proceed to define
\[ V_i := \iota_{p_i}(V_u) \subset T_{p_i} M, \]
and take $\Sigma_i^+$, $p_i^+$ to be such that:
\begin{enumerate}
	\item[(4'')] $\Theta^{1}_{\Sigma_i^+}(p_i^+, 2^{-L_{0,i}}) - \Theta^{1}_{\Sigma_i^+}(p_i^+) \to 0$, 
	\item[(5'')] $\overline{\cN}^{1}_\cC(\Sigma_i^+; p_i^+, 2^{-\ell}) \leq 1 + \delta_i$ for all $\ell \geq L_{1,i}$,
	\item[(6'')] $p_i^+ \in B_{\tau_i}(p_i)$,
	\item[(7'')] $\Sigma^+_i$ does not cross $\Sigma_i$ smoothly,
\end{enumerate}
and violating the desired conclusion, i.e.,
\begin{equation} \label{eq:freq-1-subspace-contr-1}
	p_i^+ \in B_{\tau_i}(p_i) \setminus \eta_{p_i,\tau_i}(B_\theta(V_i)).
\end{equation}
Note that (1'), (4''), (6''), (7''), and Lemma \ref{lemm:one.sided.cone} imply that
\begin{equation} \label{eq:freq-1-subspace-conv}
	\iota_{p_i}^{-1}((\Sigma_i^+)_{p_i,\tau_i}) \rightharpoonup \cC.
\end{equation}
After choosing $\mbfv_i \in B_1 \subset \RR^{n+1}$ so that
\[ p_i^+ = \eta_{p_i,\tau_i}(\iota_{p_i}(\mbfv_i)), \]
it follows from \eqref{eq:freq-1-subspace-contr-1}, \eqref{eq:freq-1-subspace-conv}, and Lemma \ref{lemm:one.sided.cone} that, after passing to a subsequence,
\begin{equation} \label{eq:freq-1-subspace-limit-spine}
	\mbfv_i \to \mbfv \in \spine \cC \cap \bar B_1 \setminus B_\theta(V_u).
\end{equation}
We may then invoke Corollaries \ref{Cor_Jac field_Freq Monoton} and \ref{Cor_Jac field_High degree pts concentrate} to find that
\begin{equation} \label{eq:freq-splitting-contr-at-Q}
	G_\cC(w; \mbfv) < 1 \text{ for all } w \in \cJ_u \setminus \{ 0 \}.
\end{equation}

Let $\iota_{p_i \to p_i^+}$ the isometry from $(T_{p_i} M, g_{p_i})$ to $(T_{p_i^+} M, g_{p_i^+})$ obtained by parallel transport along the minimizing geodesic from $p_i$ to $p_i^+$, and $ \iota_{p_i^+} := \iota_{p_i \to p_i^+} \circ \iota_{p_i}$. Define
\[ \cC_i^+ := \iota_{p_i \to p_i^+}(\cC_i) = \iota_{p_i^+}(\cC) \subset T_{p_i^+} M. \]
Note that $\cC_i^+$ is a hypercone centered at $\orig \in T_{p_i^+}M$ and \eqref{eq:freq-1-subspace-conv} plus Lemma \ref{Lem:change-bspt} imply
\begin{enumerate}
	\item[(3'')] $\mbfd_{\cC_i^+}(\Sigma_i^+; p_i^+, 2^{-L_{2,i}}) \to 0$.
\end{enumerate} 
Proceeding exactly as above, with $\Sigma_i^+$, $p_i^+$ in place of $\Sigma_i$, $p_i$, and (3''), (4''), (5'') in place of (3'), (4'), (5'), we obtain a Jacobi field $u^+ \in \Jac^*_1(\cC) \setminus \{ 0 \}$ by Proposition \ref{Prop:nonconc-Jac} and Corollary \ref{Cor_Jac field_Freq Monoton}. We introduce similar notation to before,
\[ h_{p_i^+, \tau_i}^{\Sigma_i^+} \circ \Pi := \Sdist_{\cC^+_i} \circ \iota_{p_i^+} = \Sdist_{\cC}, \]
where $\Pi$ is unchanged, and the corresponding domains are $U^{\Sigma_i^+}_{p_i^+, \tau_i} \to \cC$ as $i \to \infty$. By construction,
\begin{equation} \label{eq:freq-1-subspace-up}
	\mbfd_{\cC^+_i}(\Sigma^+_i; p^+_i, \tau_i)^{-1} h_{p_i^+, \tau_i}^{\Sigma_i^+} \to u^+ \text{ in } C^2_{\textnormal{loc}}(\cC),
\end{equation}
by passing to subsequences (not relabeled) as appropriate. 

Next, orthogonally decompose
\[ \mbfv_i =: \mbfv_i^\parallel + \mbfv_i^\perp \in \spine \cC \oplus (\spine \cC)^\perp \]
and set
\begin{equation} \label{eq:freq-1-subspace-di}
\lambda_i : = \mbfd_{\cC_i}(\Sigma_i;p_i,\tau_i) + \mbfd_{\cC_i^+}(\Sigma_i^+;p_i^+,\tau_i) + |\mbfv_i^\perp| > 0.
\end{equation}
Note that $\lambda_i \to 0$ by (3'), (3''), and \eqref{eq:freq-1-subspace-limit-spine}. After passing to a subsequence, we have
\begin{align*}
	\lambda_i^{-1} \mbfd_{\cC_i}(\Sigma_i;p_i,\tau_i) & \to c \in [0, 1], \\
	\lambda_i^{-1} \mbfd_{\cC_i^+}(\Sigma_i^+;p_i^+,\tau_i) & \to c^+ \in [0, 1], \\
	\lambda_i^{-1} \mbfv_i^\perp & \to \mbfz \in (\spine \cC)^\perp \cap \bar B_1,
\end{align*}
with
\[ c + c^+ + |\mbfz| = 1 \]
by construction. The definitions above together with \eqref{eq:freq-1-subspace-u0} and \eqref{eq:freq-1-subspace-up} imply
\begin{equation} \label{eq:freq-1-subspace-u0-lambda}
	\lambda_i^{-1} h^{\Sigma_i}_{p_i,\tau_i} \to c u \text{ in } C^2_\textnormal{loc}(\cC),
\end{equation}
\begin{equation} \label{eq:freq-1-subspace-up-lambdap}
	\lambda_i^{-1} h^{\Sigma^+_i}_{p^+_i, \tau_i} \to c^+ u^+ \text{ in } C^2_\textnormal{loc}(\cC),
\end{equation}
Note that the left hand sides are centered at different points, so the fact that $\Sigma_i$, $\Sigma^+_i$ do not cross does \emph{not} imply that $c^+ u^+ - c u$ does not change sign. We in fact have to consider the new height map
\[ h^{\Sigma_i^+}_{p_i, \tau_i} \]
for $\iota_{p_i}^{-1}((\Sigma_i^+)_{p_i, \tau_i})$ over subsets $U^{\Sigma_i^+}_{p_i, \tau_i} \to \cC$; this is well-defined by \eqref{eq:freq-1-subspace-conv}. We do then, in fact, after possibly swapping orientations get a quantity that doesn't change sign,
\begin{equation} \label{eq:freq-1-subspace-sign-0}
	h^{\Sigma^+_i}_{p_i, \tau} - h^{\Sigma_i}_{p_i, \tau_i} \geq 0
\end{equation}
on the subset $U_i = U^{\Sigma_i}_{p_i, \tau_i} \cap U^{\Sigma_i^+}_{p_i, \tau_i} \to \cC$ as $i \to \infty$. Note that \eqref{eq:freq-1-subspace-u0-lambda} tells us how to take limits on $\lambda_i^{-1}$ times the second term of \eqref{eq:freq-1-subspace-sign-0}, but \eqref{eq:freq-1-subspace-up-lambdap} falls short of telling us the same for the first term due to the change in basepoint. We claim that:
\begin{equation} \label{eq:freq-1-subspace-sign-claim}
	\lambda_i^{-1} h^{\Sigma^+_i}_{p_i, \tau} \to c^+ u^+(\cdot - \mbfv) + \langle \mbfz, \nu_\cC \rangle \text{ in } C^0_\textnormal{loc}(\cC) \, .
\end{equation}
We briefly postpone the proof and show how \eqref{eq:freq-1-subspace-sign-claim} implies the result. Together with \eqref{eq:freq-1-subspace-u0-lambda} and \eqref{eq:freq-1-subspace-sign-0} it follows that
\[ c^+ u^+(\cdot - \mbfv) + \langle \mbfz, \nu_\cC \rangle \geq c u \, . \]
Note that both sides are elements of $\Jac^*(\cC)$. Thus, there's $w \in \Jac^*_\textnormal{pos}(\cC)$ so that
\[
c^+ u^{+}(\cdot - \mbfv) + \langle \mbfz, \nu_\cC \rangle = c u + w \, .
\]
If $c^+ = 0$, then $c + |\mbfz| = 1$, so $w$ is the sum of a $0$- and $1$-homogeneous Jacobi field, at least one of which is non-zero, a contradiction to Proposition \ref{Prop_Jac field_Decomp} and Lemma \ref{Lem_Jac field_Jac_(trl, rot, pos)}. Thus, $c^+ \neq 0$, so we can rearrange our expression to
\[
u^{+}(\cdot - \mbfv) = (c^+)^{-1}(c u - \langle \mbfz ,\nu_\cC \rangle + w) \in \cJ_u \setminus \{0\}. 
\]
On the other hand, $u^{+} \in \Jac^*_1(\cC)$ by construction, which contradicts \eqref{eq:freq-splitting-contr-at-Q} and completes the proof of the theorem.

It remains to prove \eqref{eq:freq-1-subspace-sign-claim}. We will do so by a sequence of coordinate change transformations. To that end, let
\[ \cT_{p_i \to p_i^+, \tau_i} = (\eta_{p_i^+, \tau_i} \circ \iota_{p_i^+})^{-1} \circ (\eta_{p_i, \tau_i} \circ \iota_{p_i}) \]
be the change of basepoint map from $p_i$ to $p_i^+$ in scale $\tau_i$ from \eqref{eq:cob-map} in the appendix, pulled back to the corresponding $\RR^{n+1}$'s. It follows from Lemmas \ref{Lem:norm-coord-derivatives} and \ref{Lem:change-bspt} that
\[
	R^1_i(\cdot) := \cT_{p_i \to p_i^+, \tau_i}(\cdot) - (\cdot - \mbfv_i)
\]
satisfies, for all compact $K \subset \RR^{n+1}$,
\begin{equation} \label{eq:freq-1-subspace-claim-1}
	\Vert R^1_i \Vert_{C^0(K)} \leq C_K \tau_i^2.
\end{equation}
On the other hand, Lemma \ref{Lem:freq-lss-2-metric-dies} gives $C>0$ so that
\[
\tau_i^2 \leq C \mbfd_{\cC_i}(\Sigma_i;p_i,\tau_i)^2,
\]
(or the same with $\cC_i^+$, $\Sigma_i^+$, $p_i^+$) so in particular \eqref{eq:freq-1-subspace-claim-1} implies
\begin{equation} \label{eq:freq-1-subspace-claim-1-final}
	\Vert R^1_i \Vert_{C^0(K)} \leq C_K \lambda_i^2,
\end{equation}
where $\lambda_i$ is as in \eqref{eq:freq-1-subspace-di}.

Next, set:
\[
	R^2_i(\cdot) := \Sdist_{\cC}(\cdot) - \Sdist_{\cC} \left[ \cT_{p_i \to p_i^+, \tau_i}(\cdot) + \mbfv_i \right].
\]
Since $\Sdist_\cC$ is $1$-Lipschitz, \eqref{eq:freq-1-subspace-claim-1-final} implies, for all compact $K \subset \RR^{n+1}$,
\begin{equation} \label{eq:freq-1-subspace-claim-2}
	\Vert R^2_i \Vert_{C^0(K)} \leq C_K \lambda_i^2.
\end{equation}

It will be convenient to introduce some further notation. Namely, let
\[ \Phi^{\Sigma_i}_{p_i, \tau_i} : U^{\Sigma_i}_{p_i, \tau_i} \to \iota_{p_i}^{-1}((\Sigma_i)_{p_i,\tau_i}), \]
\[ \Phi^{\Sigma_i^+}_{p_i^+, \tau_i} : U^{\Sigma_i^+}_{p_i^+, \tau_i} \to \iota_{p_i^+}^{-1}((\Sigma^+_i)_{p_i^+,\tau_i}), \]
\[ \Phi^{\Sigma_i^+}_{p_i, \tau_i} : U^{\Sigma_i^+}_{p_i, \tau_i} \to \iota_{p_i}^{-1}((\Sigma^+_i)_{p_i,\tau_i}), \]
be the graph maps corresponding to the height functions we have (which are inverted by $\Pi$, on each corresponding domain). Note that
\begin{align}
	h^{\Sigma_i^+}_{p_i,\tau_i}(\cdot)
		& = \Sdist_{\cC} \left[ \Phi^{\Sigma_i^+}_{p_i, \tau_i}(\cdot) \right] \nonumber \\
		& = \Sdist_{\cC} \left[ \left(\cT_{p_i \to p_i^+, \tau_i} \circ \Phi^{\Sigma_i^+}_{p_i, \tau_i}\right)(\cdot) + \mbfv_i \right] + \left( R^2_i \circ \Phi^{\Sigma_i^+}_{p_i, \tau_i}\right)(\cdot) \nonumber \\
		& = \Sdist_{\cC} \left[ \left(\Phi^{\Sigma_i^+}_{p_i^+, \tau_i} \circ \Pi \circ \cT_{p_i\to p_i^+, \tau_i} \circ \Phi^{\Sigma_i^+}_{p_i, \tau_i} \right)(\cdot) + \mbfv_i \right] + \left( R^2_i \circ \Phi^{\Sigma_i^+}_{p_i, \tau_i} \right)(\cdot), \label{eq:freq-1-subspace-h-redefn-1}
\end{align}
as long as
\begin{equation} \label{eq:freq-1-subspace-claim-u}
	\left( \Pi \circ \cT_{p_i \to p_i^+, \tau_i} \circ \Phi^{\Sigma_i^+}_{p_i, \tau_i}\right)(\cdot) \in U^{\Sigma_i^+}_{p_i^+, \tau_i}.
\end{equation}
Choose a new slower exhaustion $U_{p_i, \tau_i} \to \cC$ such that $U_{p_i, \tau_i} \subset U^{\Sigma_i}_{p_i, \tau_i}$ and \eqref{eq:freq-1-subspace-claim-u} holds over $U_{p_i, \tau_i}$. Moreover, denote
\[ \Psi_i = \left( \Pi \circ \cT_{p_i \to p_i^+, \tau_i} \circ \Phi^{\Sigma_i^+}_{p_i, \tau_i}\right) : U_{p_i, \tau_i} \to U^{\Sigma_i^+}_{p_i^+, \tau_i}, \]
and
\[ \tilde R^2_i := \left( R^2_i \circ \Phi^{\Sigma_i^+}_{p_i, \tau_i} \right) : U_{p_i, \tau_i} \to \RR. \]
Note that \eqref{eq:freq-1-subspace-limit-spine} implies that
\begin{equation} \label{eq:freq-1-subspace-psi}
	\Psi_i(\cdot) \to (\cdot) - \mbfv \text{ in } C^0_{\textnormal{loc}}(\cC),
\end{equation}
while \eqref{eq:freq-1-subspace-claim-2} implies that, for all fixed compact $L \subset \cC$, and $i \geq i_1(L)$,
\begin{equation} \label{eq:freq-1-subspace-claim-2-final}
	\Vert \tilde R^2_i \Vert_{C^0(L)} \leq C_L \lambda_i^2.
\end{equation}
Using our definitions for $\Psi_i$ and $\tilde R^2_i$, we can rewrite \eqref{eq:freq-1-subspace-h-redefn-1} as follows:
\begin{align} \label{eq:freq-1-subspace-h-redefn-2}
	h^{\Sigma_i^+}_{p_i,\tau_i}(\cdot) 
		& = \Sdist_{\cC} \left[ \Phi^{\Sigma_i^+}_{p_i^+, \tau_i} (\Psi_i(\cdot)) + \mbfv_i \right] + \tilde R^2_i(\cdot) \nonumber \\
		& = \Sdist_{\cC} \left[ \Phi^{\Sigma_i^+}_{p_i^+, \tau_i} (\Psi_i(\cdot)) + \mbfv_i^\perp \right] + \tilde R^2_i(\cdot),
\end{align}
where, in the second step, we used the translation invariance of $\cC$ along $\mbfv_i^\parallel \in \spine \cC$.

Finally, denote over $U^{\Sigma_i^+}_{p_i^+, \tau_i}$:
\[ R^3_i(\cdot) := \Sdist_{\cC} \left[ \Phi^{\Sigma_i^+}_{p_i^+, \tau_i} (\cdot) + \mbfv_i^\perp \right] - \Sdist_{\cC} \left[ \Phi^{\Sigma_i^+}_{p_i^+, \tau_i}(\cdot) \right] - \langle \mbfv_i, \nu_\cC(\cdot) \rangle. \]
It follows from \eqref{eq:freq-1-subspace-conv}, Allard's theorem \cite{Allard:first-variation}, and Lemma \ref{Lem:affine-acting-graph} that, for all compact $L \subset \cC$ and all $i \geq i_2(L)$, 
\begin{equation} \label{eq:freq-1-subspace-claim-3}
	\Vert R^3_i \Vert_{C^0(L)} \leq \epsilon_{L,i} |\mbfv_i^\perp| \leq \epsilon_{L,i} \lambda_i,
\end{equation}
where $\lambda_i$ is as in \eqref{eq:freq-1-subspace-di} again and, for $L$ held fixed, one has
\begin{equation} \label{eq:freq-1-subspace-claim-3-eps}
	\epsilon_{L,i} \to 0 \text{ as } i \to \infty.
\end{equation}
Writing
\[ \tilde R^3_i := R^3_i \circ \Psi_i : U_{p_i, \tau_i} \to \RR, \]
note that \eqref{eq:freq-1-subspace-psi} and \eqref{eq:freq-1-subspace-claim-3} imply that
\begin{equation} \label{eq:freq-1-subspace-claim-3-final}
	\Vert R^3_i \Vert_{C^0(L)} \leq \epsilon'_{L,i} \lambda_i,
\end{equation}
for a possibly different
\begin{equation} \label{eq:freq-1-subspace-claim-3-final-eps}
	\epsilon'_{L,i} \to 0 \text{ as } i \to \infty.
\end{equation}

In summary, \eqref{eq:freq-1-subspace-h-redefn-2} can be once more rewritten as:
\begin{align}
	h^{\Sigma_i^+}_{p_i,\tau_i}(\cdot) 
		& = \Sdist_{\cC} \left[ \Phi^{\Sigma_i^+}_{p_i^+, \tau_i}(\Psi_i(\cdot)) \right] + \langle \mbfv_i, \nu_\cC(\Psi_i(\cdot)) \rangle + \tilde R^2_i(\cdot) + \tilde R^3_i(\cdot) \nonumber \\
		& = h^{\Sigma_i^+}_{p_i^+,\tau_i}(\Psi_i(\cdot)) + \langle \mbfv_i, \nu_\cC(\Psi_i(\cdot)) \rangle + \tilde R^2_i(\cdot) + \tilde R^3_i(\cdot). \label{eq:freq-1-subspace-h-redefn-3}
\end{align}
Now, the desired \eqref{eq:freq-1-subspace-sign-claim} follows from \eqref{eq:freq-1-subspace-h-redefn-3} together with \eqref{eq:freq-1-subspace-up-lambdap}, \eqref{eq:freq-1-subspace-psi}, \eqref{eq:freq-1-subspace-claim-2-final}, \eqref{eq:freq-1-subspace-claim-3-final}, and \eqref{eq:freq-1-subspace-claim-3-final-eps}. This completes the proof of the theorem.
\end{proof}

\section{Spectral quantity $\alpha(\cC)$ for general minimizing hypercones} \label{sec:kappa.lambda}

\subsection{Spectral quantity $\alpha(\cC)$} 

Our sheet separation analysis relies on Simon's partial Harnack theory for positive Jacobi fields (cf. \cite[Appendix A]{Wang:smoothing}). For minimizing hypercones $\cC \subset \RR^{n+1}$, a spectral quantity that we will call $\alpha(\cC)$ controls the separation of disjoint minimizers that are sufficiently well modeled on $\cC$. 

\begin{Def} \label{defi:kappa.lambda}
	Let $\cC \subset \RR^{n+1}$ be any minimizing hypercone. Denote by $\cL := \cC \cap \SSp^n$ the link of the regular part of $\cC$. We will be interested in the preliminary spectral quantity
	\begin{equation} \label{eq:lambda}
		\mu(\cC) := \inf \left\{ \int_\cL |\nabla_\cL \varphi|^2 - |A_\cL|^2 \varphi^2 : \varphi \in C^1_c(\cL), \; \Vert \varphi \Vert_{L^2(\cL)} = 1 \right\},
	\end{equation}
	where $\nabla_\cL$ denotes the covariant derivative on $\cL$ and $A_\cL$ denotes the second fundamental form of $\cL$ as a submanifold of $\SSp^n$.
	
	The roots of the equation $x^2 - (n-2) x - \mu(\cC) = 0$ are necessarily real when $\cC \in \mathscr{C}_n$ (the argument in \cite[Section 4]{CaffarelliHardtSimon} applies), and we denote the smaller as:
	\begin{equation} \label{eq:kappa}
		\alpha(\cC) := \frac{n-2}{2} - \sqrt{\frac{(n-2)^2}{4} + \mu(\cC)}.
	\end{equation}
\end{Def}

\begin{Rem} \label{rema:lambda.gamma.compared}
	It is worth comparing this spectral quantity with those discussed in Section \ref{Subsec_reg Min Cone} in case $\cC = \cC_\circ$ is a regular hypercone. In the language of that section:
	\[ \alpha(\cC_\circ) = -\gamma_1. \]
\end{Rem}

A key property of $\alpha(\cdot)$ is that it, among nonflat minimizing hypercones, it is precisely minimized by quadratic hypercones. Indeed:

\begin{Lem} \label{lemm:kappa.zhu}
	Define:
	\begin{equation} \label{eq:kappa.n}
		\alpha_n := \frac{n-2}{2} - \sqrt{\frac{(n-2)^2}{4} - (n-1)}.
	\end{equation}
	Then:
	\begin{enumerate}
		\item[(i)] The $\alpha_n$ form a decreasing sequence with $2 = \alpha_7 > \alpha_8 > \ldots > 1$.\footnote{See also \cite[Remark 3.9]{CMS:generic.9.10} for approximations of the first few values. The reference uses the notation $\kappa_n$ in place of $\alpha_n$.}
		\item[(ii)] For every $\cC \in \mathscr{C}_n$, $\alpha(\cC) \geq \alpha_n$ with equality if and only if $\cC \in \mathscr{C}_n^{\textnormal{qd}}$. 
	\end{enumerate}
\end{Lem}
\begin{proof}
	(i) This is an elementary computation.
	
	(ii) This was shown by J.\ Zhu \cite{Zhu:stability}; cf.\ \cite{Simons:minvar,Lawson:rigid,CherndoCarmoKobayashi,Perdomo,Wu}.
\end{proof}

The non-rigid portion of Lemma \ref{lemm:kappa.zhu} improves to:

\begin{Lem} \label{lemm:kappa.zhu.nonqd}
	Denote\footnote{It is an open question whether $\mathscr{C}_7 = \mathscr{C}_7^{\textnormal{qd}}$; if it is, then we may interpret $\Delta^{\operatorname{non-qd}}_7 = \infty$ using the convention $\inf \emptyset = \infty$.}
	\begin{equation} \label{eq:delta.nonqd.gap}
		\Delta^{\operatorname{non-qd}}_n = \inf \{ \alpha(\cC) : \cC \in \mathscr{C}_n \setminus \mathscr{C}_n^{\textnormal{qd}} \} - \alpha_n.
	\end{equation}
	Then:
	\begin{enumerate}
		\item[(i)] $\Delta^{\operatorname{non-qd}}_n>0$.
		\item[(ii)] For every $\cC \in \mathscr{C}_n \setminus \mathscr{C}_n^{\textnormal{qd}}$, $\alpha(\cC) \geq \alpha_n + \Delta^{\operatorname{non-qd}}_n$. 
	\end{enumerate}
\end{Lem}
\begin{proof}
	This was shown in \cite[Appendix A]{Wang:smoothing}.
\end{proof}

It also follows that $\alpha(\cdot)$ behaves well under cone-splitting:

\begin{Prop} \label{prop:kappa.spine}
	If $\cC \in \mathscr{C}_n \cap \Rot(\cC_\circ \times \RR^k)$, $\cC_\circ \in \mathscr{C}_{n-k}$, then
	\[ \alpha(\cC) = \alpha(\cC_\circ). \]
\end{Prop}

A proof is given below using a localized version $\alpha(\cC; \sigma)$ of $\alpha(\cC)$. For hypercones that are not regular, such as our intended cylindrical hypercones, it will help to work with a localization of $\alpha(\cC)$ to the controlled-regularity subsets $\reg_{>\sigma} \cC$, $\sigma > 0$.

\begin{Def} \label{defi:kappa.lambda.sigma}
Let $\cC \subset \RR^{n+1}$ be any minimizing hypercone, with $\cL := \cC \cap \SSp^n$ the link of the regular part. Consider, for $\sigma > 0$, the open subsets
\[ \cL_{>\sigma} := \reg_{>\sigma} \cC  \cap \SSp^n \subset \cL. \]
Then, one defines, in analogy with \eqref{eq:lambda}, \eqref{eq:kappa},
\begin{equation} \label{eq:lambda.rho}
	\mu(\cC; \sigma) := \inf \left\{ \int_\cL |\nabla_\cL \varphi|^2 - |A_\cL|^2 \varphi^2 : \varphi \in C^1_c(\cL_{>\sigma}), \; \Vert \varphi \Vert_{L^2(\cL)} = 1 \right\},
\end{equation}
\begin{equation} \label{eq:kappa.rho}
	\alpha(\cC; \sigma) := \frac{n-2}{2} - \sqrt{\frac{(n-2)^2}{4} + \mu(\cC; \sigma)}.
\end{equation}
  More generally, for every open subset $\Omega \subset \cL$, define 
  \begin{equation} \label{equ_mu(Omega)}
	\mu(\cC; \Omega) := \inf \left\{ \int_\Omega |\nabla_\cL \varphi|^2 - |A_\cL|^2 \varphi^2 : \varphi \in C^1_c(\Omega), \; \Vert \varphi \Vert_{L^2(\cL)} = 1 \right\},
  \end{equation}
  \begin{equation} \label{equ_alpha(Omega)}
	\alpha(\cC; \Omega) := \frac{n-2}{2} - \sqrt{\frac{(n-2)^2}{4} + \mu(\cC; \Omega)}.
  \end{equation}
\end{Def}

  Clearly, $\mu(\cC; \Omega)$ is monotone non-increasing in $\Omega$, $\alpha(\cC; \Omega)$ is monotone non-decreasing in $\Omega$, and if $Z\subset \cL$ is a submanifold of codimension $>2$, $\Omega_1\subset \Omega_2\subset \dots$ is an exhaustion of $\cL\setminus Z$, then by a standard cut-off argument we have 
  \begin{equation} \label{equ_alpha_exhaust limit}
    \lim_{j\to \infty}\alpha(\cC; \Omega_j) = \alpha(\cC; \cC\setminus Z) = \alpha(\cC)\,.       
  \end{equation}
  In particular,
\begin{equation} \label{eq:kappa.convergence.rho}
	\lim_{\sigma \to 0} \alpha(\cC; \sigma) = \alpha(\cC).
\end{equation}
Moreover, if $\cC_i \to \cC$ within $\mathscr{C}_n$, $\sigma_i \to \sigma$ within $(0, \infty)$, and $t \in (0, 1)$, then
\begin{equation} \label{eq:kappa.convergence.both}
	\alpha(\cC; \sigma) \leq \liminf_i \alpha(\cC_i; \sigma_i) \leq \limsup_i \alpha(\cC_i; \sigma_i) \leq \alpha(\cC; t \sigma)
\end{equation}

One can verify, using $\alpha(\cC; \sigma)$, that Proposition \ref{prop:kappa.spine} from the introduction holds: 

\begin{proof}[Proof of Proposition \ref{prop:kappa.spine}]
	Without loss of generality, we rotate $\RR^{n+1}$ so that
	\[ \cC = \cC_\circ \times \RR^k. \]
	Denote $n_\circ = n - k$, so that $\cC_\circ \subset \RR^{n_\circ + 1}$, and write $\cL$, $\cL_\circ$ for the links of $\cC$, $\cC_\circ$. 
	
    Fix an arbitrary domain $\Omega_\circ \Subset  \cL_\circ$ with smooth boundary. Denote
    \[ \Omega:= \SSp^n\cap (\operatorname{Cone}_\orig(\Omega_\circ)\times\RR^k)\subset \cL. \]
    The primary goal is to show that 
    \begin{equation} \label{equ_alpha(Omega)=alpha(Omega_0)}
      \alpha(\cC; \Omega) = \alpha(\cC_\circ; \Omega_\circ)\,. 
    \end{equation}
    It suffices to show this for $\Omega_\circ \neq \cL_\circ$.
    Note that combining \eqref{equ_alpha(Omega)=alpha(Omega_0)} and \eqref{equ_alpha_exhaust limit} proves Proposition \ref{prop:kappa.spine}. 

    To prove \eqref{equ_alpha(Omega)=alpha(Omega_0)} for $\Omega
    _\circ \neq \cL_\circ$, let $\mu_\circ:= \mu(\cC_\circ; \Omega_\circ)$ be the principal Dirichlet eigenvalue for $-\Delta_{\cL_\circ} - |A_{\cL_\circ}|^2$ on $\Omega_\circ$, $\varphi_\circ : \overline{\Omega_\circ} \to [0, \infty)$ be a corresponding principal Dirichlet eigenfunction, i.e.,
	\begin{equation} \label{eq:kappa.spine.phi}
		(\Delta_{\cL_\circ} + |A_{\cL_\circ}|^2) \varphi_\circ = -\mu_{\circ} \varphi_\circ \text{ in } \Omega_\circ, \quad \varphi_\circ = 0 \text{ on } \partial \Omega_\circ.
	\end{equation}
    Note that by strong maximum principle, $\varphi_\circ>0$ on $\Omega_\circ$. Also denote for simplicity $\alpha_\circ:= \alpha(\cC_\circ; \Omega_\circ)$. At this point, the assumption $\Omega_\circ\neq \cL_\circ$ enters to give $\alpha_\circ< \alpha(\cC_\circ)\leq (n_\circ-2)/2$. Let \[
      \psi: \operatorname{Cone}_\orig(\Omega_\circ)\times \RR^k \to [0, +\infty)\,, \quad \psi := r^{-\alpha_\circ}\varphi_\circ
    \] 
    be the $\RR^k$-invariant $(-\alpha_\circ)$-homogeneous extension of $\varphi_\circ$. Then 
	\begin{align*}
		 (\Delta_\cC + |A_\cC|^2)\psi 
        & = (\Delta_{\cC_\circ} + |A_{\cC_\circ}|^2) (r^{-\alpha_\circ} \varphi_\circ) \\
		& = (r^{-2} \Delta_{\cL_\circ} + (n_\circ-1) r^{-1} \partial_r + \partial_r^2 + r^{-2} |A_{\cL_\circ}|^2)(r^{-\alpha_\circ} \varphi_\circ) \\
		& = r^{-2} (\Delta_{\cL_\circ} + |A_{\cL_\circ}|^2) (r^{-\alpha_\circ} \varphi_\circ) + r^{-2} (\alpha_\circ^2 - (n_\circ-2) \alpha_\circ) (r^{-\alpha_\circ} \varphi_\circ) \\
		& = r^{-2} (-\mu_\circ) (r^{-\alpha_\circ} \varphi_\circ) + r^{-2} (\alpha_\circ^2 - (n_\circ-2) \alpha_\circ) (r^{-\alpha_{\circ}} \varphi_\circ) \\
		& = 0,
	\end{align*}
	where we used \eqref{eq:kappa.spine.phi}, \eqref{equ_alpha(Omega)} in the last two steps. Thus, $\psi$ is a positive Jacobi field with Dirichlet boundary conditions on $\operatorname{Cone}_\orig(\Omega_\circ) \times \RR^k \subset \cC$. Note, also, that $\psi$ being $(-\alpha_{\circ})$-homogeneous implies,
	\begin{equation} \label{eq:kappa.spine.homogeneous}
		\partial_\rho \psi = - \alpha_{\circ}\rho^{-1} \psi,
	\end{equation}
	where $\rho$ denotes the radial distance function from the origin in $\RR^{n+1}$. 
	
	Following the same computation as before but backwards, and in $\RR^{n+1}$ rather than $\RR^{n_\circ+1}$, we deduce from the Jacobi field nature of $\psi$ that:
	\begin{align*}
		0   & = (\Delta_{\cC} + |A_{\cC}|^2) \psi \\
			& = \left(\rho^{-2} \Delta_{\cL} + (n-1) \rho^{-1} \partial_\rho + \partial_\rho^2 + \rho^{-2} |A_{\cL}|^2\right) \psi \\
			& = \rho^{-2} (\Delta_{\cL} + |A_{\cL}|^2 + (\alpha_{\circ}^2 - (n-2) \alpha_{\circ})) \psi,
	\end{align*}
	where we used \eqref{eq:kappa.spine.homogeneous} in the last step. This implies that $\varphi:= \psi|_{\Omega}$ is a positive solution to 
    \begin{equation} \label{equ_varphi J_L eigenfunc}
      -(\Delta_{\cL} + |A_{\cL}|^2)\varphi = (\alpha_{\circ}^2 - (n-2) \alpha_{\circ})) \varphi \quad \text{ on }\Omega\,.         
    \end{equation}
    Hence, for every $\xi\in C^\infty_c(\Omega)$, let $\zeta:= \varphi^{-1}\xi$, we have 
    \begin{align*}
      \int_\Omega |\nabla_\cL\xi|^2 - |A_\cL|^2\xi^2 
      & = \int_\Omega |\nabla_\cL(\varphi\zeta)|^2 - |A_\cL|^2(\varphi\zeta)^2 \\
      & = \int_\Omega |\nabla_\cL\zeta|^2\varphi^2 + |\nabla_\cL\varphi|^2\zeta^2 + \varphi\nabla_\cL\varphi\cdot \nabla_\cL(\zeta^2) - |A_\cL|^2(\varphi\zeta)^2 \\
      & = \int_\Omega |\nabla_\cL\zeta|^2\varphi^2 - \zeta^2\varphi\cdot (\Delta_\cL + |A_\cL|^2)\varphi \\
      & \geq (\alpha_\circ^2 - (n-2)\alpha_\circ)\|\xi\|_{L^2(\Omega)}^2 \,.
    \end{align*}
    where we use \eqref{equ_varphi J_L eigenfunc} in the last inequality. This means \[
      \mu(\cC; \Omega) \geq \alpha_\circ^2 - (n-2)\alpha_\circ\,.
    \]
    Since by \eqref{equ_alpha(Omega)}, $\mu(\cC; \Omega) = \alpha(\cC; \Omega)^2 - (n-2)\alpha(\cC; \Omega)$ and $\alpha(\cC;\Omega)\leq (n-2)/2$, $\alpha_\circ< (n_\circ-2)/2<(n-2)/2$, this proves ``$\leq$" of \eqref{equ_alpha(Omega)=alpha(Omega_0)}.

    To prove the reverse inequality, we again use \eqref{equ_varphi J_L eigenfunc} to choose $\varphi$ as a test function after some cut-off near $\sing \cL = \bar \cL \setminus \cL$. Recall that by the definition of $\Omega$, we have \[
      \overline{\Omega}\cap \sing \cL = \SSp^n \cap (\RR^k\times \{\orig\}) \,.
    \]
    For every $\eps\in (0, 1)$, let $\eta_\eps:= \eta(r/\eps)$, viewed as a smooth function on $\cL$, where $\eta\in C^\infty(\RR; [0, 1])$ such that $\eta = 0$ on $(-\infty, 1/2]$, $\eta = 1$ on $[1, +\infty)$ and $|\eta'|\leq 10$.  
    Then $\eta_\eps\varphi$ is a Lipschitz function on $\cL$ supported away from $\sing \cL$ and vanishes outside $\Omega$, therefore by the same calculation as above, 
    \begin{align*}
      \mu(\cC; \Omega) \int_{\Omega}  (\eta_\eps\varphi)^2 
      & \leq \int_\Omega |\nabla_\cL(\eta_\eps\varphi)|^2 - |A_\cL|^2(\eta_\eps\varphi)^2 \\
      & = \int_\Omega |\nabla_\cL\eta_\eps|^2\varphi^2 - \eta_\eps^2\varphi\cdot (\Delta_\cL + |A_\cL|^2)\varphi \\
      & = \int_\Omega |\nabla_\cL\eta_\eps|^2\varphi^2 + (\alpha_\circ- (n-2)\alpha_\circ)\eta_\eps^2\varphi^2 \,.
    \end{align*}
    Recall that $\varphi = r^{-\alpha_\circ}\varphi_\circ$, where $\alpha_\circ<(n_\circ-2)/2$, we then have 
    \begin{align*}
      \int_{\Omega}(1-\eta_\eps^2)\varphi^2  
      & \leq \sum_{j=1}^\infty\int_{\Omega\cap \{2^{-j}\eps\leq r\leq 2^{-j+1}\eps\}} \|\varphi_\circ\|_{L^\infty}^2\cdot r^{-2\alpha_\circ} \\
      & \leq C(n, \varphi_\circ)\sum_{j=1}^\infty(2^{-j}\eps)^{-2\alpha_\circ} \Area(\cL \cap \{r<2^{-j+1}\eps\}).         
    \end{align*}
    Also,
    \begin{align*}
      \int_{\Omega} |\nabla_\cL\eta_\eps|^2\varphi^2 
        & \leq \int_{\Omega\cap \{\eps/2\leq r\leq\eps\}} \eps^{-2}|\eta'|^2\|\varphi_\circ\|_{L^\infty}^2\cdot r^{-2\alpha_\circ} \\
        & \leq C(n, \varphi_\circ)\eps^{-2-2\alpha_\circ} \Area(\cL \cap \{r<\eps\}).        
    \end{align*}
    By the monotonicity formula and a Vitali covering argument, \[
      \|\cL\|(\{r<\eps\})\leq C(\cL)\eps^{n_\circ} \,.  \]
    Combining these estimates and sending $\eps\to 0$, we derive, \[
      \mu(\cC; \Omega) \leq \alpha_\circ^2 - (n-2)\alpha_\circ \,.
    \]
    This proves ``$\geq$" of \eqref{equ_alpha(Omega)=alpha(Omega_0)}.
\end{proof}

\subsection{Choosing good $\alpha(\cC; \sigma)$ localizations}

Let $d_{\mathscr{C}_n}$ be a metric on $\mathscr{C}_n$ (see \eqref{eq:mathscr.n}) compatible with its standard topology, e.g., constructed from the flat norm. In analogy with the notation $B_\epsilon(\cC) \subset \mathscr{C}_n$ for open $\epsilon$-balls centered at $\cC \in \mathscr{C}_n$, we will also denote for any subset $S \subset \mathscr{C}_n$
\[ B_\epsilon(S) := \bigcup_{\cC \in S} B_\epsilon(\cC) \]
to be the standard $\epsilon$-neighborhood in $\mathscr{C}_n$ with respect to $d_{\mathscr{C}_n}$. 

It will also be convenient to decompose $\mathscr{C}_n$ in terms of spine dimensions. Specifically: 
\[
	\mathscr{C}_{n}^{k\textnormal{-cyl}} := \{ \cC \in \mathscr{C}_n : \dim \spine \cC \geq k \} 
\]
for $k = 0, 1, \ldots, n-7$. A standard consequence of the monotonicity formula is that:
\begin{equation} \label{eq:mathscr.k.cyl.closed}
	\mathscr{C}_{n}^{k\textnormal{-cyl}} \subset \mathscr{C}_{n} \text{ is compact}.
\end{equation}
Recall also the definition of $\mathscr{C}_n^{\textnormal{qd-cyl}}$ from \eqref{eq:mathscr.n.qd-cyl}.

\begin{Prop} \label{prop:alpha.sigma.good}
	The following are all true:
	\begin{enumerate}
		\item[(i)] Fix $\xi > 0$. For all sufficiently small $\epsilon, \sigma > 0$ depending on $n, \xi$, 
			\[ \cC \in B_\epsilon(\mathscr{C}_{n}^{k\textnormal{-cyl}}) \implies \alpha(\cC; \sigma) > \alpha_{n - k} - \xi. \]
		\item[(ii)] Fix $\xi > 0$. For all $\epsilon > 0$, and all sufficiently small $\sigma > 0$ depending on $n, \xi, \epsilon$,
			\[ \cC \in \mathscr{C}_{n}^{k\textnormal{-cyl}} \setminus B_\epsilon(\mathscr{C}_n^{\textnormal{qd-cyl}}) \implies \alpha(\cC; \sigma) > \alpha_{n-k} + \Delta_{n-k}^{\textnormal{non-qd}} - \xi. \]
		\item[(iii)] For all sufficiently small $\epsilon > 0$ depending on $n$,
			\[ \cC \in B_\epsilon(\mathscr{C}_n^{\textnormal{qd-cyl}}) \setminus \mathscr{C}_n^{\textnormal{qd-cyl}} \implies \cC \in B_\epsilon(\mathscr{C}_n^{k'\textnormal{-cyl}}) \text{ for } k' = \dim \spine \cC + 1. \]
	\end{enumerate}
\end{Prop}
\begin{proof}
	(i) If not, there would exist $k \in \{ 0, 1, \ldots, n-7 \}$, $\sigma_i \to 0$, $\epsilon_i \to 0$, $\cC_i \in B_{\epsilon_i}(\mathscr{C}_{n}^{k\textnormal{-cyl}})$ such that
	\begin{equation} \label{eq:alpha.sigma.good.i.1}
		\alpha(\cC_i; \sigma_i) \leq \alpha_{n-k} - \xi. 
	\end{equation}
	Pass to a subsequence so that $\cC_i \to \cC$. Then $\cC \in \mathscr{C}_{n}^{k\textnormal{-cyl}}$ by $\epsilon_i \to 0$ and \eqref{eq:mathscr.k.cyl.closed}, so
	\begin{equation} \label{eq:alpha.sigma.good.i.2}
		\alpha(\cC) \geq \alpha_{n-k}
	\end{equation}
	by Proposition \ref{prop:kappa.spine}. On the other hand, for each fixed $\sigma \in (0, 1)$, we have from \eqref{eq:alpha.sigma.good.i.1} together with \eqref{eq:kappa.convergence.both} and $\sigma_i \to 0$ that
	\[ \alpha(\cC; \sigma) \leq \lim_i \alpha(\cC_i; \sigma_i) \leq \alpha_{n-k} - \xi. \]
	Now using \eqref{eq:kappa.convergence.rho} to send $\sigma \to 0$ in the inequality above, we contradict \eqref{eq:alpha.sigma.good.i.2}.
	
	(ii) If not, there would exist $k \in \{ 0, 1, \ldots, n-7 \}$, $\sigma_i \to 0$, $\cC_i \in \mathscr{C}_{n}^{k\textnormal{-cyl}} \setminus B_{\epsilon}(\mathscr{C}_n^{\textnormal{qd-cyl}})$ such that
	\begin{equation} \label{eq:alpha.sigma.good.ii.1}
		\alpha(\cC_i; \sigma_i) \leq \alpha_{n-k} + \Delta_{n-k}^{\textnormal{non-qd}} - \xi. 
	\end{equation}
	Pass to a subsequence so that $\cC_i \to \cC$. Then $\cC \in \mathscr{C}_{n}^{k\textnormal{-cyl}} \setminus \mathscr{C}_n^{\textnormal{qd-cyl}}$ by \eqref{eq:mathscr.k.cyl.closed}. It follows from Proposition \ref{prop:kappa.spine}, and Lemma \ref{lemm:kappa.zhu.nonqd} that
	\begin{equation} \label{eq:alpha.sigma.good.ii.2}
		\alpha(\cC) \geq \alpha_{n-k} + \Delta_{n-k}^{\textnormal{non-qd}}.
	\end{equation}
	Working as in (i), one derives a contradiction from using \eqref{eq:alpha.sigma.good.ii.1} against \eqref{eq:alpha.sigma.good.ii.2}.
	
	(iii) If not, there would exist $k \in \{ 0, 1, \ldots, n-7 \}$, $\epsilon_i \to 0$, $\cC_i \in \mathscr{C}_n \setminus \mathscr{C}_n^{\textnormal{qd-cyl}}$ with $\dim \spine \cC_i = k$, and $\cC'_i \in \mathscr{C}_n^{\textnormal{qd-cyl}}$ with $d_{\mathscr{C}_n}(\cC_i, \cC_i') < \epsilon_i$, so that
	\begin{equation} \label{eq:alpha.sigma.good.iii.1}
		\cC_i \not \in B_{\epsilon_i}(\mathscr{C}_n^{(k+1)\textnormal{-cyl}}).
	\end{equation}
	As before, $\dim \spine \cC_i = k$ plus \eqref{eq:mathscr.k.cyl.closed} and passing to a subsequence implies that
	\begin{equation} \label{eq:alpha.sigma.good.iii.limit}
		\cC_i \to \cC \in \mathscr{C}_n^{k\textnormal{-cyl}}.
	\end{equation}
	In particular, \eqref{eq:alpha.sigma.good.iii.1} plus $d_{\mathscr{C}_n}(\cC_i, \cC_i') < \epsilon_i$ guarantee that
	\begin{equation} \label{eq:alpha.sigma.good.iii.2}
		\dim \spine \cC_i' \leq k.
	\end{equation}
	Let us upgrade \eqref{eq:alpha.sigma.good.iii.2} to a strict inequality for large $i$. Indeed, otherwise,
	\begin{equation} \label{eq:alpha.sigma.good.iii.3}
		\cC_i' \in \Rot(\cC'_{\circ,i} \times \RR^k), \; \cC'_{\circ,i} \in \mathscr{C}_{n-k}^{\textnormal{qd}}.
	\end{equation}
	Likewise, $\cC_i \in \mathscr{C}_n \setminus \mathscr{C}_n^{\textnormal{qd-cyl}}$ plus $\dim \spine \cC_i = k$ also imply
	\begin{equation} \label{eq:alpha.sigma.good.iii.4}
		\cC_i \in \Rot(\cC_{\circ,i} \times \RR^k), \; \cC_{\circ,i} \in \mathscr{C}_{n-k} \setminus \mathscr{C}_{n-k}^{\textnormal{qd}}.
	\end{equation}
	Now \eqref{eq:alpha.sigma.good.iii.3} and \eqref{eq:alpha.sigma.good.iii.4} plus $d_{\mathscr{C}_n}(\cC_i, \cC_i') < \epsilon_i \to 0$ imply that after possibly rotating $\cC_{\circ,i}'$ to $\tilde \cC_{\circ,i}' \in \mathscr{C}_{n-k}^{\textnormal{qd}}$, we have
	\[ d_{\mathscr{C}_n}(\cC_{\circ,i}, \tilde \cC'_{\circ,i}) \to 0 \]
	even though $\tilde \cC'_{\circ,i} \in \mathscr{C}_n^{\textnormal{qd}} \not \ni \cC_{\circ,i}$. This contradicts the isolatedness of quadratic hypercones (Corollary \ref{coro:quadratic.isolated}) as $i \to \infty$. This completes the proof that \eqref{eq:alpha.sigma.good.iii.2} is a strict inequality, as desired:
	\begin{equation} \label{eq:alpha.sigma.good.iii.5}
		\dim \spine \cC_i' \leq k-1.
	\end{equation}
	Now note that the space $\mathscr{C}_{n}^{\textnormal{qd-cyl}}$ is discrete modulo rotations. By passing to a subsequence we have that spine dimensions get preserved in \eqref{eq:alpha.sigma.good.iii.5}, so
	\begin{equation} \label{eq:alpha.sigma.good.iii.limit.prime}
		\cC_i' \to \cC' \in \mathscr{C}_n^{\textnormal{qd-cyl}}, \; \dim \spine \cC' \leq k-1.
	\end{equation}
	On the other hand, $d_{\mathscr{C}_n}(\cC_i, \cC_i') < \epsilon_i \to 0$ implies that $\cC = \cC'$, so \eqref{eq:alpha.sigma.good.iii.limit} and \eqref{eq:alpha.sigma.good.iii.limit.prime} are in contradiction. This completes the proof.
\end{proof}

\section{Covering tree structure}

Standard measure theory manipulations can be used to understand the implications of H\"older continuity of a map to the dimension of its image and the dimensions of its level sets. See, e.g., \cite[Proposition 7.7]{FROS:ihes} for a general statement. In our case, the non-uniqueness of tangent cones requires a slightly coarser approach, where dimensions and H\"older continuities are scale-dependent. We will use the following result. In our applications of it to the proof of Theorem \ref{theo:small.sing.dim} and Corollary \ref{coro:small.sing.dim}, we will take $\cS \subset \sing \mathscr{F}$, $\frT$ to be the restriction to $\cS$ of our timestamp map $\frT$, and $\mathfrak{C}_j$ to be covers of $\cS$ with subsets of balls of radius $a \cdot \theta^j$.

\begin{Prop} \label{prop:covering}
	Let $(M, g)$ be a Riemannian manifold. Take some subset $\cS \subset (M, g)$, a map $\frT : \cS \to \RR$, and some constants $\theta \in (0, 1)$, $a > 0$, and $C_1, C_2 > 0$. 
	
	Assume we can construct a countable family of coverings $\mathfrak{C}_j$, $j \in \ZZ_{\geq 0}$, of $\cS$ (i.e., for all $j \in \ZZ_{\geq 0}$, we have $\cup_{B \in \mathfrak{C}_j} B \supset \cS$) with the following structural properties:
	\begin{enumerate}
		\item[(1)] For each $j \in \ZZ_{\geq 0}$, $\mathfrak{C}_j$ is a finite collection of sets $B \subset M$ with
			\[ \operatorname{diam}_g B \leq 2 a \cdot \theta^j. \]
		\item[(2)] For each $B_{j+1} \in \mathfrak{C}_{j+1}$, $j \in \ZZ_{\geq 0}$, there exists a ``parent'' element
			\[ \mathfrak{p}(B_{j+1}) \in \mathfrak{C}_j. \]
		\item[(3)] For each $B_j \in \mathfrak{C}_j$, $j \in \ZZ_{\geq 0}$, there exists a ``coarse dimension''
			\[ \mathfrak{D}(B_j) \geq 0 \]
			such that
			\[ \# \mathfrak{p}^{-1}(B_j) \leq C_1 \cdot \theta^{-\mathfrak{D}(B_j)}. \]
		\item[(4)] For each $B_j \in \mathfrak{C}_j$, $j \in \ZZ_{\geq 0}$, there exists a ``coarse H\"older exponent''
			\[ \mathfrak{h}(B_j) \geq 0 \]
			such that
			\[ \operatorname{diam} \frT(B_j \cap \cS) \leq C_2 \cdot C_1^j \cdot \theta^{\sum_{i=1}^j \mathfrak{h}(\mathfrak{p}^i(B_j))}, \]
			where $\mathfrak{p}^i = \mathfrak{p} \circ \cdots \circ \mathfrak{p}$ taken $i$ times.
	\end{enumerate}
	Let $d > 0$, $\xi > 0$. There exists $\theta_0 > 0$ depending on $\xi$ and $C_1$ (but not on $a$, $d$, or $C_2$) such that being able to carry out the setup above with $\theta \leq \theta_0$ guarantees:
	\begin{enumerate}
		\item[(i)] If $d \geq (\xi + \mathfrak{D}(B))/\mathfrak{h}(B)$ for all $B \in \cup_j \mathfrak{C}_j$, then $\dim \frT(\cS) \leq d$.
		\item[(ii)] If $d \geq \xi + \mathfrak{D}(B) - \mathfrak{h}(B)$ for all $B \in \cup_j \mathfrak{C}_j$, then $\dim \cS \cap \frT^{-1}(t) \leq d$ for a.e. $t$.
	\end{enumerate}
\end{Prop}
\begin{proof}
	Recall from \cite[\S 2]{Simon:GMT} the definition of $\cH^d_\infty$, the $d$-dimensional Hausdorff outer measure computed using countable covers without diameter constraints. 
	
	(i) Without loss of generality, we may assume $d \leq 1$ and $C_1 > 1$. Fix $j \geq 1$. Then:
	\begin{align} \label{eq:covering.a}
		\cH^d_\infty(\frT(\cS))
			& \leq \sum_{\substack{B_j \in \mathfrak{C}_j}} \cH^d_\infty(\frT(B_j \cap \cS)) \nonumber \\
			& = \sum_{\substack{B_j \in \mathfrak{C}_j}} \frac{\prod_{i=1}^j (\# \mathfrak{p}^{-1}(\mathfrak{p}^i(B_j)))}{\prod_{i=1}^j (\# \mathfrak{p}^{-1}(\mathfrak{p}^i(B_j)))} \cdot \cH^d_\infty(\frT(B_j \cap \cS)) \nonumber \\
			& = \sum_{\substack{B_j \in \mathfrak{C}_j}} \frac{1}{\prod_{i=1}^j (\# \mathfrak{p}^{-1}(\mathfrak{p}^i(B_j)))} \cdot \prod_{i=1}^j (\# \mathfrak{p}^{-1}(\mathfrak{p}^i(B_j))) \cdot \cH^d_\infty(\frT(B_j \cap \cS)) \nonumber \\
			& = \mathbb{E}_{B_j \in \mathfrak{C}_j} \left[ \prod_{i=1}^j (\# \mathfrak{p}^{-1}(\mathfrak{p}^i(B_j))) \cdot \cH^{d}_\infty(\frT(B_j \cap \cS)) \right]
	\end{align}
	where $\mathbb{E}_{B_j \in \mathfrak{C}_j}[\cdot]$ is the expected value among all $B_j \in \mathfrak{C}_j$ visited by a uniformly random walk starting from $\mathfrak{C}_0$, descending to $\mathfrak{C}_1$, then $\mathfrak{C}_2$, etc. Note that, by assumption (3),
	\[ \prod_{i=1}^j (\# \mathfrak{p}^{-1}(\mathfrak{p}^i(B_j))) \leq C_1^j \cdot \theta^{-\sum_{i=1}^j \mathfrak{D}(\mathfrak{p}^i(B_j))}, \]
	and by (4),
	\[ \cH^d_{\infty}(\frT(B_j \cap \cS)) \leq \omega_d \cdot (\operatorname{diam} \frT(B_j \cap \cS))^d \leq \omega_d \cdot C_2^d \cdot C_1^{dj} \cdot \theta^{d \sum_{i=1}^j \mathfrak{h}(\mathfrak{p}^i(B_j))}. \]
	Plugging back into \eqref{eq:covering.a} and using $d \geq (\xi + \mathfrak{D}(B_j))/\mathfrak{h}(B_j)$, $C_1 > 1$, $d \leq 1$,
	\begin{align*}
		\cH^d_\infty(\frT(\cS)) 
			& \leq \omega_d \cdot C_2^d \cdot \mathbb{E}_{B_j \in \mathfrak{C}_j} \left[ \left( C_1^{d+1} \theta^{j^{-1} \sum_{i=1}^j (-\mathfrak{D}(\mathfrak{p}^i(B_j)) + d \cdot \mathfrak{h}(\mathfrak{p}^i(B_j)))} \right)^j \right] \\
			& \leq \omega_d \cdot C_2^d \cdot \mathbb{E}_{B_j \in \mathfrak{C}_j} \left[ \left( C_1^2 \theta^{\xi} \right)^j \right].
	\end{align*}
	Taking $\theta \leq \theta_0$ with $C_1^{2} \theta_0^\xi < 1$ and sending $j \to \infty$ implies $\cH^d_\infty(\frT(\cS)) = 0$. It follows that $\cH^d(\frT(\cS)) = 0$ by \cite[Theorem 3.6]{Simon:GMT}, so $\dim \frT(\cS) \leq d$.
	
	(ii) Fix $j \geq 1$. For each $t \in \RR$, consider the sub-covering
	\[ \mathfrak{C}_j[t] := \{ B_j \in \mathfrak{C}_j : B_j \cap \cS \cap \frT^{-1}(t) \neq \emptyset \}. \]
	Then, since $\mathfrak{C}_j[t]$ is a covering of $\cS \cap \frT^{-1}(t)$ with sets of diameter $\leq 2a \theta^j$, we can estimate
	\[ \cH^d_\infty(\cS \cap \frT^{-1}(t)) \leq \omega_d \cdot (a\theta^j)^d \cdot (\# \mathfrak{C}_j[t]). \]
	As a result, we have the following elementary upper Lebesgue integral estimate (recall that these are  subadditive, not additive):
	\begin{align*}
		\overline{\int_\RR} \cH^d_\infty(\cS \cap \frT^{-1}(t)) \, dt
			& \leq \omega_d \cdot (a\theta^j)^d \cdot \overline{\int_\RR} (\# \mathfrak{C}_j[t]) \, dt \\
			& = \omega_d \cdot (a\theta^j)^d \cdot \overline{\int_\RR} \sum_{B_j \in \mathfrak{C}_j} \mathbf{1}_{t \in \frT(B_j \cap \cS)} \, dt \\
			& \leq \omega_d \cdot (a\theta^j)^d \cdot \sum_{B_j \in \mathfrak{C}_j} \overline{\int_\RR} \mathbf{1}_{t \in \frT(B_j \cap \cS)} \, dt \\
			& \leq \omega_d \cdot (a\theta^j)^d \cdot \sum_{B_j \in \mathfrak{C}_j} \operatorname{diam} \frT(B_j \cap \cS).
	\end{align*}
	At this point the proof is quite similar to (i), and we can again rewrite:
	\begin{equation} \label{eq:covering.b}
		\overline{\int_\RR} \cH^d_\infty(\cS \cap \frT^{-1}(t)) \, dt \leq \omega_d \cdot (a\theta^j)^d \cdot \mathbb{E}_{B_j \in \mathfrak{C}_j} \left[ \prod_{i=1}^j (\# \mathfrak{p}^{-1}(\mathfrak{p}^i(B_j))) \cdot \operatorname{diam} (\frT(B_j \cap \cS)) \right]
	\end{equation}
	As before, by assumption (3),
	\[ \prod_{i=1}^j (\# \mathfrak{p}^{-1}(\mathfrak{p}^i(B_j))) \leq C_1^j \cdot \theta^{-\sum_{i=1}^j \mathfrak{D}(\mathfrak{p}^i(B_j))}, \]
	and by (4),
	\[ \operatorname{diam} \frT(B_j \cap \cS) \leq C_2 \cdot C_1^{j} \cdot \theta^{\sum_{i=1}^j \mathfrak{h}(\mathfrak{p}^i(B_j))}. \]
	Plugging back into \eqref{eq:covering.b} and using $\mathfrak{D}(B_j) \leq \mathfrak{h}(B_j) + d - \xi$:
	\begin{align*}
		\overline{\int_\RR} \cH^d_\infty(\cS \cap \frT^{-1}(t)) \, dt
			& \leq \omega_d \cdot a^d \cdot C_2 \cdot \mathbb{E}_{B_j \in \mathfrak{C}_j} \left[ \left( C_1^2 \theta^{j^{-1} \sum_{i=1}^j(d -\mathfrak{D}(\mathfrak{p}^i(B_j)) + \mathfrak{h}(\mathfrak{p}^i(B_j)))} \right)^j \right] \\
			& \leq \omega_d \cdot a^d \cdot C_2 \cdot \mathbb{E}_{B_j \in \mathfrak{C}_j} \left[ \left( C_1^2 \theta^{\xi} \right)^j \right].
	\end{align*}
	Taking $\theta \leq \theta_0$ with $C_1^2 \theta_0^\xi < 1$ and sending $j \to \infty$ implies $\cH^d_\infty(\cS \cap \frT^{-1}(t)) = 0$ for a.e. $t$. Proceeding as in (i), $\dim (\cS \cap \frT^{-1}(t)) \leq d$ for a.e. $t$.
\end{proof}

\section{Proof of Theorem \ref{theo:small.sing.dim} and Corollary \ref{coro:small.sing.dim}} \label{sec:proof}

\subsection{Setup} \label{sec:proof.setup}

Fix $(M, g)$, $\mathscr{F}$, and $\frT$ as in the statement of Theorem \ref{theo:small.sing.dim}. We note three facts:
\begin{enumerate}
	\item By rescaling $M$ and possibly excising closer to the compact subset $\sing \mathscr{F} \subset M$, we may assume that $\breve M := B_2(\sing \mathscr{F}) \subset M$ and $(M, g)$ satisfy the injectivity radius and curvature constraints laid out in the beginning of Section \ref{sec:nonlinear}.
	\item By our assumption on $\mathscr{F}$, for each $p \in \supp \mathscr{F}$ there is a unique $\Sigma \in \mathscr{F}$ such that $p \in \supp \Sigma$. This $\Sigma \in \mathscr{F}$ will just be referred to as ``corresponding'' to $p$. 
	\item By Allard's theorem \cite{Allard:first-variation}, if $p \in \sing \mathscr{F}$, then $p \in \sing \Sigma$ for its corresponding $\Sigma \in \mathscr{F}$.
\end{enumerate}
We will invoke these facts repeatedly without further comment. 

All our constants may a priori depend on $M$, $g$, and $\mathscr{F}$. One may actually lift this dependence entirely by incorporating $M$, $g$, $\mathscr{F}$ into the various contradiction arguments provided we hold the injectivity radius and curvature constraints 
from Section \ref{sec:nonlinear} fixed, as well as the Lipschitz constant of $\mathfrak{T}$. We will not need to do this, so we avoid it.

\subsection{Conical scale structure}

We will seek to understand points of $\sing \mathscr{F}$ that look sufficiently conical with respect to $\mbfd$ at some uniform dyadic scale. 

\begin{Prop} \label{prop:conical.structure}
	Fix $\sigma \in (0, \sigma_1(n)]$, $H = H(n, \sigma)$ as in Lemma \ref{lemm:jacobi.decay.cone}, and $\theta \in (0, 1)$. There exist $\eps > 0$, $L \in \ZZ_{\geq 0}$ depending on $n, \sigma, \theta$, with the following significance. 
	
	Suppose that $\ell \in \ZZ_{\geq 0}$, $\ell \geq L$, and $\cS \subset \sing \mathscr{F}$ satisfies:
	\begin{enumerate}
		\item[(1)] For all $p \in \cS$, with $p \in \sing \Sigma$ for $\Sigma \in \mathscr{F}$, we have
			\[ \mbfd_{\cC_{p,\ell}}(\Sigma; p, 2^{-\ell}) < \eps \]
			for some nonflat minimizing hypercone $\cC_{p,\ell} \subset T_p M$.
	\end{enumerate}
	Then, for all $p \in \cS$, $\ell \in \ZZ_{\geq 0}$, $\ell \geq L$, there exists some nonflat minimizing hypercone $\check \cC_{p,\ell} \subset T_p M$ so that the following all hold:
	\begin{enumerate}
		\item[(i)] If $V_{p,\ell} := \spine \check \cC_{p,\ell}$, then
			\[ \cS \cap B_{2^{-\ell}}(p) \subset \eta_{p,2^{-\ell}}(B_\theta(V_{p,\ell})). \]
		\item[(ii)] For all $p', p^+ \in \cS \cap B_{2^{-\ell}}(p)$, if $p^+ \in \sing \Sigma^+$ for $\Sigma^+ \in \mathscr{F}$, then
			\[ \reg_{>2^{-\ell}\theta \sigma} \Sigma^+ \cap \partial B_{2^{-\ell} \theta}(p') \neq \emptyset. \]
		\item[(iii)] For all $p', p^+, p^- \in \cS \cap B_{2^{-\ell}}(p)$, if $p^\pm \in \sing \Sigma^\pm$ with $\Sigma^\pm \in \mathscr{F}$, then\footnote{We emphasize that the balls in the left and right hand side are centered at different points.}
		\begin{align*}
			& \frac{d_g(\supp \Sigma^-, \reg_{> 2^{-\ell} \sigma} \Sigma^+ \cap \partial B_{2^{-\ell}}(p))}{2^{-\ell}} \\
			& \qquad \leq H \cdot \theta^{1+\alpha(\check \cC_{p,\ell}; \sigma)} \cdot \frac{d_g(\supp \Sigma^-, \reg_{> 2^{-\ell} \theta \sigma} \Sigma^+ \cap \partial B_{2^{-\ell} \theta}(p'))}{2^{-\ell} \theta}.
		\end{align*}
            \item[(iv)] For $p$ and $\ell$ whose  $\cC_{p,\ell}$ is a copy of a hypercone in $\mathscr{C}_n^{\textnormal{qd-cyl}}$, we may also take their $\check \cC_{p,\ell}$ to be a copy of the same hypercone, possibly rotated.
	\end{enumerate}
\end{Prop}
\begin{proof}
	Consider any $p_i \in \sing \mathscr{F}$, $\mbfd_{\cC_i}(\Sigma_i; p_i, 2^{-\ell_i}) < \eps_i \to 0$ and $\ell_i \to \infty$ as in the setup.
	
	Choose any linear isometry $\iota_{p_i} : \RR^{n+1} \to (T_{p_i} M, g_{p_i})$. Passing to a subsequence, and using $\ell_i \to \infty$, we may assume that
	\[ (\eta_{p_i, 2^{-\ell_i}} \circ \iota_{p_i})^{-1} \; \Sigma_i \to \Sigma, \]
	a minimizer in $\RR^{n+1}$, and
	\[ \iota_{p_i}^{-1} \; \cC_i \to \cC, \]
	a nonflat minimizing hypercone in $\RR^{n+1}$. Also, $\mbfd_{\cC_i}(\Sigma_i; p_i, 2^{-\ell_i}) < \eps_i \to 0$ implies that
	\begin{equation} \label{eq:conical.structure.limit.cone}
		(\eta_{p_i, 2^{-\ell_i}} \circ \iota_{p_i})^{-1} \; \Sigma_i \to \cC.
	\end{equation}
	
	Let us rule out (i), (ii), (iii), or (iv) failing as $i \to \infty$ with $\check \cC_i := \cC$. By passing to a further subsequence, we may assume that one of (i), (ii), (iii), (iv) fails consistently as $i \to \infty$. We immediately observe that (iv) could not have failed as $i \to \infty$ with $\check \cC_i = \cC$ since  $\mathscr{C}_n^{\textnormal{qd-cyl}}$ is discrete modulo rotations (Corollary \ref{coro:quadratic.isolated}). So we focus on (i), (ii), (iii).
	
	To set things up, consider any $p_i' \in \sing \mathscr{F}$ with $\mbfd_{\cC_i'}(\Sigma_i'; p_i', 2^{-\ell_i}) < \eps_i$ for its corresponding $\Sigma_i'$, $\cC_i'$. Then,
	\begin{equation} \label{eq:conical.structure.shift.i}
		(\eta_{p_i, 2^{-\ell_i}} \circ \iota_{p_i})^{-1} p_i' =: \mbfv_i' \in B_1.
	\end{equation}
	Passing to a subsequence again, we may arrange for 
	\begin{equation} \label{eq:conical.structure.shift}
		\mbfv_i' \to \mbfv' \in \bar B_1.
	\end{equation}
	and
	\[ (\eta_{p_i, 2^{-\ell_i}} \circ \iota_{p_i})^{-1} \; \Sigma_i' \to \Sigma', \]
	a minimizer, and
	\[ (\iota_{p_i} \circ \iota_{p_i \to p_i'})^{-1} \; \cC_i' \to \cC', \] 
	a nonflat minimizing hypercone in $\RR^{n+1}$. It follows from $\mbfd_{\cC_i'}(\Sigma_i'; p_i', 2^{-\ell_i}) < \eps_i \to 0$ that $\Sigma' - \mbfv' = \cC'$, i.e., 
	\begin{equation} \label{eq:conical.structure.limit.cone.shift}
		(\eta_{p_i, 2^{-\ell_i}} \circ \iota_{p_i})^{-1} \; \Sigma_i' \to \cC' + \mbfv'.
	\end{equation}
	Then, \eqref{eq:conical.structure.limit.cone}, \eqref{eq:conical.structure.shift.i}, \eqref{eq:conical.structure.shift}, \eqref{eq:conical.structure.limit.cone.shift}, the fact that elements of $\mathscr{F}$ do not cross smoothly, and Lemma \ref{lemm:one.sided.cone} guarantee that $\cC' = \cC$ after perhaps swapping orientations, and
	\begin{equation} \label{eq:conical.structure.shift.spine}
		\mbfv' \in \spine \cC.
	\end{equation}
	
	Let us first see why (i) cannot fail as $i \to \infty$ with $\check \cC_i := \cC$. If it did, then \eqref{eq:conical.structure.shift.i} turns into
	\[ \mbfv_i' \in B_1 \setminus B_\theta(V), \]
	with $V = \spine \cC$, and taking limits yields
	\[ \mbfv_i' \to \mbfv' \in \bar B_1 \setminus B_\theta(V). \]
	Clearly, this contradicts \eqref{eq:conical.structure.shift.spine}, so (i) could not have failed as $i \to \infty$.
	
	Suppose now that (ii) fails as $i \to \infty$. 
	The argument above with $(p_i^+, \Sigma_i^+)$ in place of $(p_i', \Sigma_i')$ shows that 
	\[ (\eta_{p_i, 2^{-\ell_i}} \circ \iota_{p_i})^{-1} \; \reg_{> \theta \sigma} \Sigma^+_i \]
	converge locally to a superset of $\reg_{>2\theta \sigma} \cC$, and the result follows by Lemma \ref{lemm:jacobi.decay.cone} (i), applied with $\cC = \theta (\cC + \mbfv')$ in view of \eqref{eq:conical.structure.limit.cone}, \eqref{eq:conical.structure.shift.i}, \eqref{eq:conical.structure.shift}, \eqref{eq:conical.structure.limit.cone.shift}, \eqref{eq:conical.structure.shift.spine}. Thus, (ii) could also not have failed as $i \to \infty$.

	Finally suppose that (iii) fails with $p_i'$, $p_i^\pm$, $\Sigma_i^\pm$ and $\check \cC_i := \cC$. The argument above with $(p_i^\pm, \Sigma_i^\pm)$ in place of $(p_i', \Sigma_i')$ implies
	\[ (\eta_{p_i, 2^{-\ell_i}} \circ \iota_{p_i})^{-1} \; \Sigma^\pm_i \to \cC. \]
	We can take the difference of the corresponding graphical functions over an exhausting sequence of subsets of $\cC$ and use the standard Harnack inequality and the connectedness of $\cC$ to define a positive Jacobi field $u$ on $\cC$. Then the failure of the distance estimate rescales to give, in the limit as $i \to \infty$, a left-hand side that is
	\[ \leq \frac{\inf_{\reg_{> 2 \sigma} \cC \cap \partial B_{1}(\mbfo)} u}{1}, \]
	and a right-hand side that is, by \eqref{eq:conical.structure.shift.i} and \eqref{eq:conical.structure.shift},
	\[ \geq H \cdot \theta^{1+\alpha(\cC; \sigma/2)} \cdot \frac{\inf_{\reg_{>\frac12 \theta \sigma} \cC \cap \partial B_\theta(\mbfv')} u}{\theta}. \]
	This contradicts Lemma \ref{lemm:jacobi.decay.cone} (ii) and thus completes the proof.
\end{proof}

\subsection{Singular set stratification via density drop}

We will use the following stratification of $\sing \mathscr{F}$ that exploits monotonicity:

\begin{Def} \label{defi:conical.stratification}
	For $\eta > 0$, $L \in \ZZ_{\geq 0}$, define
	\[ \cS(\eta, L) := \cup_{\Sigma \in \mathscr{F}} \{ p \in \sing \Sigma : \Theta^1_\Sigma(p, 2^{-L}) - \Theta_\Sigma(p) \leq \eta \}. \]
\end{Def}

\begin{Rem} \label{rema:conical.stratification}
The monotonicity property of $\Theta^1_\Sigma(p, \cdot)$ from Section \ref{sec:nonlinear} implies:
\begin{enumerate}
	\item $\cS(\eta, L_0) \subset \cS(\eta, L_1)$ whenever $L_1 \geq L_0$, and
	\item For each $\eta > 0$, $L \mapsto \cS(\eta, L)$ is an exhaustion of $\sing \mathscr{F}$.
\end{enumerate}
\end{Rem}

The advantage of this definition is that it combines well with Proposition \ref{prop:conical.structure} and allows us to invoke it in arbitrarily deep scales, and in particular iterate Proposition \ref{prop:conical.structure} (iii) to obtain our coarse H\"older estimates. This is done in the two subsequent lemmas.

\begin{Lem} \label{lemm:conical.structure}
	Fix $\sigma,  \theta, \eps$ as in Proposition \ref{prop:conical.structure}. There exist $\eta > 0$, $L \in \ZZ_{\geq 0}$ depending on $n, \sigma, \theta, \eps$ with the following significance.
	
	Let $L_0 \in \ZZ_{\geq 0}$, $L_0 \geq L$. Then, Proposition \ref{prop:conical.structure} applies with $\cS = \cS(\eta, L_0)$ and any $\ell \in \ZZ_{\geq 0}$ with $\ell \geq L_0$; i.e., for every $p \in \cS(\eta, L_0)$ with corresponding $\Sigma \in \mathscr{F}$ and $\ell \in \ZZ_{\geq 0}$, $\ell \geq L_0$, we have 
	\[ \mbfd_{\cC_{p,\ell}}(\Sigma; p, 2^{-\ell}) < \eps \]
	to some nonflat minimizing hypercone $\cC_{p,\ell} \subset T_{p} M$.
\end{Lem}
\begin{proof}
	It is a straightforward argument by contradiction.
\end{proof}

To establish the coarse H\"older continuity necessary for Proposition \ref{prop:covering}, we will iterate Proposition \ref{prop:conical.structure} along sequences of shrinking balls, specifically, with radii shrinking like $2^{-\ell} \theta^i$, $\ell$ large but fixed, and $i = 0, 1, 2, \ldots$

\begin{Lem}[Coarse H\"older continuity via iteration] \label{lemm:conical.structure.iter}
	Fix $\sigma, \theta, \eps$ as in Proposition \ref{prop:conical.structure}, with $\theta = 2^{-m}$, $m \in \ZZ_{\geq 1}$, and $\eta, L$ as in Lemma \ref{lemm:conical.structure}.
	
	Let $\ell \in \ZZ_{\geq 0}$, $\ell \geq L$. Assume that $p^\pm, p_0, \ldots, p_j \in \cS(\eta, \ell)$ satisfy:
	\begin{enumerate}
		\item[(1)] $p_{i+1} \in B_{2^{-\ell} \theta^{i}}(p_i)$ for all $i = 0, 1, \ldots, j-1$, and
		\item[(2)] $p^\pm \in B_{2^{-\ell} \theta^{i}}(p_i)$ for all $i = 0, 1, \ldots, j$.
	\end{enumerate}
	Consider $p^\pm \in \sing \Sigma^\pm$ with $\Sigma^\pm \in \mathscr{F}$. Then:
	\begin{enumerate}
		\item[(i)] We have the scale invariant distance estimate
			\[ d_j := \frac{d_g(\supp \Sigma^-, \reg_{> 2^{-\ell} \theta^j \sigma} \Sigma^+ \cap \partial B_{2^{-\ell} \theta^j}(p_j))}{2^{-\ell} \theta^j} \leq 2. \]
		\item[(ii)] We have\footnote{We have not directly used (i) to estimate $d_j \leq 2$ in (ii), as we will want to use a stronger estimate on $d_j$ when $p^\pm$ are both modeled by cylinders over quadratic hypercones in $B_{2^{-\ell} \theta^j}(p_j)$.}
			\[ |\frT(p^+) - \frT(p^-)| \leq (\operatorname{Lip} \frT) \cdot 2^{-\ell} \cdot d_j \cdot H^j \cdot \theta^{\sum_{i=0}^{j-1} (1 + \alpha(\check \cC_i; \sigma))}, \]
			where $\check \cC_i \subset T_{p_i} M$ is the hypercone at $p_i \in \cS(\eta, \ell)$, scale $2^{-\ell} \theta^i$, given by Lemma \ref{lemm:conical.structure} and then Proposition \ref{prop:conical.structure}, and $H = H(n,\sigma)$ is as in Proposition \ref{prop:conical.structure}.
	\end{enumerate}
\end{Lem}
\begin{proof}
	(i) Apply Proposition \ref{prop:conical.structure} (ii) with
		\[ p \mapsto p_{j-1}, \; p' \mapsto p_j, \; \text{our given } p^\pm, \; \text{and } \ell \mapsto \ell + \log_2 (1/\theta)^{j-1}. \]
		This is possible by Lemma \ref{lemm:conical.structure}, (1) and (2) with $i \mapsto j-1$, and Remark \ref{rema:conical.stratification}. This yields a point
		\[
			p^+_j \in \reg_{> 2^{-\ell} \theta^j \sigma} \Sigma^+ \cap \partial B_{2^{-\ell} \theta^j}(p_j).
		\]
		Then,
		\[ d_j \leq \frac{d_g(p^-, p^+_j)}{2^{-\ell} \theta^j} \leq \frac{d_g(p^-, p_j) + d_g(p_j, p^+_j)}{2^{-\ell} \theta^j} \leq 2 \]
		where we used $p^+_j \in \partial B_{2^{-\ell} \theta^j}(p_j)$, $p^- \in B_{2^{-\ell} \theta^j}(p_j)$ by (2) with $i \mapsto j$, and the triangle inequality.
	
	(ii) An iteration of Proposition \ref{prop:conical.structure} (iii) outward from smaller balls to larger balls, i.e., iteratively with $i'=1, \ldots, j$ and
	\[ p \mapsto p_{j-i'}, \; p' \mapsto p_{j-i'+1}, \; \ell \mapsto \ell + \log_2 (1/\theta)^{j-i'} = \ell + (j-i') m, \]
	which is possible by (1) and (2) with $i \mapsto j-i'$ and Remark \ref{rema:conical.stratification}, yields:
	\begin{align*} 
		\frac{d_g(\supp \Sigma^-, \reg_{> 2^{-\ell} \sigma} \Sigma^+ \cap \partial B_{2^{-\ell}}(p_0))}{2^{-\ell}} 
			& \leq d_j \cdot H^j \cdot \theta^{\sum_{i'=1}^j (1 + \alpha(\check \cC_{j-i'}; \sigma))} \nonumber \\
			& = d_j \cdot H^j \cdot \theta^{\sum_{i=0}^{j-1} (1 + \alpha(\check \cC_{i}; \sigma))}.
	\end{align*}
	
	Choose $\hat p^\pm \in \supp \Sigma^\pm$ attaining $d_g$ in the left hand side above. Rearranging,
	\[
		d_g(\hat p^+, \hat p^-) \leq 2^{-\ell} \cdot d_j \cdot H^j \cdot \theta^{\sum_{i=0}^{j-1} (1 + \alpha(\check \cC_i; \sigma))}.
	\]
	Using in this order the facts that $\frT$ is constant on each $\supp \Sigma^\pm$, as well as Lipschitz, and then the bound on $d_g(\hat p^+, \hat p^-)$, we get the chain of inequalities:
	\begin{align*}
		|\frT(p^+) - \frT(p^-)| 
			& = |\frT(\hat p^+) - \frT(\hat p^-)| \\
			& \leq (\operatorname{Lip} \frT) \cdot d_g(\hat p^+, \hat p^-) \\
			& \leq (\operatorname{Lip} \frT) \cdot 2^{-\ell} \cdot d_j \cdot H^j \cdot \theta^{\sum_{i=0}^{j-1} (1 + \alpha(\check \cC_i; \sigma))}.
\	\end{align*}
	This completes the proof.
\end{proof}

\subsection{Special case: singular points modeled by the same hypercone}

When we invoke Lemma \ref{lemm:conical.structure.iter} on $p^\pm$ modeled by a copy of the same nonflat minimizing hypercone $\cC \subset \RR^{n+1}$, we can replace the estimate $d_j \leq 2$ from Lemma \ref{lemm:conical.structure.iter} (i) to one where $d_j$ is controlled by how close each $\Sigma^\pm$ is to its corresponding copy $\cC^\pm \subset T_{p^\pm}$ of $\cC$ at scale $2^{-\ell} \theta^j$. This can improve the estimate on $|\frT(p^+) - \frT(p^-)|$ in Lemma \ref{lemm:conical.structure.iter} (ii).

Throughout this section we fix some nonflat minimizing hypercone $\cC \subset \RR^{n+1}$. 

\begin{Prop}\label{Prop:frank-transl-rot}
There exists $\sigma=\sigma(\cC) \in (0,\tfrac 1 8)$ with the following property.

If $\cC^\pm \in \Rot(\cC)$ and $\mbfx^\pm \in B_{4\sigma}$, then 
\[
(\reg_{>3\sigma}\cC^- +\mbfx^-) \cap (\reg_{>3\sigma}\cC^+ + \mbfx^+) \cap (B_1 \setminus B_{2\sigma}) \not = \emptyset. 
\]
\end{Prop}
\begin{proof}
Rotating and translating a sequence of counterexamples, there is $\cC_j \in \Rot(\cC)$ and $|\mbfx_j| < 8j^{-1}$ so that 
\begin{equation} \label{eq:frankel-trl-rot-contr-0}
(\reg_{>3j^{-1}} \cC) \cap (\reg_{>3j^{-1}} \cC_j + \mbfx_j) \cap (B_{1/2} \setminus B_{6j^{-1}}) = \emptyset. 
\end{equation}
Letting $j\to\infty$, the strong maximum principle gives that $\cC_j$ limits to $\cC$. In particular,
\[ \cC_j = e^{\mbfA_j} \cC \]
for some $\mbfA_j \in \mathfrak{so}(n+1)$ with $\mbfA_j \to \mbfO$ as $j \to \infty$. Let's then rewrite \eqref{eq:frankel-trl-rot-contr-0} as:
\begin{equation} \label{eq:frankel-trl-rot-contr}
(\reg_{>3j^{-1}} \cC) \cap (e^{\mbfA_j} \reg_{>3j^{-1}} \cC + \mbfx_j) \cap (B_{1/2} \setminus B_{6j^{-1}}) = \emptyset. 
\end{equation}
(In the middle term, the translation by $\mbfx_j$ happens after the rotation by $e^{\mbfA_j}$.) By suitably modding out by symmetries, we arrange the following:

\begin{Cla} \label{clai:frankel-trl-rot}
    There exist $\mbfx'_j \in \RR^{n+1}$, $\mbfA'_j \in \mathfrak{so}(n+1)$ such that:
    \begin{equation} \label{eq:frankel-trl-rot-yj-Bj}
	(\reg_{>3j^{-1}}\cC) \cap (e^{\mbfA'_j} \reg_{>3j^{-1}} \cC + \mbfx'_j) \cap (B_{1/4} \setminus B_{15j^{-1}}) = \emptyset,
    \end{equation}
    \begin{equation} \label{eq:frankel-trl-rot-lambda}
	0 < |\mbfx'_j| + |\mbfA'_j| =: \lambda_j \to 0,
    \end{equation}
\end{Cla}
\begin{proof}[Proof of claim]
    We decompose
\[ \mbfx_j =: \mbfx_j^\parallel + \mbfx_j^\perp \in \spine \cC \oplus (\spine \cC)^\perp. \]
Moreover, denote
\[ \mathfrak{s} := \{ \mbfB \in \mathfrak{so}(n+1) : \langle \mbfB p, \nu_\cC(p) \rangle = 0 \text{ for all } p \in \cC \}, \]
and let $\mathfrak{s}^\perp$ be any complementary subspace in $\mathfrak{so}(n+1)$, e.g., one obtained by choosing an inner product and taking orthogonal complements. Then, decompose
\[ \mbfA_j =: \mbfA_j^\parallel + \mbfA_j^\perp \in \mathfrak{s} \oplus \mathfrak{s}^\perp. \]

Choose
\[ \mbfx_j' := e^{-\mbfA_j^\parallel} \mbfx_j^\perp, \]
\[ \mbfA_j' := \mbfA_j^\perp + \mbfB_j, \]
where $\mbfB_j$ is chosen so that
\[ e^{\mbfA_j^\perp + \mbfB_j} = e^{-\mbfA_j^\parallel} e^{\mbfA_j^\parallel + \mbfA_j^\perp} (= e^{-\mbfA_j^\parallel} e^{\mbfA_j}). \]
Note that this can be arranged by Lemma \ref{Lem:exp-addition-vs-mult-matrices}, which additionally guarantees that
\begin{equation} \label{eq:frankel-trl-rot-bj}
	|\mbfB_j| \leq C |\mbfA_j^\parallel| |\mbfA_j^\perp|.
\end{equation}

We may now obtain \eqref{eq:frankel-trl-rot-yj-Bj} from \eqref{eq:frankel-trl-rot-contr}: translate by $-\mbfx_j^\parallel$, rotate by $e^{-\mbfA_j^\parallel}$ (these actions preserve the $\reg_{>3j^{-1}} \cC$ on the left), and use $|\mbfx_j| < 8j^{-1}$.

We fix these $\mbfx_j'$, $\mbfA_j'$ for the remainder of the proof, and set $\lambda_j := |\mbfx_j'| + |\mbfA_j'|$. Then, the decay assertion in \eqref{eq:frankel-trl-rot-lambda} immediately follows from $|\mbfx_j| \to 0$, $|\mbfA_j| \to 0$, and \eqref{eq:frankel-trl-rot-bj}. To see $\lambda_j > 0$, note that if $\lambda_j = 0$, then $\mbfx_j' = \orig$ and $\mbfA_j' = \mbfO$, so \eqref{eq:frankel-trl-rot-yj-Bj} cannot hold. This completes the proof of \eqref{eq:frankel-trl-rot-lambda}, and thus the claim.
\end{proof}

We also have:

\begin{Cla} \label{clai:frankel-trl-rot-jf}
    Assume the notation of Claim \ref{clai:frankel-trl-rot}. Writing 
    \[
	e^{\mbfA'_j} \cC + \mbfx'_j \supset \graph_\cC h_j,
    \]
    for some $h_j$ over a connected exhaustion of $\cC \cap B_{1/4}$, then subsequentially
    \begin{equation} \label{eq:frankel-trl-rot-linearized}
	\lambda_j^{-1} h_j \to u \in \Jac_\textnormal{trl}(\cC) \oplus \Jac_\textnormal{rot}(\cC) \text{ in } C^\infty_{\textnormal{loc}}(\cC \cap B_{1/4}),
    \end{equation}
    and after perhaps swapping unit normals:
    \begin{equation} \label{eq:frankel-trl-rot-linearized-2}
        u > 0.
    \end{equation}
\end{Cla}
\begin{proof}[Proof of claim]
    Passing to a subsequence, denote
    \begin{equation} \label{eq:frankel-trl-rot-yB}
	\lambda_i^{-1} (\mbfx_j', \mbfA_j') \to (\mbfy, \mbfB) \in ((\spine \cC)^\perp \oplus \mathfrak{s}^\perp) \setminus \{ (\orig, \mbfO) \}.
    \end{equation}
    The fact that $\mbfB \in \mathfrak{s}^\perp$ follows from $\mbfA_j^\perp \in \mathfrak{s}^\perp$ and \eqref{eq:frankel-trl-rot-bj}. 

    We relate $h_j$ to $\mbfy$, $\mbfB$. Fixing some $p \in \cC$ (to be determined), we can write $\cC$ locally as a graph over $T_p \cC$ of some function $h$ with $h(\orig) = |\nabla h(\orig)| = 0$ and $|D^2 h| \leq C r_\cC(p)^{-1}$ on $\cC \cap B_{cr_{\cC}(p)} \subset T_p \cC$. In particular $\|\nabla h\|_{C^0(\cC \cap B_t)} \leq C r_{\cC}(p)^{-1} |t|$ for $t \in (0,cr_\cC(p))$. For $j$ sufficiently large (depending on $p$), Lemma \ref{Lem:affine-acting-graph} thus applies with $p$, $\mbfx_j'$, $\mbfA_j'$ in place of $\mbfz$, $\mbfx$, $\mbfA$, to show that
\begin{equation} \label{eq:frankel-trl-rot-hj-est}
	|h_j(p) - \langle \mbfA_j' p + \mbfx_j', \nu_\cC(p) \rangle| \leq C(p) (|\mbfx_j'| + |\mbfA_j'|)^2 = C(p) \lambda_j^2
\end{equation}
Moreover, by \eqref{eq:frankel-trl-rot-yB},
\begin{equation} \label{eq:frankel-trl-rot-yjBj-est}
	|\lambda_j^{-1} \langle \mbfA_j' p + \mbfx_j', \nu_\cC(p) \rangle - \langle \mbfB p + \mbfy, \nu_\cC(p) \rangle| \leq \epsilon_{p,i} \to 0
\end{equation}
as $i \to \infty$. It follows from \eqref{eq:frankel-trl-rot-hj-est}, \eqref{eq:frankel-trl-rot-yjBj-est}, that
\begin{equation}\label{eq:frankel-trl-rot-limit-p}
\lambda_j^{-1}h_j(p) \to \langle \mbfB p + \mbfy,\nu_\cC(p)\rangle \text{ pointwise for } p \in \cC \cap B_{1/4}.
\end{equation}
Note that the right hand side above is an element of $\Jac_\textnormal{trl}(\cC) \oplus \Jac_\textnormal{rot}(\cC)$ (see Definitions \ref{defi:translation.jf}, \ref{defi:rotation.jf}). This implies \eqref{eq:frankel-trl-rot-linearized}.

Moreover, \eqref{eq:frankel-trl-rot-contr} and \eqref{eq:frankel-trl-rot-yj-Bj} imply that $h_j \geq 0$, perhaps after swapping unit normals. Thus, \eqref{eq:frankel-trl-rot-linearized-2} will follow from \eqref{eq:frankel-trl-rot-limit-p}, the Harnack inequality, and standard elliptic theory if we can show that the right hand side of \eqref{eq:frankel-trl-rot-limit-p} is not identically $0$.

Suppose not. Then, 
\[ p \mapsto \langle \mbfB p + \mbfy, \nu_\cC(p) \rangle \]
is a linear combination of a translation and a rotation Jacobi field that vanishes on $\cC \cap B_{1/4}$, and thus on $\cC$ by unique continuation. Thus, both the translation and the rotation parts vanish:
\[ p \mapsto \langle \mbfB p, \nu_\cC(p) \rangle \equiv \langle \mbfy, \nu_\cC(p) \rangle \equiv 0, \]
since one grows linearly and the other does not. But recall that $\langle \mbfy, \nu_\cC(p) \rangle \equiv 0$ implies $\mbfy \in \spine \cC$, and $\langle \mbfB p, \nu_\cC(p) \rangle \equiv 0$ implies $\mbfB \in \mathfrak{s}$. This is a contradiction to \eqref{eq:frankel-trl-rot-yB}. This completes the proof that there exists $p \in \cC \cap B_{1/4}$ so that the right hand side of \eqref{eq:frankel-trl-rot-limit-p} is nonvanishing. This completes the proof of \eqref{eq:frankel-trl-rot-linearized-2} and thus the claim.
\end{proof}

The $u$ constructed in Claim \ref{clai:frankel-trl-rot-jf} leads to a contradiction. Indeed, $u$ is bounded near $\orig$, being the sum of translation and a rotation Jacobi fields by  \eqref{eq:frankel-trl-rot-linearized}, and is positive by  \eqref{eq:frankel-trl-rot-linearized-2}. This violates a well-known result of L. Simon \cite{Simon:decay} (cf.\ \cite[Theorem 1.4]{Wang:smoothing}) about positive Jacobi fields. This contradiction completes the proof.
\end{proof}

\begin{Prop}\label{prop:close-to-cone-close-to-eachother}
There exists $\sigma_0=\sigma_0(\cC) > 0$ such that for all $\sigma \in (0, \sigma_0)$ we can find $\tau_0 > 0$, $\eps_0 > 0$, and $C_0 > 0$ depending on $\cC, \sigma$ and have the following properties. 

Consider $\tau \in (0,\tau_0)$, $p \in \breve M$, minimizing $\Sigma^\pm$, $p^\pm \in B_{\sigma \tau}(p) \cap \supp \Sigma^\pm$, and copies $\cC^\pm = \iota_{p^\pm}(\cC) \subset T_{p^\pm} M$ such that:
\begin{enumerate}
	\item[(1)] $\mbfd_{\cC^\pm}(\Sigma^\pm;p^\pm,\tau) < \eps_0$, and
	\item[(2)] $\Sigma^\pm$ do not cross smoothly.
\end{enumerate}
Then, we have the scale-invariant distance estimate:
\[
	\frac{d_g(\supp \Sigma^-, \reg_{>\tau \sigma} \Sigma^+ \cap \partial B_{\tau}(p))}{\tau} \leq C_0 \cdot (\mbfd_{\cC^+}(\Sigma^+;p^+,\tau) + \mbfd_{\cC^-}(\Sigma^-;p^-,\tau) + \tau^2).
\]
\end{Prop}
\begin{proof}
We always assume $\sigma, \tau < \tfrac12$ below.

Let $\iota_{p^\pm \to p} : (T_{p^\pm}M, g_{p^\pm}) \to (T_{p}M, g_{p})$ be the parallel transport map along the minimizing geodesic from $p^\pm$ to $p$. This is well-defined since $p \in \breve M$. Then
\[
\tilde \cC^\pm := \iota_{p^\pm\to p}(\cC^\pm) \subset T_{p}M,
\]
are copies of $\cC$ in $T_{p} M$. Recalling \eqref{eq:defn-eta-tau} and \eqref{eq:defn-g-tau} we can write
\[
p^{\pm} := \eta_{p,\frac \tau2}(\mbfv^{\pm}), \; \mbfv^\pm \in B_{2\sigma} \subset T_{p}M.
\]
Invoke Proposition \ref{Prop:frank-transl-rot} to find $\sigma := \sigma(\cC)<\frac 18$ and $\mbfz \in T_{p}M$ so that
\[
\mbfz \in (\reg_{>3\sigma} \tilde \cC^- + \mbfv^-) \cap (\reg_{>3\sigma} \tilde\cC^+ +\mbfv^+) \cap (B_{1}\setminus B_{2\sigma}) \subset T_p M.
\]
Let $z=\eta_{p,\frac\tau2}(\mbfz) \in B_{\frac \tau 2}(p)\setminus B_{\sigma\tau}(p)$. Taking $\tau_0$ small, Lemma \ref{Lem:change-bspt} gives
\[
|\eta_{p^\pm,\frac \tau2}^{-1}(z) - \iota_{p\to p^\pm}(\mbfz - \mbfv^\pm)| \leq C\tau^2
\]
where $\iota_{p\to p^{\pm}} = (\iota_{p^\pm \to p})^{-1}$ is the parallel transport map. As such, we can find $\mbfw^\pm \in \reg_{>3\sigma}\cC^\pm$ with
\[
d_{g_{p^\pm,\frac \tau2}}(\eta_{p^\pm,\frac\tau2}^{-1}(z),\mbfw^\pm) \leq C\tau^2.
\]
Taking $\tau_0$ even smaller, we have  
\[
d_g(p^\pm,\eta_{p^\pm,\frac\tau 2}(\mbfw^\pm)) \leq d_g(p^\pm,p)+ d_g(p,z) + d_g(z,\eta_{p^\pm,\frac\tau 2}(\mbfw^\pm)) \leq  \sigma\tau + \frac 12 \tau + C\tau^3 \leq \frac {3\tau}{4}.
\]
As such, if we take $\eps_0,\tau_0$ even smaller, elliptic estimates (for the graphical function of $(\Sigma^\pm)_{p^\pm,\frac\tau2}$ over $\cC^\pm$ near $\mbfw^\pm$) and the triangle inequality give
\[
d_{g_{p^{\pm},\frac\tau2}}(\eta_{p^\pm,\frac\tau2}^{-1}(z), \reg_{>2\sigma} (\Sigma^\pm)_{p^\pm,\frac\tau2}) \leq C(\mbfd_{\cC^\pm}(\Sigma^\pm;p^\pm,\tau) + \tau^2). 
\]
Thus
\[
\frac{d_{g}(z, \reg_{>\tau \sigma} \Sigma^\pm)}{\tau} \leq C(\mbfd_{\cC^\pm}(\Sigma^\pm;p^\pm,\tau) + \tau^2). 
\]
Taking $\eps_0,\tau_0$ even smaller, $z \in B_\tau(p) \setminus B_{\sigma\tau}(p)$ implies that 
\[
\frac{d_{g}(z, \reg_{>\tau \sigma} \Sigma^\pm \cap (B_{2\tau}(p) \setminus B_{\frac 12\sigma\tau}(p)))}{\tau} \leq C(\mbfd_{\cC^\pm}(\Sigma^\pm;p^\pm,\tau) + \tau^2). 
\]
Using the triangle inequality we thus conclude that
\begin{align*}
	& \frac{d_g(\reg_{>\tau \sigma} \Sigma^-, \reg_{>\tau \sigma} \Sigma^+ \cap (B_{2\tau}(p)\setminus B_{\frac12 \sigma \tau}(p)))}{\tau} \\
	& \qquad \leq C(\mbfd_{\cC^+}(\Sigma^+;p^+,\tau) + \mbfd_{\cC^-}(\Sigma^-;p^-,\tau) + \tau^2). 
\end{align*}
Taking $\eps_0$, $\tau_0$ smaller if necessary, the assertion follows from the Harnack inequality as encoded by Claim \ref{claim:harnack-to-bdry-sphere} below.
\end{proof}

\begin{Cla}\label{claim:harnack-to-bdry-sphere}
For all sufficiently small $\tau$, $\eps > 0$ and all sufficiently large $C>0$ depending on $\cC, \sigma$, we have
\begin{align*}
	& d_g(\supp \Sigma^-, \reg_{>\tau\sigma} \Sigma^+ \cap \partial B_{\tau}(p)) \\
	& \qquad \leq C \cdot d_g(\reg_{>\tau \sigma}\Sigma^-, \reg_{>\tau \sigma} \Sigma^+ \cap (B_{2\tau}(p)\setminus B_{\frac12 \tau\sigma}(p)))
\end{align*}
for all $p,p^\pm,\Sigma^\pm,\cC^\pm$ as in the statement of Proposition \ref{prop:close-to-cone-close-to-eachother}.
\end{Cla}
\begin{proof}
    If this fails, we can find $\Sigma^\pm_j,\cC^\pm_j,p^\pm_j,p_j$, and $\tau_j\to 0$, so that $\mbfd_{\cC^\pm_j}(\Sigma^\pm_j;p^\pm_j,\tau_j)\to 0$ but the assertion fails with $C = j$. By Lemmas \ref{lemm:one.sided.cone} and \ref{Lem:change-bspt} we find that $\iota_{p_j}^{-1} \; (\Sigma_j^\pm)_{p_j,\tau_j}$ both converge to $\cC$. Normalizing the difference of the graphs appropriately, we can use the Harnack inequality to find a positive Jacobi field $u$ on $\cC$ with 
    \[
    \inf_{\reg_{>2 \sigma} \cC \cap \partial B_1} u \geq 1,
    \]
    but
    \[
    \inf_{\reg_{>\frac12 \sigma} \cC \cap (B_2\setminus B_{\frac12 \sigma})}u = 0.
    \]
    This contradicts the Harnack inequality due to the connectedness of $\cC$.
\end{proof}

\subsection{Main dichotomy}

Recall the space $\mathscr{C}_n^{\textnormal{qd-cyl}}$ of minimizing hypercones that are cylindrical over a quadratic hypercone from \eqref{eq:mathscr.n.qd-cyl}.

For the reasons outlined in the introduction, we will have to have different arguments to treat points modeled by $\cC \in \mathscr{C}_n^{\textnormal{qd-cyl}}$. Moreover, among such points, our arguments will further depend on the decay order of the approximating Jacobi field. Let us denote, for fixed $k \in \{ 0, 1, \ldots, n-7 \}$,
\begin{equation} \label{eq:delta.hf.gap}
	\Delta^{\textnormal{qd}}_{n-k} := \min \{ \Delta^{>1}_{\cC} : \cC \in \mathscr{C}_n^{\textnormal{qd-cyl}}, \; \dim \spine \cC \geq k \} \in (0, 1],
\end{equation}
where $\Delta^{>1}_\cC$ is as in Remark \ref{rema:gamma.star.c}. Note that this is indeed a minimum since $\Rot(\cdot)$ does not affect $\Delta^{>1}_\cC$, and $\mathscr{C}_{n-k}^{\textnormal{qd}}$ is finite up to rotations. We also denote
\begin{equation} \label{eq:delta.hf.gap.all}
	\Delta^{\textnormal{qd}}_{} := \min \{ \Delta^{\textnormal{qd}}_{n-k} : k \in \{ 0, 1, \ldots, n-7 \} \} \in (0, 1],
\end{equation}
This depends on $n$ of course, but we omit this from the notation.

\begin{Lem}[Dichotomy] \label{lemm:main.dichotomy}
	Fix $\delta \in (0, \tfrac12 \Delta^{\textnormal{qd}}_{})$. There exist $\eta > 0$, $L \in \ZZ_{\geq 0}$ depending on $n$, $\delta$ with the following significance.
	
	Suppose that $L_0 \in \ZZ_{\geq 0}$, $L_0 \geq L$, and $p \in \cS(\eta, L_0)$. Write $p \in \sing \Sigma$ for $\Sigma \in \mathscr{F}$. Then, exactly one of the following holds:
	\begin{enumerate}
		\item[(I)] There exists a unique $\cC \in \mathscr{C}_n^{\textnormal{qd-cyl}}$ up to rotations, and a fixed copy $\cC_p \subset T_p M$ of it (independent of scale) such that:
			\[ \cN^1_{\cC_p}(\Sigma; p, 2^{-\ell}) > 1 + \Delta^{>1}_{\cC} - \delta \text{ for all } \ell \in \ZZ_{\geq 0}, \; \ell \geq L_0. \]
		\item[(II)] There exists some $L_1 \in \ZZ_{\geq 0}$, $L_1 \geq L_0$, such that for all $\cC \in \mathscr{C}_n^{\textnormal{qd-cyl}}$:
			\[ \bar \cN^{1}_{\cC}(\Sigma; p, 2^{-\ell}) \leq 1 + \delta \text{ for all } \ell \in \ZZ_{\geq 0}, \; \ell \geq L_1. \]
	\end{enumerate}
\end{Lem}
\begin{proof}
	Note that up to rotations $\mathscr{C}_n^{\textnormal{qd-cyl}}$ is a finite set. The statement then follows from Corollary \ref{Cor:unif-high-freq-quantitative} and that $\delta < \tfrac12 \Delta^{\textnormal{qd}}_{}$.
\end{proof}

\begin{Rem} \label{rema:L0.L1}
	We emphasize that in case (II) of Lemma \ref{lemm:main.dichotomy} we have no control over how large $L_1 \geq L_0$ may need to be, so we need to consider all possibilities.
\end{Rem}

\begin{Def} \label{defi:main.dichotomy}
	Let $\delta$, $\eta$, $L$ be as in Lemma \ref{lemm:main.dichotomy}. Take any $L_0 \in \ZZ_{\geq 0}$ with $L_0 \geq L$.
	\begin{enumerate}
		\item[(I)] Denote $\cS^{>1}_{\cC}(\eta, L_0) \subset \cS(\eta, L_0)$ the set of $p$ for which Lemma \ref{lemm:main.dichotomy} (I) holds with $\eta$, $L_0$, and a copy $\cC_p \subset T_p M$ of some $\cC \in \mathscr{C}_n^{\textnormal{qd-cyl}}$, i.e.:
			\[ \cN^1_{\cC_p}(\Sigma; p, 2^{-\ell}) > 1 + \Delta^{>1}_{\cC} - \delta \text{ for all } \ell \in \ZZ_{\geq 0}, \; \ell \geq L_0. \]
		\item[(II)] Denote $\cS^{=1}_{L_1}(\eta, L_0) \subset \cS(\eta, L_0)$ the set of $p$ for which Lemma \ref{lemm:main.dichotomy} (II) holds with $\eta$ and some $L_1 \in \ZZ_{\geq 0}$ with $L_1 \geq L_0$, i.e., for every $\cC \in \mathscr{C}_n^{\textnormal{qd-cyl}}$:
			\[ \bar \cN^{1}_{\cC}(\Sigma; p, 2^{-\ell}) \leq 1 + \delta \text{ for all } \ell \in \ZZ_{\geq 0}, \; \ell \geq L_1. \]
	\end{enumerate}
\end{Def}

An immediate consequence of Lemma \ref{lemm:main.dichotomy} and Definition \ref{defi:main.dichotomy} is that:

\begin{equation} \label{eq:main.dichotomy}
	\cS(\eta, L_0) = \left( \bigcup_{\cC \in \mathscr{C}_n^{\textnormal{qd-cyl}}} \cS^{>1}_\cC(\eta, L_0) \right) \cup \left( \bigcup_{L_1 \geq L_0} \cS^{=1}_{L_1}(\eta, L_0) \right).
\end{equation}

\begin{Rem} \label{rema:main.dichotomy.hf.qd}
	The first union, over $\cC \in \mathscr{C}_n^{\textnormal{qd-cyl}}$, is really a finite union since $\cS^{>1}_\cC(\eta, L_0)$ is independent of rotations of $\cC$.
\end{Rem}

\subsection{Points in $\cS^{>1}_{\cC}(\eta, L_0)$} \label{sec:hf.qd.points}

Fix $\delta \in (0, \tfrac12 \Delta^{\textnormal{qd}}_{})$ (see \eqref{eq:delta.hf.gap.all}), $\sigma \in (0, \min \{ \sigma_0(n), \sigma_1(n) \})$ (see Proposition \ref{prop:close-to-cone-close-to-eachother} and Lemma \ref{lemm:jacobi.decay.cone}), and some $\theta = 2^{-m}$, $m \in \ZZ_{\geq 1}$, to be determined.

Given these, we take $\eps = \eps(n, \sigma, \theta) > 0$, $\eta = \eta(n, \delta, \sigma, \theta) > 0$, $L = L(n, \delta, \sigma, \theta) \in \ZZ_{\geq 0}$ so that Proposition \ref{prop:conical.structure}, Lemmas \ref{lemm:conical.structure}, \ref{lemm:main.dichotomy}, Proposition \ref{prop:close-to-cone-close-to-eachother}, and Corollary \ref{Cor:unif-high-freq-quantitative} all hold for all $\cC \in \mathscr{C}_n^{\textnormal{qd-cyl}}$ (there is only a discrete choice of $\cC$, see Remark \ref{rema:main.dichotomy.hf.qd}).

Fix some $\cC \in \mathscr{C}_n^{\textnormal{qd-cyl}}$ and $L_0 \in \ZZ_{\geq 0}$ with $L_0 \geq L$. Then Corollary \ref{Cor:unif-high-freq-quantitative} and Lemma \ref{lemm:main.dichotomy} (I) imply that there exists
\[ L_2 = L_2(n, L_0, \eps) \geq L_0 \]
so that:
\begin{equation} \label{eq:hf.qd.points.eps}
	\mbfd_{\cC_p}(\Sigma; p, 2^{-\ell}) < \eps \text{ for all } \ell \geq L_2,
\end{equation}
where $p \in \cS^{>1}_\cC(\eta, L_0)$, $p \in \sing \Sigma$, $\Sigma \in \mathscr{F}$, and some copy $\cC_p \subset T_p M$ of $\cC$.

We proceed to construct iterative covers $\mathfrak{C}_j$ of $\cS^{>1}_\cC(\eta, L_0)$ by subsets of radius-$2^{-L_2} \theta^j$ balls in a manner compatible with Proposition \ref{prop:covering}. Each element
\[ B_j \in \mathfrak{C}_j \]
will be constructed to be an intersection of a chain of balls with geometrically decreasing radii:
\[ B_j = \bigcap_{i=0}^j B_{2^{-L_2} \theta^i}(p_i); \]
cf. Lemma \ref{lemm:conical.structure.iter}'s (2) for why this format is useful. 

The construction of the various $\mathfrak{C}_j$ and the parent map $\mathfrak{p} : \mathfrak{C}_{j+1} \to \mathfrak{C}_j$ is as follows:

\begin{itemize}
	\item For the base case, $j=0$, take any Vitali cover
		\[ \mathfrak{C}_0 = \{  B_{2^{-L_2}}(p_0) \}_{p_0} \text{ of } \cS^{>1}_\cC(\eta, L_0) \]
		using radius-$2^{-L_2}$ balls centered at various points $p_0 \in \cS^{>1}_\cC(\eta, L_0)$.
        \item For the inductive step, consider $j \geq 0$ and any already constructed
            \[ B_j := \bigcap_{i=0}^j B_{2^{-L_2} \theta^i}(p_i). \]
        Then, take any Vitali cover
		\[ \mathfrak{C}_{j+1}[B_j] = \{ B_{2^{-L_2} \theta^{j+1}}(p_{j+1}) \}_{p_{j+1}} \text{ of } B_j \cap \cS^{>1}_\cC(\eta, L_0) \]
        using radius-$2^{-L_2} \theta^{j+1}$ balls centered at various points $p_{j+1} \in B_j \cap \cS^{>1}_\cC(\eta, L_0)$. For each such $p_{j+1}$, consider the corresponding
        \[ B_{j+1} := B_{2^{-L_2} \theta^{j+1}}(p_{j+1}) \cap B_j = \bigcap_{i=0}^{j+1} B_{2^{-L_2} \theta^i}(p_i). \]
        If it's empty or has already occurred at the $j$-th step, discard it. Otherwise, set
        \[ \mathfrak{p}(B_{j+1}) := B_j. \]
    \item Once all $B_j \in \mathfrak{C}_j$ have been considered and their $\mathfrak{C}_{j+1}[B_j]$ has been constructed as above, define
        \[ \mathfrak{C}_{j+1} := \bigcup_{B_j \in \mathfrak{C}_j} \mathfrak{C}_{j+1}[B_j]. \]
\end{itemize}

We now seek to establish the quantitative properties for $\mathfrak{C}_j$ required by Proposition \ref{prop:covering} (3) and (4). To that end, revisit the inductive construction starting from 
\[ B_j = \bigcap_{i=0}^j B_{2^{-L_2} \theta^i}(p_i). \]
In view of \eqref{eq:hf.qd.points.eps}, we know that Proposition \ref{prop:conical.structure} applies for $p_j \in \sing \Sigma$, $\Sigma \in \mathscr{F}$, at scale $2^{-L_2} \theta^j$ with some ``a priori'' hypercone $\cC(B_j)$, and yields some ``a posteriori'' hypercone $\check \cC(B_j)$, and that both hypercones are copies of our fixed $\cC$ by (iv) of the proposition. 

By (i) of Proposition \ref{prop:conical.structure}, 
\[ B_j \cap \cS^{>1}_\cC(\eta, L_0) \subset \eta_{p_j, 2^{-L_2} \theta^j}(B_\theta(\spine \check \cC(B_j))). \]
Since $\dim \spine \check \cC(B_j) = \dim \spine \cC = k$, our curvature bounds on $(M, g)$ allow us to estimate the cardinality of the Vitali cover by
\begin{equation} \label{eq:hf.qd.points.card}
    \# \mathfrak{p}^{-1}(B_j) = \# \mathfrak{C}_{j+1}[B_j] \leq C_1' \cdot \theta^{-k}
\end{equation}
for a geometric constant $C_1' = C_1'(n)$. 

Next, we claim that, for every $p^\pm \in B_j \cap \cS^{>1}_\cC(\eta, L_0)$,
\begin{equation} \label{eq:hf.qd.points.holder}
    |\frT(p^+) - \frT(p^-)| \leq (\operatorname{Lip} \frT) \cdot 2^{-L_2} \cdot C_2' \cdot H^j \cdot \theta^{j(1 + \alpha(\cC; \sigma) + \Delta^{>1}_\cC - \delta)},
\end{equation}
with $C_2' = C_2'(n, \sigma) > 0$. We defer the proof of \eqref{eq:hf.qd.points.holder} until the section's end. (It follows by combining Lemma \ref{lemm:conical.structure.iter} with Proposition  \ref{prop:close-to-cone-close-to-eachother} and Corollary \ref{Cor:unif-high-freq-quantitative}.) Taking \eqref{eq:hf.qd.points.holder} for granted and then taking suprema over $p^\pm \in B_j \cap \cS^{>1}_\cC(\eta, L_0)$:
\begin{equation} \label{eq:hf.qd.points.diam}
    \operatorname{diam} \frT(B_j \cap \cS^{>1}_\cC(\eta, L_0)) \leq C_2 \cdot (C_1'')^j \cdot \theta^{j(1 + \alpha(\cC; \sigma) + \Delta^{>1}_\cC - \delta)},
\end{equation}
for $C_2 = C_2(n, \sigma, L_2, \operatorname{Lip} \frT) > 0$ and $C_1'' = H(n, \sigma) > 0$.
  
Fix $C_1 = \max \{ C_1', C_1'' \}$ and $C_2$ as above. Then, \eqref{eq:hf.qd.points.card}, \eqref{eq:hf.qd.points.diam} imply that Proposition \ref{prop:covering}'s (3), (4) are satisfied with
\[ \mathfrak{D}(B_j) := k, \]
\[ \mathfrak{h}(B_j) := 1 + \alpha(\cC; \sigma) + \Delta^{>1}_{\cC} - \delta, \]
for all $B_j \in \mathfrak{C}_j$, $j = 0, 1, 2, \ldots$ Thus, Proposition \ref{prop:covering} implies:

\begin{Cor} \label{coro:hf.qd.points}
	Fix $\xi > 0$, $\delta \in (0, \tfrac12 \Delta^{\textnormal{qd}}_{})$ (see Lemma \ref{lemm:main.dichotomy}) and $\sigma \in (0, \min \{ \sigma_0(n), \sigma_1(n) \})$ (see Proposition \ref{prop:close-to-cone-close-to-eachother} and Lemma \ref{lemm:jacobi.decay.cone}). Then, there exist $\eta > 0$, $L \in \ZZ_{\geq 0}$ depending on $n, \delta, \sigma, \xi$ so that, for all $L_0 \in \ZZ_{\geq 0}$ with $L_0 \geq L$, 
	\begin{equation} \label{eq:hf.qd.points.dim.epsilon}
		\dim \frT(\cS^{>1}_\cC(\eta, L_0)) \leq \frac{k + \xi}{1 + \alpha(\cC; \sigma) + \Delta^{>1}_{\cC} - \delta},
	\end{equation}
	and, for a.e. $t \in \RR$,
	\begin{equation} \label{eq:hf.qd.points.codim.epsilon}
		\dim (\cS^{>1}_\cC(\eta, L_0) \cap \frT^{-1}(t)) \leq \max \{ k - (1 + \alpha(\cC; \sigma) + \Delta^{>1}_{\cC} - \delta), 0 \} + \xi.
	\end{equation}
	We emphasize that, above, $k = \dim \spine \cC \in \{ 0, 1, \ldots, n-7 \}$. 
\end{Cor}
\begin{proof}
    Our choice of $\sigma$ fixes $C_1 = C_1(n, \sigma)$ in \eqref{eq:hf.qd.points.diam}. 
    
    We may then fix $\theta = \theta(n, \sigma, \xi)$ of the form $2^{-m}$, $m \in \ZZ_{\geq 1}$, so that Proposition \ref{prop:covering} applies with this $C_1 = C_1(n, \sigma)$ and our chosen $\xi$. 
	
	Then take $\eps = \eps(n, \sigma, \xi) > 0$ so that Propositions \ref{prop:conical.structure} and  \ref{prop:close-to-cone-close-to-eachother} both hold, and $\eta = \eta(n, \delta, \sigma, \xi) > 0$, $L = L(n, \delta, \sigma, \xi) \in \ZZ_{\geq 0}$ so that Proposition \ref{prop:conical.structure}, Lemmas \ref{lemm:conical.structure}, \ref{lemm:main.dichotomy}, and Corollary \ref{Cor:unif-high-freq-quantitative} all apply. 
    
    Then, for all $L_0 \geq L$, the coverings we constructed precisely fit into Proposition \ref{prop:covering}, which yields the required result.
\end{proof}

\begin{proof}[Proof of \eqref{eq:hf.qd.points.holder}]
	Corollary \ref{Cor:unif-high-freq-quantitative} (ii) implies the following distance to the unique tangent cones $\cC_{p^\pm}$ at $p^\pm \in \sing \Sigma^\pm$:
	\begin{equation} \label{eq:hf.qd.points.holder.1}
		\mbfd_{\cC_{p^\pm}}(\Sigma^\pm; p^\pm, 2^{-L_2} \theta^j) \leq C(n) \cdot (2^{-L_2} \theta^j \cdot  2^{L_0})^{(\Delta^{>1}_\cC - \delta)} \leq C(n) \cdot \theta^{j(\Delta^{>1}_\cC - \delta)},
	\end{equation}
	where we used $L_2 \geq L_0$. We may plug \eqref{eq:hf.qd.points.holder.1} into Proposition  \ref{prop:close-to-cone-close-to-eachother} to obtain:
	\begin{align} \label{eq:hf.qd.points.holder.2}
		d_j 
			& := \frac{d_g(\supp \Sigma^-, \reg_{> 2^{-L_2} \theta^j \sigma} \Sigma^+ \cap \partial B_{2^{-L_2} \theta^j}(p_j))}{2^{-L_2} \theta^j} \nonumber \\
			& \leq C_0 \cdot \left( 2 \cdot C(n) \cdot \theta^{j(\Delta^{>1}_\cC - \delta)} + (2^{-L_2} \theta^j)^2 \right) \nonumber \\
			& \leq C_2' \cdot \theta^{j(\Delta^{>1}_\cC - \delta)},
	\end{align}
	with $C_2' = C_2'(n, \sigma) > 0$. The result follows from Lemma \ref{lemm:conical.structure.iter} (ii) applied with $\ell = L_2$ (recall that $\cS(\eta, L_0) \subset \cS(\eta, L_2)$ by Remark \ref{rema:conical.stratification}), with  \eqref{eq:hf.qd.points.holder.2} used to bound $d_j$, and with Proposition \ref{prop:conical.structure} (iv) to see that the $\check \cC_i$ are rotations of our fixed $\cC$.
\end{proof}

\subsection{Points in $\cS^{=1}_{L_1}(\eta, L_0)$} \label{sec:lf.nonqd.points}

Fix $\sigma \in (0, \sigma_1(n))$ (see Lemma \ref{lemm:jacobi.decay.cone}) and some $\theta = 2^{-m}$, $m \in \ZZ_{\geq 1}$, to be determined. 

Now take $\delta = \delta(n, \theta) \in (0, \tfrac12 \Delta^{\textnormal{qd}}_{})$ (see \eqref{eq:delta.hf.gap.all}), $\eps = \eps(n, \sigma, \theta) > 0$, $\eta = \eta(n, \sigma, \theta) > 0$, $L = L(n, \sigma, \theta) \in \ZZ_{\geq 0}$ so that Proposition \ref{prop:conical.structure}, Lemmas \ref{lemm:conical.structure}, \ref{lemm:main.dichotomy}, and Theorem \ref{theo:freq-1-subspace} all hold for all $\cC \in \mathscr{C}_n^{\textnormal{qd-cyl}}$. In what follows, we work with $L_0, L_1 \in \ZZ_{\geq 0}$, with
\[ L_1 \geq L_0 \geq L. \]
and fix
\[ L_2 := L_1 + L. \]

We proceed to construct iterative covers $\mathfrak{C}_j$ of $\cS^{=1}_{L_1}(\eta, L_0)$ by subsets of radius-$2^{-L_2} \theta^j$ balls in a manner compatible with Proposition \ref{prop:covering}. The process is like the one in the previous section. However, rather than improve the coarse H\"older exponents (by improving Lemma \ref{lemm:conical.structure.iter} (ii) using Proposition \ref{prop:close-to-cone-close-to-eachother}), we will improve the coarse dimensions (by improving Proposition \ref{prop:conical.structure} (ii) using Theorem \ref{theo:freq-1-subspace}). The construction of the various $\mathfrak{C}_j$ and the parent map $\mathfrak{p} : \mathfrak{C}_{j+1} \to \mathfrak{C}_j$ is as follows:

\begin{itemize}
	\item For the base case, $j=0$, take any Vitali cover
		\[ \mathfrak{C}_0 = \{ B_{2^{-L_2}}(p_0) \}_{p_0} \text{ of } \cS^{=1}_{L_1}(\eta, L_0) \]
		using radius-$2^{-L_2}$ balls centered at various points $p_0 \in \cS^{=1}_{L_1}(\eta, L_0)$. 
        \item For the inductive step, consider $j \geq 0$ and any already constructed
            \[ B_j := \bigcap_{i=0}^j B_{2^{-L_2} \theta^i}(p_i). \]
        Then, take any Vitali cover
		\[ \mathfrak{C}_{j+1}[B_j] = \{ B_{2^{-L_2} \theta^{j+1}}(p_{j+1}) \}_{p_{j+1}} \text{ of } B_j \cap \cS^{=1}_{L_1}(\eta, L_0) \]
        using radius-$2^{-L_2} \theta^{j+1}$ balls centered at various points $p_{j+1} \in B_j \cap \cS^{=1}_{L_1}(\eta, L_0)$. For each such $p_{j+1}$, consider the corresponding
        \[ B_{j+1} := B_{2^{-L_2} \theta^{j+1}}(p_{j+1}) \cap B_j = \bigcap_{i=0}^{j+1} B_{2^{-L_2} \theta^i}(p_i). \]
        If it's empty or has already occurred at the $j$-th step, discard it. Otherwise, set
        \[ \mathfrak{p}(B_{j+1}) := B_j. \]
    \item Once all $B_j \in \mathfrak{C}_j$ have been considered and their $\mathfrak{C}_{j+1}[B_j]$ has been constructed as above, define
        \[ \mathfrak{C}_{j+1} := \bigcup_{B_j \in \mathfrak{C}_j} \mathfrak{C}_{j+1}[B_j]. \]
\end{itemize}

We now establish the quantitative properties for $\mathfrak{C}_j$ required by Proposition \ref{prop:covering} (3) and (4). To that end, revisit the inductive construction starting from 
\[ B_j = \bigcap_{i=0}^j B_{2^{-L_2} \theta^i}(p_i). \]
In view of Lemma \ref{lemm:conical.structure}, we know that Proposition \ref{prop:conical.structure} applies for $p_j \in \sing \Sigma$, $\Sigma \in \mathscr{F}$, at scale $2^{-L_2} \theta^j$ with some ``a priori'' hypercone $\cC(B_j)$, and yields some ``a posteriori'' hypercone $\check \cC(B_j)$. These need not be congruent hypercones anymore.

It follows from Proposition \ref{prop:conical.structure} (i) and Theorem \ref{theo:freq-1-subspace} that
\[ B_j \cap \cS^{=1}_{L_1}(\eta, L_0) \subset \eta_{p_j, 2^{-L_2} \theta^j}(B_\theta(V(B_j))) \]
where $V(B_j)$ is a linear subspace determined as follows:
\begin{itemize}
    \item If $\cC(B_j)$ is a hypercopy of a cone in $\mathscr{C}_n^{\textnormal{qd-cyl}}$, then by Theorem \ref{theo:freq-1-subspace} $V(B_j)$ can be taken to be a proper linear subspace of $\spine \cC(B_j)$, i.e.,
    \[ \dim V(B_j) \leq \dim \spine \cC(B_j)-1 = \dim \spine \check \cC(B_j)-1. \]
    The last equality is from Proposition \ref{prop:conical.structure} (iv).
    \item In all other cases, $V(B_j) = \spine \check \cC(B_j)$ by Proposition \ref{prop:conical.structure}, so
    \[ \dim V(B_j) = \dim \spine \check \cC(B_j). \]
\end{itemize}
Our curvature bounds on $(M, g)$ allow us to estimate
\begin{equation} \label{eq:lf.nonqd.points.card}
    \# \mathfrak{p}^{-1}(B_j) = \# \mathfrak{C}_{j+1}[B_j] \leq C_1' \cdot \theta^{-\dim V(B_j)}
\end{equation}
for a geometric constant $C_1' = C_1'(n)$.

It follows from Lemma \ref{lemm:conical.structure.iter} (i), (ii) and taking suprema over $B_j \cap \cS^{=1}_{L_1}(\eta, L_0)$ that
\begin{equation} \label{eq:lf.nonqd.points.diam}
    \operatorname{diam} \frT(B_j \cap \cS^{=1}_{L_1}(\eta, L_0)) \leq C_2 \cdot (C_1'')^j \cdot \theta^{\sum_{i=1}^{j} (1 + \alpha(\check \cC(\mathfrak{p}^i(B_j)); \sigma))},
\end{equation}
for $C_2 = C_2(n, L_2, \operatorname{Lip} \frT) > 0$ and $C_1'' = H(n, \sigma) > 0$.

Fix $C_1 = \max \{ C_1', C_1'' \}$ and $C_2$ as above. Then, \eqref{eq:lf.nonqd.points.card}, \eqref{eq:lf.nonqd.points.diam} imply that Proposition \ref{prop:covering}'s (3), (4) are satisfied with
\[ \mathfrak{D}(B_j) := \dim V(B_j), \]
\[ \mathfrak{h}(B_j) := 1 + \alpha(\check \cC(B_j); \sigma), \]
for all $B_j \in \mathfrak{C}_j$, $j = 0, 1, 2, \ldots$ Thus, Proposition \ref{prop:covering} implies:

\begin{Cor} \label{coro:lf.nonqd.points}
	Fix $\xi > 0$, $\sigma \in (0, \sigma_1(n))$ (see Lemma \ref{lemm:jacobi.decay.cone}). For all  sufficiently small $\delta > 0$ depending on $n, \sigma, \xi$, there exist $\eta > 0$, $L \in \ZZ_{\geq 0}$ depending on $n, \sigma, \delta, \xi$ so that for all $L_0, L_1 \in \ZZ_{\geq 0}$ with $L_1 \geq L_0 \geq L$,
	\begin{align} \label{eq:lf.nonqd.points.dim.epsilon}
		& \dim \frT(\cS^{=1}_{L_1}(\eta, L_0)) \nonumber \\
		& \qquad \leq \max \Big\{ \max_{\cC \in \mathscr{C}_{n}^{\textnormal{qd-cyl}}} \Big\{ \frac{\dim \spine \cC - 1 + \xi}{1 + \alpha(\cC; \sigma)} \Big\}, \nonumber \\
		& \qquad \qquad \qquad \sup_{\cC \in \mathscr{C}_n \setminus \mathscr{C}_n^{\textnormal{qd-cyl}}} \Big\{ \frac{\dim \spine \cC + \xi}{1 + \alpha(\cC; \sigma)} \Big\} \Big\},
	\end{align}
	and, for a.e. $t \in \RR$,
	\begin{align} \label{eq:lf.nonqd.points.codim.epsilon}
		& \dim (\cS^{=1}_{L_1}(\eta, L_0) \cap \frT^{-1}(t)) \nonumber \\
		& \qquad \leq \max \Big\{ \max_{\cC \in \mathscr{C}_{n}^{\textnormal{qd-cyl}}} \Big\{ \dim \spine \cC -(2+\alpha(\cC; \sigma)) \Big\}, \nonumber \\
		& \qquad \qquad \qquad \sup_{\cC \in \mathscr{C}_n \setminus \mathscr{C}_n^{\textnormal{qd-cyl}}} \Big\{ \dim \spine \cC - (1 + \alpha(\cC; \sigma)) \Big\}, 0 \Big\} + \xi.
	\end{align}
\end{Cor}
\begin{proof}
    Our choice of $\sigma$ forces a $C_1 = C_1(n, \sigma)$ as defined after \eqref{eq:lf.nonqd.points.diam}. 
    
    We may then fix $\theta = \theta(n, \sigma, \xi)$ of the form $2^{-m}$, $m \in \ZZ_{\geq 1}$, so that Proposition \ref{prop:covering} applies with this $C_1 = C_1(n, \sigma)$ and our chosen $\xi$. 
    
    Having fixed $\theta = \theta(n, \sigma, \xi)$, our smallness requirement on $\delta$ is:
	\begin{itemize}
		\item We allow any $\delta = \delta(n, \theta) = \delta(n, \sigma, \xi) > 0$ so that Lemma \ref{lemm:main.dichotomy} and Theorem \ref{theo:freq-1-subspace} apply.
	\end{itemize}
	
	Then take $\eps = \eps(n, \sigma, \xi) > 0$ so that Proposition \ref{prop:conical.structure} and Theorem \ref{theo:freq-1-subspace} hold, and $\eta = \eta(n, \sigma, \delta, \xi) > 0$, $L = L(n, \sigma, \delta, \xi) \in \ZZ_{\geq 0}$ so that Proposition \ref{prop:conical.structure}, Lemmas \ref{lemm:conical.structure}, \ref{lemm:main.dichotomy}, and Theorem \ref{theo:freq-1-subspace} all apply.
    
    Then, for all $L_0, L_1 \in \ZZ_{\geq 0}$ with $L_1 \geq L_0 \geq L$, the coverings we constructed for $\cS^{=1}_{L_1}(\eta, L_0)$ precisely fit into Proposition \ref{prop:covering}. Moreover, for $\cC \in \mathscr{C}_n^{\textnormal{qd-cyl}}$, we saw how to obtain the asserted 1-dimensional improvement in $\dim \spine \cC$.
\end{proof}

\subsection{Proof of Theorem \ref{theo:small.sing.dim}}

(i) Note that $\dim \frT(\sing \mathscr{F}) \leq 1$ is automatic since $\frT$ is $\RR$-valued, so it suffices to just prove $\dim \frT(\sing \mathscr{F}) \leq { \mathfrak{d}_{n}^{\textnormal{img}}}$.

Recall from the remark after Definition \ref{defi:conical.stratification} that, for any $\eta > 0$, 
\[ L_0 \mapsto \cS(\eta, L_0) \]
is a monotone exhaustion of $\sing \mathscr{F}$. In particular, since Hausdorff dimension bounds are preserved under countable unions, it suffices to prove
\begin{equation} \label{eq:small.sing.dim.a.wts}
	\dim \frT(\cS(\eta, L_0)) \leq {\mathfrak{d}_{n}^{\textnormal{img}}}
\end{equation}
for a fixed $\eta > 0$, and all large enough $L_0 \in \ZZ_{\geq 0}$.

Let $\xi > 0$ be arbitrary. Fix $\epsilon > 0$ so that Proposition \ref{prop:alpha.sigma.good} (i) and (iii) apply with this choice of $\xi$. Then choose $\sigma > 0$ sufficiently small so that Proposition \ref{prop:alpha.sigma.good} (ii) also applies with this choice of $\epsilon$, and so do Corollaries \ref{coro:hf.qd.points}, \ref{coro:lf.nonqd.points} with these $\sigma, \xi$. Finally, choose any $\delta > 0$ small enough that Lemma \ref{lemm:main.dichotomy}
and Corollaries \ref{coro:hf.qd.points}, \ref{coro:lf.nonqd.points} holds. 

Lemma \ref{lemm:main.dichotomy} and Corollaries \ref{coro:hf.qd.points}, \ref{coro:lf.nonqd.points} then yield $\eta > 0$, $L \in \ZZ_{\geq 0}$. Choose any $L_0 \in \ZZ_{\geq 0}$ where $L_0 \geq L$. By our previous discussion, it suffices to prove \eqref{eq:small.sing.dim.a.wts} with this $L_0$. To that end, we use the decomposition of $\cS(\eta, L_0)$ from \eqref{eq:main.dichotomy}, i.e.,
\begin{equation} \label{eq:small.sing.dim.decomposition}
	\cS(\eta, L_0) = \left( \bigcup_{\cC \in \mathscr{C}_n^{\textnormal{qd-cyl}}} \cS^{>1}_\cC(\eta, L_0) \right) \cup \left( \bigcup_{L_1 \geq L_0} \cS^{=1}_{L_1}(\eta, L_0) \right).
\end{equation}

We now observe that Corollary \ref{coro:hf.qd.points} can bound, for $\cC \in \mathscr{C}_n^{\textnormal{qd-cyl}}$ with $\dim \spine \cC = k$,
\begin{align} \label{eq:small.sing.dim.hf.qd.fixed}
	\dim \frT(\cS^{>1}_\cC(\eta, L_0)) 
		& \leq \frac{k + \xi}{1 + \alpha(\cC; \sigma) + \Delta^{>1}_{\cC} - \delta} \nonumber \\
		& \leq \frac{k + \xi}{1 + \alpha_{n-k} + \Delta^{\textnormal{qd}}_{n-k} - \delta - \xi}.
\end{align}
The first inequality is merely Corollary \ref{coro:hf.qd.points}'s \eqref{eq:hf.qd.points.dim.epsilon}. For the second inequality we used Proposition \ref{prop:alpha.sigma.good} (i) to estimate $\alpha(\cC; \sigma) > \alpha_{n-k} - \xi$ and \eqref{eq:delta.hf.gap} to estimate $\Delta^{>1}_{\cC} \geq \Delta^{\textnormal{qd}}_{n-k}$. Taking unions over all $\cC \in \mathscr{C}_n^{\textnormal{qd-cyl}}$ (finitely many up to rotation, see Remark \ref{rema:main.dichotomy.hf.qd}):
\begin{equation} \label{eq:small.sing.dim.hf.qd}
	\dim \frT \left( \bigcup_{\cC \in \mathscr{C}_n^{\textnormal{qd-cyl}}} \cS^{>1}_\cC(\eta, L_0) \right) \leq \max_{k=0,1,\ldots,n-7} \Big\{ \frac{k + \xi}{1 + \alpha_{n-k} + \Delta^{\textnormal{qd}}_{n-k} - \delta - \xi} \Big\}.
\end{equation}

We also observe that Corollary \ref{coro:lf.nonqd.points} can bound, for all $L_1 \geq L_0$,
\begin{align} \label{eq:small.sing.dim.lf.nonqd.fixed}
		& \dim \frT(\cS^{=1}_{L_1}(\eta, L_0)) \nonumber \\
		& \qquad \leq \max \Big\{ \max_{\cC \in \mathscr{C}_{n}^{\textnormal{qd-cyl}}} \Big\{ \frac{\dim \spine \cC - 1 + \xi}{1 + \alpha(\cC; \sigma)} \Big\}, \sup_{\cC \in \mathscr{C}_n \setminus \mathscr{C}_n^{\textnormal{qd-cyl}}} \Big\{ \frac{\dim \spine \cC + \xi}{1 + \alpha(\cC; \sigma)} \Big\} \Big\} \nonumber \\
		& \qquad \leq \max_{k=0,1,\ldots,n-7} \Big\{ \frac{k - 1 + \xi}{1 + \alpha_{n-k} - \xi}, \frac{k + \xi}{1 + \alpha_{n-k} + \Delta_{n-k}^{\textnormal{non-qd}} - \xi} \Big\}.
\end{align}
The first inequality is again merely Corollary \ref{coro:lf.nonqd.points}'s \eqref{eq:lf.nonqd.points.dim.epsilon}. For the second inequality, need a combination of arguments. 
\begin{itemize}
	\item For $\cC \in \mathscr{C}_{n}^{\textnormal{qd-cyl}}$ with $\dim \spine \cC = k$, Proposition \ref{prop:alpha.sigma.good}'s (i) estimates
		\[ \alpha(\cC; \sigma) > \alpha_{n-k} - \xi, \]
		so
		\[ \frac{\dim \spine \cC - 1 + \xi}{1 + \alpha(\cC; \sigma)} \leq \frac{k-1+\xi}{1 + \alpha_{n-k} - \xi}, \]
		which is present in the final maximum expression.
	\item For $\cC \in \mathscr{C}_n \setminus \mathscr{C}_n^{\textnormal{qd-cyl}}$ with $\dim \spine \cC = k$ which \emph{is not} $\epsilon$-close to $\mathscr{C}_n^{\textnormal{qd-cyl}}$ in the metric sense of Proposition \ref{prop:alpha.sigma.good}, conclusion (ii) of Proposition \ref{prop:alpha.sigma.good} estimates
		\[ \alpha(\cC; \sigma) > \alpha_{n-k} + \Delta_{n-k}^{\textnormal{non-qd}} - \xi, \]
		so
		\[ \frac{\dim \spine \cC + \xi}{1 + \alpha(\cC; \sigma)} \leq \frac{k + \xi}{1 + \alpha_{n-k} + \Delta_{n-k}^{\textnormal{non-qd}} - \xi}, \]
		which is present in the final maximum expression.
	\item For $\cC \in \mathscr{C}_n \setminus \mathscr{C}_n^{\textnormal{qd-cyl}}$ with $\dim \spine \cC = k$ which \emph{is} $\epsilon$-close to $\mathscr{C}_n^{\textnormal{qd-cyl}}$ in the metric sense of Proposition \ref{prop:alpha.sigma.good}, conclusion (iii) of Proposition \ref{prop:alpha.sigma.good} guarantees that $\cC$ is $\epsilon$-close to a cone $\cC' \in \mathscr{C}_n$ with $k' := \dim \spine \cC' \geq k+1$. Then, conclusion (i) of Proposition \ref{prop:alpha.sigma.good} estimates
		\[ \alpha(\cC; \sigma) > \alpha_{n-(k+1)} - \xi, \]
		so
		\[ \frac{\dim \spine \cC + \xi}{1 + \alpha(\cC; \sigma)} \leq \frac{k + \xi}{1 + \alpha_{n-(k+1)} - \xi} = \frac{(k+1)-1 + \xi}{1 + \alpha_{n-(k+1)} - \xi}, \]
		which is present in the final maximum expression (since $k+1 \leq k' \leq n-7$).		
\end{itemize}
Now that we have justified \eqref{eq:small.sing.dim.lf.nonqd.fixed}, we recall Definition \ref{defi:main.dichotomy} and take countable unions over $L_1 \in \ZZ_{\geq 0}$ with $L_1 \geq L_0$ to obtain:
\begin{align} \label{eq:small.sing.dim.lf.nonqd}
		& \dim \frT \left( \bigcup_{L_1 \geq L_0} \cS^{=1}_{L_1}(\eta, L_0) \right) \nonumber \\
		& \qquad \leq \max_{k=0,1,\ldots,n-7} \Big\{ \frac{k - 1 + \xi}{1 + \alpha_{n-k} - \xi}, \frac{k + \xi}{1 + \alpha_{n-k} + \Delta_{n-k}^{\textnormal{non-qd}} - \xi} \Big\}.
\end{align}
Together, \eqref{eq:small.sing.dim.decomposition}, \eqref{eq:small.sing.dim.hf.qd}, \eqref{eq:small.sing.dim.lf.nonqd} and the definition of ${ \mathfrak{d}_{n}^{\textnormal{img}}}$ in \eqref{eq:delta.n} imply:
\[
	\dim \frT(\cS(\eta, L_0)) \leq { \mathfrak{d}_{n}^{\textnormal{img}}} + \Psi(\xi, \delta),
\]
where $\Psi(a, b)$ denotes an expression which $\to 0$ as $|a| + |b| \to 0$. Sending the arbitrary $L_0 \to \infty$, and recalling Remark \ref{rema:conical.stratification}, we deduce that
\[ \dim \frT(\sing \mathscr{F}) = \dim \frT \left( \bigcup_{L_0} \cS(\eta, L_0) \right) \leq {\mathfrak{d}_{n}^{\textnormal{img}}} + \Psi(\xi, \delta). \]
The result now follows by sending $\delta \to 0$, $\xi \to 0$.

(ii) The proof is essentially the same, so we describe the modifications. Instead of \eqref{eq:small.sing.dim.hf.qd.fixed}, we prove that, for a.e. $t \in \RR$,
\[
	\dim (\cS^{>1}_\cC(\eta, L_0) \cap \frT^{-1}(t)) \leq \max \{ k - (1 + \alpha_{n-k} + \Delta^{\textnormal{qd}}_{n-k} - \delta - \xi), 0 \} + \xi.
\]
by using Corollary \ref{coro:hf.qd.points}'s \eqref{eq:hf.qd.points.codim.epsilon} instead of \eqref{eq:hf.qd.points.dim.epsilon}. Instead of \eqref{eq:small.sing.dim.lf.nonqd.fixed}, we prove that for a.e. $t \in \RR$, 
\begin{align*}
		& \dim (\cS^{=1}_{L_1}(\eta, L_0) \cap \frT^{-1}(t)) \\
		& \qquad \leq \max_{k=0,1,\ldots,n-7} \{ k - (2 + \alpha_{n-k} - \xi), k - (1 + \alpha_{n-k} + \Delta_{n-k}^{\textnormal{non-qd}} - \xi) \}, 0 \Big\} + \xi.
\end{align*}
by using Corollary \ref{coro:lf.nonqd.points}'s \eqref{eq:lf.nonqd.points.codim.epsilon} instead of \eqref{eq:lf.nonqd.points.dim.epsilon}. Then, since $\Delta_{n-k}^{\textnormal{qd}} \in (0, 1]$, it follows from these two inequalities and the definition of ${\mathfrak{d}_{n}^{\textnormal{dom}}}$ in \eqref{eq:cee.n} that for a.e. $t \in \RR$:
\[
	\dim (\sing \mathscr{F} \cap \frT^{-1}(t)) \leq {\mathfrak{d}_{n}^{\textnormal{dom}}} + \Psi(\xi, \delta).
\]
We may thus conclude as before. This completes the proof of Theorem \ref{theo:small.sing.dim}.

\subsection{Proof of Corollary \ref{coro:small.sing.dim}}

	We use that $\Delta^{\textnormal{non-qd}}_{n-k}, \Delta^{\textnormal{qd}}_{n-k} > 0$ for all $k \leq n-7$ without further comment. This follows from Lemma \ref{lemm:kappa.zhu.nonqd} and Remark \ref{rema:gamma.star.c}, respectively.
	
	(i) Note that $n+1 \leq 11$ implies $k \leq 3$. We study the fractions in the ``$\max$'' expression for $\mathfrak{d}_{n}^{\textnormal{img}}$ as $0 \leq k \leq 3$ varies and show they are all $< 1$. There are two cases:
	\begin{itemize}
		\item $k \leq 2$. The first fraction has numerator $k \leq 2$, and denominator $> 1 + \alpha_{n-k} > 2$ by Lemma \ref{lemm:kappa.zhu}, so its value is $< 1$. The second fraction has numerator $k-1 \leq 1$ and denominator still $> 2$, so this fraction's value, too, is $< 1$.
		\item $k = 3$. Then, $n=10$. The first fraction has numerator $k = 3$ and denominator $> 1 + \alpha_7 = 3$ by Lemma \ref{lemm:kappa.zhu}, so its value is $< 1$. The second fraction has numerator $k-1 = 2$ and denominator $= 3$, so this fraction's value, too, is $< 1$.
	\end{itemize}
	Altogether, $\mathfrak{d}_{n}^{\textnormal{img}} < 1$.
	
	(ii) Likewise, we show that each term in the ``$\max$'' expression for $\mathfrak{d}_{n}^{\textnormal{dom}}$, as $0 \leq k \leq n-7$ varies,  is $\leq n-10-\epsilon_n$ where 
	\begin{align*}
		\epsilon_n 
			& := \min \Big\{ \alpha_7 + \min \{ \Delta^{\textnormal{qd}}_7, \Delta^{\textnormal{non-qd}}_7 \} - 2, \\
			& \qquad \qquad \min_{k=0,1,\ldots,n-8} \{ \alpha_{n-k} +  \min \{ \Delta^{\textnormal{qd}}_{n-k}, \Delta^{\textnormal{non-qd}}_{n-k} \} - 1 \} \Big\}.
	\end{align*}
	Note that $\epsilon_n > 0$. There are two cases:
	\begin{itemize}
		\item $k \leq n-8$. The parenthetical term is
			\[ 1 + \alpha_{n-k} + \min \{ \Delta^{\textnormal{qd}}_{n-k}, \Delta^{\textnormal{non-qd}}_{n-k} \} \geq 2 + \epsilon_n.  \]
			Thus, the full term is $\leq k - (2 + \epsilon_n) \leq n-10 - \epsilon_n$.
		\item $k = n-7$. The parenthetical term is
			\[ 1 + \alpha_7 + \min \{ \Delta^{\textnormal{qd}}_{n-k}, \Delta^{\textnormal{non-qd}}_{n-k} \} \geq 3 + \epsilon_n. \]
			Thus, the full term is $\leq k - (3 + \epsilon_n) \leq n - 10 - \epsilon_n$.
	\end{itemize}
This completes the proof.

\section{Proof of Theorem \ref{theo:main.rn.11} (Plateau problem in $\RR^{n+1}$)} \label{sec:plateau.rn}

In this section we provide the proof of Theorem \ref{theo:main.rn.11}. 

\begin{Prop} \label{prop:plateau.rn}
	Let $\Gamma \subset \RR^{n+1}$ be as in Theorem \ref{theo:main.rn.11} and $T$ be a minimizing integral $n$-current with $\partial T = \llbracket \Gamma \rrbracket$. Assume that $T$ additionally satisfies:
	\begin{enumerate}
		\item[(1)] $T$ is uniquely minimizing, and
		\item[(2)] $T$ has multiplicity one on $\reg T$.
	\end{enumerate}
	There exist $C^\infty$-small perturbations $\Gamma'$ of $\Gamma$ for which all minimizing integral $n$-currents $T'$ with $\partial T' = \llbracket \Gamma' \rrbracket$ have the desired regularity of Theorem \ref{theo:main.rn.11}, i.e., $T' = \llbracket \Sigma' \rrbracket$ for a smooth, precompact, oriented hypersurface $\Sigma' \subset \RR^{n+1}$ with $\partial \Sigma' = \Gamma'$ and
	\[ \sing \Sigma' = \emptyset \text{ if } n+1 \leq 11, \text{ else } \dim \sing \Sigma' \leq n-10-\epsilon_n, \]
	where $\epsilon_n > 0$ is a dimensional constant.
\end{Prop}

\begin{proof}[Proof of Theorem \ref{theo:main.rn.11} assuming Proposition \ref{prop:plateau.rn}]
This reduction was carried out in \cite[Section 6]{CMS:generic.codim.plus.2}, and we recall the main steps here.

We may perturb  $\Gamma \mapsto \Gamma'$ so that minimizing integral $n$-currents $T'$ with $\partial T' = \llbracket \Gamma' \rrbracket$ satisfy the two extra constraints of Proposition \ref{prop:plateau.rn}. Specifically, (1) gets arranged as in \cite[Section 6.1]{CMS:generic.codim.plus.2}, and (2) as in \cite[Section 6.3]{CMS:generic.codim.plus.2}. The argument relied on the Hardt--Simon boundary regularity theorem for minimizing integral $n$-currents and White's improvement, \cite{Hardt-Simon:boundary-regularity, White:boundary-regularity-multiplicity}.\footnote{See \cite[Proof of Theorem 4.1 assuming Lemmas 4.1, 4.2]{CMS:generic.9.10} for another approach.} 

Thus, we may apply Proposition \ref{prop:plateau.rn} with $\Gamma', T'$ in place of $\Gamma, T$. This yields $\Gamma''$ and $\Sigma''$ with the asserted improved regularity. 
\end{proof}

\begin{proof}[Proof of Proposition \ref{prop:plateau.rn}]
	First we note:
	
	\begin{Cla} \label{clai:plateau.rn}
	We may find a precompact open set $M \subset \RR^{n+1}$ with smooth boundary $\partial M$ and the following properties:
	\begin{enumerate}
		\item[(i)] $\Gamma \subset \partial M$, $\spt T \setminus \Gamma \subset M$;
		\item[(ii)] $\spt T$ and $\partial M$ meet smoothly and transversely along $\Gamma$;
		\item[(iii)] $T \mres M$ is a minimizing boundary in $M$.
	\end{enumerate}
	We may also construct a smooth deformation $(\Gamma_t)_{t \in [0,1]} \subset \partial M$ of $\Gamma_0 := \Gamma$ so that:
	\begin{enumerate}
		\item[(iv)] for all $t \in [0,1]$, every minimizing $T_t$ with $\partial T_t = \llbracket \Gamma_t \rrbracket$ satisfies (i)-(iii) with $T_t$ and $\Gamma_t$ in place of $T$ and $\Gamma$;
		\item[(v)] for all pairs $t^\pm \in [0,1]$, every pair of distinct minimizing $T^\pm$ with $\partial T^\pm = \llbracket \Gamma_{t^\pm} \rrbracket$ have disjoint supports in $M$;
		\item[(vi)] there exists $c > 0$ such that, for all $t^\pm \in [0,1]$, $T^\pm$ as in (v),
			\[ |t^+-t^-| \leq c |\mbfx^+ - \mbfx^-|, \; \text{ for all } \mbfx^\pm \in \spt T^\pm. \]
	\end{enumerate}
	\end{Cla}
	\begin{proof}[Proof of claim]	
	Statements (i)-(iv) follow in a standard way. Indeed, from the uniqueness of $T_0$ (assumption (1)), any choice of minimizers $T_{t_i}$ with $t_i \to 0$ converge to $T_0$. Moreover, from the unit density of $T_0$ (assumption (2)) and Allard's interior \cite{Allard:first-variation} and boundary regularity \cite{Allard:boundary-regularity}, the convergence $T_{t_i} \to T_0$ is smooth up to the boundary.\footnote{See the proof of \cite[Lemma 4.3]{CMS:generic.9.10} for a full exposition.} 
	
	It remains to establish (v)-(vi). To that end, we will take the deformation $(\Gamma_t)_{t \in [0,1]}$ to be via a local foliation along $\partial M$ such that:
		\begin{enumerate}
			\item[(I)] the normal speed of the foliation $(\Gamma_t)_{t \in [0,1]}$ is positive;
			\item[(II)] for all $t^\pm \in [0,1]$ with $t^+ > t^-$, $\llbracket \Gamma_{t^+} \rrbracket - \llbracket \Gamma_{t^-} \rrbracket$ is the boundary of a positively oriented smooth subdomain in $\partial M$.
		\end{enumerate}
		
	Now, (v) follows from the standard cut-and-paste trick (see \cite[Lemma 2.8]{CMS:generic.9.10}), whose boundary monotonicity assumptions are guaranteed by (II) above.

	Finally, in view of (I), the inequality in (vi) is equivalent (for a different $c$) to:
		\begin{equation} \label{eq:plateau.rn.1}
			d(\Gamma_{t^+}, \Gamma_{t^-}) \leq c \cdot d(\spt T^+, \spt T^-).
		\end{equation}
		The validity of \eqref{eq:plateau.rn.1} follows by elementary arguments since, by the maximum principle, the distance between $\spt T^\pm$ is attained at either boundary, and all $T^\pm$ are uniformly graphical as discussed for (i)-(iv).
	\end{proof}
	
	We wish to invoke Corollary \ref{coro:small.sing.dim} on:
	\begin{itemize}
		\item the manifold $M$ with the Euclidean metric,
		\item the collection of minimizing boundaries
		\[ \mathscr{F} := \bigcup_{t \in [0,1]} \{ T \mres M : T \text{ minimizing with } \partial T = \llbracket \Gamma_t \rrbracket \}, \]
		\item the function
		\[ \frT(\mbfx) := t, \; \mbfx \in M \cap \spt T_t, \; t \in [0, 1]. \]
	\end{itemize}
	Note that the conditions of Corollary \ref{coro:small.sing.dim} (i.e., of Theorem \ref{theo:small.sing.dim} therein) are satisfied: (1) follows from (iv), (2) from (v), (3) is trivial, and (4) follows from (vi). Upon applying Corollary \ref{coro:small.sing.dim}, the result follows by choosing $t$ close enough to $0$.
\end{proof}

\section{Proof of Theorem \ref{theo:main.homology.11} (Homology minimization problem)} \label{sec:homology.min}

In this section we provide the proof of Theorem \ref{theo:main.homology.11}. 

\begin{Prop} \label{prop:homology.min}
	Let $(N^{n+1}, g)$ and $[\alpha] \in H_n(N, \ZZ) \setminus \{ [0] \}$ be as in Theorem \ref{theo:main.homology.11} and $T$ be a $g$-minimizing integral $n$-current with in $[\alpha]$. Assume that $T$ additionally satisfies:
	\begin{enumerate}
		\item[(1)] $T$ is uniquely minimizing, and
		\item[(2)] $T$ has multiplicity one on $\reg T$.
	\end{enumerate}
	There exist $C^\infty$-small perturbations $g'$ of $g$ for which all $g'$-minimizing integral $n$-currents $T'$ in $[\alpha]$ have the desired regularity of Theorem \ref{theo:main.homology.11}, i.e., $T' = \llbracket \Sigma' \rrbracket$ for a smooth, precompact, oriented hypersurface $\Sigma' \subset N$ with 
	\[ \sing \Sigma' = \emptyset \text{ if } n+1 \leq 11, \text{ else } \dim \sing \Sigma' \leq n-10-\epsilon_n, \]
	where $\epsilon_n > 0$ is a dimensional constant.
\end{Prop}

\begin{proof}[Proof of Theorem \ref{theo:main.homology.11} assuming Proposition \ref{prop:homology.min}]
This reduction was essentially carried out in \cite{CMS:generic.9.10} but not in \cite{CMS:generic.codim.plus.2}, and we recall it here.

\begin{Cla} \label{clai:homology.min.unique}
	Let $T$ be $g$-minimizing in $[\alpha]$. There exist $C^\infty$-small perturbations $g'$ of $g$ such that $T$ is uniquely $g'$-minimizing in $[\alpha]$.
\end{Cla}
\begin{proof}[Proof of claim]
	One may take $g' := (1+\eps u^2) g$, where $\eps > 0$ is small and $u : N \to [0, \infty)$ is smooth and $u = 0$ precisely on $\supp T$. See \cite[Lemma 5.2]{CMS:generic.9.10}.
\end{proof}

Use Claim \ref{clai:homology.min.unique} to replace $g, T$ with $g', T' = T$ so that (1) of Proposition \ref{prop:homology.min} holds. Write
\begin{equation} \label{eq:homology.min.tprime}
	T' =: \sum_{i=1}^Q k'_i \llbracket \Sigma'_i \rrbracket,
\end{equation}
where the $\Sigma'_i$ are disjoint, connected, smooth, precompact, oriented hypersurfaces (with the usual a priori regularity $\dim \sing \Sigma_i' \leq n-7$). Note that
\[ T_1' := \sum_{i=1}^Q \llbracket \Sigma'_i \rrbracket \]
is $g'$-minimizing in some $[\alpha_1]$ by \cite[Theorem 33.4]{Simon:GMT}, and it clearly satisfies (2) of Proposition \ref{prop:homology.min} by construction. It also satisfies (1) of Proposition \ref{prop:homology.min}, otherwise $T'$ would not. 

We may invoke Proposition \ref{prop:homology.min} with $g', [\alpha_1] \ni T_1'$ in place of $g, [\alpha] \ni T$. We obtain $g''$ with the improved regularity property for all $g''$-minimizers in $[\alpha_1]$. Choose some $g''$-minimizing $T_1''$ in $[\alpha_1]$. Since $T_1'$ satisfies (1) and (2), if $g''$ is close enough to $g'$, then we may write
\[ T_1'' =: \sum_{i=1}^Q \llbracket \Sigma''_i \rrbracket \]
where, for each $i$,
\begin{equation} \label{eq:homology.min.homologous}
	\llbracket \Sigma''_i \rrbracket \text{ is homologous to } \llbracket \Sigma'_i \rrbracket,
\end{equation}
in addition to having the desired improved regularity
\begin{equation} \label{eq:homology.min.improved}
	\sing \Sigma''_i = \emptyset \text{ if } n+1 \leq 11, \text{ else } \dim \sing \Sigma''_i \leq n-10-\epsilon_n.
\end{equation}
Fix such $g'', T_1''$. Define
\[ T'' := \sum_{i=1}^Q k_i' \llbracket \Sigma''_i \rrbracket, \]
with $k'_i$ from \eqref{eq:homology.min.tprime}. It follows that $T'' \in [\alpha]$ in view of \eqref{eq:homology.min.homologous}. Moreover, $T''$ is also $g''$-minimizing in $[\alpha]$, in view of \cite[Corollary 3]{White:boundary-regularity-multiplicity} since $T_1''$ is $g''$-minimizing in $[\alpha_1]$. Note that $T''$ has the desired improved regularity in view of \eqref{eq:homology.min.improved}. To conclude, we may guarantee that $T''$ is uniquely $g''$-minimizing in $[\alpha]$ by applying Claim \ref{clai:homology.min.unique} once more.
\end{proof}

Our proof of Proposition \ref{prop:homology.min} mirrors our proof of Proposition \ref{prop:plateau.rn}:

\begin{proof}[Proof of Proposition \ref{prop:homology.min}]
	Retracing the proof of \cite[Lemma 5.7]{CMS:generic.9.10}, we may arrange for an open subset $\breve N \subset N$ with the following properties:
	\begin{enumerate}
		\item[(i)] $\spt T \subset \breve N$;
		\item[(ii)] $T$ is a $g$-minimizing boundary in $\breve N$.
	\end{enumerate}
	The result now follows as in Proposition \ref{prop:plateau.rn} provided we prove:
	\begin{Cla} \label{clai:homology.min}
		There exists $w \in C^\infty(N)$ so that the smooth deformation
		\[ g_t := e^{2tw} g, \; t \in [0,1], \]
		of $g_0 := g$ exhibits the following properties:
		\begin{enumerate}
			\item[(iii)] for all $t \in [0,1]$, all $g_t$-minimizing $T$ in $[\alpha]$ satisfy (i)-(ii) with $g_t$ in place of $g$;
			\item[(iv)] for all $t^\pm \in [0,1]$, all distinct $g_{t^\pm}$-minimizing $T^\pm$ in $[\alpha]$ have disjoint supports;
			\item[(v)] there exists $c' > 0$ such that, for all $t^\pm \in [0,1]$, $T^\pm$ as in (iv),
			\[ d_g(\mbfx^+, \mbfx^-) \geq c' |t^+ - t^-|, \; \text{ for all } \mbfx^\pm \in \spt T^\pm. \]
		\end{enumerate}
	\end{Cla}
	
	Obtaining $\breve N$ satisfying (i)-(ii) follows from \cite[Lemma 5.7]{CMS:generic.9.10} and assumption (2). Moreover, (iii) is a standard consequence of assumption (1).
	
	It remains to establish (iv)-(v). To that end, we will take
		\[ g_t := e^{2tw} g, \; t \in [0,1], \]
		for suitable $w \in C^\infty(N)$. We also fix $p \in \reg T$, and let $r, h > 0$ be such that
		\[ \Phi : D_{r} \times (-h, h) \stackrel{\approx}{\to} C_{r,h} \subset N \]
		are well-defined Fermi coordinates off of a radius-$r$ normal ball $D_{r} \subset \reg T$ centered at $p$, to a height-$h$ cylinder over it. We may take $w \in C^\infty(N)$ to satisfy, for some $c > 1$:
		\begin{enumerate}
			\item[(I)] $\spt w \cap \breve N \subset C_{r,h} \cap \breve N$;
			\item[(II)] $C_{r/c,h} \cap \breve N \subset \{ w \neq 0 \}$;
			\item[(III)] $\partial_z w \leq -c^{-1}(|\nabla_g w|_g + |w|)$ on $\breve N$, with $z$ the vertical Fermi coordinate.
		\end{enumerate}
		These conditions can be arranged by taking $w$ to be $(-1)$ times a suitable smoothing of the indicator of $B^g_{\varsigma/2}(\Phi(p, \varsigma/3))$, where $0 < \varsigma < \min\{r,h\}$, and a posteriori shrinking $\breve N$ ($\supset \spt T$) and enlarging $c$ as appropriate. Note that to arrange for (III) it is convenient to shrink $\breve N$ so far such that $B^g_{\varsigma/4}(\Phi(p, \varsigma/3)) \cap \breve N = \emptyset$. In view of assumption (1), we may take a small enough multiple of $w$ (note this does not affect (I)-(III)) to also guarantee:
		\begin{enumerate}
			\item[(IV)] for all $t \in [0,1]$, every $g_t$-minimizing $T$ in $[\alpha]$ satisfies
				\[ \spt T \subset \breve N; \]
			\item[(V)] for all $t$, $T$ as above, there exists a smooth $u : D_r \to (-h, h)$ with
				\[ \spt T \cap C_{r,h} = \graph_{D_r} u, \]
				\[ \|u\|_{C^2(D_r)} + \|w\|_{C^1(C_{r, h})} < \delta, \]
				and
				\begin{equation} \label{eq:homology.min.graph.grad}
					g(\nu, \nabla_g w) \leq - c^{-1} |\partial_z w| \text{ on } C_{r,h},
				\end{equation}
				where $c$ is from (III), $\delta \in (0, h)$ is small, and $\nu$ is the upward pointing normal to $\spt T$. Graphs are always with respect to the fixed background metric $g$.  
		\end{enumerate}
	
	\begin{proof}[Proof of Claim \ref{clai:homology.min} (iv)]
		Consider, first, the following weaker version of (iv):
		\begin{enumerate}
			\item[(iv')] If $t \in [0,1]$, and $T$, $T'$ are distinct $g_t$-minimizers in $[\alpha]$, then their supports are disjoint.
		\end{enumerate}
		We know that this holds from the standard cut-and-paste trick (see \cite[Lemma 2.8]{CMS:generic.9.10}) on a manifold with a fixed metric. For distinct metrics, we assert:
		\begin{enumerate}
			\item[(iv'')] If $t^\pm \in [0,1]$, $t^+ > t^-$, $T^\pm$ are $g_{t^\pm}$-minimizers in $[\alpha]$, and $u^\pm$ are as in (V) for $T^\pm$, then the supports of $T^\pm$ are disjoint, with $u^+ > u^-$ on $D_r$.
		\end{enumerate}
		Let us reorient the boundaries $T^\pm$ so that, in $C_{r,h}$, the upward pointing unit normals $\nu^\pm$ from (V) are both inward pointing (with respect to the regions they bound). Write $T^\pm =: \partial \llbracket E^\pm \rrbracket$. We may then construct new minimizing boundaries,
		\[ \hat T^+ := \partial \llbracket E^+ \cap E^- \rrbracket, \]
		\[ \hat T^- := \partial \llbracket E^+ \cup E^- \rrbracket. \]
		A weaker version of (V) holds for $\hat T^\pm$. The Lipschitz functions
		\[ \hat u^+ := \max \{ u^+, u^- \}, \]
		\[ \hat u^- := \min \{ u^+, u^- \}. \]
		satisfy
		\begin{equation} \label{eq:homology.min.decomp.1}
			\spt \hat T^\pm \cap C_{r,h} = \graph_{D_r} \hat u^\pm. 
		\end{equation}
		  By construction and our assumption, we have
		\begin{equation}
			u^+ \leq \hat u^+ \text{ on } D_r, 
		\end{equation}
		By construction, $T^\pm$, $\hat T^\pm$ are all homologous and, with respect to any metric, we have (see \cite[Decomposition Theorem]{White:boundary-regularity-multiplicity}) the following equality of Radon measures:
		\begin{equation} \label{eq:homology.min.decomp.2}
			\Vert T^+ \Vert + \Vert T^- \Vert = \Vert \hat T^+ \Vert + \Vert \hat T^- \Vert.
		\end{equation}
        		Note that \eqref{eq:homology.min.decomp.2} and the $g_{t^-}$-minimizing nature of $T^-$ imply
		\begin{equation} \label{eq:homology.min.area.1}
			\MM_{g_{t^-}}(T^+) - \MM_{g_{t^-}}(\hat T^+) = \MM_{g_{t^-}}(\hat T^-) - \MM_{g_{t^-}}(T^-) \geq 0.
		\end{equation}
		Likewise, $T^+$ being $g_{t^+}$-minimizing implies
		\begin{equation} \label{eq:homology.min.area.2}
			\MM_{g_{t^+}}(T^+) - \MM_{g_{t^+}}(\hat T^+) \leq 0.
		\end{equation}
		On the other hand, we also have
		\begin{align} \label{eq:homology.min.area.3}
			\MM_{g_{t^+}}(T^+) - \MM_{g_{t^+}}(\hat T^+)
				& = \MM_{g_{t^+}}(T^+) - \MM_{g_{t^-}}(T^+) \nonumber \\
				& \qquad + \MM_{g_{t^-}}(T^+) - \MM_{g_{t^-}}(\hat T^+) \nonumber \\
				& \qquad + \MM_{g_{t^-}}(\hat T^+) - \MM_{g_{t^+}}(\hat T^+) \nonumber \\
				& \geq \MM_{g_{t^+}}(T^+) - \MM_{g_{t^-}}(T^+) \nonumber \\
				& \qquad + \MM_{g_{t^-}}(\hat T^+) - \MM_{g_{t^+}}(\hat T^+),
		\end{align}
		where, for the inequality, we invoked \eqref{eq:homology.min.area.1}. Thus, \eqref{eq:homology.min.area.3} implies: 
		\begin{align} \label{eq:homology.min.area.4}
			\MM_{g_{t^+}}(T^+) - \MM_{g_{t^+}}(\hat T^+)
				& \geq (\MM_{g_{t^+}}(T^+) - \MM_{g_{t^-}}(T^+)) \nonumber \\
				& \qquad - (\MM_{g_{t^+}}(\hat T^+) - \MM_{g_{t^-}}(\hat T^+)).
		\end{align}
		We evaluate the right hand side of \eqref{eq:homology.min.area.4} within $C_{r,h}$ using (I) and (IV):
		\begin{align} \label{eq:homology.min.area.5}
			\MM_{g_{t^+}}(T^+) - \MM_{g_{t^+}}(\hat T^+)
				& \geq (\MM_{g_{t^+}}(\graph_{D_r} u^+) - \MM_{g_{t^-}}(\graph_{D_r} u^+)) \nonumber \\
				& \qquad - (\MM_{g_{t^+}}(\graph_{D_r} \hat u^+) - \MM_{g_{t^-}}(\graph_{D_r} \hat u^+)).
		\end{align}
		We will study the right hand side of \eqref{eq:homology.min.area.5} via the following interpolating families: for $\sigma, \tau \in [0,1]$
        \begin{align*}
           u^{(\sigma)} & := \hat u^+ + \sigma (u^+ - \hat u^+),  \\
           w^{(\tau)} & := (t^- + \tau(t^+-t^-)) w, \\
           g^{(\tau)} & := g_{t^- + \tau (t^+ - t^-)} = e^{2w^{(\tau)}} g.
        \end{align*}
		  Define $A : [0,1] \times [0,1] \to \RR$ via
		\[ A(\sigma, \tau) = \MM_{g^{(\tau)}}(\graph_{D_r} u^{(\sigma)}) = \int_{\graph_{D_r} u^{(\sigma)}} e^{nw^{(\tau)}}. \]
		All integrals are with respect to the measure induced by the fixed background metric $g$. We see that \eqref{eq:homology.min.area.5} implies:
		\begin{align} \label{eq:homology.min.area.6}
			\MM_{g_{t^+}}(T^+) - \MM_{g_{t^+}}(\hat T^+)
				& \geq (A(1,1) - A(1,0)) - (A(0,1) - A(0,0)) \nonumber \\
				& = \int_0^1 A_\tau(1,\tau) \, d\tau - \int_0^1 A_\tau(0,\tau) \, d\tau \nonumber \\
				& = \int_0^1 \int_0^1 A_{\tau \sigma}(\sigma, \tau) \, d\sigma \, d\tau.
		\end{align}
        To take the differentials, we rewrite, using the area formula
        \begin{align*}
            A(\sigma, \tau) = \int_{D_r} F[u^{(\sigma)}]\cdot (e^{nw^{(\tau)}}\circ\Phi_\sigma)
        \end{align*}
        where:
        \begin{itemize}
        	\item $\Phi_\sigma(x):= \Phi(x, u^{(\sigma)}(x))$ is the graphical parametrization of $\graph_{D_r} u^{(\sigma)}$, and 
			\item $F=F(x, z, p)$ satisfies the estimate on the derivatives on $D_r$:
	        \begin{align}
            F(x, 0 ,0) = 1, \quad 
            F_z(x, 0, 0) = F_p(x, 0, 0) = 0; \quad
            \|F\|_{C^2}\leq C(T, N); \label{eq:AreaIntegrandEst}
	        \end{align}
                see \cite[(A.12)]{ChodoshMantoulidis:ac.3d} for an exact formula for $F$. And $F[u]:= F(x, u(x), \nabla u(x))$.
	    \end{itemize}
        Differentiating $A$ in $\tau$ and rearranging:
		\[ \frac{A_\tau(\sigma, \tau)}{n(t^+-t^-)} = \int_{D_r} F[u^{(\sigma)}]\cdot \big( (e^{n w^{(\tau)}}w) \circ\Phi_\sigma \big). \]
		Then differentiating in $\sigma$ too and letting $v:= u^+ - \hat u^+(\leq 0)$:
		\begin{align*}
			\frac{A_{\tau \sigma}(\sigma, \tau)}{n(t^+-t^-)} 
			& = \int_{D_r} \left(F_z(x, u^{(\sigma)}, \nabla u^{(\sigma)})v + F_p(x, u^{(\sigma)}, \nabla u^{(\sigma)})\cdot \nabla v\right)\cdot \big( (e^{n w^{(\tau)}}w)\circ\Phi_\sigma \big) \\
            & \qquad + \int_{D_r} F[u^{(\sigma)}]\cdot \big(\partial_z(e^{n w^{(\tau)}}w)\circ\Phi_\sigma \big) \cdot v ,
		\end{align*}

        Let $\Omega:= \{x\in D_r: u^+<\hat u^+\}$. We now assert that, for $\delta$ sufficiently small depending on $N, g, T, D_r, c$ in (V), we have $\Omega\cap \Phi_\sigma^{-1}(\spt w) = \emptyset$ for every $\sigma\in [0, 1]$.
		Suppose for contradiction that $\bar\sigma\in [0, 1]$ such that $\Omega \cap \Phi_{\bar\sigma}^{-1}(\spt w) \neq \emptyset$.       
        To handle the term with $\nabla v$ in the equation of $A_{\tau\sigma}$ above, first note that $v = u^+ - \hat u^+$ is smooth on $\Omega$ and vanishes on $D_r\setminus \Omega$. Furthermore $\nabla v = 0$ a.e.~on $D_r\setminus \Omega$. 
        Let $\Omega_j\Subset \Omega$ be a increasing smooth exhaustion of $\Omega$ such that $\|w\circ\Phi_\sigma\|_{C^1(\partial \Omega_j)}\cdot \|v\|_{C^0(\partial \Omega_j)}$ tends to $0$ as $j\to \infty$.
        Then
        \begin{align*}
            & \int_{D_r} \Big(F_p(x, u^{(\sigma)}, \nabla u^{(\sigma)})\cdot \nabla v\Big)\cdot \big( (e^{n w^{(\tau)}}w)\circ\Phi_\sigma \big) \\
            & \qquad = \lim_{j\to \infty} \int_{\Omega_j} (F_p(x, u^{(\sigma)}, \nabla u^{(\sigma)})\cdot \nabla v)\cdot \big( (e^{n w^{(\tau)}}w)\circ\Phi_\sigma \big) \\
            & \qquad = \lim_{j\to \infty} \int_{\Omega_j} -\Div \Big( F_p(x, u^{(\sigma)}, \nabla u^{(\sigma)})\cdot \big( (e^{n w^{(\tau)}}w)\circ\Phi_\sigma \big) \Big) v \\
            & \qquad \qquad + \lim_{j\to \infty} \int_{\partial\Omega_j} \left(F_p(x, u^{(\sigma)}, \nabla u^{(\sigma)})\cdot \zeta_j\right)\cdot \big( (e^{n w^{(\tau)}}w)\circ\Phi_\sigma \big) v \\
            & \qquad = - \int_{\Omega} \Div \Big( F_p(x, u^{(\sigma)}, \nabla u^{(\sigma)})\cdot \big( (e^{n w^{(\tau)}}w)\circ\Phi_\sigma \big) \Big) v
        \end{align*}
        where $\zeta_j$ denotes the outward unit normal field of $\partial\Omega_j$ in $\Omega$. Thus we find, 
        \begin{align}
            \frac{A_{\tau \sigma}(\sigma, \tau)}{n(t^+-t^-)} = \int_{\Omega} \Big( L(\sigma, \tau) + 
            F[u^{(\sigma)}]\cdot \big(\partial_z(e^{n w^{(\tau)}}w)\circ\Phi_\sigma \big) \Big)v, \label{eq:A_(tau sigma)expression}            
        \end{align}
        where 
        \begin{align*}
            L(\sigma, \tau) 
            	& := F_z(x, u^{(\sigma)}, \nabla u^{(\sigma)})\cdot \big( (e^{n w^{(\tau)}}w)\circ\Phi_\sigma\big) \\
				& \qquad - \Div \Big( F_p(x, u^{(\sigma)}, \nabla u^{(\sigma)})\cdot \big( (e^{n w^{(\tau)}}w)\circ\Phi_\sigma \big) \Big)
        \end{align*}
        Note that on $\Omega$, $u^{(\sigma)}=u^- + \sigma(u^+-u^-)$ is smooth, hence together with (\ref{eq:AreaIntegrandEst}), we have pointwise estimate on $\Omega$, 
        \begin{align}
            |L(\sigma, \tau)|\leq C(T, N, D_r)\left(\|u^+\|_{C^2(D_r)}+\|u^-\|_{C^2(D_r)}\right)(|w|+|\nabla_g w|_g)\circ \Phi_\sigma\,. \label{eq:ErrorEstL(sigma, tau)}
        \end{align}
        While by (III), $\partial_z w<0$, then by \eqref{eq:AreaIntegrandEst} and (V) with $\delta(N, g, T, D_r)>0$ small enough, we have
        \begin{align}
            F[u^{(\sigma)}]\cdot \left(\partial_z(e^{n w^{(\tau)}}w)\circ\Phi_\sigma \right) \leq \big(\tfrac12\partial_z w + C(n)\|w\|_{C^1(C_{r, h})}|w| \big)\circ \Phi_\sigma\,. \label{eq:ErrorEstMain}
        \end{align}
        Since we assumed $\Omega\cap \Phi_{\bar\sigma}^{-1}(\spt w) \neq \emptyset$ (for contradiction) and we have that $v <0$ in $\Omega$, we can combine \eqref{eq:A_(tau sigma)expression}-\eqref{eq:ErrorEstMain} with (III) and (V) to choose $\delta(N, g, T, D_r, c)>0$ small enough so that $A_{\tau\sigma}(\bar\sigma, \tau)>0$. Since $A_{\tau\sigma}(\sigma, \tau)$ is continuous in $\sigma$, this contradicts \eqref{eq:homology.min.area.2} and \eqref{eq:homology.min.area.6}.

  		Now we fix the choice of $\delta$ in (V) so that $\Omega \cap \Phi_{\sigma}^{-1}(\spt w)=\emptyset$ for all $\sigma\in [0, 1]$, i.e., so that we have the weak version of the inequality asserted in (iv''):
		\[ u^+ = \hat u^+\geq u^- \text{ on } D_r\cap (\cup_{\sigma\in [0, 1]}\Phi_{\sigma}^{-1}(\spt w)). \]
		Plugging this back into \eqref{eq:A_(tau sigma)expression}, \eqref{eq:homology.min.area.6} and recalling \eqref{eq:homology.min.area.2} yields
		\[ \MM_{g_{t^+}}(T^+) = \MM_{g_{t^+}}(\hat T^+). \]
		Therefore, $T^+$, $\hat T^+$ are both $g_{t^+}$-minimizers in $[\alpha]$.
		
		We claim that $T^+$, $\hat T^+$ coincide. Indeed, if they did not, then they would be disjoint by (iv'). Thus, $\hat T^+ = T^-$ would be disjoint from $T^+$ and $\hat u^+ = u^- > u^+$, this contradicts the weak inequality $u^+\geq u^-$ on $D_r\cap (\cup_{\sigma\in [0, 1]}\Phi_{\sigma}^{-1}(\spt w))$ that we have already concluded. Thus $\hat T^+ = T^+$. 
		
		We assert that strict inequality $u^+ > u^-$ holds on $D_r$ and that the supports of $T^+$, $T^-$ are disjoint. Indeed, carrying out a conformal metric change computation shows that
		\begin{equation} \label{eq:homology.min.mean.curv}
			H(T^-, g_{t^+}) = (n-1) (t^+-t^-) e^{-t^+ w} \langle \nabla_{g} w, \nu \rangle_g,
		\end{equation}
		where $\nu$ is the upward unit normal to $\spt T^-$ with respect to $g$. Using (II), (III), and (V)'s \eqref{eq:homology.min.graph.grad}, we see that $H(T^-, g_{t^+}) \leq 0$ and that this inequality is strict over $D_{r/c}$. Since $H(T^+, g_{t^+}) = 0$, the maximum principle implies that the supports of $T^+$, $T^-$ are disjoint, with $u^+ > u^-$. This completes the proof of (iv'') which yields \emph{(iv)}.
	\end{proof}
	
	\begin{proof}[Proof of Claim \ref{clai:homology.min} (v)]	
		We pick up from \eqref{eq:homology.min.mean.curv}. By a standard PDE sliding argument (see, e.g., the proof of \cite[Proposition 5.6 (f)]{CMS:generic.9.10}), and the fact that $H(T^-, g_{t^+}) < 0$ over $D_{r/c}$, we deduce that
		\begin{equation} \label{eq:homology.min.dist.1}
			u^+ - u^- \geq c'(t^+-t^-) \text{ over } D_{r},
		\end{equation}
		for a constant $c' > 0$ independent of $t^\pm$, $T^\pm$. Note that \eqref{eq:homology.min.dist.1} trivially implies
		\begin{equation} \label{eq:homology.min.dist.2}
			d_{g}(\spt \partial (T^+ \mres (N \setminus C_{r,h})), \spt \partial (T^- \mres (N \setminus C_{r,h}))) \geq c'(t^+-t^-),
		\end{equation}
		for a possibly smaller $c' > 0$. In view of condition (I), both these currents are minimizing, so arguing as in \eqref{eq:plateau.rn.1}, we deduce from \eqref{eq:homology.min.dist.2} that
		\begin{equation} \label{eq:homology.min.dist.3}
			d_{g}(\spt T^+, \spt T^-) \geq c'(t^+-t^-),
		\end{equation}
		for a different yet $c' > 0$. This implies the required result.
	\end{proof}
	
	The result now follows as in Proposition \ref{prop:plateau.rn}. We wish to invoke Corollary \ref{coro:small.sing.dim} on:
	\begin{itemize}
		\item the manifold $(M, g) := (\breve N \setminus C_{r,h}, g)$,
		\item the collection of minimizing boundaries
		\[ \mathscr{F} := \bigcup_{t \in [0,1]} \{ T \mres (\breve N \setminus C_{r,h}) : T \text{ $g_t$-minimizing in } [\alpha] \}, \]
		\item the function
		\[ \frT(\mbfx) := t, \; \mbfx \in (\breve N \setminus C_{r,h}) \cap \spt T_t, \; t \in [0, 1]. \]
	\end{itemize}
	Note that the conditions of Corollary \ref{coro:small.sing.dim} (i.e., of Theorem \ref{theo:small.sing.dim} therein) are satisfied: (1) follows from (iii), (2) from (iv') and (iv''), (3) is trivial, and (4) follows from (v). Upon applying Corollary \ref{coro:small.sing.dim}, the result follows by choosing $t$ close enough to $0$.
\end{proof}

\appendix

\section{Quadratic hypercones} \label{app:quadratic}

We recall that quadratic hypercones are precisely rotations of 
\[ \cC_{p,q} = \{ (x, y) \in \RR^{p+1} \times \RR^{q+1} : q |x|^2 = p |y|^2 \} \subset \RR^{n+1}, \]
where $p, q > 0$ are integers with $p+q = n-1$. 

We shall recall facts about quadratic hypercones that are pertinent to this paper, but we direct the interested reader to \cite{Simoes:quadratic, SimonSolomon:quadratic, EdelenSpolaor:quadratic} for a comprehensive treatment.
\begin{Prop}
  $\cC_{p,q}$ is stable if and only if $p+q\geq 6$ (or equivalently $n+1\geq 8$);

  $\cC_{p,q}$ is minimizing if and only if $p+q\geq 6$ and $(p, q)\neq (1, 5),\ (5, 1)$.
\end{Prop}
\begin{proof}
  The stability of $\cC_{p,q}$ follows from \cite{Simons:minvar}; The minimizing property of $\cC_{p, q}$ follows from \cite{BDG:Simons, Lawson72, Simoes:quadratic}.
\end{proof}

\begin{Prop} \label{prop:quadratic.strictly.stable}
	Each stable quadratic hypercone $\cC$ is strictly stable (Definition \ref{defi:strictly.stable}).
\end{Prop}
\begin{proof}
	This follows from the explicit spectral characterization in \cite[Theorem 4.5]{CaffarelliHardtSimon} and the fact that $|A|^2 = n-1 < \tfrac{(n-2)^2}{4}$ (for $n+1 \geq 8$) on the link of every such $\cC$.
\end{proof}

\begin{Prop} \label{prop:quadratic.strictly.minimizing}
	Each minimizing quadratic hypercone $\cC$ is strictly minimizing (Definition \ref{def:strictly-minimizing}).
\end{Prop}
\begin{proof}
See \cite{Lin:strict-min} or alternatively \cite{Lawlor:sufficient} (see the discussion in (3) on p.\ 100 there and then the calculation on p.\ 69--70).
\end{proof}

\begin{Prop} \label{prop:quadratic.integrable}
	Each quadratic hypercone $\cC$ is strongly integrable (Definition \ref{defi:strongly.integrable}).
\end{Prop}
\begin{proof}
	This follows from \cite[Proposition 2.7]{SimonSolomon:quadratic}.
\end{proof}

\begin{Cor} \label{coro:quadratic.isolated}
	Each quadratic hypercone is $\cC$ isolated in the space of minimal hypercones.
\end{Cor}

\section{A Frankel property of nonflat minimizing hypercones}

\begin{Lem} \label{lemm:one.sided.cone}
	Suppose that $\cC \subset \RR^{n+1}$ is a minimizing hypercone, and $\Sigma$ is a minimizing boundary in $\RR^{n+1}$ that does not cross $\cC$ smoothly that is conical about some $p \in \sing \Sigma$. Then,
	\[ \Sigma = \pm \cC \text{ and } p \in \spine \cC. \]
\end{Lem}
\begin{proof}
	This is contained in the proof of  \cite[Proposition 3.3]{CMS:generic.9.10}. We briefly recall the strategy. By blowing down $\Sigma$, one obtains a minimizing hypercone $\cC' \subset \RR^{n+1}$ that does not cross $\cC$ smoothly (or else $\Sigma$ would have since $\cC$ is dilation invariant). Therefore, by Simon's strong maximum principle for minimizers (\cite{Simon:smp}), $\cC' = \pm \cC$. 
	
	Since $\Sigma$ was conical about $p$ to begin with, we must have
	\[ \Sigma = \cC' + p = \cC + p. \]
	If $p = \orig$, we are done. Otherwise, consider the rescaled minimizers
	\[ \Sigma_t := t \Sigma = \cC + tp, \; t \in (0, 1). \]
	By sending $t \to 0$ and writing $\Sigma_t$ as graphs locally over compact subsets of $\cC$ and using the regular Harnack inequality, we obtain the Jacobi field $\langle \nu_\cC, p \rangle$ on $\cC$. Since $\cC$ is dilation invariant, none of the $\Sigma_t$ cross $\cC$ smoothly, and therefore
	\[ \langle \nu_\cC, p \rangle \geq 0 \text{ on } \cC. \]
	By the strong maximum principle on the Jacobi field $\langle \nu_\cC, p \rangle$, $\langle \nu_\cC, p \rangle \equiv 0$ or $\langle \nu_\cC, p \rangle > 0$. The latter is impossible: $\langle \nu_\cC, p \rangle$ would be a bounded Jacobi field on $\cC$, in contradiction to Simon's Harnack theory for positive Jacobi fields on nonflat minimizing hypercones \cite{Simon:decay} (cf. \cite[Theorem 1.4]{Wang:smoothing}). Thus, $\langle \nu_\cC, p \rangle \equiv 0$, so $\cC$ splits a line parallel to $p$, so $p \in \spine \cC$, and $\Sigma = \cC + p = \cC$. This completes the proof.
\end{proof}

\section{$\alpha(\cC; \sigma)$ as decay rate of positive Jacobi fields}

\begin{Lem} \label{lemm:jacobi.decay.cone}
	There exist $\sigma_1 = \sigma_1(n)$ such that for all $\cC \in \mathscr{C}_n$:
	\begin{enumerate}
		\item[(i)] $\reg_{>4 \sigma_1} \cC \cap \partial B_1(\orig) \neq \emptyset$, and
		\item[(ii)] for all $\sigma \in (0, \sigma_1]$, there exists $H = H(n, \sigma)$ such that
			\[ \inf_{\reg_{> 2 \sigma} \cC \cap \partial B_1(\orig)} u < H \rho^{\alpha(\cC; \sigma/2)} \inf_{\reg_{> \frac12 \rho \sigma} \cC \cap \partial B_\rho(\mbfy)} u \]
			 for all positive Jacobi fields $u$ on $\cC$, $\mbfy \in \spine \cC \cap \bar B_1(\orig)$, and $\rho \in (0,1)$.
	\end{enumerate}
\end{Lem}
\begin{proof}
    (i) follows from \cite[Proposition 3.10]{CMS:generic.9.10}.

    For (ii), note that the proofs of \cite[Lemmas 3.8, B.1]{CMS:generic.9.10} apply with $\Omega = \reg_{>\sigma}$ and with $\alpha(\cC; \sigma)$ in place of the cruder $\kappa_n^*$ to yield
    \[ \inf_{\reg_{> \sigma} \cC \cap \partial B_1(\mbfy)} u \leq \rho^{\alpha(\cC; \sigma)} \sup_{\reg_{> \sigma} \cC \cap \partial B_\rho(\mbfy)} u. \]
    The result now follows by replacing $\sigma$ with $\sigma/2$ and two invocations of Harnack (see, e.g., \cite[Proposition 3.10]{CMS:generic.9.10}) to adjust the left and right hand sides to the desired ones.
\end{proof}

\section{A splitting lemma for homogeneous polynomials}
    
    Recall that for a polynomial $p \neq 0$ on $\RR^k$ and $y_0 \in \RR^k$ we denote by
    \[ \Ord_p(y_0) \in \ZZ_{\geq 0} \]
    the degree of the lowest-order nonzero term in the expansion of $p$ about $y_0$. First we record a trivial consequence of Taylor expansions:
    
    \begin{Lem} \label{lemm:polynomial.ord.characterization}
    	Let $p \neq 0$ be a polynomial on $\RR^k$ and $y_0 \in \RR^k$. Then
	\[ \Ord_p(y_0) \geq d \iff  D^\alpha p(y_0)=0 \text{ for every } \alpha \in (\ZZ_{\geq 0})^k \text{ with } |\alpha| \leq d-1. \]
    \end{Lem}
    
    We also have the following splitting result for homogeneous polynomials:
    
    \begin{Lem}\label{Lem_homog-poly-splitting}
    Let $p\neq 0$ be a homogeneous polynomial on $\RR^k$.
    \begin{enumerate}
    	\item[(i)] For every $y_0 \in \RR^k$, we have
	\[ \Ord_p(y_0) \geq d \iff D^\alpha p(y_0)=0 \text{ for every } \alpha \in (\ZZ_{\geq 0})^k \text{ with } |\alpha| = d-1. \]
        \item[(ii)] Assume $\deg p = d$. Then, the set
        \[ V = \{ y_0 \in \RR^k : \Ord_p(y_0) \geq d \} \]
        is a linear subspace of $\RR^k$, and $p$ is invariant under translations in $V$.
      \end{enumerate}
    \end{Lem}
    \begin{proof}
    
    {(i)} In view of Lemma \ref{lemm:polynomial.ord.characterization}, we need only check ``$\impliedby$''. The proof for that follows by induction on $d$. The case $d=1$ is again handled by Lemma \ref{lemm:polynomial.ord.characterization}. Now suppose that $d \geq 2$ and that the assertion holds for $d-1$ in place of $d$. Fix $y_0 \in \RR^k$ with $D^\alpha p(y_0) = 0$ for all $|\alpha| = d-1$. Then for every $i \in \{ 1, \ldots, k \}$, $D_i p$ is a homogeneous polynomial with
    \[ D^\beta (D_i p)(y_0) = 0 \text{ for all } |\beta| = d-2. \]
    By the inductive assumption, then,
    \[ \Ord_p D_i p(y_0) \geq d-1 \text{ for all } i \in \{ 1, \ldots, k \}. \]
    By Lemma \ref{lemm:polynomial.ord.characterization} again, for all $i \in \{ 1, \ldots, k \}$,
    \[ D^\beta (D_i p)(y_0) = 0 \text{ for all } |\beta| \leq d-2, \]
    i.e.,
    \[ D^\alpha p(y_0) = 0 \text{ for all } 1 \leq |\alpha| \leq d-1, \]
    For $\alpha=0$ it suffices to note that by homogeneity and the case of $|\alpha|=1$ above,
    \[ (\deg p)\cdot p(y_0) = \sum_{i=1}^k (y_0)_i D_i p(y_0) = 0, \]
    which completes the inductive step and thus the proof.
    
    {(ii)} Clearly $V$ is a linear subspace since $D^\alpha p$ is linear if $|\alpha|=d-1 = \deg p - 1$. It remains to show the translation invariance along $V$. To that end, fix $y_0 \in V$. Then 
    \[ \Ord_p(y_0) \geq d = \deg p \implies \Ord_p(y_0) = \deg p = d \]
    in turn implies that $p(\cdot+y_0)$ is homogenous. Now for $t \neq 0$ we can use homogeneity at $y_0$ and then homogeneity at the origin to conclude that,  
    \[
    p(y_0+y) = t^{-d} p(y_0+ty) =  p(t^{-1} y_0 + y_0)
    \]
    for all $y \in \RR^k$. Sending $t \to \infty$ yields $p(y_0 + y) = p(y)$ for all $y \in \RR^k$.
    \end{proof}

\section{Normal coordinates} \label{app:normal.coordinates}

We consider $(M^d,g)$ satisfying
\begin{equation}\label{eq:curv-assumption-app}
|\Rm| + |\nabla \Rm| + |\nabla^2\Rm| \leq C^{-1} \, ,
\end{equation}
We will choose $C = C(d)$ later, in Lemma \ref{Lem:norm-coord-derivatives}. We will work with the subset:
\[ \breve M := \{ p \in M : \injrad(M, g, p) \geq 2 \}. \]

\subsection{Basics} 

Let $p \in \breve M$. For $r \in (0, 2]$, write $B_r \subset T_pM$ for the radius $r$ ball centered at $\orig \in T_p M$ (with respect to the inner product $g_p$), and $B_r(p) \subset M$ for the radius $r$ ball centered at $p \in M$. For $\tau > 0$, let
\begin{equation}\label{eq:defn-eta-tau}
\eta_{p,\tau} : (B_{2\tau^{-1}} \subset T_p M) \to (B_2(p) \subset M), \qquad \mbfv \mapsto \exp_p(\tau \mbfv).
\end{equation}
The inverse $\eta_{p,\tau}^{-1} : B_2 (p) \to B_{2\tau^{-1}}$ is also well-defined and smooth. We define 
\begin{equation}\label{eq:defn-g-tau}
g_{p,\tau} : = \tau^{-2} \eta_{p,\tau}^*g \, .
\end{equation} 
This is a Riemannian metric on $B_{2\tau^{-1}} \subset T_p M$. If $g_p$ denotes the flat metric on $T_p M$ obtained by the inner product induced by $g$ at $p$, then:

\begin{Lem}\label{Lem:norm-coord-derivatives}
There is $C=C(d)$ for \eqref{eq:curv-assumption-app} with the following property.

For every $p\in \breve M$, $k\in \{0,1,2\}$, and $R, \tau >0$ with $R \tau \leq 2$, we have
\[
\rVert \nabla^k(g_{p,\tau}-g_p)\rVert_{C^0(B_R)} \leq C R^{2-k} \tau^2.
\]
The covariant derivative is taken with respect to $g_p$.
\end{Lem}
\begin{proof}
Fix $p,\tau,R$ and a $g_p$-orthonormal basis $\mbfe_1,\dots,\mbfe_d$ for $T_pM$. Our goal is to prove the asserted estimate at $\mbfv = (v_1,\dots,v_d) \in B_R$. To this end, define
\[
\Gamma(t,s_1,\dots,s_d) := \exp_p(t\tau(\mbfv + s_1\mbfe_1 + \dots + s_d\mbfe_d)). 
\]
Let $\mbfT=\partial_t \Gamma$ and $\mbfS_i = \partial_{s_i}\Gamma$. We define $\mbfE_i(t)$ by parallel transporting $\mbfe_i$ along $\gamma(t) : =\Gamma(t,0,\dots,0)$. 

Noting that $\Gamma(1,s_1,\dots,s_d) = \eta_{p,\tau}(s_1+v_1,\dots,s_d+v_d)$, we have\footnote{We write $D_{s_i}$ for the derivative along the curve $s \mapsto \Gamma(t,s_1,\dots,s_{i-1},s,s_{i+1},\dots,s_d)$. (Alternatively, this is the pullback connection $\Gamma^*\nabla$ in the direction $\partial_{s_i}$.)}
\begin{align*}
\mbfS_i(1,0,\dots,0) & = \partial_{i} \eta_{p,\tau}|_{\mbfv}\\
D_{s_i} \mbfS_j(1,0,\dots,0)&  = \nabla_{\partial_i \eta_{p,\tau}} \partial_j \eta_{p,\tau}|_\mbfv\\
D_{s_i}D_{s_j} \mbfS_k(1,0,\dots,0) & = \nabla_{\partial_i \eta_{p,\tau}} \nabla_{\partial_j \eta_{p,\tau}} \partial_k \eta_{p,\tau}|_\mbfv.
\end{align*}
We observe the following boundary conditions for $\mbfS_i$ and its derivatives:
\begin{align*}
\mbfS_i(0,0,\dots,0) & = 0 \\
D_t\mbfS_i(0,0,\dots,0) & = D_{s_i} \mbfT(0,0,\dots,0) = \tau \mbfe_i \\
D_{s_i} \mbfS_j(0,0,\dots,0) & = 0\\
D_t D_{s_i} \mbfS_j(0,0,\dots,0) & = (\Rm(\mbfT, \mbfS_i)\mbfS_j + D_{s_i} D_{s_j} \mbfT)(0,0,\dots,0) = 0\\
D_{s_i} D_{s_j} \mbfS_k(0,0,\dots,0) & = 0 \\
D_t D_{s_i} D_{s_j} \mbfS_k(0,0,\dots,0) & 
= (\Rm(\mbfT, \mbfS_i) D_{s_j}\mbfS_k  + D_{s_i}(\Rm(\mbfT,\mbfS_j)\mbfS_k)\\
&\qquad + D_{s_i}D_{s_j}D_{s_k}\mbfT)(0,0,\dots,0) = 0. 
\end{align*}
Note also that $\mbfS_i$ satisfies the Jacobi equation
\[
D^2_t \mbfS_i + \Rm(\mbfS_i,\mbfT)\mbfT = 0.
\]
Combined with $|\mbfT(t,0,\dots,0)|\leq R \tau$, the boundary conditions for $\mbfS_i$, and curvature estimates \eqref{eq:curv-assumption-app} we obtain 
\begin{equation}\label{eq:norm-coord-app-est-0th-order-S_i}
|\mbfS_i(t,0,\dots,0) - t \tau \mbfE_i(t)| + |D_t\mbfS_i(t,0,\dots,0) - \tau \mbfE_i(t)| \leq CR^2\tau^3. 
\end{equation}
Indeed, for $C>0$ to be chosen these estimates obviously hold for $t>0$ very small. If the first time they fail is $t_0 \in (0,1)$ then on $(0,t_0]$ we find that
\begin{align*}
 | D_t^2(\mbfS_i(t) - t\tau \mbfE_i(t))| & \leq \lvert\Rm(\mbfS_i(t),\mbfT)\mbfT\rvert \\
& \leq \lvert\Rm(t\tau \mbfE_i(t),\mbfT)\mbfT\rvert + \lvert\Rm(\mbfS_i(t)-t\tau \mbfE_i(t),\mbfT)\mbfT\rvert \\
& \leq C^{-1} t R^2 \tau^3 + R^4 \tau^5 \\
& \leq (C^{-1} + 4)R^2 \tau^3
\end{align*}
Integrating this twice and taking $C>0$ sufficiently large proves that the asserted estimates hold past $t_0$, a contradiction. 

Now, we consider the higher derivatives. Differentiating the Jacobi field equation and then commuting we obtain
\begin{align*}
0 & = D_{s_i} D^2_t\mbfS_j + D_{s_i}(\Rm(\mbfS_j,\mbfT)\mbfT)\\
& = D_t^2D_{s_i} \mbfS_j  + D_t(\Rm(\mbfS_i,\mbfT)\mbfS_j) + \Rm(\mbfS_i,\mbfT)D_t\mbfS_j+ D_{s_i}(\Rm(\mbfT,\mbfS_j)\mbfT). 
\end{align*}
As before, we obtain that 
\begin{equation}\label{eq:norm-coord-app-est-1st-order-S_i}
|D_{s_i}\mbfS_j(t,0,\dots,0)|+|D_tD_{s_i}\mbfS_j(t,0,\dots,0)|\leq CR \tau^3
\end{equation}
possibly after increasing $C$. A similar argument yields
\begin{equation}\label{eq:norm-coord-app-est-2nd-order-S_i}
|D_{s_i}D_{s_j}\mbfS_k(t,0,\dots,0)|+|D_tD_{s_i}D_{s_j}\mbfS_k(t,0,\dots,0)|\leq C \tau^3
\end{equation}

To complete the proof, we now observe that
\[
g_{\tau,p}(\mbfe_i,\mbfe_j) = \tau^{-2}g(\partial_i\eta_{p,\tau},\partial_j\eta_{p,\tau})
\]
so
\[
g_{\tau,p}(\mbfe_i,\mbfe_j)|_{\mbfv} = \tau^{-2}g(\mbfS_i,\mbfS_j)|_{(1,0,\dots,0)}
\]
Using \eqref{eq:norm-coord-app-est-0th-order-S_i} we thus obtain (after taking $C$ larger if necessary):
\[
|g_{\tau,p}(\mbfe_i,\mbfe_j)|_{\mbfv} - \delta_{ij}| \leq C R^2 \tau^2
\]
Similarly, we have that
\[
\partial_i (g_{\tau,p}(\mbfe_j,\mbfe_k)) |_{\mbfv} = \tau^{-2}(g(D_{s_i}\mbfS_j,\mbfS_k) + g(\mbfS_j,D_{s_i}\mbfS_k))|_{(1,0,\dots,0)}
\]
so \eqref{eq:norm-coord-app-est-0th-order-S_i} and \eqref{eq:norm-coord-app-est-1st-order-S_i} give
\[
|\partial_i (g_{\tau,p}(\mbfe_j,\mbfe_k)) |_{\mbfv}| \leq C R \tau^2. 
\]
Using also \eqref{eq:norm-coord-app-est-2nd-order-S_i} the second derivatives follow similarly. This proves the assertion. 
\end{proof}

\subsection{Change of basepoint} Let $p \in \breve M$, $q \in B_2(p)$.

We define the \textbf{change of basepoint} map, from $p$ to $q$ at some scale $\tau > 0$, to be
\begin{equation}\label{eq:cob-map}
\eta_{q,\tau}^{-1} \circ \eta_{p,\tau}  : (B_{(2-d(p,q))\tau^{-1}} \subset T_p M) \to (B_{2\tau^{-1}} \subset T_q M).
\end{equation}

In a purely Euclidean background we would have had
\[ (\eta_{q,\tau}^{-1} \circ \eta_{p,\tau})(\mbfv) = \mbfv - \tau^{-1} (q-p). \]
In our Riemannian setting, there is an analogous estimate that makes use of the linear isometry that parallel transports along the minimizing geodesic from $p$ to $q$:
\[ \iota_{p\to q} : (T_p M, g_p) \to (T_q M, g_q) \]

\begin{Lem}\label{Lem:change-bspt}
Choose $C=C(d)$ for \eqref{eq:curv-assumption-app} such that Lemma \ref{Lem:norm-coord-derivatives} holds.

Let $p \in \breve M$, $R,\tau > 0$, $B_\tau(p) \subset \breve M$, $\mbfq \in B_1 \subset T_p M$, $q = \eta_{p,\tau}(\mbfq)$. If $(1+R)\tau < 2$, then 
\[ \|(\eta_{q,\tau}^{-1} \circ \eta_{p,\tau})(\cdot) - \iota_{p \to q}(\cdot - \mbfq) \|_{C^0(B_R)} \leq C'\tau^2, \]
where $C' = C'(C, R)$.\end{Lem}
\begin{proof}
Fix any $\mbfv \in B_{R} \subset T_p M$ and denote $v = \eta_{p,\tau}(\mbfv)$. Also let 
\[ \gamma(s) : = \eta_{p,\tau}(s\mbfq) = \exp_p(s\tau \mbfq) \]
be the minimal geodesic between $p$ and $q$ parametrized by constant speed for $s\in [0,1]$. 

For $s \in [0,1]$, we have
\begin{equation} \label{eq:change-bspt-dist}
	d(\gamma(s), v) \leq d(p, \gamma(s)) + d(p, v) \leq d(p, q) + d(p, v) < (1+R) \tau.
\end{equation}
Recalling that, by assumption, $(1+R)\tau < 2$ and $\gamma([0,1]) \subset B_\tau(p) \subset \breve M$, there exists a smooth vector field $\mbfV(s)$ along $\gamma(s)$ so that
\[ \eta_{\gamma(s),\tau}(\mbfV(s)) = v. \] 
Note that $(\eta_{q,\tau}^{-1} \circ \eta_{p,\tau})(\mbfv) = \mbfV(1)$

On the other hand, let $\bar \mbfV(s)$ denote the parallel transport of $\mbfv$ along $\gamma(s)$, i.e.,
\[ \bar \mbfV(s) = \iota_{p \to \gamma(s)}(\mbfv). \]
Note that $\iota_{p \to q}(\mbfv - \mbfq) = \iota_{p \to q}(\mbfv) + \mbfp = \bar \mbfV(1) + \mbfp$, where $p = \eta_{q,\tau}(\mbfp)$.

Since $\mbfp = -\tau^{-1}\gamma'(1)$, if we define
\[
\hat \mbfV(s) : = \mbfV(s) - \bar\mbfV(s) + s \tau^{-1}\gamma'(s),
\]
then we will be done if we show that $|\hat \mbfV(1)| \leq C \tau^2$. Note that
\[
D_s \hat \mbfV(s) = D_s \mbfV(s) + \tau^{-1}\gamma'(s). 
\]
To estimate this quantity, we define the geodesic variation
\[
\Gamma(s,t) : = \exp_{\gamma(s)}(t\tau \mbfV(s)).
\]
In particular, $\mbfS : =\partial_s \Gamma$ is a Jacobi field along each $t\mapsto \Gamma(s,t)$. Note that $\tau \mbfV(s) = \mbfT(s,0)$ for $\mbfT=\partial_t\Gamma$. Let $\bar \mbfG(s,t)$ denote the parallel transport of $\gamma'(s) \in T_{\Gamma(s,0)}M$ along $t\mapsto\Gamma(s,t)$. Then $\hat \mbfS(s,t) : = \mbfS(s,t) - (1-t)\bar \mbfG(s,t)$ satisfies 
\begin{equation}\label{eq:geo-boundary-values}
    \hat \mbfS(s,0) = \orig \ \ \text{and}\ \  \hat\mbfS(s,1) = \orig
\end{equation}
 (note that $\mbfS(s,0) = \gamma'(s)$ and $\mbfS(s,1) = \orig$). We compute
\[
D^2_t \hat\mbfS = - \Rm(\mbfS,\mbfT)\mbfT = -\Rm(\hat\mbfS,\mbfT)\mbfT - (1-t) \Rm(\bar \mbfG,\mbfT)\mbfT .
\]
Because $|\mbfT| \leq (1+R)\tau$ by \eqref{eq:change-bspt-dist} and $|\bar \mbfG| = d(p,q) \leq \tau$, we can combine this with \eqref{eq:curv-assumption-app}, \eqref{eq:geo-boundary-values} to prove that $|\hat\mbfS(s,t)| + |D_t\hat\mbfS(s,t)| \leq C' \tau^3$. 
In particular, we find that 
\begin{align*}
D_s\hat \mbfV(s) & = D_s \mbfV(s) + \tau^{-1}\gamma'(s) \\
& = \tau^{-1} (D_s\mbfT(s,0) + \mbfS(s,0)) \\
& = \tau^{-1} (D_t\mbfS(s,0) + \mbfS(s,0))\\
& = \tau^{-1} (D_t\hat\mbfS(s,0)  - \bar \mbfG(s,0) + \hat \mbfS(s,0) + \bar \mbfG(s,0) )\\
& = \tau^{-1} (D_t\hat\mbfS(s,0)  + \hat \mbfS(s,0)  ). 
\end{align*}
Thus $|D_s\hat \mbfV(s)|\leq C'\tau^2$. Since $\hat \mbfV(0)=\orig$ we can integrate this to get
\[ |\hat\mbfV(s)|\leq C'\tau^2. \]
Evaluating this at $s=1$ yields the desired estimate.
\end{proof}

\section{The exponential map of a matrix} 

For $\mbfA \in M_d(\RR)$ (space of $d\times d$ matrices), we define the usual exponential map
\[ \exp \mbfA = e^\mbfA = \sum_{k=0}^\infty \frac{\mbfA^k}{k!}. \]
For a small neighborhood of $\operatorname{\textbf{Id}}$ we can define $\log := \exp^{-1}$. 
\begin{Lem}\label{Lem:exp-addition-vs-mult-matrices}
There is $\lambda = \lambda(d) >0$ so that
\[ \vert\log (e^{-\mbfA_0} e^{\mbfA_0+\mbfA_1}) - \mbfA_1\vert \leq 2\vert\mbfA_0\vert\vert\mbfA_1\vert \]
for all $\mbfA_0,\mbfA_1 \in M_d(\RR)$ with $\vert\mbfA_0\vert\vert\mbfA_1\vert < \lambda$.
\end{Lem}
\begin{proof}
Define 
\[
F(\mbfA_0,\mbfA_1) = \log (e^{-\mbfA_0} e^{\mbfA_0+\mbfA_1}) - \mbfA_1.
\]
We take $\lambda >0$ sufficiently small so that $F$ is smooth for $\vert\mbfA_0\vert,\vert\mbfA_1\vert < 2\lambda$. Define
\[ f(s,t) := F(s\mbfA_0,t\mbfA_1), \]
and note that
\[
f(s,0) = \log (\operatorname{\mathbf{Id}}) - \mathbf{O} = \mathbf{O}, \qquad f(0,t) = \log(e^{t\mbfA_1}) - t\mbfA_1 = \mathbf{O},
\]
and that
\[
\frac{\partial^2f}{\partial s \partial t}(s,t) = \frac{\partial}{\partial s}(DF|_{t\mbfA_0,s\mbfA_1}(\mbfA_0,\mathbf{O})) = D^2F|_{t\mbfA_0,s\mbfA_1}((\mbfA_0,\mathbf{O}),(\mathbf{O},\mbfA_1))
\]
so we find that 
\[
\left\vert \frac{\partial^2f}{\partial s \partial t}(s,t) \right\vert \leq C \vert\mbfA_0\vert\vert\mbfA_1\vert
\]
for $C=C(m)$. Then, we can write
\begin{align*}
f(s,t) & = \int_0^t \frac{\partial f}{\partial t}(s,\tau) d\tau\\
& = \int_0^s \int_0^t \frac{\partial^2 f}{\partial s\partial t}(\sigma,\tau) d\tau d\sigma + \int_0^t \frac{\partial f}{\partial t}(0,\tau) d\tau\\
& = \int_0^s \int_0^t \frac{\partial^2 f}{\partial s\partial t}(\sigma,\tau) d\tau d\sigma 
\end{align*}
where we used $f(s,0) = f(0,t) = 0$ in several places. Taking the norm of both sides and using the estimate for the second derivatives yields the assertion. 
\end{proof}

\section{Affine transformation acting on a graph} 

\begin{Lem}\label{Lem:affine-acting-graph}
For $\delta \in (0,1)$ and $d \in \ZZ_{\geq 1}$, there is $\lambda=\lambda(d,\delta)>0$, $C=C(d)>0$ with the following property. Suppose that $\mbfz \in B_2 \subset \RR^d$, $\mbfA \in B_\lambda \subset M_d(\RR)$, $\mbfx \in B_\lambda \subset \RR^d$, $\Pi\subset \RR^d$ is a codimension $1$ subspace with fixed unit normal $\nu$, and $w_0 \in C^\infty(B_\delta \subset \Pi)$ with $|w_0| + \delta |\nabla w_0| \leq \delta$ on $B_\delta$. Letting $\Sigma_0$ denote the $\nu$-normal graph of $w_0$ over $\mbfz + \Pi$ then there exists a unique $w\in \RR$ such that,
\[
\mbfz + w \nu \in e^{\mbfA} \Sigma_0 + \mbfx
\]
where 
\[
|w - w_0(\orig) - \langle \mbfA(\mbfz  + w_0(\orig)\nu),\nu \rangle - \langle \mbfx,\nu\rangle| \leq C( |\mbfA| + |\mbfx|) (|\mbfA| + \|\nabla w_0\|_{C^0(B_{C(|\mbfA|+|\mbfx|)})}). 
\]
\end{Lem}
\begin{proof}
Taking $\lambda$ sufficiently small, it's clear that a unique such $w$ exists (conversely, one may derive the existence of $w$ from the subsequent analysis). There is $\mbfw \in B_\delta$ so that 
\begin{equation}\label{eq:rot-trans-graph-orig-old-graphing-fct}
\mbfz + w \nu = e^\mbfA(\mbfz +\mbfw +  w_0(\mbfw)\nu) + \mbfx
\end{equation}
Rearranging \eqref{eq:rot-trans-graph-orig-old-graphing-fct} and projecting to $\Pi$ we obtain
\[
 \mbfw = \proj_\Pi((e^{-\mbfA}-\Id) \mbfz) + w\proj_\Pi(e^{-\mbfA}\nu) - \proj_\Pi(e^{-\mbfA}\mbfx)
\]
so 
\[
|\mbfw | \leq C(|\mbfA| (1+|w|) + |\mbfx|)
\]
with $C=C(m)$. On the other hand, projecting \eqref{eq:rot-trans-graph-orig-old-graphing-fct} to $\nu$ yields 
\begin{align*}
w & = \langle e^\mbfA(\mbfz + \mbfw + w_0(\mbfw)\nu) - \mbfz,\nu\rangle + \langle \mbfx ,\nu\rangle\\
& = w_0(\mbfw) + \langle \mbfA(\mbfz + \mbfw + w_0(\mbfw)\nu),\nu \rangle + \langle \mbfx,\nu\rangle\\
&  + \langle(e^\mbfA - \mbfA - \Id)(\mbfz + \mbfw + w_0(\mbfw)\nu),\nu \rangle
\end{align*}
A coarse estimate thus gives 
\[
|w| \leq \delta + C(|\mbfA|(1+|\mbfw|) + |\mbfx|),
\]
so we obtain
\[
|\mbfw| \leq C(|\mbfA| + |\mbfx|)
\]
as long as $\lambda$ is sufficiently small. Using this and Taylor's theorem in the $\nu$-projected equation to estimate terms involving $\mbfw$, we have proven the asserted expression.
\end{proof}

\section{Simon's non-concentration estimate in the Riemannian setting}	

  We use the same notations as in Sections \ref{Subsec_reg Min Cone} and \ref{Subsec_Cylind Min Cone}. Let $\cC_\circ\subset \RR^{n_\circ+1}$ be a strictly stable (cf.\ Definition \ref{defi:strictly.stable}) and strictly minimizing (cf.\ Definition \ref{def:strictly-minimizing} below) hypercone, with cross section $\cL_\circ:= \cC_\circ\cap \SSp^{n_\circ}$. Recall that we continue to identify minimizers with their regular part, and thus $\cC_\circ, \cL_\circ$ are smooth. We parametrize $\cC_\circ$ by 
  \[
    (0, +\infty)\times \cL_\circ \to \cC_\circ, \ \ \ (r, \omega) \mapsto x= r\omega\,.
  \]
  Fix a continuous unit normal field $\nu_{\cC_\circ}$ of $\cC_\circ$. Let $U_\circ^\pm$ be the connected components of $\RR^{n_\circ+1}\setminus \overline{\cC_\circ}$, where $\nu_\circ$ points into $U_\circ^+$.
  
  In the cylindrical setting, we fix an integer $k\geq 1$ and take $n=n_\circ +k$ and
  \[ \cC:= \cC_\circ\times \RR^k\subset \RR^{n_\circ+1} \times \RR^k = \RR^{n+1} \]
   parametrized by \[
     (0, +\infty)\times \cL_\circ\times \RR^k \to \RR^{n+1}, \quad (r, \omega, y)\mapsto (r\omega, y)\,.
   \]
   We also write $(x,y) \in \RR^{n_\circ+1}\times \RR^k$. Throughout this section, we will use $r$ to denote
    \[ r(x,y) = |x|, \]
    including when $(x,y) \not \in \cC$. We will also work with balls in any of $\RR^{n_\circ+1}$, $\RR^k$, $\RR^{n+1}$, correspondingly denoted 
    \[ B^{n_\circ+1}_s(x), \; B^k_s(y), \; B_s(x,y). \]
    It will be convenient to work with cylinders
   \[ Q_s(x,y) := B^{n_\circ+1}_s(x) \times B^k_s(y). \]
   The center point will be suppressed if it is the origin. We will also use without comment the fact that $Q_s \subset B_{\sqrt{2} s} \Subset B_{2s}$. We will write $\Sdist_\cC$ for the signed distance to $\cC$ (cf.\ \eqref{eq:Sdist-signed}). 

   We will study $g$-minimizing boundaries $\Sigma$ in $B_2$ with respect to smooth Riemannian metrics $g$ satisfying
\begin{equation}\label{eq metric assumption}
\max\{|g_{ij} - \bar{g}_{ij}| , |\partial_k g_{ij}| , |\partial^2_{kl}g_{ij}|\} \leq \delta \leq 10^{-1},
\end{equation}
    where $\bar g$ denotes the standard Euclidean metric.  We always assume \eqref{eq metric assumption} below. In applications we will always take $g$ to be (rescaled) normal coordinates in a Riemannian metric (curvature bounds as in \eqref{eq:curv-assumption-app} guaranteeing that \eqref{eq metric assumption} holds on small scales) but for the sake of generality we do not assume this here.

  \begin{Thm}[{\cite[Theorem 2.1]{Simon:uniqueness.some}}] \label{Thm_Simon's L^2 Noncon}
    There exist $\eps(\cC)$, $C(\cC)$, $\delta_1(\cC) > 0$ with the following property. Let $\delta \leq \delta_1$ in \eqref{eq metric assumption}, and $\Sigma$ is a $g$-minimizing boundary in $B_2$ with $\supp \Sigma \cap B_1 \neq \emptyset$ and  $\int_{\Sigma} \Sdist_\cC^2 d\|\Sigma\|_{\bar g} \leq \eps$, then there exist a $C^2_{\textnormal{loc}}$ function $u$ on a subset of $\cC$ and an open $U\supset Q_{1/2}\cap \{r>1/20\}$ so that $\graph_\cC u = \Sigma \cap U$ and 
    \begin{multline*}
      \int_{Q_{1/2}} r^{-2}\Sdist_\cC^2 d\|\Sigma\|_{\bar g} + \int_{Q_{1/2} \cap \Dom(u)} (|\nabla_\cC u|^2 + r^{-2}u^2)\ d\|\cC\|_{\bar g} \\
      \leq C(\cC) \left( \int_{Q_1} \Sdist_\cC^2 \ d\|\Sigma\|_{\bar g} + \delta \right) \, . 
    \end{multline*}
  \end{Thm}
The remainder of this appendix contains a proof of Theorem \ref{Thm_Simon's L^2 Noncon} including the adaptations to the non-flat background. We will use the spectral analysis on $\cC_\circ$ developed in Section \ref{Subsec_reg Min Cone}.

  \subsection{Weighted norm}\label{Subsec_weighted}
  
For $\Omega \subset \cC_\circ$ and $v \in C^2_\textnormal{loc}(\Omega)$ we define 
\[
\|v\|_{C^2_*(\Omega)} : = \sup_{x \in \Omega} \left( |x|^{-1}|v(x)| + |\nabla v(x)| + |x| |\nabla^2 v(x)| \right).
\]
If $\Omega = \cC_\circ$, we will write $v = O^*_2(r^\alpha)$ when  $r \|v\|_{C^2_*(\cC_\circ\setminus B_r)} = O(r^\alpha)$ for $r > 0$.

We also extend this definition to the cylindrical setting. For $\Omega \subset \cC$ and $v \in C^2_\textnormal{loc}(\Omega)$, we define
\[
\|v\|_{C^2_*(\Omega)} : = \sup_{(x,y) \in \Omega} \left( |x|^{-1}|v(x,y)| + |\nabla v(x,y)| + |x| |\nabla^2 v(x,y)| \right).
\]
If $\Omega = \cC$, we will write $v = O^*_2(r^\alpha)$ if $r \|v\|_{C^2_*((\cC_\circ\setminus B_r) \times \RR^k)} = O(r^\alpha)$ for $r > 0$. 

In both cases, $\Omega$ might be denoted $\Dom(v)$ or else suppressed if it can be understood from the context.

  \subsection{Hardt-Simon Foliation} \label{Subsec_HS foliation}
   By \cite{HardtSimon:isolated}, there exists a pair of real numbers
   \[ \gamma(\pm) = \gamma(\cC_\circ, \pm)\in \{\gamma_1, 2-n_\circ -\gamma_1\}, \]
   where $\gamma_1 = -\alpha(\cC_\circ)>-\frac{n_\circ - 2}{2}$ is defined in Section \ref{Subsec_reg Min Cone}, a unique pairwise-disjoint family of minimizing hypersurfaces $\{{S}_\circ(\lambda)\}_{\lambda\in \RR}$ foliating $\RR^{n_\circ+1}\setminus \{0\}$ parametrized by $\lambda\in\RR$, as well as a family of $C^\infty_{\textnormal{loc}}(\cC_\circ)$ functions $\{(f_\circ)_\lambda\}_{\lambda\in \RR}$ such that
   \begin{enumerate}[(i)]
    \item ${S}_\circ(0)=\cC_\circ$, $(f_\circ)_0 \equiv 0$; ${S}_\circ(\lambda)\in U_\circ^{\textrm{sgn}(\lambda)}$ for every $\lambda\neq 0$.
    \item For every $t>0$ and $\lambda\neq 0$,
    \begin{align}
      {S}_\circ(t\lambda) & = t^{\frac1{1-\gamma(\textrm{sgn}\lambda)}} \cdot {S}_\circ(\lambda)\,, \label{Equ_HS_scal H_0(lambda)}\\
      (f_\circ)_{t\lambda}(x) & = t^{\frac{1}{1-\gamma(\textrm{sgn}\lambda)}}\cdot (f_\circ)_{\lambda}(t^{-\frac{1}{1-\gamma(\textrm{sgn}\lambda)}}x) \label{Equ_HS_scal f_lambda} \,.
    \end{align} 
    \item There exists $R=R(\cC_\circ)>1$ and $\delta=\delta(\cC_\circ)\in (0, 1)$ such that for every $\lambda\neq 0$, outside the ball of radius $R|\lambda|^{\frac{1}{1-\gamma(\textrm{sgn}\lambda)}}$, we have 
    \begin{align}
     \begin{split}
      {S}_\circ(\lambda) & = \graph_{\cC_\circ} (f_\circ)_\lambda \\
      (f_\circ)_\lambda(x) & = \lambda\cdot r^{\gamma(\textrm{sgn}\lambda)}\psi_1(\omega) + O_2^*\left(|\lambda|^{1+\frac{\delta}{1-\gamma(\textrm{sgn}\lambda)}}\cdot r^{\gamma(\textrm{sgn}\lambda) -\delta} \right)\,,  
     \end{split} \label{Equ_HS_f_lambda = lambda r^gamma}
    \end{align}
   \end{enumerate}
   where $\psi_1>0$ is the $L^2$-unit first eigenfunction of $-(\Delta_{\cL_\circ} + |A_{\cL_\circ}|^2)$ as in section \ref{Subsec_reg Min Cone}. 
   
   We assume from now on that $\cC_\circ$ is strictly stable and strictly minimizing in the sense of \cite{HardtSimon:isolated}:

   \begin{Def}\label{def:strictly-minimizing}
   We say $\cC_\circ$ is \textbf{strictly minimizing} if $\gamma(+)=\gamma(-) = \gamma_1$ above. 
   \end{Def}

Let $\nu_{{S}_\circ}$ be the unit $C^1_{\textnormal{loc}}$ vector field on $\RR^{n_\circ+1}\setminus \{0\}$ such that
\[ \nu_{{S}_\circ}(x)\perp T_x{S}_\circ(\lambda) \text{ whenever } x\in {S}_\circ(\lambda), \]
\[ \nu_{{S}} = \nu_{\cC_\circ} \text{ along } \cC_\circ. \]
Note that $\nu_{{S}_\circ}$ is homogeneous with degree $0$, i.e. \[
     \nu_{{S}_\circ}(t x) = \nu_{{S}_\circ}(x) \,, \qquad \text{ for every }t>0 \text{ and } x\in \RR^{n_\circ+1}\setminus \{0\}\, . 
   \] 
   Also define $\frT_{{S}_\circ}(x)\in C^0(\RR^{n_\circ+1}) \cap C^1_\textrm{loc}(\RR^{n_\circ+1}\setminus \{0\})$ by \begin{equation}\label{eq:FH definition}
     \frT_{{S}_\circ}(x) := \lambda \qquad \text{ if }x\in {S}_\circ(\lambda) \,.
   \end{equation}
   
   We now extend everything to the cylindrical case, where
   \[ \cC:= \cC_\circ\times \RR^k\subset \RR^{n+1=n_\circ+1+k}. \]
    Let
   \[ U^\pm:= U^\pm_\circ\times \RR^k \]
   be the connected components of $\RR^{n+1}\setminus \overline{\cC}$, and
   \[ {S}(\lambda):= {S}_\circ(\lambda)\times \RR^k \]
   be the minimizing hypersurfaces foliating $\RR^{n+1}$. Then we write
   \[ f_\lambda, \; \frT_{{S}}, \; \nu_{{S}} \]
   for the $y$-invariant extensions of $(f_\circ)_\lambda, \frT_{{S}_\circ}, \nu_{{S}_\circ}$ to $\cC$, $\RR^{n+1}$, $\RR^{n+1} \setminus (\{0\} \times \RR^k)$, respectively. Note that, since all the ${S}(\lambda)$ are minimal, we have 
   \begin{align}
     \Div_{\RR^{n+1},\bar g}(\nu_{{S}}) \equiv 0\,. \label{HS_div(X_H)=0}
   \end{align}
 
   \begin{Prop}\label{Prop_HS_Est of F_H and |D F_H|}
     Let $\cC, U^\pm, {S}(\lambda), f_\lambda, \frT_{{S}}$ be defined as above.
     Then for every $(x, y)\in \RR^{n+1}\setminus(\{0\}\times \RR^k)$ and every unit vector $v\in \RR^{n+1}$, we have
     \begin{align}
        \frT_{{S}}(t x, y) & = t^{1-\gamma_1} \frT_{{S}}(x, y)\,; \label{Equ_HS_F_H 1-gamma homogen} \\
        C(\cC_\circ)^{-1}\cdot\frac{\dist_{\cC_\circ}(x)}{|x|^{\gamma_1}} & \leq \frT_{{S}}(x, y) \leq C(\cC_\circ)\cdot\frac{\dist_{\cC_\circ}(x)}{|x|^{\gamma_1}}\,; \label{Equ_HS_F_H sim r^(gamma)dist_C} \\
        |\nabla_{v^{\perp}}\frT_{{S}}(x, y)|^2 & \leq C(\cC_\circ) |x|^{-2\gamma_1}\left(1\pm \langle v, \nu_{{S}}(x, y) \rangle_{\bar g} \right) \label{Equ_HS_|D F_H|<r^(-2gamma)(1+<v,X_H>)}\,
     \end{align}
     where $\nabla_{v^{\perp}}\frT_{{S}}(x, y) := (\nabla \frT_{{S}}(x, y))^{\perp v}$ is the Euclidean gradient in $v^\perp$ direction and \eqref{Equ_HS_|D F_H|<r^(-2gamma)(1+<v,X_H>)} holds for either choice of sign.
   \end{Prop}
   \begin{proof}
     Without loss of generality we can assume that $y=0$. Also, since $\frT_{{S}}\in C^1_{\textnormal{loc}}$ and $\nu_{{S}}\in C^0_{\textnormal{loc}}$, it suffices to consider when $x\notin \cC_\circ$. Now we see that \eqref{Equ_HS_F_H 1-gamma homogen} follows directly from \eqref{Equ_HS_scal H_0(lambda)} and \eqref{eq:FH definition}, while \eqref{Equ_HS_F_H sim r^(gamma)dist_C} follows directly from definition and \eqref{Equ_HS_f_lambda = lambda r^gamma}. We now prove \eqref{Equ_HS_|D F_H|<r^(-2gamma)(1+<v,X_H>)}. Since $\frT_{{S}}$ is constant on every ${S}(\lambda)$  and independent of $y$, we have \[
       \nabla_{\RR^{n+1}}\frT_{{S}}(x, y) = c(x) \nu_{{S}}(x, y) \,,  \]
     for some $c(x)\in \RR$. We now determine the value of $c(x)$. 
    Differentiating both sides of \eqref{Equ_HS_F_H 1-gamma homogen} with respect to $t$ and evaluating at $t=1$ we find \[
       c(x)\langle \nu_{{S}}(x, y), (x, 0) \rangle = \langle \nabla_{\RR^{n+1}}\frT_{{S}}(x, y), (x, 0) \rangle = (1-\gamma_1)\frT_{{S}}(x, y) \,.
     \]
     Solving for $c(x)$ (recall that $\gamma_1 = -\alpha(\cC_\circ)<-1$ by \eqref{eq:kappa.n}) we find 
     \[
     |\nabla_{v^\perp} \frT_{{S}}(x,y)|^2  = \frac{(1-\gamma_1)^2 \frT_{{S}}(x, y)^2}{\langle \nu_{{S}}(x,y),(x,0)\rangle^2}|\nu_{{S}}(x,y)^{\perp v}|^2.
     \]
     For any unit vectors $v,X$ we have $|X^{\perp v}|^2 = 1-(X\cdot v)^2 \leq 2(1\pm X\cdot v)$ (this holds for either choice of sign). Moreover, by (\ref{Equ_HS_f_lambda = lambda r^gamma}) and $\gamma_1 \neq 1$ we find: \[
       \langle \nu_{{S}_\circ}(x), x\rangle \geq C(\cC_\circ)^{-1} \lambda |x|^{\gamma_1} = C(\cC_\circ)^{-1} \frT_{{S}}(x, y) |x|^{\gamma_1}  \,.
     \]
     Putting these expressions together proves (\ref{Equ_HS_|D F_H|<r^(-2gamma)(1+<v,X_H>)}).     
   \end{proof}
   
     \subsection{Divergence}
   Fix a metric $g$ on $Q_4$ satisfying \eqref{eq metric assumption}. For a vector field $X$ on $Q_4$ we will write $\Div_{\RR^{n+1},g} X$ for the divergence of $X$ with respect to $g$ and $\Div_{\RR^{n+1},\bar g} X$ for the divergence of $X$ the Euclidean metric $\bar g$. Similarly, for $\Sigma\subset Q_4$ a hypersurface or varifold we write $\Div_{\Sigma,g} X$ for the divergence of $X$ along $\Sigma$ with respect to $g$ and $\Div_{\Sigma,\bar g}X$ for the divergence of $X$ along $\Sigma$ with respect to the Euclidean metric $\bar g$.    Set
   \[
   \cE_{\RR^{n+1}}(X,g) : = \Div_{\RR^{n+1},g} X - \Div_{\RR^{n+1},\bar g} X
   \]
   and
\[
\cE_{\Sigma}(X,g) : = \Div_{\Sigma,g} X - \Div_{\Sigma,\bar g} X 
\]
The following lemma is easily verified
\begin{Lem}\label{lemm div g bar g}
We have $|\cE_{\RR^{n+1}}(X,g)| \leq C_n \delta |X|$ and $|\cE_{\Sigma}(X,g)| \leq C_n  \delta(|X|+|\nabla_{\bar g} X|) $.
\end{Lem}

  \subsection{$L^2$ non-concentration}
   Now we are ready to prove Simon's $L^2$ non-concentration estimate in the Riemannian setting. We continue to assume that the background metric $g$ satisfies \eqref{eq metric assumption}. 
   
   The following formula was essentially derived in \cite[Lemma 2.2, (24)]{Simon:uniqueness.some} with different notation.\footnote{Lemma \ref{Lem_Noncon_1st variation} also applies to stationary integral varifolds $\Sigma$ but we will not need that level of generality.}
   
   \begin{Lem} \label{Lem_Noncon_1st variation}
     Assume that \eqref{eq metric assumption} holds. Let $\Sigma$ be $g$-minimizing in $B_2$, $0\leq \zeta\in C^1_c(B^k_1)$, $0\leq \varphi \in C^1_c([0, 1))$, and $f\in \Lip(\RR^{n+1})$. Then we have 
     \begin{align}
      \begin{split}
        & \int_{Q_1} \zeta(y)\varphi(r)f(x, y)^2 (n_\circ + |\nu_y|^2)\ d\|\Sigma\|_{\bar g} \\
        =\ & \int_{Q_1} \Big[ \langle \nabla_{\bar g}\zeta(y), \nu_y  \rangle_{\bar g} \varphi(r)r\partial_r^\perp - \zeta(y)\varphi'(r)r(1-(\partial_r^\perp)^2) \Big]f^2  \\ & \qquad \qquad - \zeta(y)\varphi(r) \langle \nabla_{\Sigma,\bar g} (f^2), (x, 0) \rangle_{\bar g} \ d\|\Sigma\|_{\bar g}  \\
& + O\left( \int_{Q_1} (  | \zeta\varphi f^2| + r |\nabla_{\bar g}(\zeta\varphi f^2)| ) d\Vert \Sigma\Vert_{\bar g} \right) \delta 
      \end{split} \label{Equ_Noncon_1st variation}
     \end{align}
     where the implied constant depends only on $n$. Here 
     \[
       \partial_r^\perp:= \langle (\displaystyle\tfrac{x}{|x|}, 0),  \nu_{\Sigma,\bar g} \rangle_{\bar g},  \quad  \nu_y := \displaystyle\sum_{i=1}^k \langle \partial_{y_i}, \nu_{\Sigma,\bar g} \rangle_{\bar g} \partial_{y_i}\,,  \] 
     where $\nu_{\Sigma, \bar g}$ is a unit normal field on $\Sigma$ under $\bar g$ defined almost\footnote{Note that the formula is unchanged if we replace $\nu_\Sigma$ by $-\nu_\Sigma$.} everywhere. 

     There is $\kappa_\cC$ is sufficiently small depending only on $\cC$ so that if $\Sigma$ is locally written as a graph of some function $u$ with $\|u\|_{C^1_*} < \kappa_\cC$ over a subset of $\cC$, namely $(\hat{x}, y):= (x+u(x, y)\nu_{\cC_\circ}(x), y)\in \spt \Sigma$, then for some choice of sign $\pm$, we have
     \begin{align}
        \langle \partial_{y_i}, \nu_{\Sigma,\bar g} \rangle_{\bar g}|_{(\hat{x}, y)} & = \pm (1+O(\|u\|_{C^1_*}))\cdot \partial_{y_i}u(x, y)\,, \;\;\;\;\; \text{ for every }1\leq i\leq k\,; \label{Equ_Noncon_partial_y^perp}\\
        \partial_r^\perp |_{(\hat{x}, y)} & = \mp (1+O(\|u\|_{C^1_*}))\cdot \partial_ru(x, y) \,. \label{Equ_Noncon_partial_r^perp}
     \end{align}
   \end{Lem}
   \begin{proof}
    To prove (\ref{Equ_Noncon_1st variation}), we consider $X = \zeta(y)\varphi(r)f(x, y)^2\cdot (x, 0)$ to be a test vector field in the first variation formula,
    \[
     \int \diverg_{\Sigma,g} X \ d\|\Sigma\|_g = 0\,.
   \]
   Note that Lemma \ref{lemm div g bar g} allows us to replace our $g$'s with $\bar g$'s and obtain
   \[
     \int \diverg_{\Sigma,\bar g} X \ d\|\Sigma\|_{\bar g} = O\left( \delta \cdot \int_{Q_1} (  | \zeta\varphi f^2| + r |\nabla_{\bar g}(\zeta\varphi f^2)| ) d\Vert \Sigma\Vert_{\bar g} \right).
   \]
   In particular, it suffices to study the term on the left hand side. Observe that
\[
  \diverg_{\Sigma,\bar g} (x,0) = n - \diverg_{\Sigma,\bar g} (0, y)   = n - \sum_{i=1}^k |\partial_{y_i}^T|^2 = n_\circ + |\nu_y|^2,
\]
where $\partial^T_{y_i}$ is the orthogonal porjection of $\partial_{y_i}$ onto $T\Sigma$ under $\bar g$. Note also that
\begin{align*}
  \langle \nabla_{\Sigma,\bar g}\zeta(y)  ,(x,0) \rangle_{\bar g} & = - 
  \langle \nabla_{\bar g} \zeta(y), \nu_y \rangle_{\bar g} r \partial_r^\perp 
\end{align*}
Thus,
\begin{align*}
  \diverg_{\Sigma,\bar g} X & = \zeta(y)\varphi(r) f(x,y)^2 (n_\circ + |\nu_y|^2)\\
  & + \zeta(y) \varphi'(r)f(x,y)^2 r(1-(\partial_r^\perp)^2)\\
  & + \zeta(y)\varphi(r) \langle \nabla_{\Sigma,\bar g}(f^2),(x,0) \rangle_{\bar g} \\
  & - \varphi(r) \langle \nabla_{\bar g} \zeta(y),\nu_y \rangle_{\bar g} r \partial_r^\perp f(x,y)^2
\end{align*}
As explained, this proves \eqref{Equ_Noncon_1st variation}. The formulae \eqref{Equ_Noncon_partial_y^perp} and \eqref{Equ_Noncon_partial_r^perp} follow via a standard computation.
\end{proof}

Let us fix a sufficiently good choice of $\kappa_\cC$ in addition to satisfying \eqref{Equ_Noncon_partial_y^perp} and \eqref{Equ_Noncon_partial_r^perp} of Lemma \ref{Lem_Noncon_1st variation}. First, note that combining \eqref{Equ_HS_f_lambda = lambda r^gamma} and \eqref{Equ_Noncon_partial_r^perp} and by taking $\kappa_\cC$ smaller if necessary we can arrange (as in Proposition \ref{Prop_HS_Est of F_H and |D F_H|}) that
\begin{equation}
\langle r\partial_r, \nu_{{S}}(x) \rangle_{\bar g} \leq C(\cC_\circ) \cdot \dist(x, \cC_\circ)
\label{eq radial XH comparable dist}
\end{equation}
whenever $\dist(x, \cC_\circ) < \kappa_\cC |x|$. For $\kappa_\cC > 0$ smaller yet and the domain
\[
D_\Phi := \{ (x, y, t) \in \cC_\circ \times \RR^k \times \RR : |t| < \kappa_\cC |x| \},
\]
the graph map over hypercone 
\begin{equation} \label{eq:app.noncon.tub.nghd}
	\Phi: D_\Phi \to \RR^{n+1}\,,\quad (x, y, t)\mapsto ( x+t\nu_{\cC_\circ, \bar g}(x), y)\,,
\end{equation}
is a diffeomorphism onto its image. Having fixed $\kappa_\cC$, we know that for sufficiently small $\eps_0(\cC) > 0,\delta_0(\cC)$, any $g$-minimizing boundary $\Sigma$ in $B_2$ so that $g$ satisfies \eqref{eq metric assumption} and $\Sigma$ satisfies $\int_{\Sigma}\Sdist_\cC^2 d\|\Sigma\|_{\bar g} \leq \eps_0$ then there exists an open subset of $\cC$ and a $C^2$ function $u$ on it with
\begin{equation} \label{eq:app.noncon.eps0}
	\Sigma \cap (Q_1 \setminus \{ r < 1/4 \}) = \graph_\cC u, \qquad \| u \|_{C^2_*} < \kappa_\cC.
\end{equation}
We fix these $\kappa_\cC$, $\eps_0(\cC),\delta_0(\cC)$ for the rest of the appendix.

   The following lemma is a modified version of \cite[Lemma 1.4]{Simon:uniqueness.some}. We first introduce the following notation.  Let $Z\subset \RR^{n+1}$ be a closed subset and $h:Z\to \RR_{\geq 0}$ be an upper-semi-continuous function.  Denote by \begin{equation}\label{eq:def.base.set}
       \bar{B}_h = \bar B_h(Z) := \bigcup_{z\in Z} \bar{B}_{h(z)}(z) \subset \RR^{n+1}\,,
   \end{equation}
   which is also a closed subset (note, for $h(z) = 0$, $\bar B_{h(z)}(z)$ is a point). We will also need to work with annuli
   \[ A^{n_\circ+1}_{s_1, s_2} = B^{n_\circ+1}_{s_2} \setminus \bar B^{n_\circ+1}_{s_1} \]
   in the first $\RR^{n_\circ+1}$ factor. 

   Note that since $\Sigma$ is mininal with respect to $g$, its mean curvature with respect to $\bar g$ is controlled by $C\delta r^{-1}_\Sigma$.

   \begin{Lem} \label{Lem_Noncon_Maximal graphical region}
     Let $\alpha > 4-n_\circ$. There exist $\eps_1(\cC)\leq \eps_0(\cC)$, $\delta_1(\cC, \alpha)\leq \delta_0(\cC)$, $C(\cC, \alpha)$ with the following property. Let $\delta \leq \delta_1$ in \eqref{eq metric assumption} and $\Sigma$ be a $g$-minimizing boundary in $B_2$ with $\supp \Sigma \cap B_{1} \neq \emptyset$ and $\int_{\Sigma}\Sdist_\cC^2\, d\|\Sigma\|_{\bar g} \leq \eps_1$. Then there exists an upper-semi-continuous $s_\Sigma: \bar B_1^k \to [0, 1/20)$ and $u$ defined on an open subset of $\cC$ such that
     \[ \graph_\cC u = \Sigma \cap U \text{ with } \|u\|_{C^2_*} < \kappa_\cC \]
     where
     \[
       U := Q_1\setminus \bar{B}_{s_\Sigma}(\{0\}\times \bar{B}_1^k) \supset A_{1/2, 1}^{n_\circ+1} \times B_1^k \,.
     \]
     Moreover, for $1/2 \leq \tau_1 < \tau_2 \leq 1$, the following estimate holds if $\delta\leq \delta_1$ in \eqref{eq metric assumption}:
     \[
     \begin{split}
        \int_{Q_{\tau_1} \cap \Dom(u)} r^\alpha (|\nabla u|^2 + r^{-2}u^2)\ d\|\cC\|_{\bar g} + \int_{Q_{\tau_1} \cap\bar{B}_{2s_\Sigma}(\{0\}\times B_1^k)}r^\alpha\ d\|\Sigma\|_{\bar g} 
       \\ \leq\ \frac{C(\cC, \alpha)}{(\tau_2-\tau_1)^2} \cdot \int_{Q_{\tau_2}} r^{\alpha-2}\Sdist_\cC^2 \ d\|\Sigma\|_{\bar g} + C(\cC,\alpha) \cdot \delta^2 \,. 
     \end{split}
     \]
   \end{Lem}

   \begin{proof}
     For each $y\in \bar{B}^k_1$, we define 
     \begin{multline*}
       s_\Sigma(y):= \inf\big\{s\in (0, 1/2): \forall s'\in [s,1/2] \textrm{ we have } \Sigma \cap (A^{n_\circ+1}_{s', {2s'}}\times B_{2s'}^k(y)) = \graph_\cC u\\
       \text{for some function $u$ defined on an open subset of }\cC \text{ with } \|u\|_{C^2_*} < \kappa_\cC  \big\} \,. 
     \end{multline*}
     It is clear that $s_\Sigma$ is upper-semi-continuous in $y$. By taking $\eps_1,\delta_1$ even smaller, we may assume $s_\Sigma\leq 1/20$. By rescaling, a compactness argument, unique continuation and interior estimates for the minimal surface equation, we always have for every $y\in {B}_{\tau_1}^k \cap \{s_\Sigma >0\}$, 
 \begin{align}
       \int_{Q_{\tau_2} \cap (A^{n_\circ+1}_{s_\Sigma(y), 2s_\Sigma(y)}\times B^k_{2s_\Sigma(y)}(y))} r^{\alpha-2} \Sdist_\cC^2 \ d\|\Sigma\|_{\bar g} + \delta^2 \cdot s_\Sigma(y)^{n+\alpha} \geq C(\cC,\alpha)^{-1}s_\Sigma(y)^{n+\alpha} \,. \label{Equ_Noncon_L^2 dist control smallest scale vol weight}
\end{align}
The monotonicity formula yields for every $y\in {B}_{\tau_1}^k \cap \{s_\Sigma >0\}$,
\begin{equation}
  \int_{Q_{\tau_1} \cap B_{10s_\Sigma(y)}(0,y)} r^\alpha \, d\|\Sigma\|_{\bar g} \leq C(n, \alpha)s_\Sigma(y)^{n+\alpha}\,,\label{Equ_Noncon_L^2 bad balls r to alpha control}
\end{equation}
so when $\delta\leq \delta_1(\cC,\alpha)$, by combining \eqref{Equ_Noncon_L^2 dist control smallest scale vol weight} and \eqref{Equ_Noncon_L^2 bad balls r to alpha control} with the Vitali covering lemma we obtain
\[
  \int_{Q_{\tau_1} \cap\bar{B}_{2s_\Sigma}}r^\alpha\ d\|\Sigma\|_{\bar g} 
       \ \leq\ C(C(\cC,\alpha)^{-1}-\delta^2)^{-1} \int_{Q_{\tau_2}} r^{\alpha-2}\Sdist_\cC^2 d\|\Sigma\|_{\bar g}\, ,
\]
 where $\bar B_{2s_\Sigma}$ is defined via \eqref{eq:def.base.set} with $Z=\bar B^k_1$. Using $\Vert u \Vert_{C^2_*} < \kappa_\cC$ on $\Dom(u)$, a similar argument gives
\begin{multline}\label{eq noncon W12 estimates near bad}
\int_{Q_{\tau_1} \cap \Dom(u) \cap \overline B_{2s_\Sigma}} r^\alpha (|\nabla u|^2 + r^{-2} u^2) \ d\|\cC\|_{\bar g} \\
\leq\ C(C(\cC,\alpha)^{-1}-\delta^2)^{-1}  \int_{Q_{\tau_2}} r^{\alpha-2}\dist_\cC^2 \ d\|\Sigma\|_{\bar g} . 
\end{multline}
On the other hand, for any $(x,y) \in  (Q_{\tau_1} \cap \Dom(u))\setminus \overline B_{2s_\Sigma}$, it follows from the definition of $s_\Sigma$ that there exists a universal $c > 1$ such that
\[ Q_{|x|/c}(x,y) \subset \Dom(u), \]
\[ Q_{(\tau_2-\tau_1)|x|/c} \subset Q_{\tau_2}. \]
Interior estimates for the minimal surface equation yield 
\begin{multline*}
	\int_{Q_{(\tau_2-\tau_1) |x|/2c}(x,y)}  (r^2|\nabla u|^2 + u^2) \ d\|\cC\|_{\bar g} \\
	\leq C\frac{C(\cC)}{(\tau_2-\tau_1)^2} \cdot \int_{Q_{(\tau_2-\tau_1)|x|/c}(x,y)} u^2 \ d\|\cC\|_{\bar g} +C(\cC) \cdot |x|^{n-2} \cdot \delta^2.
\end{multline*}
Since $r$ and $|x|$ are comparable on $Q_{|x|/c}(x,y)$, i.e., $C^{-1} r \leq |x| \leq C r$ for a $C > 1$ (which depends on $c$), the previous estimate directly implies:
\begin{align*}
	& \int_{Q_{(\tau_2-\tau_1)|x|/2c}(x,y)}  r^\alpha (|\nabla u|^2 + r^{-2} u^2) \ d\|\cC\|_{\bar g} \\ 
	& \qquad \leq C^{\alpha-2} \cdot |x|^{\alpha-2} \cdot \int_{Q_{(\tau_2-\tau_1) |x|/2c}(x,y)}  (r^2 |\nabla u|^2 + u^2) \ d\|\cC\|_{\bar g} \\ 
	& \qquad \leq C^{\alpha-2} \cdot |x|^{\alpha-2} \cdot \left( \frac{C(\cC)}{(\tau_2-\tau_1)^2} \int_{Q_{(\tau_2-\tau_1) |x|/2c}(x,y)} u^2d\|\cC\|_{\bar g} + C(\cC) \cdot |x|^{n-2} \cdot \delta^2 \right) \\
	& \qquad \leq \frac{C(\cC,\alpha)}{(\tau_2-\tau_1)^2} \int_{Q_{(\tau_2-\tau_1) |x|/c}(x,y)} r^{\alpha-2}u^2 \ d\|\cC\|_{\bar g} \\
    & \qquad + C(\cC,\alpha) \cdot \left( \int_{Q_{(\tau_2-\tau_1) |x|/c}(x,y)} r^{\alpha - 4} \ d\|\cC\|_{\bar g} \right) \cdot \delta^2.
\end{align*}
Using the Besicovitch covering lemma we find that
\begin{align}\label{eq noncon W12 estimates far bad}
\begin{split}
	& \int_{Q_{\tau_1} \cap \Dom(u) \setminus \overline B_{2s_\Sigma}} r^\alpha(|\nabla u|^2 + r^{-2}u^2)\ d\|\cC\|_{\bar g} \\ 
	& \qquad \leq\ C \frac{C(\cC, \alpha)}{(\tau_2-\tau_1)^2} \cdot \int_{Q_{\tau_2} \cap \Dom(u)} r^{\alpha-2} u^2 \ d\|\cC\|_{\bar g} + C(\cC,\alpha) \cdot \left(\int_{Q_{\tau_2}} r^{\alpha-4} \ d\|\cC\|_{\bar g}\right) \cdot \delta^2 \\
	& \qquad \leq\ \frac{C(\cC, \alpha)}{(\tau_2-\tau_1)^2} \cdot \int_{Q_{\tau_2}} r^{\alpha-2} \dist_\cC^2 \ d\|\Sigma\|_{\bar g} + C(\cC,\alpha) \cdot \delta^2.
\end{split}
\end{align}
In the last step, we used the fact that $u$ is comparable to $\dist_\cC$ and $d\Vert \cC \Vert_{\bar g}$ to $d\Vert \Sigma \Vert_{\bar g}$ on $\Dom(u)$ and that $\alpha > 4-n_\circ$ guarantees the integrability of $r^{\alpha-4}$ on $\cC$. The assertion thus follows by combining estimates \eqref{eq noncon W12 estimates near bad} and  \eqref{eq noncon W12 estimates far bad}.
 \end{proof}

Below we fix $\eps_1(\cC),\delta_1(\cC,1) > 0$ as in Lemma \ref{Lem_Noncon_Maximal graphical region} with $\alpha=1$.

   \begin{Cor} \label{Cor_Noncon_Vol diff est btwn M and C}
      There exists $C(\cC) > 0$ with the following significance.
     
      If $\delta \leq \delta_1$ in \eqref{eq metric assumption}, $\Sigma$ is a $g$-minimizing boundary in $B_2$ with $\supp \Sigma \cap B_{1} \neq \emptyset$ and $\int_{\Sigma}\Sdist_\cC^2\, d\|\Sigma\|_{\bar g} \leq \eps_1$, $1/2 \leq \tau_1<\tau_2 \leq 1$, $\tau_3 = (\tau_1+\tau_2)/2$, $\zeta\in C_c^1(B_{\tau_3}^k)$, $\psi\in C_c^2([0, \tau_3))$, and $\psi|_{[0,\tau_1]} \equiv 1$, then
     \begin{align*}
       & \int_{Q_{\tau_2}} \zeta(y)\psi(r)\ (d\|\Sigma\|_{\bar g} - d\|\cC\|_{\bar g}) \\ 
       & \qquad \leq \frac{C(\cC)}{(\tau_2-\tau_1)^2}\cdot (\|\zeta\|_{C^1}\cdot\|\psi\|_{C^2})\int_{Q_{\tau_2}} r^{-1}\dist_\cC^2\ d\|\Sigma\|_{\bar g} + C(\cC) \cdot (\|\zeta\|_{C^1}\cdot\|\psi\|_{C^2}) \cdot\delta \,.
     \end{align*}
   \end{Cor}
   \begin{proof}
     Take $f=1$ and $\varphi=\psi$ in Lemma \ref{Lem_Noncon_1st variation} we get
     \begin{align*}
      &  \int_{Q_{\tau_2}} \zeta(y)\psi(r)(n_\circ + |\nu_y|^2)\ d\|\Sigma\|_{\bar g} \\
       & \qquad \leq \underbrace{\int_{Q_{\tau_2}} \big( |\nabla_{\bar g}\zeta(y)|\psi(r) |\nu_y||r\partial_r^\perp| + \zeta(y)|\psi'(r)|r|\partial_r^\perp|^2\big) \ d\|\Sigma\|_{\bar g}}_{\text{(I)}} \\
       & \qquad \qquad + \underbrace{\int_{Q_{\tau_2}} \zeta(y)\psi'(r)r\ \big(d\|\cC\|_{\bar g} - d\|\Sigma\|_{\bar g}\big) }_{\text{(II)}} - 
       \underbrace{\int_{Q_{\tau_2}} \zeta(y)\psi'(r)r\ d\|\cC\|_{\bar g}}_{\text{(III)}} \\
       & \qquad \qquad + C(\cC) \cdot \delta \cdot \Vert \zeta \Vert_{C^1}  \Vert \psi\Vert_{C^1}.
     \end{align*}
     To deal with (III), notice that if we apply Lemma \ref{Lem_Noncon_1st variation} with $\Sigma=\cC$, $g = \bar g$, $f=1$ and $\zeta, \psi$ as above, then since $|\nu_y| = |\partial_r^\perp| = 0$, we get 
     \begin{align}
       \text{(III)} = -n_\circ \int_{Q_{\tau_2}} \zeta(y)\psi(r)\ d\|\cC\|_{\bar g} \,.   \label{Equ_Noncon_Pf vol diff bd, (III)}
     \end{align}
     To deal with (I), recall that $\supp(\zeta(y)\psi(r)) \subset Q_{\tau_3}$. Let $U$ be the domain given by Lemma \ref{Lem_Noncon_Maximal graphical region} with $\Sigma \cap U = \graph_\cC u$. We have 
     \begin{align}
      \begin{split}
       \text{(I)} & \leq 2\|\zeta\|_{C^1}\cdot \|\psi\|_{C^1} \cdot \int_{Q_{\tau_3}}r(|\partial_r^\perp|+|\nu_y|)^2 \ d\|\Sigma\|_{\bar g} \\
       & \leq 2\|\zeta\|_{C^1}\cdot \|\psi\|_{C^1} \cdot \left( \int_{Q_{\tau_3}\cap U} r(|\partial_r^\perp|+|\nu_y|)^2\ d\|\Sigma\|_{\bar g} + \int_{Q_{\tau_3}\setminus U} 4r\ d\|\Sigma\|_{\bar g} \right) \\
       & \leq C(\cC)\cdot\|\zeta\|_{C^1}\cdot \|\psi\|_{C^1} \cdot \left( \int_{Q_{\tau_3}\cap \Dom(u)} r|\nabla u|^2\ d\|\cC\|_{\bar g} + \int_{Q_{\tau_3}\setminus U} r\ d\|\Sigma\|_{\bar g} \right) \\
       & \leq \frac{C(\cC)}{(\tau_2-\tau_1)^2}\cdot\|\zeta\|_{C^1}\cdot \|\psi\|_{C^1} \cdot \int_{Q_{\tau_2}} r^{-1}\dist_\cC^2\ d\|\Sigma\|_{\bar g} + C(\cC)\cdot \|\zeta\|_{C^1}\cdot \|\psi\|_{C^1}\cdot  \delta^2
      \end{split} \label{Equ_Noncon_Pf vol diff bd, (I)}
     \end{align}
     where the third inequality follows from (\ref{Equ_Noncon_partial_y^perp}) and (\ref{Equ_Noncon_partial_r^perp}) by parametrizing $\Sigma \cap U = \graph_\cC u$, and the last inequality follows from Lemma \ref{Lem_Noncon_Maximal graphical region} with $\alpha=1$ and $\tau_3$ in place of $\tau_1$.

     To deal with (II), recall that $\spt(\psi') \subset \{\tau_1 \leq r\leq \tau_3\}$, and that by \eqref{eq:app.noncon.eps0} we know that $\Sigma$ is a $\kappa_\cC$-$C^2_*$ graph of some function $u$ over $\cC$ in $\spt(\psi')$. Therefore, if $F(x, z, p)$ denotes the corresponding Jacobian factor of the graph,
     \begin{align}
     \begin{split}
       \text{(II)} & = \int_{Q_{\tau_3}} \zeta(y)\left(r \psi'(r) - \sqrt{r^2+u^2}\psi'(\sqrt{r^2+u^2}) F(x, u, \nabla u) \right)\ d\|\cC\|_{\bar g} \\
       & \leq C(\cC)\cdot\|\zeta\|_{C^0}\cdot \|\psi\|_{C^2} \int_{Q_{\tau_3}\cap \Dom(u)} r(|\nabla u|^2 + r^{-2}u^2)\ d\|\cC\|_{\bar g} \\
       & \leq \frac{C(\cC)}{(\tau_2-\tau_1)^2}\cdot\|\zeta\|_{C^0}\cdot \|\psi\|_{C^2} \int_{Q_{\tau_2}} r^{-1}\dist_\cC^2\ d\|\Sigma\|_{\bar g} + C(\cC)\cdot \delta^2 \cdot \|\zeta\|_{C^0}\cdot \|\psi\|_{C^2}.         
     \end{split} \label{Equ_Noncon_Pf vol diff bd, (II)}
     \end{align}
     where the last inequality follows from Lemma \ref{Lem_Noncon_Maximal graphical region} with $\alpha =1$ and $\tau_3$ in place of $\tau_1$. Combining (\ref{Equ_Noncon_Pf vol diff bd, (III)}), (\ref{Equ_Noncon_Pf vol diff bd, (I)}) and (\ref{Equ_Noncon_Pf vol diff bd, (II)}) proves the Corollary.
   \end{proof}

   \begin{Lem} \label{Lem_Noncon_weak Noncon}
    There exists $C(\cC) > 0$ with the following significance. 
    
    If $\delta \leq \delta_1$ in \eqref{eq metric assumption} and $\Sigma$ is $g$-minimizing boundary in $B_2$ with $\supp \Sigma \cap B_1 \neq \emptyset$ and $\int_{\Sigma}\Sdist_\cC^2\, d\|\Sigma\|_{\bar g} \leq \eps_1$ and $1/2 \leq\tau_1<\tau_2\leq 1$, then
    \begin{align*}
      \int_{Q_{\tau_1}} r^{-2}\dist_\cC^2 \ d\|\Sigma\|_{\bar g} 
      & \leq \frac{C(\cC)}{(\tau_2-\tau_1)^5}\cdot \left( \int_{Q_{\tau_2}} r^{-1}\dist_\cC^2\ d\|\Sigma\|_{\bar g} + \delta \right) 
    \end{align*}
   \end{Lem}
   \begin{proof}
    Let $\tau_3:=(\tau_1+\tau_2)/2$. Choose
    \[ \zeta\in C^2_c(B^k_{\tau_3}), \; \zeta|_{B^k_{\tau_1}} = 1, \; 0 \leq \zeta \leq 1, \]
    \[ \psi\in C_c^2([0, \tau_3)), \; \psi|_{[0, \tau_1]} = 1, \; 0 \leq \psi \leq 1. \]
    Note that we can assume that
    \[ \|\psi\|_{C^l}+\|\zeta\|_{C^l}\leq 100 (\tau_2-\tau_1)^{-l}, \; l\in \{0, 1, 2\}. \]
    Also let $f(x, y):= \frT_{{S}}(x, y)$ as in \eqref{eq:FH definition} and $\varphi(r) := \psi(r)r^{2\gamma_1-2}$. Note that even though $\varphi$ diverges as $r \to 0$, it does so precisely with the decay bounds on $f^2$ as $r = |x| \to 0$ shown in Proposition \ref{Prop_HS_Est of F_H and |D F_H|}. By an elementary truncation argument, Lemma \ref{Lem_Noncon_1st variation}'s \eqref{Equ_Noncon_1st variation} is still applicable with these $\zeta, \varphi, f$. By combining with Proposition \ref{Prop_HS_Est of F_H and |D F_H|} we get 
    \begin{align*}
     & C(\cC_\circ)^{-1}\int_{Q_{\tau_2}} \zeta(y)\psi(r)r^{-2}\dist_\cC^2\cdot (n_\circ + 2\gamma_1-2 + |\nu_y|^2)\ d\|\Sigma\|_{\bar g} \\
     & \qquad \leq \int_{Q_{\tau_2}} \zeta(y)\psi(r)r^{2\gamma_1-2}\frT_{{S}}(x, y)^2 \cdot (n_\circ + 2\gamma_1-2 + |\nu_y|^2)\ d\|\Sigma\|_{\bar g} \\
     & \qquad \leq C(\cC_\circ) \cdot \int_{Q_{\tau_2}} \Big[ \Big(|\nabla_{\bar g} \zeta(y)|\psi(r) + \zeta(y)|\psi'(r)|\Big)r^{-1}\dist_\cC^2 \\
     & \qquad\qquad\qquad\qquad  + 2\zeta(y)\psi(r)\left(r^{-1}\dist_\cC \right) r^{\gamma_1}|\nabla_{\Sigma,\bar g} \frT_{{S}}|\ \Big] d\|\Sigma\|_{\bar g}\\
     & \qquad \qquad + C(\cC) \cdot \delta \cdot \int_{Q_{\tau_2}} (1+|\nabla \zeta| + |\psi'|) \ d\|\Sigma\|_{\bar g}.
    \end{align*}
    Since $\cC_\circ$ is strictly stable, we have $n_\circ+ 2\gamma_1-2>0$. Using AM-GM on the second term in the final expression, we thus obtain
    \begin{align*}
     & \int_{Q_{\tau_2}} \zeta(y)\psi(r)r^{-2}\dist_\cC^2\ d\|\Sigma\|_{\bar g} \\
     & \qquad \leq C(\cC_\circ) \cdot \int_{Q_{\tau_2}} \Big[ \Big(|\nabla \zeta(y)|\psi(r) + \zeta(y)|\psi'(r)|\Big)r^{-1}\dist_\cC^2 \\
     & \qquad \qquad \qquad \qquad + \zeta(y)\psi(r) r^{2\gamma_1}|\nabla_\Sigma \frT_{{S}}|^2 \Big] \ d\|\Sigma\|_{\bar g} \\
     & \qquad + C(\cC) \cdot \delta \cdot  \int_{Q_{\tau_2}} (1+|\nabla \zeta| + |\psi'|) \ d\|\Sigma\|_{\bar g}.
    \end{align*}
    Using \eqref{Equ_HS_|D F_H|<r^(-2gamma)(1+<v,X_H>)} and the definition of $\zeta,\psi$ we have
    \begin{align}
    \begin{split}
      & \int_{Q_{\tau_2}} \zeta(y)\psi(r)r^{-2}\dist_\cC^2\ d\|\Sigma\|_{\bar g} \\
      & \qquad \leq \frac{C(\cC_\circ)}{\tau_2-\tau_1} \cdot \int_{Q_{\tau_2}} r^{-1}\dist_\cC^2 \ d\|\Sigma\|_{\bar g} \\
      & \qquad \qquad + C(\cC_\circ) \cdot \int_{Q_{\tau_2}}\zeta(y)\psi(r) (1- \langle \nu_{\Sigma,\bar g}, \nu_{{S}} \rangle_{\bar g}) \ d\|\Sigma\|_{\bar g} \\
      & \qquad \qquad + \frac{C(\cC)}{\tau_2-\tau_1}\cdot \delta
    \end{split} \label{Equ_Noncon_Pf r^(-2)dist < r^(-1)dist_1st variation}
    \end{align}
    It remains to estimate the second to last term of \eqref{Equ_Noncon_Pf r^(-2)dist < r^(-1)dist_1st variation}. We may rewrite it as
        \begin{align}\label{eq estimating noncon rewrite}
    \begin{split}
     & \int_{Q_{\tau_2}}\zeta(y)\psi(r) (1- \langle \nu_\Sigma, \nu_{{S}} \rangle_{\bar g}) \ d\|\Sigma\|_{\bar g} \\
     & \qquad = \int_{Q_{\tau_2}}\zeta(y)\psi(r) \ d\|\Sigma\|_{\bar g} - \int_{Q_{\tau_2}}\zeta(y)\psi(r) \langle \nu_{\Sigma,\bar g}, \nu_{{S}} \rangle_{\bar g} \ d\|\Sigma\|_{\bar g} \\
     & \qquad = \int_{Q_{\tau_2}}\zeta(y)\psi(r) \ d\|\cC\|_{\bar g} - \int_{Q_{\tau_2}}\zeta(y)\psi(r) \langle \nu_{\Sigma,\bar g}, \nu_{{S}} \rangle_{\bar g} \ d\|\Sigma\|_{\bar g} \\
     & \qquad \qquad + \int_{Q_{\tau_2}}\zeta(y)\psi(r) \ (d\|\Sigma\|_{\bar g} - d\|\cC\|_{\bar g}) \\
     & \qquad = \int_{Q_{\tau_2}}\zeta(y)\psi(r) \underbrace{\langle \nu_{\cC,\bar g}, \nu_{{S}} \rangle_{\bar g}}_{=1} \ d\|\cC\|_{\bar g} - \int_{Q_{\tau_2}}\zeta(y)\psi(r) \langle \nu_{\Sigma,\bar g}, \nu_{{S}} \rangle_{\bar g} \ d\|\Sigma\|_{\bar g} \\
     & \qquad \qquad + \int_{Q_{\tau_2}}\zeta(y)\psi(r) \ (d\|\Sigma\|_{\bar g} - d\|\cC\|_{\bar g}).     
    \end{split}
    \end{align}
    We may estimate the second to last line in \eqref{eq estimating noncon rewrite} using Stokes's theorem. To that end, note first that, by \eqref{HS_div(X_H)=0}, we have 
    \begin{align*}
       \Div_{\RR^{n+1},\bar g}\left(\zeta(y)\psi(r)  \nu_{{S}}(x, y)\right)
      & = \zeta(y) \psi'(r) \langle \partial_r, \nu_{{S}}(x, y) \rangle_{\bar g} 
    \end{align*}
    and that this is supported in $Q_{\tau_3} \cap \{r \geq \tau_1\}$. Recall our graph map $\Phi$ from \eqref{eq:app.noncon.tub.nghd} and how, by \eqref{eq:app.noncon.eps0}, there exists a map $u$ such that
    \[ \Sigma \cap (Q_{\tau_2} \setminus \{r < 1/4\}) = \graph_\cC u \text{ with } \| u \|_{C^2_*} < \kappa_\cC. \]
    Let's denote $u^{\pm}:= \max\{0, \pm u\}$. Then:
        \begin{align}
    \begin{split}
      & \int_{Q_{\tau_2}} \zeta(y)\psi(r) \langle \nu_{\cC,\bar g}, \nu_{{S}} \rangle_{\bar g} \ d\|\cC\|_{\bar g} - \int_{Q_{\tau_2}} \zeta(y)\psi(r) \langle \nu_{\Sigma,\bar g}, \nu_{{S}} \rangle_{\bar g} \ d\|\Sigma\|_{\bar g} \\
      & \qquad \leq \int_{Q_{\tau_3} \cap \Dom(u)}\int_{-u^-}^{u^+} \zeta(y) \big|\psi'(\sqrt{r^2+t^2}) \langle \partial_r, \nu_{{S}} \rangle_{\bar g}|_{\Phi}\big|\ | \det(D\Phi)| \ dt\ d\|\cC\|_{\bar g}\\
        & \qquad \leq C(\cC) \cdot \|\psi\|_{C^1} \cdot \int_{Q_{\tau_3}\cap \Dom(u)}\int_{-u^-}^{u^+} r^{-1}|t|\ dt \ d\|\cC\|_{\bar g}  \\
        & \qquad \leq \frac{C(\cC)}{\tau_2-\tau_1} \cdot \int_{Q_{\tau_3} \cap \Dom(u)} r^{-1}u^2 \ d\|\cC\|_{\bar g} \\
        & \qquad \leq \frac{C(\cC)}{(\tau_2-\tau_1)^3} \cdot \int_{Q_{\tau_2}} r^{-1}\dist_\cC^2 \ d\|\Sigma\|_{\bar g} + \frac{C(\cC)\cdot \delta^2}{\tau_2-\tau_1} \,.
    \end{split} \label{Equ_Noncon_Est term using Stokes Thm}
    \end{align}
    In the second inequality we used \eqref{eq radial XH comparable dist}. In the last inequality, we used Lemma \ref{Lem_Noncon_Maximal graphical region}. 

The result follows by combining \eqref{Equ_Noncon_Pf r^(-2)dist < r^(-1)dist_1st variation}, \eqref{eq estimating noncon rewrite}, \eqref{Equ_Noncon_Est term using Stokes Thm}, and Corollary \ref{Cor_Noncon_Vol diff est btwn M and C}.
   \end{proof}

   \begin{proof}[Proof of Theorem \ref{Thm_Simon's L^2 Noncon}.]
     Assume $\delta \leq \delta_1$ in \eqref{eq metric assumption}. For each $\tau\in [1/2, 1]$, denote for simplicity, 
     \begin{align*}
       I(\tau):= \int_{Q_\tau} \dist_\cC^2\ d\|\Sigma\|_{\bar g} \ + \delta\,, & & 
       J(\tau):= \int_{Q_\tau} r^{-2}\dist_\cC^2\ d\|\Sigma\|_{\bar g} \ +\delta \,. 
     \end{align*}
     H\"older's inequality and Lemma \ref{Lem_Noncon_weak Noncon} guarantee, for every $1/2\leq \tau_1 < \tau_2\leq 1$,
     \begin{equation}\label{eq:moser-iteration-nonconc}
       J(\tau_1) \leq \frac{C(\cC)}{(\tau_2-\tau_1)^5}\sqrt{J(\tau_2)I(1)} \,.
     \end{equation}
     Let $t_q:= 1-2^{-q}$, $q\in \ZZ_{\geq 1}$. 
     Taking $\tau_1 =t_q,\tau_2=t_{q+1}$ in \eqref{eq:moser-iteration-nonconc}, we obtain \[
       \log J(t_q) \leq C(\cC)(1+q) + 2^{-1}\log J(t_{q+1}) + 2^{-1}\log I(1) \,.
     \]
     Multiplying both side by $2^{-q}$ and summing with respect to $q\in \ZZ_{\geq 1}$ we obtain \[
       \log J(1/2) = \log J(t_1) \leq C(\cC) + \log I(1) \,.
     \]
     Combining with Lemma \ref{Lem_Noncon_Maximal graphical region}, this proves Theorem \ref{Thm_Simon's L^2 Noncon}.
   \end{proof}

\bibliographystyle{alpha}
\bibliography{references}

\end{document}